\documentclass[twocolumn]{svjour3}          
\smartqed  
\usepackage{amsmath,amssymb,ifthen}
\usepackage{graphicx,psfrag,subfigure}
\usepackage{color}
\usepackage{moreverb}
\usepackage{mathrsfs}
\usepackage{hhline}
\usepackage[english]{babel}
\usepackage{tikz}
\usepackage{tikz-3dplot}
\usepackage[utf8]{inputenc}
\usepackage{amsfonts}
\usepackage{enumerate}
\usepackage[unicode,colorlinks=false,linkcolor=blue,pagebackref=false]{hyperref}
\usepackage{algorithm}
\usepackage[noend]{algpseudocode}
\usepackage{pst-node}
\usepackage{longtable}

\newcommand{\R}{{\mathbb R}}

\newcommand\N{{\mathbb N}}

\newcommand\diam{{\rm diam}}
\newcommand\const[1]{C_{\rm #1}}
\newcommand\norm[2]{\|#1\|_{#2}}
\newcommand\seminorm[2]{\vert #1\vert_{#2}}
\newcommand\set[2]{\big\{#1\,:\,#2\big\}}
\newcommand\dual[2]{\langle#1\,;\,#2\rangle}

\newcommand\supp{{\rm supp}}
\newcommand\dist{{\rm dist}}

\renewcommand\paragraph[1]{\noindent\textbf{\textit{#1.}}}

\newcommand\F{{\bf F}}
\newcommand\dpa{{\hat d}}
\newcommand\dph{d}
\renewcommand\AA{{\bf A}}
\newcommand\bb{{\bf b}}
\newcommand\xx{{\bf x}}
\newcommand\yy{{\bf y}}
\renewcommand\d{\textrm{d}}

\newcommand\MM{\mathcal M}

\newcommand\Q{\mathbb{Q}}
\newcommand\BB{{\mathcal B}}

\newcommand\II{{\mathcal I}}
\newcommand\JJ{J}

\newcommand\NN{{\mathcal N}}
\newcommand\HH{{\mathcal H}}
\newcommand\OO{{\mathcal O}}

\newcommand\QQ{{\mathcal{Q}}}
\newcommand\RR{\mathcal{R}} 
\renewcommand\SS{\mathcal{S}} 
\newcommand\TT{{\mathcal T}}

\newcommand\VV{{\mathcal V}}

\newcommand\q[1]{q_{\rm #1}}
\newcommand\ro[1]{\rho_{\rm #1}}
\renewcommand\k{k}

\newcommand\trunc{{\rm trunc}}
\newcommand\Trunc{{\rm Trunc}}
\newcommand\refine{{\tt refine}}
\newcommand\osc{{\rm osc}}
\renewcommand\div{{\rm div}}
\newcommand\normal{{\boldsymbol{\nu}}}
\renewcommand\dist{{\rm d\!l}}

\newcommand\coarse{{}}
\newcommand\coarsecomma{{}}
\newcommand\fine{+}
\newcommand\meshidx{\times}
\algnewcommand{\LineComment}[1]{\State \(\triangleright\) #1}

\def\M3AS{Math.\ Models\ Methods\ Appl.\ Sci.\ }

\newcommand{\spmh}[1]{\hat{\mathbb{S}}_{\bf p}({\bf \kv}^{#1})}
\newcommand{\hbbasis}{\mathcal{H}}
\newcommand{\thbbasis}{\mathcal{T}}
\newcommand{\hmesh}{\mathcal{Q}}
\newcommand{\myspan}{\mbox{span}}

\newcommand{\projH}[1]{\hat J_{#1,\hat{\hmesh}}^{\,\rm H}} 
\newcommand{\projHfine}[1]{\hat J_{#1,{\hat\hmesh_\fine}}^{\,\rm H}} %

\newcommand{\sextTHB}[1]{{S}_{\rm ext}^*(#1)}
\newcommand{\extsuppTHB}{\hat\omega_{\bf i, p}^\ell}
\newcommand{\mot}{\mathop{\mathrm{mot}}}

\let\hat\widehat
\let\tilde\widetilde

\DeclareFontFamily{U}{mathx}{\hyphenchar\font45}
\DeclareFontShape{U}{mathx}{m}{n}{
      <5> <6> <7> <8> <9> <10>
      <10.95> <12> <14.4> <17.28> <20.74> <24.88>
      mathx10
      }{}
\DeclareSymbolFont{mathx}{U}{mathx}{m}{n}
\DeclareMathAccent{\widecheck}{0}{mathx}{"71}
\let\check\widecheck

\spnewtheorem{assumption}{Assumption}{\bf}{\it}

\newcommand{\knot}{t}
\newcommand{\kv}{T}
\newcommand{\param}{t}
\newcommand{\bparam}{\mathbf{\param}}

\newcommand{\spbasis}{\hat{\mathcal{B}}}
\newcommand{\spu}[2]{\hat{\mathbb{S}}_{#1}(#2)}
\newcommand{\spm}{\hat{\mathbb{S}}_{\bf p}(\mathbf{\kv})}
\newcommand{\proju}[2]{\hat{J}_{#1,#2}}
\newcommand{\projm}{\hat{J}_{\mathbf{p},\mathbf{\kv}}}
\newcommand{\brkset}{Z}
\newcommand{\brkpnt}{z}
\newcommand{\mesh}{\mathcal{Q}}
\newcommand{\elem}{Q}

\newcommand{\Omegap}{\hat \Omega} 
\newcommand{\sext}[1]{S_{\rm ext}(#1)}
\newcommand{\weight}{w} 
\newcommand{\basisfun}{B} 
\newcommand{\appl}{g} 
\newcommand{\mypatch}{\pi} 
\newcommand{\myparam}{{\bf F}} 

\newcommand{\indices}{\mathcal{I}}
\newcommand{\Tmesh}{\check{\mathcal{Q}}}  
\newcommand{\Tmeshp}{\hat{\mathcal{Q}}} 
\newcommand{\spT}{\hat{\mathbb{S}}^{\rm T}_{\mathbf{p}}(\Tmesh,\mathbf{\kv}^0)} 
\newcommand{\spTargs}[2]{\hat{\mathbb{S}}^{\rm T}_{\mathbf{p}}(#1,#2)}
\newcommand{\Omindex}{\check{\Omega}_{\mathrm{ind}}} 
\newcommand{\Omip}{\check{\Omega}_{\mathrm{ip}}} 
\newcommand{\elemi}{{\check{Q}}} 
\newcommand{\elemp}{\hat{Q}} 
\newcommand{\AR}{\check{\Omega}_{\mathrm{act}}} 
\newcommand{\bisect}{{\tt bisect}}
\newcommand{\dirb}[1]{\mathrm{dir}(#1)} 
\newcommand{\levelT}[1]{\mathrm{lev}(#1)}
\newcommand{\sk}{\mathtt{skel}} 
\newcommand{\nodes}{\mathcal{A}} 
\newcommand{\anchor}{\mathbf{z}} 
\newcommand{\GIV}{\mathcal{I}^{\mathrm{gl}}} 
\newcommand{\LIV}{\mathcal{I}^{\mathrm{loc}}} 
\newcommand{\projT}[1]{\hat J^{\,\rm T}_{#1,\Tmesh}} 
\newcommand{\neig}{\mathcal{N}} 
\newcommand{\markedT}{\check{\mathcal{M}}} 
\newcommand{\markedTp}{\hat{\mathcal{M}}} 
\newcommand{\indtpar}[1]{\mathtt{param}(#1)} 
\newcommand{\partind}[1]{\mathtt{index}(#1)} 
\newcommand{\admp}{\hat{\mathbb{Q}}} 
\newcommand{\bext}[2]{\mathtt{prol}(#1)} 

\newcommand\refineind{{\tt refine\_index}}

\newcommand*\patchAmsMathEnvironmentForLineno[1]{%
  \expandafter\let\csname old#1\expandafter\endcsname\csname #1\endcsname
  \expandafter\let\csname oldend#1\expandafter\endcsname\csname end#1\endcsname
  \renewenvironment{#1}%
     {\linenomath\csname old#1\endcsname}%
     {\csname oldend#1\endcsname\endlinenomath}}%
\newcommand*\patchBothAmsMathEnvironmentsForLineno[1]{%
  \patchAmsMathEnvironmentForLineno{#1}%
  \patchAmsMathEnvironmentForLineno{#1*}}%
\AtBeginDocument{%
\patchBothAmsMathEnvironmentsForLineno{equation}%
\patchBothAmsMathEnvironmentsForLineno{align}%
\patchBothAmsMathEnvironmentsForLineno{flalign}%
\patchBothAmsMathEnvironmentsForLineno{alignat}%
\patchBothAmsMathEnvironmentsForLineno{gather}%
\patchBothAmsMathEnvironmentsForLineno{multline}%
}
\usepackage[mathlines]{lineno}

 \journalname{Noname}
\begin{document}

\title{Mathematical Foundations of\\ Adaptive Isogeometric Analysis \thanks{
The authors AB and RV were partially supported by the European Research Council (ERC) through the Advanced Grant CHANGE n.~694515. RV additionally acknowledges support of the Swiss National Science Foundation (SNF) through grant 200021\_188589. 
The authors GG and DP acknowledge support through the Austrian Science Fund (FWF) under grant P29096, grant W1245-N25 and grant SFB F65. 
GG additionally acknowledges support through the Austrian Science Fund (FWF) under grant J4379-N. 
CG  was partially supported by Istituto Nazionale di Alta Matematica (INdAM) through Gruppo Nazionale per il Calcolo Scientifico (GNCS) and Finanziamenti Premiali SUNRISE. CG and RV are members of the INdAM Research group GNCS. 
}
}


\author{Annalisa Buffa      \and
        Gregor Gantner		\and
        Carlotta Giannelli	\and
        Dirk Praetorius		\and
        Rafael V\'azquez	
}


\institute{
Annalisa Buffa \at
\'Ecole polytechnique f\'ed\'erale de Lausanne, Institute of Mathematics, 1015 Lausanne, Switzerland\\
Istituto di Matematica Applicata e Tecnologie Informatiche ``E. Magenes'' del CNR, Pavia, Italy \\
	\email{annalisa.buffa@epfl.ch}
	\and
Gregor Gantner (corresponding author) \at
University of Amsterdam, Korteweg--de Vries Institute for Mathematics, 
1090 GE Amsterdam, the Netherlands\\
              \email{g.gantner@uva.nl}           
           \and
Carlotta Giannelli \at
Universit\`a degli Studi di Firenze, Dipartimento di Matematica e Informatica ''U. Dini``, 
50134 Florence, Italy\\
	\email{carlotta.giannelli@unifi.it}
	\and
Dirk Praetorius \at
TU Wien, Institute of Analysis and Scientific Computing, 
1040 Vienna, Austria \\
	\email{dirk.praetorius@asc.tuwien.ac.at} 
	\and
Rafael V\'azquez \at
\'Ecole polytechnique f\'ed\'erale de Lausanne, Institute of Mathematics, 1015 Lausanne, Switzerland\\
Istituto di Matematica Applicata e Tecnologie Informatiche ``E. Magenes'' del CNR, Pavia, Italy \\
	\email{rafael.vazquez@epfl.ch}
	\and
}

\date{Received: date / Accepted: date}

\maketitle

\begin{abstract}
This paper reviews the state of the art and discusses recent developments in the field of adaptive isogeometric analysis, with special focus on the mathematical theory.
This includes an overview of available spline technologies for the local resolution of possible singularities as well as the state-of-the-art formulation of convergence and quasi-optimality of adaptive algorithms for both the finite element method (FEM) and the boundary element method (BEM) in the frame of isogeometric analysis (IGA).
\keywords{Isogeometric analysis \and hierarchical splines \and T-splines \and adaptivity \and finite element method \and boundary element method}

\subclass{41A15 \and 65D07\and 65N12 \and 65N30 \and 65N38 \and 65N50 \and 65Y20}
\end{abstract}

\setcounter{tocdepth}{3}
\tableofcontents

\newpage


\section{Introduction}

\subsection{Isogeometric analysis}

Isogeometric analysis (IGA) was introduced in 2005 in the seminal work \cite{hughes2005} and since then has been a very successful area of research including mathematical discoveries, computational mechanics challenges as well as a rather unique joint effort to tackle problems that fall outside one single research community.

By using the same building blocks employed in standard Computer-Aided Design (CAD), namely B-splines, Non-Uniform Rational B-splines  (NURBS) and variants thereof, the final goal of IGA is to provide an end-to-end methodology that unifies geometrical design with the analysis of partial differential equations (PDEs) for computational engineering. While this is still a widely open issue, in the last decade an extensive amount of research has been dedicated to IGA in various different fields.
We refer, e.g.,  to the special issue \cite{SpecialIssueCMAME}  for a review of the most prominent works published in recent years. B-spline based formulations are now built on solid mathematical foundations (see, e.g. \cite{bbsv14,hughes2005}) and have demonstrated their capabilities in many different areas of engineering. Moreover, since B-splines are nothing but (possibly smooth) piecewise polynomials of a given degree, methods based on them (including IGA) are potentially high-order. 


The starting point of IGA is a description of the computational geometry as a collection of (possibly trim\-med) patches. A patch is a geometric entity characterized by a spline (or more generally by a non uniform rational spline) parametrization. IGA stands for the class of methods which use spline discretization techniques over such  geometric descriptions. Thus, it includes, and it is not restricted to, second or higher order  PDEs defined in $d$-dimensional  domains \cite{IGA-book}, PDEs defined on manifolds such as the ones describing shells \cite{MR3345251} or membranes \cite{MR3896166} and also boundary integral equations \cite{pgkbf09}. IGA methods and their applications are now a rather large research area in computational mechanics and numerical analysis so that we refrain from trying to list all relevant contributions to the field. 

Indeed, this paradigm has raised significant mathematical challenges. Some of them have only been partially addressed by the community until now, e.g., the construction of $C^1$ basis functions with optimal approximation properties \cite{NgPe16,KaSaTa19,WZTSLMEH18},  optimal reparametri\-zation for trimmed surfaces \cite{Marussig-review,HINZ201848,HINZ2020112740,LIU2015108}   and the construction and manipulation of spline volumes \cite{MASSARWI201636,AnBuMa20,ZHANG2012185,PCT2020}.  Instead, other research topics have reached a more advanced maturity, e.g., the approximation estimates of splines of arbitrary degrees \cite{MR4080513} or the construction of locally refined splines and their use within an adaptive paradigm, which is the topic of this review paper. The literature on the subject is today very wide and covers several different (integro-) differential problems. This review aims at describing, with a careful mathematical perspective, some of the very many approaches existing in the literature, with a different level of details.

\subsection{Adaptivity}

As soon as the (given) data or the (unknown) solution $u$ of a PDE have singularities, the possible high-order 
convergence rate of isogeometric methods is significantly reduced down to rates which could also be achieved by low-order methods. However, at least for standard finite element methods (FEM), it is known that better rates --and usually even optimal algebraic convergence rates-- can 
be regained by an appropriate local mesh grading of the underlying mesh 
towards these singularities. 

If the singularities and the required local mesh grading are \textsl{a~priori} unknown, the local mesh adaptation
can be automated by so-called adaptive algorithms. Usually, these adaptive algorithms rely on \textsl{a~posteriori}
error estimators which provide computable (lower and upper) bounds on the error of an already computed approximation
$U \approx u$. Localizing these bounds to related elements of the underlying mesh (resp.\ specific isogeometric basis functions),
one can extract the necessary information of where to locally refine the mesh (resp.\ where to add additional basis
functions).

\subsubsection{Modules of adaptive loop}

Starting from a given initial mesh $\QQ_0$, 
adaptive algorithms aim to improve the accuracy of a discrete solution by iterating the so-called adaptive loop
\begin{align}\label{eq:intro:semr}
 \boxed{\texttt{solve}}
 \longrightarrow
 \boxed{\texttt{estimate}}
 \longrightarrow
 \boxed{\texttt{mark}}
\longrightarrow
 \boxed{\texttt{refine}}
\end{align}
The module $\texttt{solve}$ computes a discrete solution $U_k \approx u$ (indexed by some step counter $k \in \N_0$) related to the current mesh $\QQ_k$. 

The module $\texttt{estimate}$ computes for all elements $Q \in \QQ_k$ the local contributions $\eta_k(Q)$ of some \textsl{a~posteriori} error estimator $\eta_k := \big( \sum_{Q \in \QQ_k} \eta_k(Q)^2 \big)^{1/2}$ which, at least heuristically, provides a measure of the discretization error $\| u - U_k \|$. The so-called \emph{refinement indicators} $\eta_k(Q)$ depend usually on the computed discrete solution $U_k$ and the known problem or mesh data, but are independent of the unknown solution $u$. 

Having computed all refinement indicators, the module $\texttt{mark}$ selects elements $Q \in \QQ_k$ for refinement. 

Finally, the module $\texttt{refine}$ adapts the underlying mesh and generates a new mesh $\QQ_{k+1}$ by refinement of, at least, all marked elements. We stress that usually, besides the marked elements, also non-marked elements are refined to preserve structural properties of the mesh (e.g., avoidance of certain hanging nodes, preservation of local mesh grading, etc.).

\subsubsection{Analysis of adaptive algorithms}\label{sec:adaptive analysis}

Empirically, it has already been observed in the seminal papers on \textsl{a~posteriori} error estimation~\cite{MR483395,MR529976,MR615532,MR745088,MR880421} that adaptive algorithms regain the optimal convergence rate, understood as the decay of the error with respect to the number of degrees of freedom.
However, since adaptive algorithms usually do not guarantee that all elements are refined (so that the local mesh size becomes infinitesimally fine everywhere), one cannot rely on \textsl{a~priori} error estimates to ensure that the error tends to zero $\| u - U_k \| \to 0$ as the adaptive step counter $k \to \infty$ increases.

A first convergence result for adaptive finite elements for a 1D boundary value problem already dates back to~\cite{MR745088}. However, it took more than a decade until~\cite{MR1393904,MR1770058} proved plain convergence for the lowest-order FEM for the Poisson model problem in 2D. Generalizing those arguments, the works \cite{MR2413035,MR2832786} proved plain convergence $\| u - U_k \| \to 0$ for a large class of PDE model problems. 

Moreover, it took almost two decades to mathematically understand optimal convergence in the sense that $\| u - U_k \| = \OO((\#\QQ_k)^{-s})$, where $\#\QQ_k$ is proportional to the numbers of the degrees of freedom and the algebraic convergence rate $s > 0$ is as large as possible. The seminal work~\cite{bdd04} proves convergence with optimal algebraic rates for the 2D Poisson problem, discretized by lowest-order elements. While the analysis of~\cite{bdd04} requires an additional mesh coarsening step to prove optimal rates, this has been proved unnecessary in the work~\cite{stevenson07}, which was the first work that proved optimal convergence rates for the standard adaptive loop~\eqref{eq:intro:semr}. 
We note, however, that these developments originated from groundbreaking results on adaptive wavelet discretizations~\cite{MR1803124,MR1907380,MR2035007,MR2299773}, which analyzed optimality for a variety of problems in terms of the best $N$-term approximation.

The seminal ideas of~\cite{stevenson07} have then been extended to finite element methods for symmetric second-order linear elliptic PDEs in~\cite{ckns08}, general second-order linear elliptic PDEs in the setting of the Lax--Milgram lemma~\cite{cn12,ffp14}, and even for well-posed indefinite PDEs like the Helmholtz problem~\cite{bhp17}, see also~\cite{nv11} for an easy introduction to the topic focussing on the Poisson model problem. 
For standard boundary element methods (BEMs) based on piecewise polynomials, \cite{fkmp13,gantumur13,ffkmp14,ffkmp15,affkmp17} obtained similar results. 

All these developments led to the identification of a unified framework of optimal adaptivity~\cite{cfpp14}, which consists of four \emph{axioms of adaptivity} that guarantee convergence of the adaptive loop~\eqref{eq:intro:semr} with optimal algebraic rates.

While all mentioned works consider optimal adaptivity with respect to the number of the degrees of freedom, in practice, optimal adaptivity with respect to the computational time is of more importance.  This question is mathematically well-understood for wavelet discretizations (see, e.g.,~\cite{MR1803124,MR1907380,MR2035007,MR2299773}), but the numerical analysis for non-wavelet FEM (or BEM) discretizations still has to be developed. First results, where the adaptive algorithm does not only steer the mesh-refinement but also the iterative and inexact solution, include~\cite{fhps19} for standard BEM as well as~\cite{ghps20} for an abstract framework based on contractive iterative solvers (like optimally preconditioned CG solvers).

\subsection{Adaptive isogeometric analysis}


Although adaptive algorithms of type~\eqref{eq:intro:semr} have a long history in the finite element theory, their application in 3D is often very complex and some developments do not provide real computational tools. Reasons are of practical type, e.g., splitting a tetrahedral mesh is not an easy task and adaptive approaches may generate several unwanted elements 
when the refinement of the mesh fails to be aligned with the steep gradient of the solution. Sometimes, in the case of three dimensional finite elements, the generation of a tetrahedral mesh following a certain metric is preferred over the adaptive loop~\eqref{eq:intro:semr}.


The situation is different in IGA. The mesh is not as flexible as a tetrahedral mesh, but it is a locally structured and globally unstructured hexahedral mesh. Local refinement and the use of locally refined splines is a viable option to keep the structure of splines (including the isoparametric paradigm) while adapting the mesh to the structure of the solution. 

Once locally refined splines are used, the development of adaptive algorithms is not a tremendous overhead on a computational code, and can immensely improve the accuracy of the solution. 
Indeed, singularities of the PDE solution might significantly spoil the possible high-order convergence rate of isogeometric methods. 
Thus, we believe that the use of adaptive algorithms in IGA  holds the promise of becoming ubiquitous in isogeometric codes.

\subsubsection{Splines suited for adaptivity}
The tensor-product structure of B-splines and NURBS is essentially non-local, because the bisection of one single element extends the refinement through the whole domain. Adaptive IGA methods must be based on suitable extensions of B-splines that break their tensor-product structure and allow local refinement. Such extensions were already available in CAD for the design of small details in large objects, and they were applied in IGA in the last years.

Among this kind of splines with local refinement properties, we mention the following: hierarchical B-splines (HB-splines), introduced in \cite{forsey88} and first used in IGA in \cite{vgjs11}, which realize local refinement by using splines of different levels, from coarsest to finest; truncated hierarchical B-splines (THB-splines) \cite{gjs12}, which span the same space as hierarchical splines in \cite{vgjs11} with a more local basis; T-splines, for which basis functions are directly defined on a mesh with T-junctions (or hanging nodes), introduced for CAD in \cite{sederberg2003,sederberg2004}, and applied first to IGA in \cite{bazilevs2010,doerfel2010}; locally refined-splines (LR-splines), first defined in \cite{dlp13} and almost immediately applied to IGA \cite{jkd14}, which are similar to T-splines with the difference that the functions are defined on a different mesh that contains information about the continuity of the splines across edges or faces; finally, polynomial splines over hierarchical T-meshes (PHT-splines), introduced in \cite{DeChFe06} and first used in IGA in \cite{WaXuDeCh11,NgNgBoRa11}, which are also defined on a mesh with hanging nodes, but which have lower continuity on the interfaces between elements than the previous variants. 

\subsubsection{Available convergence results}

As far as convergence of adaptive IGA methods is concerned, the first result goes back to~\cite{bg16} which considers IGAFEM with (truncated) hierarchical B-splines for the Poisson model problem. Optimal algebraic convergence rates have been proved independently in~\cite{bg17,ghp17}. In particular, the work~\cite{ghp17} provides a general framework for finite element discretizations guaranteeing that the residual error estimator for general second-order linear elliptic PDEs satisfies the axioms of adaptivity from~\cite{cfpp14}. Based on this framework, the recent work~\cite{gp18} also proves convergence of adaptive IGAFEM with T-splines using the refinement strategy from~\cite{mp15,morgenstern16}. 

Optimal adaptive IGABEM in 2D has been analyzed in~\cite{fghp17} for weakly-singular integral equations and in~\cite{gps19} for hyper-singular integral equations, where these works additionally consider adaptive smoothness control to locally reduce the differentiability of the discrete spline space. 
First results on optimal adaptive IGABEM in 3D are found in~\cite{gantner17}. In the spirit of~\cite{ghp17}, the recent work~\cite{gp20} provides an abstract framework for boundary element discretizations guaranteeing that the residual error estimator for weakly-singular integral equations satisfies the axioms of adaptivity from~\cite{cfpp14}. The application to IGABEM with (truncated) hierarchical B-splines is proved in~\cite{gp20+}, and the application to T-splines will be addressed in the present manuscript. 

The main goal of this work is to provide a summary of all these convergence results and the underlying adaptive spline methodologies, i.e., hierarchical splines and T-splines. 
We will also provide some further information and references on other adaptive spline methodologies in Section~\ref{subsec:others}.



\color{black}
\subsection{Outline and contributions}
As a brief outline, Section~\ref{sec:splines} and~\ref{sec:model} present the basics on tensor-product B-splines and their application in IGA, respectively. 
In Section~\ref{sec:adaptive-splines}, we present hierarchical splines and T-splines along with corresponding refinement algorithms and with special focus on their mathematical properties. Section~\ref{sec:abstract} gives the abstract framework and the properties that guarantee optimal convergence of adaptive algorithms. This framework is applied in Section~\ref{sec:adaptive igafem} to IGAFEM and in Section~\ref{sec:igabem} to IGABEM, considering both hierarchical splines and T-splines for either method.

More in detail, Section~\ref{sec:splines} recalls the definition of non-uniform (rational) multivariate splines along with well-known properties and quasi-interpolation operators. 
It starts with univariate splines and their basis of B-splines in Section~\ref{sec:splines-univariate}.
Via tensor-products, multivariate (B-)splines are introduced in Section~\ref{sec:splines-tensor}. 
In Section~\ref{sec:NURBS}, we briefly mention that the quasi-interpolation results immediately extend to NURBS. 
Then, in Section~\ref{sec:parametrization}, we explain how geometries of arbitrary dimension can be parametrized using these NURBS functions. 

In the following Section~\ref{sec:model}, we introduce the considered model problems along with the required setting and  present standard isogeometric discretizations with multivariate splines on uniform tensor meshes as in Section~\ref{sec:splines}. 
Section~\ref{sec:parametrization_assumptions} considers NURBS parametrizations of the physical domain, which can be either a single-patch or multi-patch geometry. 
In the case of FEM, the physical domain is a Lipschitz domain, while for BEM, it is the boundary thereof.
Section~\ref{sec:IGAFEM-intro} introduces the considered PDEs in case of IGAFEM and introduces standard isogeometric ansatz functions.  
Although adaptivity will only be considered in a later section, the used {\sl a posteriori} error estimator is already formulated 
on uniform tensor meshes. 
Section~\ref{sec:IGABEM-intro} is structured analogously for IGABEM for weakly-singular integral equations arising from Dirichlet boundary value problems: We first introduce the boundary integral equation of the model problem and its discretization with standard IGA methods, and then we formulate the used error estimator.

Splines on adaptive meshes are discussed in Section~\ref{sec:adaptive-splines}. 
We mainly focus on hierarchical splines (Section~\ref{subsec:hb}) and T-splines (Section~\ref{subsec:tsplines}), and we also provide, without entering into details, several comments and references on other constructions such as LR-splines in Section~\ref{subsec:others}. 
For hierarchical splines, we define in Section~\ref{subsec:hb} two well-known bases of the same space, namely hierarchical B-splines 
and truncated hierarchical B-splines. 
We further recall refinement strategies and resulting admissible hierarchical meshes, 
and we present results on the hierarchical quasi-interpolation operator from~\cite{sm16}. 
We also mention the construction of simplified hierarchical splines from~\cite{bgarau16_1}, following a refinement strategy that marks basis functions instead of elements. 
In Section~\ref{subsec:tsplines}, we recall T-splines on T-meshes which are defined as span of T-spline blending functions.
The latter are in general not linearly independent, and therefore we also consider two- and three-dimensional dual-compatible T-splines, 
which indeed provide a basis.  
We consider a refinement strategy generating admissible meshes that yield dual-compatible T-splines,
and we also present a new result 
stating that elements in an admissible T-mesh consist of at most two B\'ezier elements. 
Finally, 
we mention several extensions of T-splines. 

Section~\ref{sec:abstract} gives an abstract formulation of an adaptive mesh-refining algorithm and states and discusses the axioms of adaptivity (Section~\ref{sec:axioms}) which guarantee convergence of adaptive mesh refinement strategies at optimal algebraic convergence rates. 
Restricted to weighted-residual error estimators these axioms are simplified and adapted in the frame of IGA to FEM (Section~\ref{sec:afem}) and BEM (Section~\ref{sec:abem}), which translates into a collection of required mesh, refinement, and space properties.

In Section~\ref{sec:adaptive igafem}, we finally consider adaptive IGAFEM using the adaptive splines and refinement strategies of Section~\ref{sec:adaptive-splines}. 
Section~\ref{sec:H-igafem} deals with hierarchical splines, and Section~\ref{sec:T-igafem} deals with T-splines. 
In each case, we provide a basis of the corresponding ansatz space for homogeneous Dirichlet problems. 
Moreover, we state that both approaches fit into the abstract framework of Section~\ref{sec:afem}, where the employed weighted-residual estimator is reliable and efficient, i.e., equivalent to the total error (consisting of energy error + data oscillations). 
These results are mostly cited, but especially for hierarchical splines on THB-admissible meshes, some minor new arguments are required. 
Further, we make the new observation that the optimal convergence rate of the total error for hierarchical splines does not depend on the considered admissibility class of the meshes. 
Indeed, it coincides with the optimal rate for arbitrary hierarchical meshes without any grading assumption. 
For hierarchical splines, all results can be relatively easily transferred to the multi-patch case, 
which in particular requires an adaptation of the single-patch refinement algorithms given in Section~\ref{sec:adaptive-splines}. 
We conclude  Section~\ref{sec:H-igafem} with three typical numerical examples 
for adaptive IGAFEM with hierarchical splines. 
Especially, we discuss the choice of either HB-splines or THB-splines and give some explanation on the expected optimal convergence rate. 

Section~\ref{sec:igabem} considers adaptive IGABEM and is similarly structured as Section~\ref{sec:adaptive igafem}.
Again, we state that hierarchical splines (Section~\ref{sec:H-igabem}) and T-splines (Section~\ref{sec:T-igabem}) fit into the abstract framework of Section~\ref{sec:abem}. 
While the implied optimal convergence of the corresponding adaptive IGABEM is known for hierarchical splines on HB-admissible meshes of class 2 in the literature, 
it is completely new for hierarchical splines on other HB-admissible meshes of different class or THB-admissible meshes as well as for T-splines on admissible T-meshes.  
The proof builds on the already known case and uses some arguments of Section~\ref{sec:adaptive igafem}. 
Again, we present two numerical experiments in the case of hierarchical splines. 
Finally, Section~\ref{sec:elementary_bem} presents recent results on an adaptive IGABEM in 2D which uses both $h$-refinement and multiplicity increase to steer the local smoothness of the employed standard splines. 
Although this approach does not fit exactly into the framework of Section~\ref{sec:abem}, similar techniques can be used to prove again optimal convergence rates for the weighted-residual error estimator. 
We conclude the section with a  numerical example. 

Finally, Section~\ref{sec:conclusion} provides our conclusion. 
There, we also discuss several open questions in the context of adaptive IGAFEM as well as IGABEM.

\subsection{General notation}

Throughout the paper and without any ambiguity, $|\cdot|$ denotes the absolute value of scalars, the Euclidean norm of vectors, or the measure of a set.
We write $A\lesssim B$ to abbreviate $A\le cB$ with some generic constant $c>0$, which is clear from the context.
Moreover, $A\simeq B$ abbreviates $A\lesssim B\lesssim A$. 
Throughout, we use indices for non-generic meshes, e.g., $\QQ_+$ typically denotes a refinement of some given mesh $\QQ$ and $\QQ_k$ denotes the $k$-th mesh generated by the adaptive algorithm. 
Corresponding quantities have the same index, e.g., $\eta_+$ and $\eta_k$ denote the error estimators corresponding to the meshes $\QQ_+$ and $\QQ_k$, respectively. 
We often use  \,$\widehat{\cdot}$\, for notation on the parametric domain.
We employ standard notation for Sobolev spaces, e.g., $H^1(\Omega)$ denotes the space of square-integrable functions on some domain $\Omega$ whose weak derivative is square-integrable as well. 
In Section~\ref{sec:sobolev}, we briefly recall Sobolev spaces on the boundary. 
A list of acronyms is given in the following Section~\ref{sec:symbols}. 
The most important symbols are listed in the following Section~\ref{sec:symbols}.

%

\onecolumn

\subsubsection{List of symbols}\label{sec:symbols}
\begin{longtable}{lll}
\hline
Name & Description &First appearance \\
\hline
${\bf A}$ & diffusion matrix & Section~\ref{sec:FEM problem}\\
$\mathcal{A}_{\bf p}(\check\QQ,{\bf T}^0)$ & anchors in T-mesh & Section~\ref{sec:T-splines defined} \\
${\bf b}$ & drift vector & Section~\ref{sec:FEM problem}\\
$\hat B_{i,p}$ & univariate B-spline & Section~\ref{sec:univariate-properties}\\
$\hat B[T_{i,p}]$ & (local) univariate B-spline & Section~\ref{sec:univariate-properties}\\
$\hat B_{{\bf i},{\bf p}}$ & multivariate B-spline & Section~\ref{sec:univariate-properties}\\
$\hat B_{{\bf i},{\bf p}}^\ell$  & hierarchical B-spline & Section~\ref{subsec:def hb}\\
$\hat B_{{\bf z},{\bf p}}$ & T-spline blending function & Section~\ref{sec:T-splines defined}\\
$\hat{\mathcal{B}}^\ell$ & uniformly refined multivariate B-splines & Section~\ref{subsec:def hb}\\
$\hat {\mathcal{B}}_p(T)$ & univariate B-splines & Section~\ref{sec:univariate-properties}\\
$\hat {{\mathcal{B}}}_{\bf p}({\bf T})$ & multivariate B-splines & Section~\ref{sec:univariate-properties}\\
$c$ & reaction coefficient & Section~\ref{sec:FEM problem}\\
$\const{apx}(s)$ & approximation constant for estimator & Section~\ref{sec:abstract main}\\
$\const{apx}^{\rm tot}(s)$  & approximation constant for total error & Section~\ref{sec:oscillations}\\
$\widehat d$ & dimension of parametric domain & Section~\ref{sec:splines-tensor definition}\\
$d$ & dimension of physical domain  & Section~\ref{sec:parametrization}\\
$\dist$ & perturbation term of meshes & Section~\ref{sec:the axioms}\\
$\mathscr{D}_\nu$ & conormal derivative & Section~\ref{sec:estimator fem}\\
${\bf F}$ & NURBS parametrization  & Section~\ref{sec:parametrization}\\
$\F_m$ & NURBS parametrization of patch & Section~\ref{sec:multi-patch}\\
$G$ & fundamental solution of PDE & Section~\ref{sec:model problem bem}\\
$h$ & volume/boundary mesh-size function & Section~\ref{sec:abstract setting fem}/\ref{sec:abstract setting bem}\\
$\widehat h$ & element size in parametric domain & Section~\ref{sec:splines-tensor definition}\\ 
$h_Q$ & element size & Section~\ref{sec:parametrization_assumptions}\\
$\hat{\mathcal{H}}_{\bf p}(\hat{\mathcal{Q}},{\bf T}^0)$ & hierarchical B-splines & Section~\ref{subsec:def hb}\\
$\hat J_{p,T}$ & quasi-interpolant for univariate splines & Section~\ref{sec:qi-1d}\\
$\hat J_{{\bf p},{\bf T}}$ & quasi-interpolant for multivariate splines & Section~\ref{sec:qi-2d}\\
$\hat J_{\bf p,\hat \QQ}^{\,\rm H}$ & quasi-interpolant for hierarchical splines & Section~\ref{sec:hierarchical interpolation}\\
$\hat J_{{\bf p},\check\QQ}^{\,\rm T}$ & quasi-interpolant for T-splines & Section~\ref{sec:dual-compatible}\\
$\mathscr{K}$ & double-layer operator & Section~\ref{sec:model problem bem}\\
${\rm lev}$ & level of elements in hierarchical/T-mesh & Section~\ref{subsec:def hb}/\ref{sec:T-meshes defined}\\
${\rm mot}$ & mother B-spline of truncated hierarchical B-spline & Section~\ref{subsec:thb}\\
$\NN(\check Q)$ & neighbors for T-splines in index domain & Section~\ref{sec:dual-compatible}\\
$\NN(\hat Q)$ & neighbors for T-splines in parametric domain & Section~\ref{sec:dual-compatible}\\
$\mathcal{N}(Q)$ & neighbors for volume/boundary multi-patches & Section~\ref{sec:H-multipatches}/\ref{sec:H-igabem}\\
$\mathcal{N}_{\mathcal{H}}(\hat Q,\mu)$ & neighbors for HB-splines & Section~\ref{sec:hierarchical refine}\\
$\mathcal{N}_{\mathcal{T}}(\hat Q,\mu)$ & neighbors for THB-splines & Section~\ref{sec:hierarchical refine}\\
$\osc$ & oscillations & Section~\ref{sec:oscillations}\\
$p$ & polynomial degree & Section~\ref{sec:univariate-properties}\\
${\bf p}$ & polynomial degree vector & Section~\ref{sec:splines-tensor definition}\\
${\bf p}_{\bf F}$ & polynomial degree vector for parametrization & Section~\ref{sec:parametrization_assumptions}\\
$\mathscr{P}$ & PDE operator & Section~\ref{sec:FEM problem}\\
$\check\QQ_0$ & initial T-mesh of index domain & Section~\ref{sec:T-meshes defined}\\
$\hat\QQ_0$ & initial hierarchical/T-mesh in parametric domain & Section~\ref{sec:hierarchical refine}/\ref{sec:T refine}\\
$\QQ_0$ & initial mesh & Section~\ref{sec:meshes_axioms}\\
$\hat{\mathcal{Q}}_\F$ & mesh of parametric domain induced by knots of parametrization & Section~\ref{sec:parametrization_assumptions}\\
$\hat \QQ^\ell$ & uniformly refined mesh of  parametric domain & Section~\ref{subsec:def hb}\\
$\QQ_{\bf F}$ & mesh induced by parametrization & Section~\ref{sec:parametrization_assumptions}\\
$\hat \Q$ & admissible hierarchical/T-meshes & Section~\ref{sec:hierarchical refine}/\ref{sec:T refine}\\
$\Q$ & admissible meshes & Section~\ref{sec:meshes_axioms}\\
$\hat\Q_m$ & admissible meshes of volume/boundary patch in parametric domain\ & Section~\ref{sec:H-multipatches}/\ref{sec:H-igabem}\\
$\Q_m$ & admissible meshes of volume/boundary patch & Section~\ref{sec:H-multipatches}/\ref{sec:H-igabem}\\
$S_{\rm ext}(\hat Q)$ & support extension (for B-splines and T-splines) & Section~\ref{sec:splines-tensor definition}/\ref{sec:dual-compatible}\\
$S_{\rm ext}(\widehat Q,k)$ & multilevel support extension & Section~\ref{sec:hierarchical refine}\\
$S_{\rm ext}^*(\hat Q)$ & modified support extension & Section~\ref{sec:hierarchical interpolation}\\
$\mathbb{S}$ & FEM/BEM ansatz space & Section~\ref{sec:FEM problem}/\ref{sec:model problem bem}\\
$\hat {\mathbb{S}}_p(T)$ & space of univariate splines & Section~\ref{sec:univariate-properties}\\
$\hat {\mathbb{S}}_{\bf p}({\bf T})$   & space of multivariate splines & Section~\ref{sec:splines-tensor definition}\\
$\hat{\mathbb{S}}_{\bf p}^{\rm H}(\hat Q,{\bf T}^0)$ & space of hierarchical splines & Section~\ref{subsec:def hb}\\
$\hat{\mathbb{S}}_{\bf p}^{\rm T}(\check\QQ,{\bf T}^0)$ & space of T-splines & Section~\ref{sec:T-splines defined}\\
$T$ & knot vector & Section~\ref{sec:univariate-properties}\\
$T_0$ & initial knot vector & Section~\ref{subsec:concrete refinement bem2}\\
$\hat T_{{\bf i},{\bf p}}^\ell$  & truncated hierarchical B-spline & Section~\ref{subsec:thb}\\
${\rm Trunc}^{\ell+1}$ & truncation operator & Section~\ref{subsec:thb}\\
${\bf T}$ & vector of knot vectors & Section~\ref{sec:splines-tensor definition}\\
${\bf T}^0$ & initial vector of knot vectors & Section~\ref{sec:T-meshes defined}\\
${\bf T}_\F$ & vector of knot vectors for parametrization & Section~\ref{sec:parametrization_assumptions}\\
${\bf T}^\ell$ & uniformly refined vector of knot vectors & Section~\ref{subsec:def hb}\\
$\hat{\mathcal{T}}_{\bf p}(\hat{\mathcal{Q}},{\bf T}^0)$ & truncated hierarchical B-splines & Section~\ref{subsec:thb}\\
$\mathbb{T}$ & admissible knot vectors for univariate refinement & Section~\ref{subsec:concrete refinement bem2}\\
$u$ & PDE solution & Section~\ref{sec:FEM problem}\\
$U$ & Galerkin FEM approximation & Section~\ref{sec:FEM problem}\\
$\mathcal{V}$ & vertices of mesh  & Section~\ref{sec:igabem1d setting}\\
$\mathcal{V}_\F$ & vertices of geometry & Section~\ref{sec:assumption-BEM}\\
$\mathscr{V}$ & single-layer operator & Section~\ref{sec:model problem bem}\\
$Z$ & breakpoints & Section~\ref{sec:univariate-properties}\\
$\hat\gamma_0$ & shape-regularity constant & Section~\ref{subsec:concrete refinement bem2}\\
$\hat\Gamma$ & parametric domain for BEM & Section~\ref{sec:parametrization_assumptions}\\
$\Gamma$ & physical domain for BEM  & Section~\ref{sec:parametrization_assumptions}\\
$\Gamma_{m,m'}$ & interface between NURBS patches & Section~\ref{sec:multi-patch}\\
$\eta$ & error estimator for FEM/BEM & Section~\ref{sec:estimator fem}/\ref{sec:estimator bem}\\
$\hat\lambda_{i,p}$ & univariate dual functional & Section~\ref{sec:qi-1d}\\
$\hat\lambda_{{\bf i},{\bf p}}$ & multivariate dual functional & Section~\ref{sec:qi-2d}\\
$\hat\lambda_{{\bf i}, {\bf p}}^\ell$ & dual functional for hierarchical splines  & Section~\ref{sec:hierarchical interpolation}\\
$\hat\lambda_{{\bf z},{\bf p}}$ & dual functional for T-splines & Section~\ref{sec:dual-compatible}\\
$\mu$ & admissibility parameter for hierarchical meshes & Section~\ref{sec:hierarchical refine}\\
$\nu$ & outer normal vector & Section~\ref{sec:estimator fem}\\
$\pi^q$ & volume/boundary element-patch & Section~\ref{sec:abstract setting fem}/\ref{sec:abstract setting bem}\\
$\Pi^q$ & volume/boundary element-patch (elements) & Section~\ref{sec:abstract setting fem}/\ref{sec:abstract setting bem}\\
$\phi$ & solution of boundary integral equation  & Section~\ref{sec:model problem bem}\\
$\Phi$ & Galerkin BEM approximation & Section~\ref{sec:IGA-basics}\\
$\hat\Omega$ & parametric domain & Section~\ref{sec:parametrization_assumptions}\\
$\Omega$ & physical domain & Section~\ref{sec:parametrization_assumptions}\\
$\check\Omega_{\rm act}$ & active region for definition of T-splines& Section~\ref{sec:T-splines defined}\\
$\check\Omega_{\rm ind}$ & index domain for definition of T-splines & Section~\ref{sec:T-meshes defined}\\
$\check\Omega_{\rm ip}$ & index/parametric domain for definition of T-splines & Section~\ref{sec:T-meshes defined}\\
$\widehat \Omega^\ell$ & nested subsets of parametric domain  & Section~\ref{subsec:def hb}\\
$\Omega_m$ & NURBS patch & Section~\ref{sec:multi-patch}\\
$\preceq$ & refinement relation & Section~\ref{subsec:def hb}\\
$\#$ & multiplicity of a breakpoint & Section~\ref{sec:univariate-properties}\\
$\nabla_\Gamma$ & surface gradient & Section~\ref{sec:sobolev}\\
$[\cdot]$ & jump & Section~\ref{sec:estimator fem}\\
$\dual{\cdot}{\cdot}_{\mathscr{P}}$ & bilinear form induced by PDE & Section~\ref{sec:FEM problem}\\
\hline
\end{longtable}

\twocolumn


\newpage



\section{Splines on tensor meshes}
\label{sec:splines}

The main purpose of this section is to introduce some basic concepts and notation that will be used throughout the paper. 
In Section~\ref{sec:splines-univariate} and \ref{sec:splines-tensor}, we recall the definition as well as elementary properties of univariate and multivariate splines and their B-spline basis.
In Section~\ref{sec:NURBS}, we introduce non-rational splines along with the NURBS basis, which are then used in Section~\ref{sec:parametrization} to define NURBS parametrizations.
For a more detailed introduction and proofs, we refer, e.g., to  \cite{boor86,boor01,schumaker07}. 

\subsection{Univariate B-splines}\label{sec:splines-univariate}

\subsubsection{Definition and properties} \label{sec:univariate-properties}
Given two integers $p \ge 0$ and $n > 0$, we define a \emph{knot vector} as an ordered vector of the form
\[
\kv = (\knot_1, \ldots, \knot_{n+p+1}),
\]
with $\knot_j \le \knot_{j+1}$ for all $1 \le j \le n+p$. We say that $T$ is an \emph{open (or $p$-open) knot vector}, if the first and last knots are repeated exactly $p+1$ times, i.e., $\knot_1 = \ldots =\knot_{p+1} < \knot_{p+2}$ and $\knot_n < \knot_{n+1} = \ldots = \knot_{n+p+1}$. For simplicity, we will assume that $\knot_{1} = 0$ and $\knot_{n+p+1} = 1$ in the following.

We also introduce the ordered set of \emph{breakpoints} $\brkset = \{ \brkpnt_1, \ldots, \brkpnt_{n'} \}$, which accounts for knots without repetitions. We denote by $\# z_j$ the \emph{multiplicity} of the breakpoint $\brkpnt_j$, such that $\sum_{j=1}^{n'} \# z_j = n+p+1$ and 
\begin{equation*}
\kv = (\underbrace{\brkpnt_1, \dots, \brkpnt_{1}
}_{\#z_1 \text{ times}},\underbrace{\brkpnt_2, \dots, \brkpnt_{2} }_{\# z_2
  \text{ times}},\ldots, \underbrace{\brkpnt_{n'}, \dots, \brkpnt_{{n'}} }_{\# z_{n'}
  \text{ times}}).
\end{equation*}
For $2 \le j \le {n'}-1$, i.e., for all internal knots, the multiplicity satisfies $\# z_j \le p+1$. Later on, and in particular for FEM, we will require lower multiplicity.

From the knot vector $\kv$, a set of $n$ \emph{B-splines} is defined using the Cox--de Boor recursion formula. We start defining the piecewise constant functions
\begin{equation*}
\hat B_{i,0}(\param) := \left \{
\begin{array}{ll}
1 & \text{ if } \knot_i \leq \param < \knot_{i+1}, \\
0 & \text{ otherwise}.
\end{array}
\right.
\end{equation*}
For $1 \le k \leq p$, the \emph{B-spline functions} are defined by the recursion
\begin{equation*}
\hat B_{i,k}(\param) = \frac{\param - \knot_i}{\knot_{i+k} - \knot_i} \hat B_{i,k-1}(\param) + \frac{\knot_{i+k+1} - \param}{\knot_{i+k+1} - \knot_{i+1}} \hat B_{i+1,k-1}(\param),
\end{equation*}
where we use the convention that fractions with zero denominator are equal to zero.

Among many other properties, the B-splines are non-negative and satisfy the partition of unity (see \cite[Theorem~4.20]{schumaker07})
\[
\sum_{i=1}^n \hat B_{i,p}(\param) = 1, \quad \text{for all } \param \in (0,1),
\]
they have local support (see \cite[Theorem~4.17]{schumaker07}), in particular
\begin{equation} \label{eq:support_univariate}
\supp( \hat B_{i,p}) = [\knot_i,\knot_{i+p+1}], \; \text{ for }i = 1, \ldots, n,
\end{equation}
they are locally linearly independent in the sense that for any open set  $O \subseteq (0,1)$ the functions $\{B_{i,p}|_{O} : B_{i,p}|_{O} \neq 0 \}$ are linearly independent (see \cite[Chapter~IX, (47)]{boor01} and
\cite[Theorem~4.18]{schumaker07}), and they form a basis of the space of piecewise polynomials of degree $p$ with $p - \# z_j$ continuous derivatives at the breakpoints $\brkpnt_j$, for each $j = 2, \ldots,{n'}-1$ (see \cite[Chapter~IX, (44)]{boor01}). Notice that the maximum and minimum allowed continuity at the breakpoints are $C^{p-1}$ and $C^{-1}$, which correspond to multiplicity $\# z_j = 1$ and $\# z_j = p+1$, respectively. We denote the \emph{basis of B-splines} as
\[
\spbasis_p(T) := \{ \hat B_{i,p} : i=1,\ldots, n \},
\]
and the \emph{spline space} spanned by them as
\begin{equation*} 
\spu{p}{\kv} := \mathrm{span} (\spbasis_p(T) ).
\end{equation*}

It is easy to see, from the recursion formula in the definition, that the definition of the B-spline $\hat B_{i,p}$, for $i = 1, \ldots, n$, depends only on the \emph{local knot vector} $\kv_{i,p} = (\knot_i, \ldots, \knot_{i+p+1})$, which is closely related to the support of the function \eqref{eq:support_univariate}. When necessary, and in particular when dealing with T-splines, we will stress this fact by using the equivalent notation
\begin{equation} \label{eq:spline-1d}
\hat B[\kv_{i,p}] := \hat B_{i,p}.
\end{equation}

Finally, we note that the breakpoints in $\brkset$ generate a partition of the interval $(0, 1)$, and we denote by $I_j := (\brkpnt_j, \brkpnt_{j+1})$ the local elements for $j = 1, \ldots, {n'}-1$, and by $\widehat h_j := \brkpnt_{j+1} - \brkpnt_j$ their corresponding element sizes. 
For each element $I_j$, which can be uniquely written as $(\knot_i, \knot_{i+1})$ for a certain index $p+1 \le i \le n$, we introduce its \emph{support extension}
\begin{equation} \label{eq:supp_ext}
 \sext{{I}_j} := [\knot_{i-p}, \knot_{i+p+1}] ,
\end{equation}
being the union of the supports of B-splines that do not vanish on $I_j$.

Assuming that the maximum multiplicity of the internal knots is less than
or equal to the degree $p$,  i.e., the B-spline functions are at least continuous, the (right-hand) derivative of each B-spline $\hat B_{i,p}$ is given by the expression \cite[Sect.~4.2]{schumaker07}
\begin{equation*}
\hat B'_{i,p}= \frac{p}{\knot_{i+p}-\knot_i} \hat B_{i,p-1} - \frac{p}{\knot_{i+p+1}-\knot_{i+1}} \hat B_{i+1,p-1}.
\end{equation*}

\subsubsection{Quasi-interpolation operators}\label{sec:qi-1d}
Let $C_{\rm locuni}\ge1$ be such that the following local quasi-uniformity is satisfied 
\begin{align}\label{eq:lqi}
C_{\rm locuni}^{-1} \le \frac{\widehat h_j}{\widehat h_{j+1}} \le C_{\rm locuni}
\end{align}
for all $j=1,\dots,n'-2$.
Clearly, for a given a knot vector $\kv$, such a constant always exists and we will use it to stress certain dependencies on the ratios $\widehat h_j/\widehat h_{j+1}$.

There are several ways to define quasi-interpolation and projection operators onto the space of splines $\spu{p}{\kv}$. In this work, we are interested in the theoretical properties of these operators, and not in their actual computation. For this reason, we will focus on two particular operators, and refer the reader to \cite{Sabl05} for further discussion on quasi-interpolation operators.

To define the quasi-interpolation operators, we first need to define a set of linear functionals $\hat{\lambda}_{i,p}$ associated to the B-splines.
Then, the quasi-interpolant takes the form
\begin{equation} \label{eq:qi-1D}
\proju{p}{\kv}: L^2(0,1) \rightarrow \spu{p}{\kv}, \quad \hat v\mapsto \sum_{i=1}^n \hat{\lambda}_{i,p}(\hat v) \hat B_{i,p}.
\end{equation}
Notice that, when $\hat{\lambda}_{i,p} (\hat B_{j,p}) = \delta_{ij}$, with $\delta_{ij}$ the Kronecker symbol, the linear functionals form a \emph{dual basis}, and the quasi-interpolant becomes a projector, i.e.,
\begin{equation*}
\proju{p}{\kv} \hat v = \hat v \quad \text{ for all } \hat v \in \spu{p}{\kv}.
\end{equation*}

The first operator that we use was introduced in \cite{deboor75} (see also \cite[Sect.~4.6]{schumaker07}) and is the one traditionally used in IGA \cite{bbchs06,bbsv14}.
We will denote it by $\proju{p}{\kv}^{\,\rm dB}$ (where dB stands for de Boor). In this case, the functionals are defined as
\begin{equation}
   \label{eq:lambda-hat-shumacker-1D}
\hat{\lambda}_{i,p}(\hat v) \equiv \hat{\lambda}^{\,\rm dB}_{i,p}(\hat v) :=  \int_{\knot_i}^{\knot_{i+p+1}} \hat v(s) D^{p+1}\psi_{i}(s)~{\rm d}s,  
\end{equation}
  where $D^{k}$ stands for the $k$-th derivative, and $\psi_{i}(\param)=G_{i}(\param)\phi_{i}(\param)$, with
\[
\phi_i(\param) := \frac{(\param - \knot_{i+1}) \dots (\param - \knot_{i+p})}{p!}, 
\] 
and 
\[
G_{i}(\param) := g\left(\frac{2\param-\knot_i-\knot_{i+p+1}}{\knot_{i+p+1}-\knot_{i}}\right),
\]
where $g$ is the transition function defined in \cite[Theorem 4.37]{schumaker07}. Note that it is trivial to see from \eqref{eq:lambda-hat-shumacker-1D} that 
\begin{equation*}
\hat v|_{\supp (\hat B_{i,p})} = 0 \implies \hat{\lambda}^{\,\rm dB}_{i,p}(\hat v) = 0.
\end{equation*}
Moreover, we notice that the definition of each dual functional is based on the local knot vector, and we will stress this fact with the alternative notation
\begin{equation} \label{eq:dual-basis-kv}
\hat{\lambda}^{\,\rm dB}[\kv_{i,p}] := \hat{\lambda}^{\,\rm dB}_{i,p}.
\end{equation}

The second operator is defined in \cite{bggs16,tsete15}, to which we refer for the details.  We will denote it by $\proju{p}{\kv}^{\,\rm Bp}$, since it is sometimes called B\'ezier projection. We start defining, for each element $I_j$, the local $L^2$-projection $P_{I_j}$ into the space of polynomials of degree $p$ on $I_j$. Since B-splines span piecewise polynomials, the local $L^2$-projection $P_{I_j}$ can be written as in \eqref{eq:qi-1D} considering the restriction of the functions to $I_j$, namely
\begin{equation} \label{eq:local-proj}
P_{I_j} (\hat v|_{I_j}) = \sum_{\substack{i=1 \\ \supp( \hat B_{i,p}) \cap I_j \not = \emptyset}}^n \hat{\lambda}^{I_j}_{i,p} (\hat v) \hat B_{i,p}|_{I_j}.
\end{equation}
Then, the functionals $\hat{\lambda}_{i,p} \equiv \hat{\lambda}^{\,\rm Bp}_{i,p}$ are defined as convex combinations of the corresponding functionals of the local projection
\[
\hat{\lambda}_{i,p} := \sum_{\substack{j=1 \\I_j \cap \supp( \hat B_{i,p}) \not = \emptyset}}^{n'-1} c_{i,I_j} \hat{\lambda}^{I_j}_{i,p},
\]
with coefficients $c_{i,I_j} \ge 0$ and $\sum_{j=1}^{{n'}-1} c_{i,I_j} = 1$ for $1 \le i \le n$.
The functionals form a dual basis. 
For the following results, the concrete choice of the coefficients $c_{i,I_j}$ is not relevant. Among the three suggested choices given in \cite[Section~6]{bggs16}, we consider the following one: for each basis function $\hat B_{i,p}$,  we choose a local element $I_{k(i)} \subseteq \supp( \hat B_{i,p})$ such that
\[
| I_{k(i)}| \simeq |\supp( \hat B_{i,p})|.
\]
In our case, this is valid for any element thanks to \eqref{eq:lqi}, with  hidden constants that depend only on the degree $p$ and the constant $C_{\rm locuni}$. 
Then, the coefficients are taken as
\[
c_{i,I_j} := \left \{
\begin{array}{ll}
1 & \text{ if } I_j = I_{k(i)}, \\
0 & \text{ otherwise},
\end{array}
\right.
\]
and the dual functionals simply become $\hat{\lambda}^{\,\rm Bp}_{i,p} = \hat{\lambda}^{I_{k(i)}}_{i,p}$.


The importance of these two quasi-interpolants comes from the following stability result.
The proofs can be found in \cite[Propositions~2.2]
{bbsv14} and \cite[Theorem~2]{bggs16}, respectively.
\begin{proposition} \label{prop:schumaker bezier}
Let either $\proju{p}{\kv}= \proju{p}{\kv}^{\,\rm dB}$ or $\proju{p}{\kv}= \proju{p}{\kv}^{\,\rm Bp}$.
Then, for any interval $I_j$, it holds that
\[
\| \proju{p}{\kv} \hat v \|_{L^2(I_j)} \le C \| \hat v \|_{L^2(\sext{I_j})}  \; \text{ for all } \hat v\in L^2(0,1),
\]
where the constant $C
>0$ depends only on the degree $p$ and the constant $C_{\rm locuni}$.
\end{proposition}

\subsection{Multivariate B-splines} \label{sec:splines-tensor}
The generalization of univariate B-splines to the multivariate setting is done by tensorization. In this section, we introduce the notation for the tensor-product basis functions and spaces.

\subsubsection{Definition and properties}\label{sec:splines-tensor definition}
Let $\dpa$ be the space dimension, which will be $\dpa = 2, 3$  in practical cases. Let the integers $p_j\ge 0$ and $n_j > 0$, and  let $\kv_j = (\knot_{j,1} , \ldots, \knot_{j,n_j+p_j+1} )$ be a $p_j$-open knot vector for each $j = 1, \ldots, \dpa$. We set the degree vector $\mathbf{p} := (p_1, \ldots, p_\dpa)$ and $\mathbf{\kv} := (\kv_1, \ldots , \kv_\dpa)$. Then, \emph{multivariate B-splines} are defined as products of the form
\[
  \hat B_{{\bf i},{\bf p}}(\bparam) := \hat
    B_{i_1,p_1}(\param_1) \cdots \hat B_{i_\dpa,p_\dpa}(\param_\dpa),
\]
for ${\bf i} = (i_1, \ldots, i_\dpa)$ and $1 \le i_j \le n_j$ for each $j = 1,\ldots, \dpa$, where it is understood that $\hat B_{i_j,p_j}$ is defined from the knot vector $T_j$. Analogously to the univariate case, we will denote the \emph{B-spline basis} as
\[
\spbasis_{\bf p}(\mathbf{\kv}) := \{\hat B_{{\bf i},{\bf p}} : \, {\bf i} = (i_1, \ldots, i_\dpa), \, 1 \le i_j \le n_j \},
\]
while the \emph{spline space} is the spanned space, which is denoted by
\[
\spm := \mathrm{span} (\spbasis_{\bf p}(\mathbf{\kv})). 
\]
It is worth noting that $ \spm = \otimes_{j=1}^\dpa \spu{p_j}{\kv_j}$, i.e., it can defined as tensor-product of the univariate spaces. Multivariate B-splines have basically the same properties as univariate B-splines: they are non-negative and form a partition of unity, they have local support, and they are locally linearly independent.

Analogously to the univariate case, from the knot vector in each direction we define the set of breakpoints, or knots without repetitions, $\brkset_j := \{ \brkpnt_{j,1}, \ldots, \brkpnt_{j,n'_j}\}$, for $j = 1, \ldots, \dpa$. Analogously to the partition of the interval in the univariate case, the breakpoints form a rectilinear grid of the form 
\begin{align*}
\hat \mesh := & \{ \hat \elem_{\bf k} = I_{1,k_1} \times \ldots \times I_{\dpa, k_\dpa} : \\
& I_{j, k_j} = (\brkpnt_{j,k_j},
  \brkpnt_{j,k_j+1}) \text{  for  } 1\leq k_j \leq n'_j-1 \}.
\end{align*}
For a generic element $\hat \elem_{\bf k} \in \hat \mesh$, we define the element size as 
\[
\widehat h_{\hat \elem_{\bf k}} := |\hat \elem_{\bf k}|^{1/\dpa}.
\] 
We also define its support extension as the union of the (open) supports of basis functions that do not vanish in $\hat \elem_{\bf k}$, and due to the tensor-product structure, this is defined from the univariate support extensions as
\begin{equation}\label{eq:multivariatese}
\sext{\hat \elem_{\bf k}} := \sext{I_{1,k_1}} \times \ldots \times \sext{I_{\dpa,k_\dpa}} 
\end{equation}
for ${\bf k} = (k_1, \ldots, k_\dpa)$.
Here, $\sext{I_{j,k_j}}$ is the univariate support extension in the $j$-th direction given by \eqref{eq:supp_ext}.


\subsubsection{Quasi-interpolation operators} \label{sec:qi-2d}
The quasi-interpolation operators and dual bases from Section~\ref{sec:qi-1d} can be generalized to the multivariate setting. The first quasi-interpolant $\projm^{\,\rm dB} : L^2((0,1)^\dpa) \rightarrow \spm$ is defined as tensor-product
\[
\projm^{\,\rm dB} := \proju{p_1}{\kv_1}^{\,\rm dB} \otimes \ldots \otimes \proju{p_\dpa}{\kv_\dpa}^{\,\rm dB},
\]
where the tensorization is interpreted in the sense of \cite[Chapter~XVII]{boor01}, see also \cite[Section~2.2]{bbsv14}. This kind of quasi-interpolant will be used for T-splines in Section~\ref{subsec:tsplines}.

For the second quasi-interpolant, instead of applying tensorization, we define it in a similar way as in the univariate case. 
As in  \eqref{eq:qi-1D}, it is defined by constructing a dual basis. To define the dual basis, for each basis function $\hat{B}_{\mathbf{i},\mathbf{p}}$, we choose an element $\hat Q_{{\bf k}({\bf i})}\subseteq \supp(\hat{B}_{\mathbf{i},\mathbf{p}})$ with size equivalent to the size of the support. 
Then, introducing a local projector in $\hat Q_{{\bf k}({\bf i})}$, as in \eqref{eq:local-proj}, and with an analogous notation for the local dual basis, the dual functional associated to this basis function is given by 
$\hat{\lambda}^{\,\rm Bp}_{\mathbf{i},\mathbf{p}} := \hat{\lambda}_{\mathbf{i},\mathbf{p}}^{\hat Q_{{\bf k}({\bf i})}}$.
This type of quasi-interpolant will be used for hierarchical B-splines in Section~\ref{subsec:hb}.

Since the multivariate quasi-interpolation operators are defined from a dual basis they are also projectors, i.e.,
\begin{equation*}
\projm \hat v = \hat v \quad \text{ for all } \hat v \in \spm,
\end{equation*}
where we can choose $\projm$ either equal to $\projm^{\,\rm dB}$ or to $\projm^{\,\rm Bp}$. Moreover, for the two operators we have a stability 
result
analogous to the one already presented in the univariate setting in Proposition~\ref{prop:schumaker bezier}. %
The 
 result for $\projm=\projm^{\,\rm dB}$ of the following proposition is proved in \cite[Lemma~3.2]{bbchs06}. 
For the second quasi-interpolant $\projm=\projm^{\,\rm dB}$, the result 
is proved in \cite[Theorem~2]{bggs16}, see also \cite[Section~3.1]{bg17}.
\begin{proposition}\label{prop:schumaker bezier in nd}
Let either $\projm=\projm^{\,\rm dB}$ or $\projm=\projm^{\,\rm Bp}$.
Then, for any element $\hat \elem \in \hat \mesh$, it holds that
\[
\| \projm \hat v \|_{L^2(\hat \elem)} \le C  \| \hat v \|_{L^2(\sext{\hat \elem})} \; \text{ for all } v\in L^2((0,1)^\dpa).
\]
The constant $C>0$ depends only on the polynomial degrees $p_1,\dots,p_\dpa$ and local quasi-uniformity \eqref{eq:lqi} in each direction.
\end{proposition}

\subsection{Non-uniform rational B-splines}
\label{sec:NURBS}
Non-uniform rational B-splines (NURBS) are a generalization of B-splines. When used to build geometry parametrizations, as we will do in Section~\ref{sec:parametrization}, they have the advantage of giving exact representations of conic sections, which cannot be achieved with piecewise polynomials, see \cite[Section 1.4]{Piegl} for more details.

We start in the univariate setting. Given the B-spline basis, we define the \emph{weight function} as a linear combination of B-splines 
\[
\hat W := \sum_{i=1}^n \weight_i \hat B_{i,p},
\]
with positive coefficients $\weight_i > 0$ for $i=1, \ldots, n$. Then, the set of NURBS basis functions is formed by the rational functions
\[
\hat R_{i,p} := \frac{\weight_i \hat B_{i,p}}{\sum_{j=1}^n \weight_j \hat B_{j,p}} = \frac{\weight_i \hat B_{i,p}}{\hat W}.
\]
Analogously, in the multivariate case, if we introduce the set of multi-indices ${\cal I} := \{ {\bf i} = (i_1, \ldots, i_\dpa) :  1 \le i_j \le n_j \}$, the weight function is defined as
\begin{equation}\label{eq:weight} 
\hat W := \sum_{{\bf i} \in {\cal I}} w_{\bf i} \hat B_{{\bf i},{\bf p}}
\end{equation}
and provides the multivariate NURBS basis functions
\begin{equation}\label{eq:NURBS_basis}
\hat R_{{\bf i},{\bf p}} := \frac{\weight_{\bf i} \hat B_{{\bf i},{\bf p}}}{\sum_{{\bf j} \in {\cal I}} \weight_{\bf j} \hat B_{{\bf j},{\bf p}}} = \frac{\weight_{\bf i} \hat B_{{\bf i},{\bf p}}}{\hat W}.
\end{equation}
Note that, although the NURBS basis functions are defined starting from B-splines, they are not constructed by tensor-product due to the presence of the weights.

Finally, we can also define a quasi-interpolant for NURBS. Using the generic notation $\projm$ for a B-spline quasi-interpolant, we define the corresponding NURBS quasi-interpolant by
\[
\projm^{\,\hat W} (\hat v) := \frac{\projm(\hat W \, \hat v)}{\,\hat W}.
\]
It can be readily seen that this operator is a projector onto the NURBS space provided that $\projm$ is a projector onto the spline space. 
Moreover, if $\projm$ is as in Proposition~\ref{prop:schumaker bezier} or Proposition~\ref{prop:schumaker bezier in nd}, $\projm^{\,\hat W}$ satisfies the same stability and approximation properties as $\projm$, where the constants depend additionally on $\widehat W$, see \cite{bbchs06} and \cite[Section 4]{bbsv14} for details.

\subsection{B-splines and NURBS geometries}
\label{sec:parametrization}
A spline or NURBS geometry is built as a linear combination of B-splines or NURBS basis functions, by associating a control point to each basis function. More precisely, let the set of $\dpa$-variate NURBS be defined as in \eqref{eq:NURBS_basis} and let ${\bf C}_{\bf i} \in \mathbb{R}^\dph$ with $\dph \ge \dpa$ be the associated control points. The parametrization of the NURBS geometry is then given by
\begin{equation} \label{eq:parametrization}
{\bf F}({\bf t}) := \sum_{{\bf i} \in {\cal I}} {\bf C}_{\bf i} \hat{R}_{\bf i, p} ({\bf t}).
\end{equation}
The parametrization of a spline geometry is built completely analogously, replacing the rational basis functions by B-splines. Note that, as mentioned above, particular choices of the weight function $\hat W$ will allow the exact representation of conic geometries by NURBS. Moreover, it is also worth noting that, for a NURBS geometry, each component $({\bf F})_i$ belongs to a space of rational splines, namely
\begin{align*}
(\F)_i\in \set{\widehat S/\widehat W}{\widehat S\in\widehat{\mathbb{S}}_{{\bf p}}(\mathbf{\kv})}, \text{ for } i = 1,\ldots, \dph.
\end{align*}

Examples of a spline curve with $\dpa=1$ and $\dph=2$, and a spline surface with $\dpa=2$ and $\dph=3$, are respectively given in Figures~\ref{fig:spline_curve} and~\ref{fig:spline_surface}. For more details on the properties of NURBS geometries and different methods to construct them, we refer to \cite{Piegl,Farin,Hoschek-Lasser,cohen2001geometric}.
\begin{figure}
\centering
\includegraphics[width=0.4\textwidth,trim=1cm 0cm 1cm 0cm, clip]{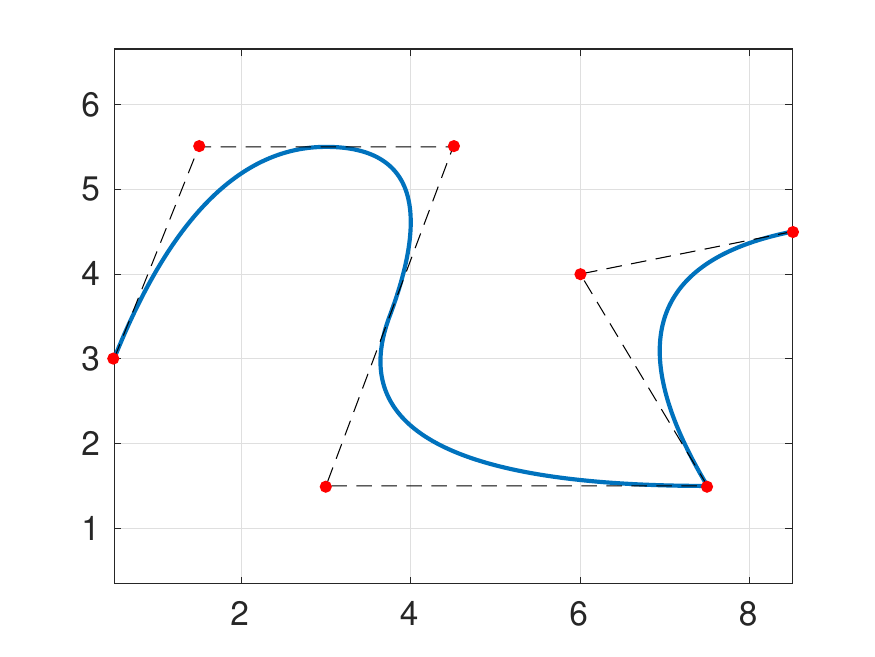}
\caption{Quadratic spline curve, constructed from the knot vector $T = (0, 0, 0, 1/4, 2/4, 3/4, 3/4, 1, 1, 1)$, along with its control points in $\mathbb{R}^2$.}
\label{fig:spline_curve}
\end{figure}
\begin{figure}
\centering
\includegraphics[width=0.4\textwidth,trim=1cm 0cm 1cm 0cm, clip]{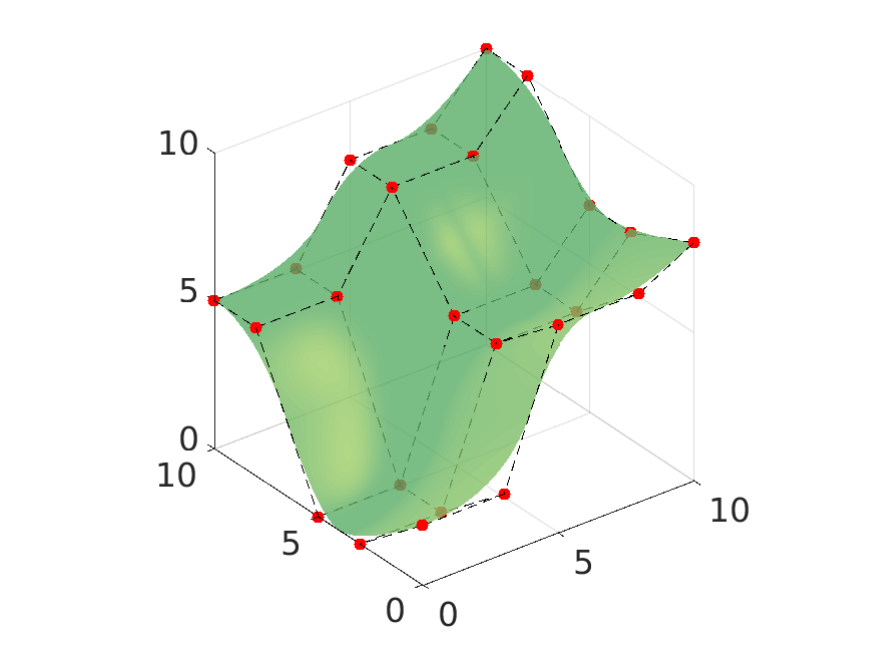}
\caption{Quadratic spline surface, constructed from the knot vectors $T_1 = T_2 = (0, 0, 0, 1/3, 2/3, 1, 1, 1)$, along with its control points in $\mathbb{R}^3$.}
\label{fig:spline_surface}
\end{figure}

\newpage


\section{Model problems and isogeometric analysis}
\label{sec:model}
In this section we introduce the basic concepts of IGA plus some important assumptions required for the numerical analysis of the method. 
In Section~\ref{sec:parametrization_assumptions}, we start with an explanation on the description of the considered geometry, i.e., a Lipschitz domain in the case of FEM and its boundary in the case of BEM, along with some important assumptions on the NURBS parametrizations that define it. Then, in Section~\ref{sec:IGAFEM-intro} we present the concept of IGA in the setting of FEM: we give a model problem written in terms of a PDE, we show how it is discretized with isogeometric methods, and present a residual-based error estimator. Finally, we present in Section~\ref{sec:IGABEM-intro} analogous ideas in the setting of isogeometric BEM for the discretization of a model problem written as a boundary integral equation.

\subsection{Parametrization of the physical domain} \label{sec:parametrization_assumptions}
We introduce here the assumptions of the physical domain. We start introducing the assumptions for the single-patch case, which will be valid throughout the paper. Then, we describe the assumptions required for multi-patch domains, and finally introduce a further assumption which is needed for BEM.
\subsubsection{General setting and single-patch domains}
In the following, we will always assume that our geometry is described through a spline or NURBS parametrization as defined in Section~\ref{sec:parametrization}. Let ${\bf p}_{\F}$ be the vector of polynomial degrees, ${\bf \kv}_{\F} = (\kv_{\F,1}, \ldots, \kv_{\F,\dpa})$ the multivariate open knot vector, with multiplicity smaller or equal to $p_{\F,j}$ for the internal knots in the $j$-th direction, and let $\widehat\QQ_{\F}$ be the corresponding tensor-mesh of the \emph{parametric domain} 
\[
\hat \Omega := (0,1)^\dpa.
\]
Introducing a weight function $\hat W_{\F}$ as in \eqref{eq:weight}, let ${\bf F}$ be a NURBS parametrization as in \eqref{eq:parametrization}, with control points in $\mathbb{R}^\dph$. We define the \emph{physical domain} as
\[
\Omega := {\bf F}(\hat \Omega) \subset \mathbb{R}^\dph.
\]
In the case of FEM, it holds that $d=\hat d$ and $|\cdot|$ will denote the $d$-dimensional volume.
In the case of BEM, where we will only work on the boundary $\Gamma$ of some Lipschitz domain, we will write $\hat\Gamma$ and $\Gamma$ instead of $\hat\Omega$ and $\Omega$.
Then, $\hat d=d-1$ and $|\cdot|$ will denote the $(d-1)$-dimensional surface measure.

The image through ${\bf F}$ of the mesh in the parametric domain automatically defines a mesh in the physical domain
\begin{equation} \label{eq:IGA-mesh}
\QQ_{\F} := \{Q = {\bf F}(\hat Q) : \hat Q \in \widehat\QQ_{\F} \}; 
\end{equation}
see an example in Figure~\ref{fig:parametrization}. Moreover, for any element $\elem \in \QQ_{\bf F}$ we define the \emph{element size} as 
\begin{equation*} 
h_Q := |Q|^{1/\dpa}.
\end{equation*}
These definitions are trivially extended to any mesh in the parametric domain. 
\begin{figure}
\centering
\includegraphics[width=0.22\textwidth]{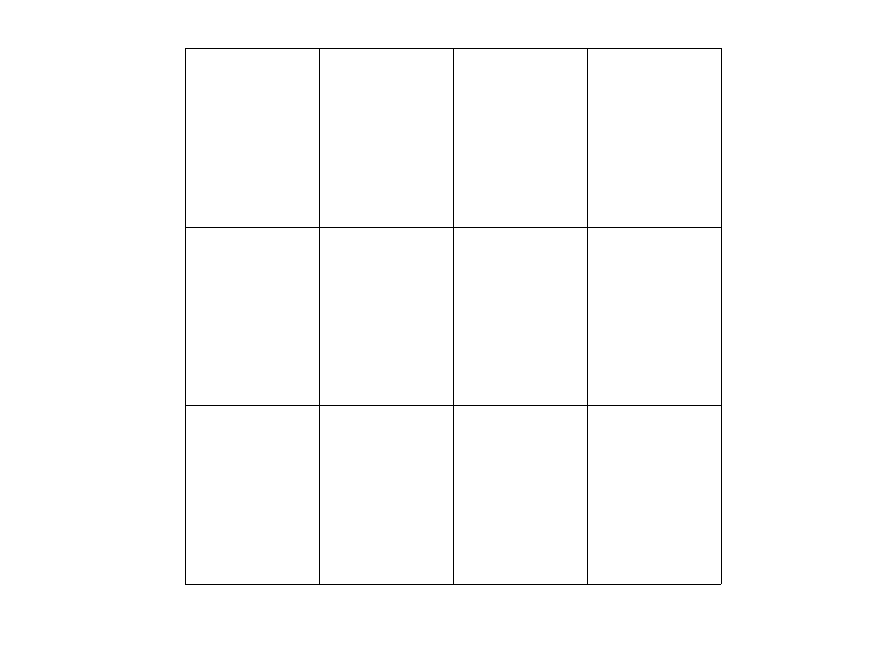}
\includegraphics[width=0.22\textwidth]{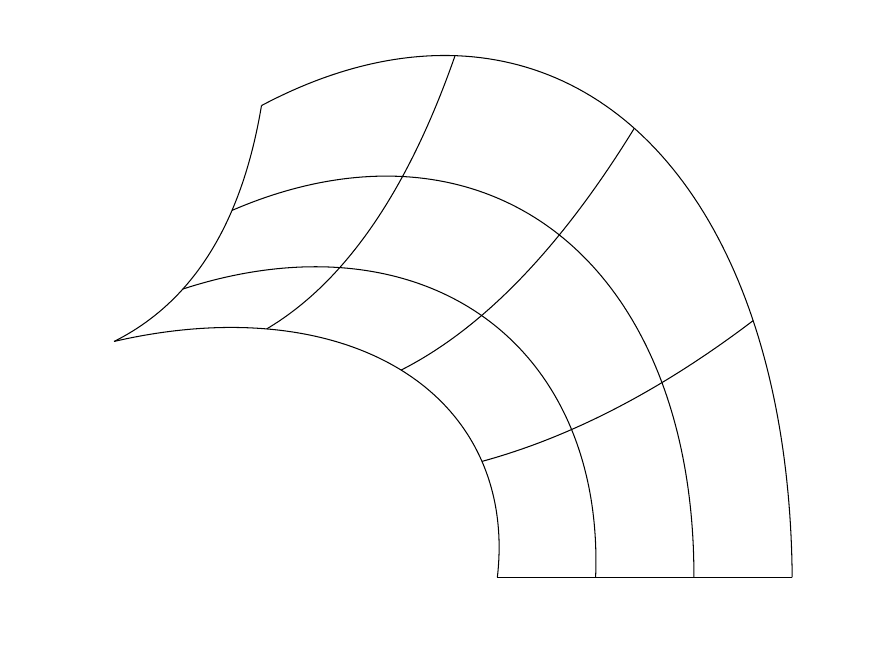}
\caption{Mesh in the parametric domain (left) and its image through ${\bf F}$ in the physical domain (right).} \label{fig:parametrization}
\end{figure}

By definition of NURBS (and B-splines), it is obvious that
\[
{\bf F}|_{\overline{\widehat Q}}\in \left(C^\infty(\overline{\widehat Q}) \right)^\dph \quad \text{for all }\widehat Q\in\widehat \QQ_{\F}, 
\]
where $\overline{\hat Q}$ denotes the closure of $\hat Q$.
However, in order to have a valid mesh, it is necessary to avoid the presence of singularities in the (inverse of the) parametrization, for which further assumptions are required.

In the following, we assume that ${\bf F}$ is a bi-Lipschitz homeomorphism\footnote{For $\widehat\omega\subseteq \R^\dpa$ and $\omega\subseteq \R^\dph$, a mapping $\gamma:\widehat\omega\to\omega$ is bi-Lipschitz if it is bijective and $\gamma$ as well as its inverse $\gamma^{-1}$ are Lipschitz continuous.}, which in particular implies that the inverse ${\bf F}^{-1}$ exists. Moreover, it implies that the Gram determinant is bounded from above and from below, namely there exists a constant $\const{\bf F} > 0$ such that
\begin{subequations}\label{eq:Cgamma}
\begin{align}\label{eq:Cgamma1}
\const{\bf F}^{-\dpa} \le \sqrt{\det(D{\bf F}^\top({\bf t}) D{\bf F}({\bf t}))} \le \const{\bf F}^{\dpa} 
\end{align}
for almost all  ${\bf t} \in \hat \Omega$,
where $D{\bf F}$ is the Jacobian matrix of the parametrization. Note that when $\dpa=\dph$ the Gram determinant reduces to $| \det(D{\bf F}({\bf t})) |$.
When $\dpa=\dph$, we additionally assume that 
\[
{\bf F}^{-1}|_{\overline{Q}} \in \left(C^2(\overline{Q})\right)^\dpa \quad\text{for all } Q\in \QQ_{\F},
\]
so that $\QQ_\F$-elementwise second derivatives of spline functions in the physical domain are well-defined.
Moreover, being bi-Lipschitz guarantees the boundedness of the first derivatives of ${\bf F}$ and its inverse. 
In particular, these assumptions imply the existence of $\const{\bf F}>0$ with \eqref{eq:Cgamma1} and for all $i,j,k\in\{1,\dots,\dpa\}$, 
\begin{align}
\begin{split}
&\Big\|\frac{\partial}{\partial t_j}(\F)_i\Big\|_{L^\infty(\widehat\Omega)}\le C_{\boldsymbol\F}, \quad \Big\|\frac{\partial}{\partial x_j}(\F^{-1})_i\Big\|_{L^\infty(\Omega)}\le C_{\boldsymbol\F},\\
&\Big\|\frac{\partial^2}{\partial t_j\partial t_k }(\F)_i\Big\|_{L^\infty(\widehat\Omega)}\le C_{\boldsymbol\F}, \Big\|\frac{\partial^2}{\partial x_j\partial x_k }(\F^{-1})_i\Big\|_{L^\infty(\Omega)}\le C_{\boldsymbol\F},
\end{split}
\end{align}
\end{subequations}
where $(\F)_i$ and $(\F^{-1})_i$ respectively denote the $i$-th component of $\F$ and $\F^{-1}$, and the second derivatives are defined elementwise.

Finally, we remark that under the assumptions on the parametrization ${\bf F}$ the size of the elements in the parametric and the physical domain is comparable, i.e., for any element $\elem = {\bf F}(\elemp)$ it holds that 
\[
h_\elem \simeq \hat h_{\elemp},
\]
and the hidden constants depend only on ${\bf F}$.

\subsubsection{Extension to multi-patch domains} \label{sec:multi-patch}
A single mapping ${\bf F}$ can only be used to parametrize simple domains that are images of the unit square or cube. To deal with more complex geometries, we introduce the concept of multi-patch domains, where each patch is constructed with a NURBS parametrization. 

In detail, we assume that the domain $\Omega$ is constructed with a partition into $M \in \mathbb{N}$ patches in the sense that
\[
\overline \Omega = \bigcup_{m=1}^M \overline {\Omega_m},
\]
where each patch $\Omega_m$ is defined with a NURBS parame-trization of the form
\[
{\bf F}_m : \hat \Omega \longrightarrow \Omega_m,
\]
and the assumptions made in Section~\ref{sec:parametrization_assumptions} are valid for each ${\bf F}_m$. 
Again, in the case of BEM, we will write $\Gamma_m$ instead of $\Omega_m$.
We denote by ${\bf p}_{\F_{m}}$ and ${\bf \kv}_{\F_m}$ the degree and the knot vector associated to the parametrization of each patch, and by $\hat \QQ_{\F_m}$ and $\spbasis_{{\bf p}_{\F_m}}({\bf \kv}_{\F_m})$ 
the corresponding mesh and the B-spline basis, respectively. Then, defining $\QQ_{\F_{m}}$ as in \eqref{eq:IGA-mesh}, we can define the \emph{multi-patch mesh} 
\begin{equation} \label{eq:multipatch-mesh}
\QQ_\F:= \bigcup_{m=1}^M \QQ_{\F_{m}}.
\end{equation}
As before, this definition can be trivially extended to refined meshes.

In order to construct suitable discrete spaces in the multi-patch domain, we must require that the meshes are conforming at the interfaces, and the patches glue together with $C^0$ continuity. Let us denote the interfaces by $\Gamma_{m,m'} := \overline{\Omega_m} \cap \overline{\Omega_{m'}}$ for $m \not = m'$. We assume that the two following conditions hold true for all $m,m'$ with $m\neq m'$:
\begin{enumerate}[(1)]
\renewcommand{\theenumi}{P\arabic{enumi}}
\bf\item\rm \label{P:conforming-mesh}
$\Gamma_{m,m'}$ is either empty, or a vertex, or the image of a full edge, or the image of a full face of $\hat \Omega$ for both parametrizations.
\bf\item\rm \label{P:conforming-basis-param} 
For each B-spline $\hat \beta_m \in \spbasis_{{\bf p}_{\F_{m}}}({\bf \kv}_{\F_{m}})$ such that
\[
(\hat \beta_m \circ {\bf F}_m^{-1})|_{\Gamma_{m,m'}} \not = 0,
\]
there exists a unique function $\hat \beta_{m'} \in \spbasis_{{\bf p}_{\F_{m'}}}({\bf \kv}_{\F_{m'}})$ such that $(\hat \beta_m \circ {\bf F}_m^{-1}) |_{\Gamma_{m,m'}} = (\hat \beta_{m'} \circ {\bf F}_{m'}^{-1}) |_{\Gamma_{m,m'}}$.
\end{enumerate}
The assumptions imply that the meshes are conforming at the interfaces and the coincident knot vectors are related by an affine transformation, including also knot repetitions. Moreover, the control points and the weights associated to the interface functions of adjacent patches must also coincide. As a consequence, the mesh $\QQ_\F$ is globally unstructured, but locally structured on each patch, see Figure~\ref{fig:multi-patch}.
\begin{figure}
\centering
\includegraphics[width=0.23\textwidth,trim=3cm 1cm 2.5cm 0cm, clip]{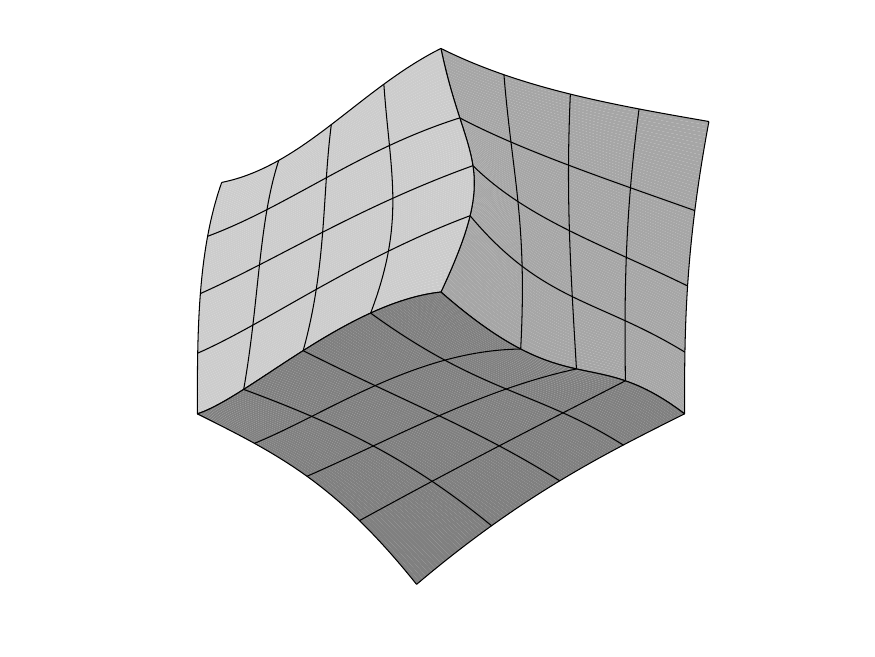}
\includegraphics[width=0.23\textwidth,trim=3cm 1cm 2.5cm 0cm, clip]{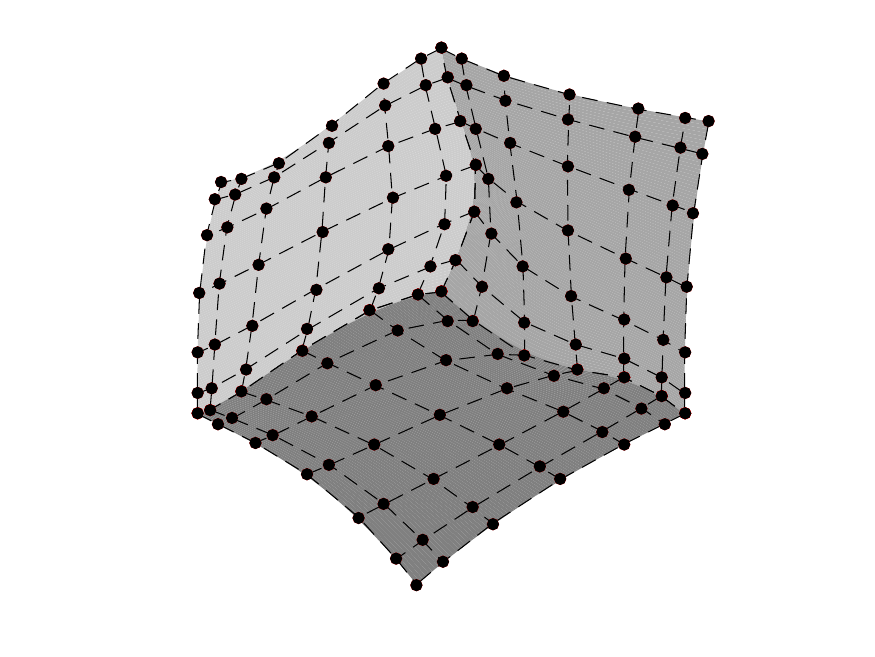}
\caption{An example of a multi-patch domain formed by three patches (left), and their corresponding control points (right). The control points associated to interface functions of adjacent patches coincide.}
\label{fig:multi-patch}
\end{figure}

\subsubsection{A further assumption for BEM} \label{sec:assumption-BEM}
%
In the case of BEM, we require a further assumption.  
Here, the boundary $\Gamma$ of some $\dph$-dimensional Lipschitz domain is defined as a multi-patch geometry through NURBS parametrizations. More precisely, we have that $\Gamma = \bigcup_{m=1}^M \overline{\Gamma_m}$, where each
\[
{\bf F}_m :(0,1)^{\dph-1} \rightarrow \Gamma_m \subset \mathbb{R}^\dph
\]
is a NURBS parametrization. 
Let us denote by
\begin{align*}
{\cal V}_{\bf F} := \bigcup_{m=1}^M \{ {\bf F}_m(\hat{{\bf z}}): \hat{{\bf z}} \in \{0,1\}^{\dph-1}\},
\end{align*}
the set of vertices of the geometry.
For each vertex ${\bf z} \in {\cal V}_{\bf F}$, we define  the subdomain covered by its neighboring elements as
\[
\mypatch_{\bf F}({\bf z}) := \bigcup \set{\overline Q}{Q\in\QQ_\F\wedge {\bf z}\in\overline Q} 
\]
Following \cite[Section 5.4.1]{gantner17}, we assume that the following condition holds true:
\begin{enumerate}[(1)]
\renewcommand{\theenumi}{P\arabic{enumi}}
\setcounter{enumi}{2}
\bf\item\rm \label{P:flattened} 
For every vertex ${\bf z} \in {\cal V}_{\bf F}$, there exists a set $\widehat{\mypatch}_{\bf F}({\bf z}) \subset \mathbb{R}^{\dph-1}$ that is an interval for $d=2$ and a polygon for $d=3$ and a bi-Lipschitz mapping
\[
\myparam_{\bf z} : \widehat{\mypatch}_{\bf F}({\bf z}) \longrightarrow \mypatch_{\bf F}({\bf z})
\]
such that $\myparam_z^{-1} \circ {\bf F}_m|_{\widehat Q}$ is an affine mapping for all $m \in \{1, \ldots, M \}$ and all $\widehat Q\in\widehat \QQ_{{\bf F}_m}$ with $Q:={\bf F}_m(\widehat Q)$ $\subset \mypatch_{\bf F}({\bf z})$. 
\end{enumerate}
The assumption means that each subdomain $\mypatch_{\bf F}({\bf z})$ can be flattened and that the inverse of the bi-Lipschitz mapping $\myparam_{\bf z}$ restricted to $Q$ essentially coincides with the inverse of ${\bf F}_m^{-1}$, see Figure~\ref{fig:Sauter-Schwab}. In particular, this prevents the case $\mypatch_{\bf F}({\bf z}) = \Gamma$. We stress that the same assumption is also made in \cite[Assumption 4.3.25]{ss11} for curvilinear triangulations.

\begin{figure}
\includegraphics[width =0.49\textwidth]{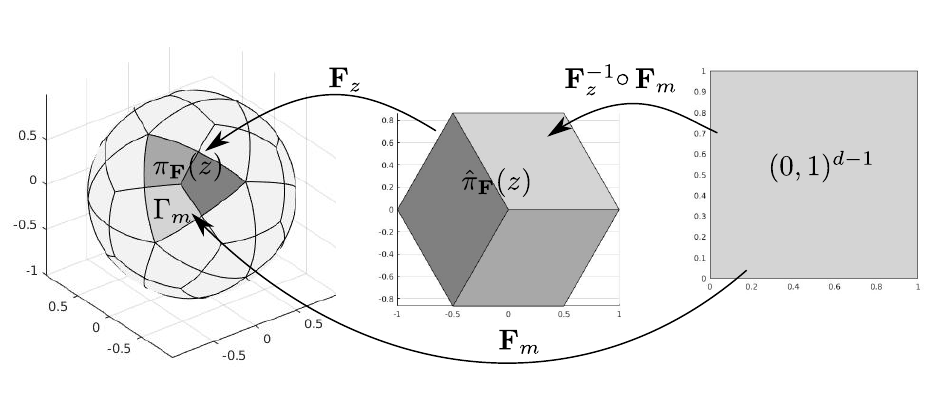}
\caption{Graphical representation of assumption \eqref{P:flattened}, in a parametrization of the sphere with 60 patches and one single element per patch. The three elements forming $\mypatch_{\bf F}({\bf z})$ on the left are colored in different tones of gray, and the corresponding polygon $\widehat{\mypatch}_{\bf F}({\bf z})$ is the hexagon shown in the middle. The mapping $\myparam_{\bf z}^{-1} \circ \myparam_m$ is an affine transformation.}
\label{fig:Sauter-Schwab}
\end{figure}

\subsection{Isogeometric analysis for FEM (IGAFEM)}\label{sec:IGAFEM-intro}
We now describe IGA based on tensor-product B-splines, i.e., without adaptive refinement. For more details about IGA we refer to \cite{hughes2005,IGA-book,bbsv14}.

\subsubsection{Model problem and Galerkin approximation}\label{sec:FEM problem}

Let $\Omega\subset\R^\dph$ with $\dph\ge 2$ be a bounded Lipschitz domain as in \cite[Definition~3.28]{mclean00}. In practice, $\Omega$ is a multi-patch domain defined as in Section~\ref{sec:multi-patch} with $\dpa=\dph$.
We consider a general second-order linear elliptic PDE  with homogenous Dirichlet boundary condition
\begin{align}\label{eq:problem}
\begin{split}
\mathscr{P}u&=f\quad \text{in }\Omega,\\
u&=0\quad\text{on }\Gamma:=\partial\Omega,
\end{split}
\end{align}
where
\begin{align}\label{eq:defP}
\mathscr{P}u:=-\div (\AA\nabla u)+\bb\cdot\nabla u +cu,
\end{align}
with $\AA\in W^{1,\infty}(\Omega)^{\dph\times \dph}$ and symmetric, $\bb\in L^\infty(\Omega)^{\dph}$, and $c\in L^\infty(\Omega)$.

We interpret  $\mathscr{P}$ in its weak form and define the corresponding bilinear form 
\begin{align*}
\begin{split}
\dual{w}{v}_{\mathscr{P}}&:=\int_\Omega (\AA \nabla w)\cdot \nabla v + (\bb\cdot\nabla w) v+c w v\,\d\xx.
\end{split}
\end{align*}
The bilinear form is clearly continuous, i.e., there exists a positive constant $\const{cont} > 0$ such that
\begin{align*}
\begin{split}
\dual{w}{v}_{\mathscr{P}}\le \const{cont}\norm{w}{H^1(\Omega)}\norm{v}{H^1(\Omega)}\\
\text{for all $v,w\in H^1(\Omega)$.}
\end{split}
\end{align*}
Additionally, we suppose ellipticity of $\dual{\cdot}{\cdot}_{\mathscr{P}}$ on $H_0^1(\Omega)$, i.e., there exists $\const{ell} > 0$ such that
\begin{align*}
\dual{v}{v}_{\mathscr{P}}\ge \const{ell}\norm{v}{H^1(\Omega)}^2\quad\text{for all }v\in H_0^1(\Omega).
\end{align*}
Note that ellipticity is for instance satisfied if the matrix $\AA$ is uniformly positive definite and  the vector $\bb \in  {\bf H}({\rm div},\Omega)$ satisfies that $-\frac12\,{\rm div}\,\bb+c\ge0$ almost everywhere in $\Omega$.

According to the Lax--Milgram theorem, for arbitrary  $f\in L^2(\Omega)$ problem~\eqref{eq:problem} admits a unique solution $u\in H_0^1(\Omega)$ to the weak formulation
\begin{align}\label{eq:weak fem}
 \dual{u}{v}_{\mathscr{P}} = \int_\Omega f v\,\d\xx
 \quad\text{for all }v\in H_0^1(\Omega).
\end{align}
Finally, we note that the additional regularity $\AA\in W^{1,\infty}(\Omega)^{\dph\times \dph}$ (instead of  only the natural assumption $\AA\in L^{\infty}(\Omega)^{\dph\times \dph}$) is only required for the well-posedness of the residual {\sl a~posteriori} error estimator, see Section~\ref{sec:estimator fem} below.

Let $\mathbb{S}\subset H_0^1(\Omega)$ be an arbitrary discrete subspace and let $U\in\mathbb{S}$ be the corresponding Galerkin approximation to the solution $u\in H_0^1(\Omega)$, i.e.,
\begin{align}\label{eq:Galerkin fem}
 \dual{U}{V}_{\mathscr{P}} = \int_\Omega f V\,\d\xx
 \quad\text{for all }V\in\mathbb{S}.
\end{align}
We note the Galerkin orthogonality
\begin{align*}
 \dual{u-U}{V}_{\mathscr{P}} = 0
 \quad\text{for all }V\in\mathbb{S},
\end{align*}
as well as the resulting C\'ea type quasi-optimality
\begin{align*}
 \norm{u-U}{H^1(\Omega)}
 \le C_{\text{C\'ea}}\min_{V\in\mathbb{S}}\norm{u-V}{H^1(\Omega)}
\end{align*}
with $C_{\text{C\'ea}} := {\const{cont}}/{\const{ell}}$.

\subsubsection{Isogeometric discretization} \label{sec:IGA-basics}
For the discretization of the model problem with the IGA method, we start with the case of a single-patch domain, and then generalize the method to the multi-patch case. 

\smallskip\paragraph{The single-patch case} Let us assume that $\Omega = {\bf F}(\hat \Omega)$, with a NURBS parametrization $\F$ of degree ${\bf p}_{\F}$ constructed from the knot vector $\mathbf{\kv}_{\F}$ as in Section~\ref{sec:parametrization_assumptions}. 
We consider a discrete space of splines $\spm \supseteq \hat{\mathbb{S}}_{{\bf p}_{\F}} (\mathbf{\kv}_{\F})$, which is obtained by refinement of the space used to build the parametrization. We note that both $h$-refinement and $p$-refinement can be applied, see \cite{hughes2005} for  details.

We will however use a milder assumption for the discrete space $\spm$, and allow to use a lower degree than for the parametrization, while the mesh and the continuity given by ${\bf F}$ must be respected.
In particular, we assume that $\hat \QQ_{\bf F}$ and $\hat \QQ$, the meshes respectively associated to the discrete spaces $\hat{\mathbb{S}}_{{\bf p}_{\F}} (\mathbf{\kv}_{\F})$ and $\spm$, are nested, in the sense that the corresponding sets of breakpoints satisfy $Z_{{\bf F},j} \subseteq Z_j$ for $j=1,\ldots, \dpa$. We also assume that the continuity of $\spm$ along the knot lines of $\hat \QQ_{\bf F}$ is always less or equal than the one of $\hat{\mathbb{S}}_{{\bf p}_{\F}} (\mathbf{\kv}_{\F})$. Note that this is always satisfied if $\spm \supseteq \hat{\mathbb{S}}_{{\bf p}_{\F}} (\mathbf{\kv}_{\F})$. Moreover, to obtain conforming spaces in $H^1(\Omega)$ we assume that the continuity across elements is not lower than $C^0$.

The discrete space in the physical domain is defined by push-forward using the NURBS parametrization, namely 
\begin{align}\label{eq:push S}
\mathbb{S}_{\bf p}({\bf \kv}) :=\big\{ V = \hat V \circ {\bf F}^{-1} : \hat V \in \spm \big\}.
\end{align}
We can easily define a basis for this space by push-forward of the B-spline basis functions, that is
\begin{align}\label{eq:push B}
\begin{split}
{\cal B}_{\bf p}({\bf \kv}) := 
  \big\{ B_{{\bf i},{\bf p}} = \hat B_{{\bf i},{\bf p}} \circ {\bf F}^{-1} : \hat B_{{\bf i},{\bf p}} \in \spbasis_{\bf p}({\bf \kv})  \big\}.
 \end{split}
\end{align}
For the solution of the discrete problem \eqref{eq:Galerkin fem}, we define the discrete space with vanishing boundary conditions
\[
\mathbb{S} := \mathbb{S}_{\bf p}({\bf \kv}) \cap H_0^1(\Omega).
\]
In practice, and thanks to the use of the open knot vectors, vanishing boundary conditions are enforced by removing the first and last basis functions from the univariate B-spline spaces. 

It is worth noting that the space $\spm$ is associated to a mesh in the parametric domain, which we denote by $\hat \QQ$ and which is a refinement of $\hat \QQ_{\F}$. As in \eqref{eq:IGA-mesh}, this mesh is mapped through ${\bf F}$ to define the mesh $\QQ$ of $\Omega$ associated to the space $\mathbb{S}$.




\begin{remark}
The assumption on the continuity along the knot lines of $\hat \QQ_{\bf F}$ is in fact a condition on the knots. Let us assume for simplicity the same degrees $p$ and $p_{\bf F}$ in every direction, and the same multiplicities of the internal knots, $m$ and $m_{\bf F}$, referring respectively to spaces $\spm$ and $\hat{\mathbb{S}}_{{\bf p}_{\F}} (\mathbf{\kv}_{\F})$. Then the condition reads
\[
p - m \le p_{\bf F} - m_{\bf F}.
\]
It is important to note that, if the condition is not respected, the optimal convergence rate may not be achieved, even for smooth solutions, see the numerical tests in \cite{Buffa_Sangalli_Vazquez}.
\end{remark}

\begin{remark}
In IGA, it is common to follow the isoparametric paradigm, and to define the discrete space as the push-forward of a NURBS space \cite{hughes2005}. Although our parametrization is constructed via NURBS, we have preferred to limit ourselves to (non-rational) spline spaces for the sake of clarity and to avoid the cumbersome presence of the weight during the mathematical analysis. The analysis of IGA with uniform NURBS discretizations has already been carried out in \cite{bbchs06}, see also \cite[Section~4]{bbsv14}. The results of this work can be extended to adaptive methods with rational splines without major (but notational) difficulties.
\end{remark}

\smallskip\paragraph{The multi-patch case}
For the definition of the multi-patch space we follow the same approach as in \cite[Section~3]{bbsv14}, see also \cite{kleissieti}. 
For each patch, let $\hat{\mathbb{S}}_{{\bf p}_m}({\bf T}_m)$ satisfy the same assumptions with respect to $\hat{\mathbb{S}}_{{\bf p}_{\F_{m}}}({\bf T}_{\F_{m}})$ as in the single-patch case.
By push-forward, we define the corresponding space ${\mathbb{S}}_{{\bf p}_m}({\bf \kv}_m)$ and its local basis ${\cal B}_{{\bf p}_m}({\bf \kv}_m)$ as in the single-patch case.
Then, the multi-patch discrete space is given as
\begin{align*}
\tilde{\mathbb{S}} := \big\{ V \in C^0(\Omega) : V|_{\Omega_m} \in {\mathbb{S}}_{{\bf p}_m}({\bf \kv}_m), \quad \\
\text{ for } m = 1, \ldots, M\big\},
\end{align*}
and finally the discrete space with vanishing boundary conditions is simply
\[
\mathbb{S} := \tilde{\mathbb{S}} \cap H_0^1(\Omega).
\]

Since each local space is associated to a mesh, which we denote by $\QQ_m$, we can define the multi-patch mesh analogously to \eqref{eq:multipatch-mesh}, i.e., $\QQ:= \bigcup_{m=1}^M \QQ_m$.

In order to construct a global basis for the multi-patch space, besides the assumptions on the parametrization given in Section~\ref{sec:multi-patch}, we need an analogous assumption to guarantee that the refined meshes are conforming. In particular, we assume that the following condition holds true:
\begin{enumerate}[(1)]
\renewcommand{\theenumi}{P\arabic{enumi}'}
\setcounter{enumi}{1}
\bf\item\rm \label{P:conforming-basis-space} 
For each $\beta_m \in {\cal B}_{{\bf p}_m}({\bf \kv}_m)$ such that $\beta_m|_{\Gamma_{m,m'}} \not = 0$,
there exists a unique function $\beta_{m'} \in {\cal B}_{{\bf p}_{m'}}({\bf \kv}_{m'})$ such that $\beta_m|_{\Gamma_{m,m'}} = \beta_{m'}|_{\Gamma_{m,m'}}$.
\end{enumerate}

With this assumption, we can build a basis of the multi-patch space $\mathbb{S}$ by gluing together functions of adjacent patches in a procedure which is analogous to the construction of the connectivity array in standard finite elements. To define a basis of $\tilde{\mathbb{S}}$, let us denote by $n$ the dimension of $\tilde{\mathbb{S}}$. We define for each patch a mapping
\[
\appl_m: {\cal B}_{{\bf p}_m}({\bf \kv}_m) \rightarrow \{1, \ldots, n\}\quad \text{for } m = 1, \ldots, M,
\]
in such a way that, for any $\beta_m \in {\cal B}_{{\bf p}_m}({\bf \kv}_m)$ and $\beta_{m'} \in {\cal B}_{{\bf p}_{m'}}({\bf \kv}_{m'})$ with $m\neq m'$,
\begin{align*}
\appl_m(\beta_m) = \appl_{m'}(\beta_{m'}) \iff \Gamma_{m,m'} \not = \emptyset \text{ and } \\
\beta_m |_{\Gamma_{m,m'}} = \beta_{m'} |_{\Gamma_{m,m'}}.
\end{align*}
Then, we define the basis of the multi-patch basis
\[
{\cal B} := \{ \basisfun_j: \, j = 1, \ldots, n \}, 
\]
where each basis function is given by
\[
\basisfun_j |_{\overline {\Omega_m}} := \left\{
\begin{array}{ll}
\beta_m  & \text{ if } \appl_m(\beta_m) = j,\\
0 & \text{otherwise}.
\end{array}
\right.
\]
The conditions described above guarantee that the basis functions are continuous at the interfaces, see an example in Figure~\ref{fig:multi-patch-function}.
\begin{figure}
\centering
\includegraphics[width=0.3\textwidth,trim=3cm 2cm 2.5cm 3cm, clip]{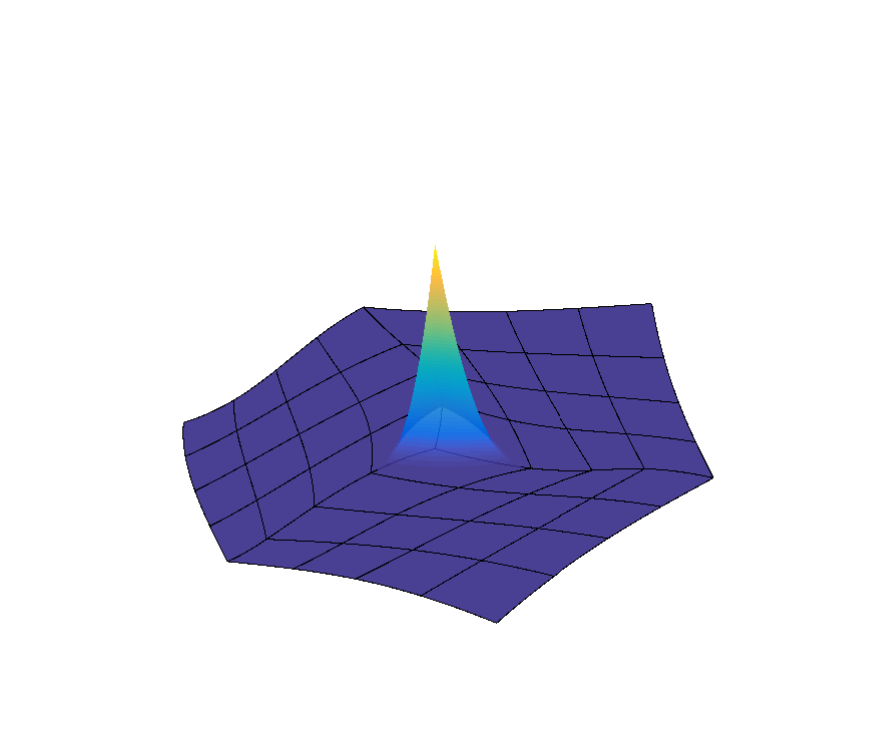}
\caption{An example of a $C^0$ basis function of the multi-patch space, defined in the same domain as in Figure~\ref{fig:multi-patch}.}
\label{fig:multi-patch-function}
\end{figure}
Once we have the basis for $\tilde{\mathbb{S}}$, a basis for $\mathbb{S}$ is easily constructed by removing the basis functions that do not vanish on the boundary similarly to the single-patch case.

\begin{remark}
The construction of splines with $C^1$ continuity (or higher) in multi-patch domains is an important subject of research not only in IGA but in general in computer aided geometric design. Different kinds of constructions have recently been proposed in the literature.
For the interested reader, we mention \cite{NgKaPe16,NgPe16,KaSaTa18,KaSaTa19,WZTSLMEH18,ToSpHu17,MaCi17}. The analysis of adaptive methods in multi-patch domains with high continuity is beyond the current state of the art, with preliminary steps in \cite{BrGiKaVa19}, and in particular beyond the scope of this paper.
\end{remark}

\subsubsection{{\sl A posteriori} error estimator}\label{sec:estimator fem}
Despite not having introduced the spline spaces with local refinement, we can already introduce the error estimator that will drive the adaptive refinement. Let the mesh $\QQ$ be defined as above, and let $Q\in\QQ$.
For almost every $\xx\in\partial Q\cap\Omega$ on the interior skeleton of the mesh, there exists a unique element $Q'\in\QQ$ with $\xx\in \partial Q'$ and $Q' \not = Q$.
We denote the corresponding outer normal vectors by $\normal$ and $\normal'$. 
With the notation 
\begin{align*}
\begin{split}
\mathscr{D}_{\normal} (\cdot):=(\AA\nabla(\cdot))\cdot\normal,\quad
\mathscr{D}_{\normal'} (\cdot):=(\AA\nabla(\cdot))\cdot\normal',
\end{split}
\end{align*}
we define 
the \textit{normal jump} as 
\begin{align*}
[\mathscr{D}_{\normal} U](\xx)
: = (\mathscr{D}_{\normal} U|_{Q})(\xx)+(\mathscr{D}_{\normal'}   U|_{Q'})(\xx).
\end{align*}
With this definition, we employ the \textit{weighted-{residual {\sl a~posteriori} error estimator}}
\begin{subequations}\label{eq:eta}
\begin{align}
\begin{split}
 &\eta := \eta(\QQ)\\
&\quad\text{with}\quad 
 \eta(\SS)^2:=\sum_{Q\in\SS} \eta(Q)^2
 \text{ for all }\SS\subseteq\QQ,
\end{split}
\end{align}
where, for all $Q\in\QQ$ with element size $h_\elem$, the local refinement indicators read 
\begin{align}
\eta(Q)^2:=h_Q^{2} \norm{f-\mathscr{P}U}{L^2(Q)}^2+h_Q \norm{[\mathscr{D}_{\normal} U]}{L^2(\partial Q\cap \Omega)}^2.
\end{align}
\end{subequations}
We refer, e.g., to the monographs~\cite{ao00,verfuerth13} for the analysis of the residual {\sl a~posteriori} error estimator~\eqref{eq:eta} in the frame of standard FEM with piecewise polynomials of fixed order.

\begin{remark}
The additional regularity $\AA$ $\in$ $W^{1,\infty}(\Omega)^{\dph\times \dph}$ (instead of only $\AA\in L^\infty(\Omega)^{\dph\times \dph}$) is needed to ensure that  $\mathscr{D}_{\normal}(\cdot)$ is well-defined. 
\end{remark}

\begin{remark}\label{rem:C11}
If $\mathbb{S}\subset C^1(\Omega)$, then the jump contributions in~\eqref{eq:eta} vanish and $\eta(Q)$ consists only of the volume residual, i.e., $\eta(Q)^2 = h_Q^2 \norm{f-\mathscr{P}U}{L^2(Q)}^2$.
\end{remark}%

\subsection{Isogeometric analysis for BEM (IGABEM)} \label{sec:IGABEM-intro}
The potential benefits of using IGA for the solution of boundary integral equations were already mentioned in the conclusions of \cite{hughes2005}, but it has only been considered first in \cite{pgkbf09}.
The research on IGABEM has steadily grown since then, although not as fast as for IGAFEM, with applications in acoustics \cite{SiScTaThLi14,DoHaKuScWo18,chen20,vk20}, elasticity~\cite{bdl15,nguyen17}, electromagnetics \cite{Vazquez20124757,SiLiVaEv18,DoKuScWo19}, lifting flow \cite{chouliaras21}, potential flow \cite{HeArDe14,KoGiPoKa15,KoGiPoKa17}, and solid mechanics \cite{ScSiEvLiBoHuSe13,marussig15}, see also the recent book \cite{Beer2020} for a comprehensive survey of the topic and a complete review of the existing literature. 
An implementation of (non-adaptive) IGABEM is available in the open-source library Bembel \cite{BEMBEL}. 
Although some of the previously mentioned works consider locally refined T-splines, the mathematical research on adaptive IGABEM methods is rather limited.
Results for the two-dimensional case are found in \cite{fgp15,fghp16,fghp17,fgps19,gps19}, where \cite{fgp15} is also the first work that considers Galerkin instead of collocation IGABEM.
The three-dimensional case has only recently been considered in \cite{gantner17,gp20,gp20+}.

\subsubsection{Sobolev spaces for BEM}\label{sec:sobolev}
For arbitrary $\dph\ge2$, let $\Omega\subset\R^\dph$ be a bounded Lipschitz domain as in \cite[Definition~3.28]{mclean00} and $\Gamma:= \partial\Omega$ its boundary.
In practice, $\Gamma$ is a multi-patch domain defined as in Section~\ref{sec:multi-patch} with $\dpa=\dph-1$.
Before we give the model problem and discuss its discretization, we have to introduce the involved Sobolev spaces on $\Gamma$.
For $\sigma\in[0,1]$, we define the Hilbert spaces $H^{\pm\sigma}(\Gamma)$ with corresponding norms as in \cite[page~99]{mclean00} by use of Bessel potentials on $\R^{\dph-1}$ and liftings via bi-Lipschitz mappings 
that describe $\Gamma$.
For $\sigma=0$, this procedure yields that $H^0(\Gamma)=L^2(\Gamma)$ with equivalent norms.
Therefore, we set $\norm{\cdot}{H^0(\Gamma)}:=\norm{\cdot}{L^2(\Gamma)}$.

For $\sigma\in(0,1]$,  any measurable subset $\omega\subseteq\Gamma$, and all $v\in H^\sigma(\Gamma)$, we define
the associated   Sobolev--Slobodeckij norm
\begin{align*}
 \norm{v}{H^{\sigma}(\omega)}^2
 := \norm{v}{L^2(\omega)}^2
 + |v|_{H^{\sigma}(\omega)}^2
 \end{align*}
 with 
 \begin{align*}
 |v|_{H^{\sigma}(\omega)}^2 :=\begin{cases} \int_\omega\int_\omega\frac{|v(\xx)-v(\yy)|^2}{|\xx-\yy|^{\dph-1+2\sigma}}\,\d\xx \,\d\yy&\text{ if }\sigma\in(0,1),\\ \norm{\nabla_\Gamma v}{L^2(\omega)}^2&\text{ if }\sigma=1.\end{cases}
\end{align*}
Here, $\nabla_\Gamma(\cdot)$ denotes the usual (weak) surface gradient which is  well-defined for almost all $\xx\in\Gamma$.
It is well known that $\norm{\cdot}{H^\sigma(\Gamma)}$ provides an equivalent norm on $H^\sigma(\Gamma)$, see, e.g., \cite[Lemma~2.19]{steinbach08} and \cite[Theorem~3.30 and page 99]{mclean00} for $\sigma\in(0,1)$ and \cite[Theorem~2.28]{mitscha14} for $\sigma=1$.

For $\sigma\in(0,1]$, $H^{-\sigma}(\Gamma)$ is a realization of the dual space of $H^{\sigma}(\Gamma)$ according to \cite[Theorem~3.30 and page~99]{mclean00}.
With the duality bracket $\dual{\cdot}{\cdot}$, we define the following equivalent norm on $H^{-\sigma}(\Gamma)$
\begin{align*}
\norm{\psi}{H^{-\sigma}(\Gamma)}&:=\sup \set{\dual{v}{\psi}}{v\in H^\sigma(\Gamma), \norm{v}{H^\sigma(\Gamma)}=1} \notag \\
&\hspace{28mm}\text{for all } \psi\in H^{-\sigma}(\Gamma).
\end{align*}

In \cite[page~76]{mclean00}, it is stated that $H^{\sigma_1}(\Gamma)\subset H^{\sigma_2}(\Gamma)$ for $-1\le \sigma_1<\sigma_2\le 1$,  where the inclusion is continuous, dense, and compact.
In particular, $H^{\sigma}(\Gamma)\subset L^2(\Gamma)\subset H^{-\sigma}(\Gamma)$ forms a Gelfand triple in the sense of \cite[Section~2.1.2.4]{ss11} for all $\sigma\in(0,1]$, where $\psi\in L^2(\Gamma)$ is interpreted as a function in $H^{-\sigma}(\Gamma)$ via 
\begin{align*}
\begin{split}
\dual{v}{\psi}:=\dual{v}{\psi}_{L^2(\Gamma)}=\int_\Gamma v\,\psi \,\d\xx \\
\text{for all }v\in H^\sigma(\Gamma),\psi\in L^2(\Gamma).
\end{split}
\end{align*}

{The spaces $H^\sigma(\Gamma)$ can also be defined as trace spaces or via interpolation, where the resulting norms are always equivalent with constants which depend only on the dimension $\dph$ and the boundary $\Gamma$.}
For a more detailed introduction to Sobolev spaces on the boundary, the reader is referred to \cite{mclean00,ss11,steinbach08,hw08}.

\subsubsection{Model problem and Galerkin approximation}\label{sec:model problem bem}
Again, we consider a general second-order linear  PDE on the $\dph$-dimensional bounded Lipschitz domain $\Omega$ with partial differential operator  
\begin{align*}
\begin{split}
\mathscr{P}u:=-\div (\AA\nabla u)+\bb\cdot\nabla u +cu,
\end{split}
\end{align*}
where the coefficients $\AA\in\R^{\dph\times \dph}, \bb\in\R^\dph$, and $c\in\R$
now additionally supposed to be constant. 
Moreover, we assume that $\AA$ is symmetric and positive definite.

Let $G:\R^\dph\setminus\{0\}\to \R$ be a corresponding fundamental solution in the sense of \cite[page~198]{mclean00}, i.e., a distributional solution of $\mathscr{P}G=\delta$, where $\delta$ denotes the Dirac delta function.
For $\psi\in L^\infty(\Gamma)$, we define the \textit{single-layer operator}
as 
\begin{align*}
({\mathscr{V}}\psi)(\xx):=\int_{\Gamma} G(\xx-\yy) \psi(\yy) \,\,\d\yy\quad\text{for all }\xx\in\Gamma.
\end{align*}
According to \cite[pages 209 and 219--220]{mclean00} and \cite[
Corollary~3.38]{hmt10}, 
 this operator can be extended for arbitrary $\sigma\in(-1/2,1/2$] to a bounded linear operator 
\begin{align}\label{eq:single layer operator}
\mathscr{V}:
H^{-1/2+\sigma}(\Gamma)\to H^{1/2+\sigma}(\Gamma).
\end{align}
In \cite[Theorem~7.6]{mclean00}, it is stated that $\mathscr{V}$ is always elliptic up to some compact perturbation.
 We assume that it is elliptic even without perturbation, i.e., 
\begin{align}\label{eq:ellipticity bem}
\hspace{-1mm}\dual{\mathscr{V}\psi}{\psi}
\ge \const{ell}\norm{\psi}{H^{-1/2}(\Gamma)}^2\text{ for all }\psi\in H^{-1/2}(\Gamma).
\end{align}
This is particularly satisfied for the Laplace problem or for the linear elasticity problem, where the case $\dph=2$ requires an additional scaling of the geometry $\Omega$, see, e.g., \cite[Chapter~6]{steinbach08}.
Moreover, the bilinear form $\dual{\mathscr{V}\,\cdot}{\cdot}$
is continuous due to \eqref{eq:single layer operator}, i.e.,  it holds with $\const{cont}:=\norm{\mathscr{V}}{H^{-1/2}(\Gamma)\to H^{1/2}(\Gamma)}$ that 
\begin{align}
\begin{split}\label{eq:continuity bem}
\dual{\mathscr{V}\psi}{\xi}
\le \const{cont}\norm{\psi}{H^{-1/2}(\Gamma)}\norm{\xi}{H^{-1/2}(\Gamma)} \\
\text{for all }\psi,\xi\in H^{-1/2}(\Gamma).
\end{split}
\end{align}

Given a right-hand side $f\in H^{1}(\Gamma)$, 
we consider the weakly-singular boundary integral equation
\begin{align}\label{eq:strong}
 \mathscr{V}\phi = f.
\end{align}
Such equations arise from the solution of Dirichlet problems of the form $\mathscr{P} u=0$ in $\Omega$ with $u=g$ on $\Gamma$ for some $g\in H^{1}(\Gamma)$, see, e.g., \cite[pages 226--229]{mclean00}. 
The normal derivative $\phi:= (\AA\nabla u)\cdot\normal$ of the  weak solution $u$ then satisfies the integral equation \eqref{eq:strong} with $f:=(\mathscr{K}+1/2)g$, i.e.,
\begin{equation}\label{eq:Symmy interior}
\mathscr{V}\phi =(\mathscr{K}+1/2) g, 
\end{equation}
where 
\begin{align}\label{eq:double layer mapping}
\mathscr{K}: H^{1/2}(\Gamma)\to H^{1/2}(\Gamma)
\end{align}
denotes the \textit{double-layer operator} \cite[pages~218--223]{mclean00}.
If $\Gamma$ is piecewise smooth and if $g\in L^\infty(\Gamma)$, for all ${\bf x}\in\Gamma$ where $\Gamma$ is locally smooth and $g$ is continuous there holds the representation 
\begin{align*}
\begin{split}
\mathscr{K}g({\bf x})= \int_{\Gamma_{}} g({\bf y})
\big(\AA\nabla_{\bf y}G({\bf x},{\bf y}) + \bb \,G({\bf x},{\bf y}) \big)\cdot\nu({\bf y})
 \,{\rm d}{\bf y};
\end{split}
\end{align*}
see \cite[Section~3.3.3]{ss11}.
Due to \eqref{eq:ellipticity bem}--\eqref{eq:continuity bem} the Lax--Milgram lemma guarantees existence
and uniqueness of the solution $\phi\in H^{-1/2}(\Gamma_{})$ of the equivalent variational formulation of~\eqref{eq:strong}
\begin{align*}
\dual{\mathscr{V}\phi}{\psi}
=\dual{f}{\psi}
 \quad\text{for all }\psi\in H^{-1/2}(\Gamma_{}).
\end{align*}
In particular, we see that $\mathscr{V}:H^{-1/2}(\Gamma)\to H^{1/2}(\Gamma)$ is an isomorphism.

In the Galerkin BEM, the test
space $H^{-1/2}(\Gamma_{})$ is replaced by some discrete subspace  $\mathbb{S}\subset {L^{2}(\Gamma_{})}\subset H^{-1/2}(\Gamma_{})$.
Again, the Lax--Milgram lemma applies and guarantees the existence and uniqueness of the solution
$\Phi\in\mathbb{S}$ of the discrete variational formulation
\begin{align}\label{eq:discrete}
\dual{\mathscr{V}\Phi}{\Psi}
 = \dual{f}{\Psi}
 \quad\text{for all }\Psi\in\mathbb{S}. 
\end{align}
In fact, $\Phi$ can be computed by solving a linear system of equations.
Note that \eqref{eq:single layer operator} even implies that $\mathscr{V} \Psi\in H^1(\Gamma)$ for arbitrary $\Psi\in \mathbb{S}$.
The additional regularity $f\in H^1(\Gamma)$ instead of $f\in H^{1/2}(\Gamma)$ is only needed to define the residual error estimator~\eqref{eq:eta bem} below, which requires that $f-\mathscr{V}\in H^1(\Gamma)$. As for the FEM problem, we also note the Galerkin orthogonality
\begin{align}\label{eq:galerkin bem}
 \dual{f-\mathscr{V}\Phi}{\Psi} = 0
 \quad\text{for all }\Psi\in\mathbb{S},
\end{align}
as well as the resulting C\'ea-type quasi-optimality
\begin{align*}
\hspace{-0.1mm} \norm{\phi-\Phi}{H^{-1/2}(\Gamma)}
 \le C_{\text{C\'ea}}\min_{\Psi\in\mathbb{S}}\norm{\phi-\Psi}{H^{-1/2}(\Gamma)},
\end{align*}
where $ C_{\text{C\'ea}} := {\const{cont}}/{\const{ell}}$.
For a more detailed introduction to boundary integral equations and BEM, the reader is referred to the monographs \cite{mclean00,ss11,steinbach08,hw08}.



\subsubsection{Isogeometric discretization} \label{sec:isogeometric-discretization-BEM}
For the solution of the discrete problem with isogeometric methods, we  assume that the boundary of the domain $\Gamma = \partial \Omega \subset \mathbb{R}^\dph$ (and not necessarily $\Omega$) is defined as a multi-patch geometry through NURBS parametrizations.
More precisely, we suppose that $\Gamma = \bigcup_{m=1}^M \overline {\Gamma_m}$, where 
\[
{\bf F}_m :(0,1)^{\dph-1} \rightarrow \Gamma_m \subset \mathbb{R}^\dph
\]
is a NURBS parametrization and the assumptions of Section~\ref{sec:parametrization_assumptions} are valid.
In particular, each ${\bf F}_m$ is a bi-Lipschitz homeomorphism. 
Moreover,  we suppose the properties \eqref{P:conforming-mesh}--\eqref{P:flattened} regarding the conformity of the meshes in multi-patch domains given in Section~\ref{sec:multi-patch} hold true. 

On each patch, we first define the local space of mapped splines ${\mathbb{S}}_{{\bf p}_m}({\bf \kv}_m)$ with the local basis ${\cal B}_{{\bf p}_m}({\bf \kv}_m)$ via push-forward as in the IGAFEM case \eqref{eq:push S}--\eqref{eq:push B}.
Then, we define the discrete isogeometric space as
\begin{align*}
\mathbb{S} := \{V \in L^2(\Gamma) : V|_{\Gamma_m} \in {\mathbb{S}}_{{\bf p}_m}({\bf \kv}_m), \quad \\
\text{ for } m = 1, \ldots, M \}.
\end{align*}
Note that, in contrast to IGAFEM, continuity of the discrete functions at the interfaces is not required for the weakly-singular boundary integral equation~\eqref{eq:strong} as $\mathbb{S}$ only needs to be contained in $L^2(\Gamma)$.
A basis for this space is clearly given by 
\[
{\cal B} := \bigcup_{m=1}^M {\cal B}_{{\bf p}_m}({\bf \kv}_m).
\]

\begin{remark}
In contrast to weakly-singular integral equations, hypersingular integral equations, which result from Neumann problems (see, e.g., \cite[Chapter~7]{mclean00}), require continuous trial functions. Assuming also the conformity property \eqref{P:conforming-basis-space}, corresponding basis functions can be constructed as for IGAFEM in Section~\ref{sec:IGA-basics}.
\end{remark}



\subsubsection{{\sl A posteriori} error estimator}\label{sec:estimator bem}
Let $\QQ$ be the mesh on $\Gamma$, defined as above.
Due to the regularity assumption $f\in H^1(\Gamma)$, the mapping property \eqref{eq:single layer operator}, and $\mathbb{S}\subset L^2(\Gamma)$, the residual satisfies that $f-\mathscr{V}\Psi\in H^1(\Gamma)$ for all $\Psi\in\mathbb{S}$.
This allows to employ the  \textit{weighted-residual {\sl a~posteriori} error estimator}
\begin{subequations}\label{eq:eta bem}
\begin{align}
\begin{split}
& \eta := \eta(\QQ)\\
 &\quad\text{with}\quad 
 \eta(\SS)^2:=\sum_{Q\in\SS} \eta(Q)^2
 \text{ for all }\SS\subseteq\QQ,
 \end{split}
\end{align}
where, for all $Q\in\QQ$ with element size $h_\elem$, the local refinement indicators read 
\begin{align}
\eta(Q)^2:=h_Q \seminorm{f-\mathscr{V}\Phi}{H^1(Q)}^2.
\end{align}
\end{subequations}
This estimator goes back to the works \cite{cs96,c97}, where reliability 
is proved for standard 2D BEM with piecewise polynomials on polygonal geometries, while the corresponding result for standard 3D BEM is found in \cite{cms01}.
The recent work \cite{gp20} generalizes these results to PDEs beyond the Laplace equation and beyond standard discretizations based on piecewise polynomials. 

\newpage


\section{Splines on adaptive meshes} \label{sec:adaptive-splines}

The design of adaptive isogeometric methods requires suitable adaptive spline spaces that enable local mesh refinement. Here, we focus on two of the main solutions that break the structure of standard multivariate tensor-product splines: \emph{hierarchical splines} in Section~\ref{subsec:hb} and \emph{T-splines} in Section~\ref{subsec:tsplines}. 
We stress that, at the moment and up to our knowledge, a thorough analysis on optimal convergence of resulting adaptive algorithms is only available for these two.
Section~\ref{subsec:others} finally collects alternative adaptive spline models and briefly comments on them.

\subsection{Hierarchical splines}\label{subsec:hb}

Hierarchical spline surfaces were introduced in \cite{forsey88} by considering a sequence of overlays to enable an efficient local editing of the geometric model. A simple selection algorithm to properly identify the
B-splines at different refinement levels needed to define a suitable basis for hierarchical spline spaces was proposed 
in \cite{kraft1997,kraft1998}. More recently, a slightly different hierarchical B-spline basis was proposed in \cite{vgjs11} and since then the hierarchical approach was widely used by different authors in IGA, see, e.g., \cite{vgjs11,schillinger2012,scott2014a,hcsh17}. In order to overcome some limitations of hierarchical B-splines, the truncated basis for the same hierarchical spline space was introduced in \cite{gjs12} leading to the definition of {\it truncated hierarchical B-splines} (THB-splines). Their application in IGA has been investigated by several authors for second order \cite{giannelli2016,hennig2016,dkrr18,bracco2019} and fourth order PDEs \cite{hennig2016,coradello20c}, and also for trimmed domains \cite{deprenter20,qarariyah19,coradello20,coradello20b}. Implementation aspects related to (T)HB-splines were addressed in \cite{bc13,kiss2014a,bm17,bgv18,garau2018}.


\subsubsection{Definition and properties}
\label{subsec:def hb}

Let 
\begin{equation}
\spmh{0} \subset \spmh{1} \subset \ldots \subset \spmh{N-1}
\label{eq:nestedspaces}
\end{equation}
be a nested sequence of $N$ tensor-product spline spaces $\spmh{\ell}$, for $\ell = 0, \ldots, N-1$, defined without loss of generality on the open hyper-cube $\hat{\Omega} := (0,1)^\dpa$. 

At any level $\ell$, we consider the B-spline basis $\spbasis^\ell:=\spbasis_{{\bf p}}({\bf \kv}^\ell)$ of degree $\mathbf{p}$ defined on the rectilinear grid $\hat{\cal Q}^\ell$, analogously to the one-level case described in Section~\ref{sec:splines-tensor}. Any (non-empty) element $\hat{Q}$ of the grid $\hat{\cal Q}^\ell$ is the Cartesian product of $\dpa$ open intervals defined by consecutive breakpoints. 
We abbreviate its level $\levelT{\widehat Q}:=\ell$.
The knot vector $\kv^\ell_i$ in the coordinate direction $i$, for $i=1,\ldots,\dpa$, is associated to the
grid at level $\ell$ and contains non-decreasing real numbers so that each breakpoint $\brkpnt_j^\ell$ appears in the knot vector as many times as specified by a certain multiplicity.  For $\dpa = 1$, an example of grids and B-spline bases of three different levels is shown in Figure~\ref{fig:exm1D}.

\begin{figure}[!t]\centering
\subfigure[$\hat{\cal Q}^0$, $\hat{\cal Q}^1$, and $\hat{\cal Q}^2$ (from top to bottom)]{
\includegraphics[width=0.49\textwidth, trim=8cm 0cm 4cm 0cm, clip]{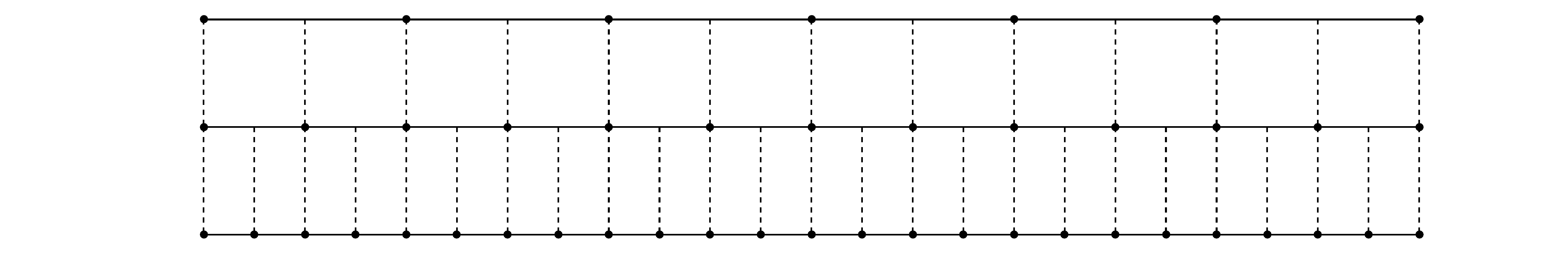}}
\subfigure[$\hat{\cal B}^0$]{
\includegraphics[width=0.49\textwidth, trim=8cm 0cm 4cm 0cm, clip]{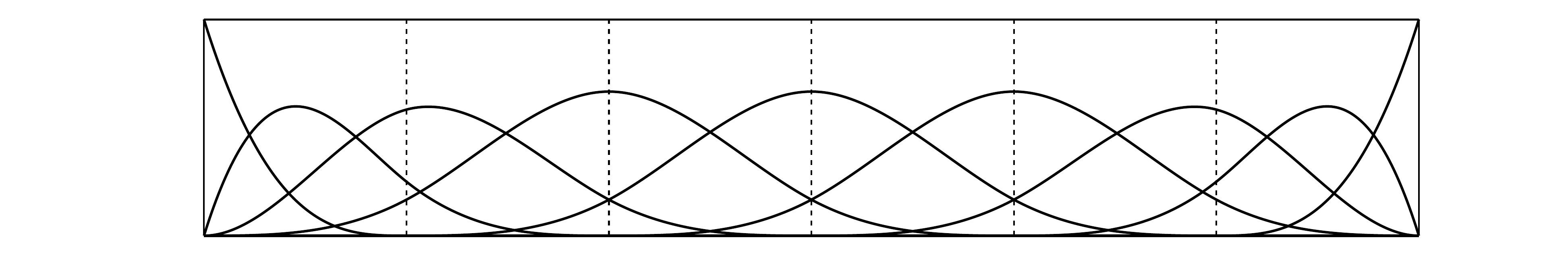}}
\subfigure[$\hat{\cal B}^1$]{
\includegraphics[width=0.49\textwidth, trim=8cm 0cm 4cm 0cm, clip]{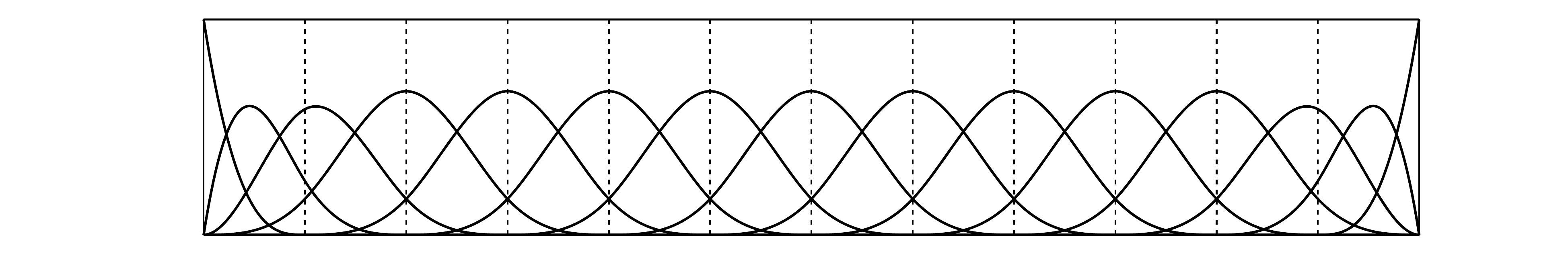}}
\subfigure[$\hat{\cal B}^2$]{
\includegraphics[width=0.49\textwidth, trim=8cm 0cm 4cm 0cm, clip]{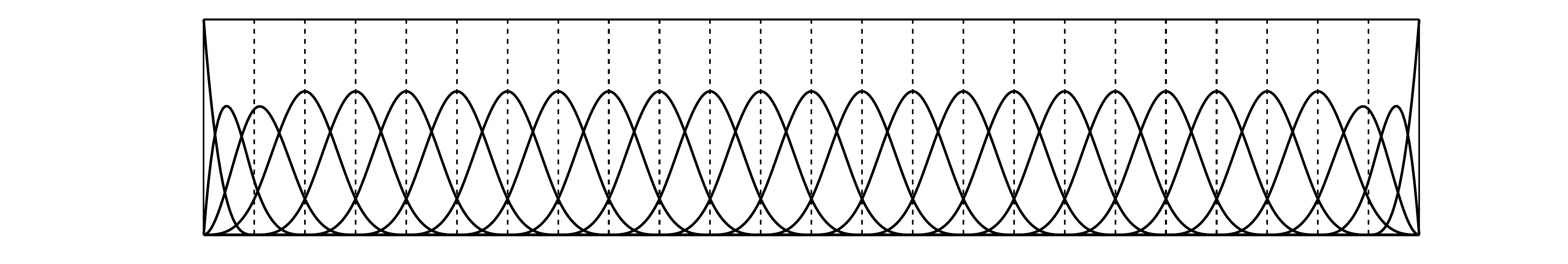}}
\caption{An example of grids (a) of three hierarchical levels for $\dpa=1$. The univariate B-splines of degree 3 defined on level 0, 1 and 2 are shown in (b), (c) and (d), respectively.
All internal knots have multiplicity one.}
\label{fig:exm1D}
\end{figure}

We assume open knot vectors in any direction at level $0$ and multiplicities of internal knots between one and $p_i$. To guarantee the nested nature of the spline spaces given by \eqref{eq:nestedspaces}, we also assume 
dyadic mesh refinement between consecutive hierarchical levels so that an element of level $\ell$ is uniformly refined in $2^\dpa$ elements of level $\ell+1$, see Figure~\ref{fig:exm1D} and \ref{fig:hmesh} for $\dpa=1$ and $\dpa=2$, respectively. In addition, any newly inserted knot appears with multiplicity one.

\begin{remark} Note that more general refinement possibilities can also be covered within the hierarchical spline model \cite{giannelli2014}.
\end{remark}

\color{black}
\begin{figure}[!t]\centering
\subfigure[$\hat{\cal Q}^0$ and $\hat{\Omega}^0$]{
\begin{tikzpicture}[scale=0.725]
\draw [fill=lightgray] (0,0) rectangle (5,5);
\draw (0,0) grid (5,5);
\end{tikzpicture}}
\hspace*{.5cm}
\subfigure[$\hat{\cal Q}^1$ and $\hat{\Omega}^1$]{
\begin{tikzpicture}[scale=0.725]
\draw [fill=lightgray] (0,0) rectangle (2,2);
\draw [fill=lightgray] (1,1) rectangle (3,3);
\draw [fill=lightgray] (2,2) rectangle (4,4);
\draw [fill=lightgray] (3,3) rectangle (5,5);
\draw (0,0) grid (5,5);
\foreach \a in {1,3,5,7,9}
	\draw (0,\a/2) -- (5,\a/2);
\foreach \a in {1,3,5,7,9}
	\draw (\a/2,0) -- (\a/2,5);
\end{tikzpicture}}\\
\subfigure[$\hat{\cal Q}^2$ and $\hat{\Omega}^2$]{
\begin{tikzpicture}[scale=0.725]
\draw [fill=lightgray] (0,0) rectangle (1,1);
\draw [fill=lightgray] (1/2,1/2) rectangle (3/2,3/2);
\draw [fill=lightgray] (2/2,2/2) rectangle (4/2,4/2);
\draw [fill=lightgray] (3/2,3/2) rectangle (5/2,5/2);
\draw [fill=lightgray] (4/2,4/2) rectangle (6/2,6/2);
\draw [fill=lightgray] (5/2,5/2) rectangle (7/2,7/2);
\draw [fill=lightgray] (6/2,6/2) rectangle (8/2,8/2);
\draw [fill=lightgray] (7/2,7/2) rectangle (9/2,9/2);
\draw [fill=lightgray] (8/2,8/2) rectangle (10/2,10/2);
\draw (0,0) grid (5,5);
\foreach \a in {1,3,5,7,9} \draw (0,\a/2) -- (5,\a/2);
\foreach \a in {1,3,5,7,9,11,13,15,17,19} \draw (0,\a/4) -- (5,\a/4);
\foreach \a in {1,3,5,7,9} \draw (\a/2,0) -- (\a/2,5);
 \foreach \a in {1,3,5,7,9,11,13,15,17,19} \draw (\a/4,0) -- (\a/4,5);
\end{tikzpicture}}
\hspace*{.5cm}
\subfigure[$\hat\hmesh$]{
\begin{tikzpicture}[scale=0.725]
\draw (0,0) grid (5,5);
\foreach \a in {1} \draw (0,\a/2) -- (2,\a/2); 
\foreach \a in {3} \draw (0,\a/2) -- (3,\a/2);
\foreach \a in {5} \draw (1,\a/2) -- (4,\a/2);
\foreach \a in {7} \draw (2,\a/2) -- (5,\a/2);
\foreach \a in {9} \draw (3,\a/2) -- (5,\a/2);
\foreach \a in {1}\draw (\a/2,0) -- (\a/2,2);
\foreach \a in {3}\draw (\a/2,0) -- (\a/2,3);
\foreach \a in {5}\draw (\a/2,1) -- (\a/2,4);
\foreach \a in {7}\draw (\a/2,2) -- (\a/2,5);
\foreach \a in {9}\draw (\a/2,3) -- (\a/2,5);
\foreach \a in {1} \draw (0,\a/4) -- (1,\a/4); 
\foreach \a in {3} \draw (0,\a/4) -- (1+1/2,\a/4);
\foreach \a in {5} \draw (1/2,\a/4) -- (2,\a/4);
\foreach \a in {7} \draw (1,\a/4) -- (5/2,\a/4);
\foreach \a in {9} \draw (1.5,\a/4) -- (3,\a/4);
\foreach \a in {11} \draw (2,\a/4) -- (3.5,\a/4);
\foreach \a in {13} \draw (2.5,\a/4) -- (4,\a/4);
\foreach \a in {15} \draw (3,\a/4) -- (4.5,\a/4);
\foreach \a in {17} \draw (3.5,\a/4) -- (5,\a/4);
\foreach \a in {19} \draw (4,\a/4) -- (5,\a/4);
\foreach \a in {1}\draw (\a/4,0) -- (\a/4,1);
\foreach \a in {3}\draw (\a/4,0) -- (\a/4,1.5);
\foreach \a in {5}\draw (\a/4,0.5) -- (\a/4,2);
\foreach \a in {7}\draw (\a/4,1) -- (\a/4,2.5);
\foreach \a in {9}\draw (\a/4,1.5) -- (\a/4,3);
\foreach \a in {11}\draw (\a/4,2) -- (\a/4,3.5);
\foreach \a in {13}\draw (\a/4,2.5) -- (\a/4,4);
\foreach \a in {15}\draw (\a/4,3) -- (\a/4,4.5);
\foreach \a in {17}\draw (\a/4,3.5) -- (\a/4,5);
\foreach \a in {19}\draw (\a/4,4) -- (\a/4,5);
\end{tikzpicture}}
\caption{An example of grids and domains (gray regions) of levels 0 (a), 1 (b), 2 (c) for $\dpa=2$. The hierarchical mesh is also shown (d).}
\label{fig:hmesh}
\end{figure}
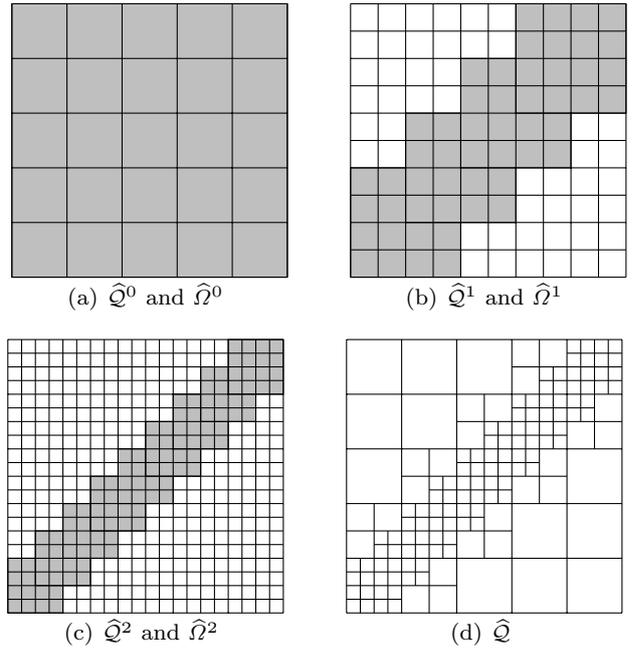 

In order to define the spline hierarchy, we consider a nested sequence of closed subsets of ${\hat\Omega}^0 := \overline{\hat\Omega}$, given by
\[
{\hat\Omega}^0\supseteq{\hat\Omega}^1\supseteq\ldots\supseteq\hat{\Omega}^{N-1}
\supseteq\hat{\Omega}^{N} = \emptyset,
\]
where we assume that $\hat{\Omega}^\ell$ is the union of the closure of elements of level $\ell-1$.
By considering 
the set of active elements at level $\ell$, for $\ell=0,\ldots,N-1$, 
we can define the \emph{hierarchical mesh} as follows:
\begin{align}\label{eq:hmesh}
\begin{split}
{\hat{\cal Q}} := \big\{
\hat{Q}\in \hat{\cal Q}^\ell :\,& 
\hat{Q} \,\subseteq\,\hat{\Omega}^\ell \wedge
\hat{Q} \,\not\subseteq\,\hat{\Omega}^{\ell+1},\\ 
&\qquad\quad\ell  =0,\ldots,N-1\big\}.
\end{split}
\end{align}
For $\dpa=2$, an example of domain hierarchy on three refinement levels is shown in Figure~\ref{fig:hmesh}.

We say that a mesh $\hat\QQ_\fine$ is a \emph{refinement} of $\hat\QQ$, and we denote it by $\hat\QQ \preceq \hat\QQ_\fine$ (or $\hat \QQ_\fine \succeq \hat \QQ$), if it is obtained from $\hat \QQ$ by successive splitting via dyadic refinement of some of its elements. Note that, under our assumptions, the fine mesh is associated to an enlargement of the subdomains $(\Omegap^\ell_\fine)_{\ell=0, \ldots,N_\fine-1}$, such that $N \le N_\fine$, $\Omegap^0 = \Omegap^0_\fine$, and $\Omegap^\ell \subseteq \Omegap^\ell_\fine$ for $\ell=1, \ldots, N$.

Given a hierarchical mesh $\hat\hmesh$, the set of \emph{hierarchical B-splines (HB-splines)} $\hat\hbbasis_{\bf p}(\hat\hmesh,{\bf \kv}^{0}):=\hat\hbbasis^{N-1} $ can be constructed according to the following steps:
\begin{enumerate}
\item $\hat\hbbasis^0 := \spbasis^0$;
\item for $\ell=0,\ldots, N-2$ 
\[
\hat\hbbasis^{\ell+1} := \hat\hbbasis_A^{\ell+1}  \cup \hat\hbbasis_B^{\ell+1},
\]
\end{enumerate}
where
\begin{align*}
\hat\hbbasis_A^{\ell+1}&:= \left\{
\hat{B}_{{\bf i},{\bf p}}^\ell \in \hat\hbbasis^{\ell} : \supp (\hat{B}_{{\bf i},{\bf p}}^\ell) \not\subseteq\hat\Omega^{\ell+1}\right\},\\
\hat\hbbasis_B^{\ell+1}&:= \left\{
\hat{B}_{{\bf i},{\bf p}}^{\ell+1} \in \spbasis^{\ell+1} : \supp (\hat{B}_{{\bf i},{\bf p}}^{\ell+1}) \subseteq\hat\Omega^{\ell+1}\right\}.
\end{align*}
Steps 1--2 define a selection mechanism which activates and deactivates B-splines at different levels of resolution by taking into account the hierarchical domain configuration.
After initializing the set of hierarchical B-splines with the B-splines of level $0$, for any subsequent level $\ell$, the set $\hat\hbbasis^{\ell+1}$ of HB-splines of level $\ell+1$ includes
\begin{itemize}
\item B-splines of coarser levels whose support is not contained in $\hat\Omega^{\ell+1}$ ($\hat\hbbasis_A^{\ell+1}$);
\item B-splines of level $\ell+1$ whose support is contained in $\hat\Omega^{\ell+1}$ ($\hat\hbbasis_B^{\ell+1}$).
\end{itemize}
Note that the HB-spline basis $\hat\hbbasis_{\bf p}(\hat\hmesh,{\bf \kv}^{0})$ with respect to the mesh $\hat\hmesh$ can also be defined as 
\begin{align*}
\hat\hbbasis_{\bf p}(\hat\hmesh,{\bf \kv}^{0}) &= 
\left\{
 \hat{B}_{{\bf i},{\bf p}}^\ell \in {\hat{\cal B}}^\ell : 
\supp  (\hat{B}_{{\bf i},{\bf p}}^\ell)\subseteq{\hat\Omega}^\ell\right. \\
&\quad\;\wedge 
\left.\supp  (\hat{B}_{{\bf i},{\bf p}}^\ell)\not\subseteq {\hat\Omega}^{\ell+1},
\, \ell =0,\ldots,N-1
\right\}.
\end{align*}
Figure~\ref{fig:hbspline1D} shows an example of cubic hierarchical B-splines for $\dpa=1$.

\begin{figure}[!t]\centering
\subfigure[$\hat{\Omega}^0$, $\hat{\Omega}^1$, and $\hat{\Omega}^2$ (from top to bottom)]{
\includegraphics[width=0.49\textwidth, trim=8cm 0cm 4cm 0cm, clip]{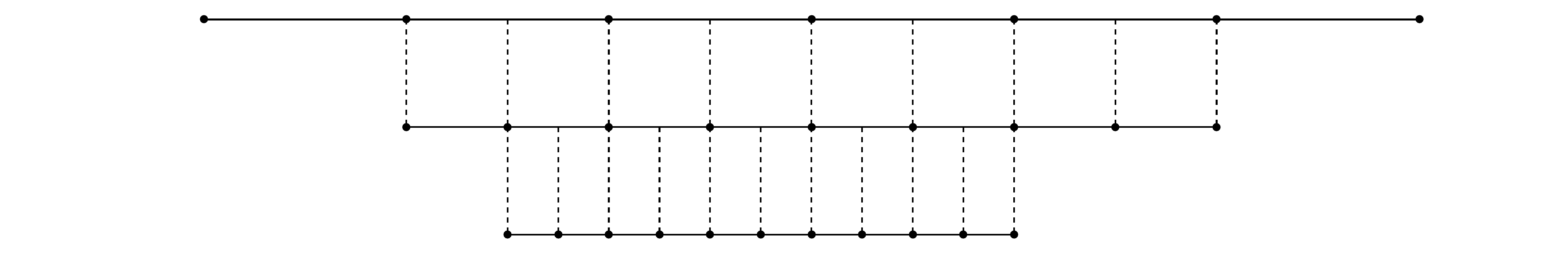}}
\subfigure[HB-splines]{
\includegraphics[width=0.49\textwidth, trim=8cm 0cm 4cm 0cm, clip]{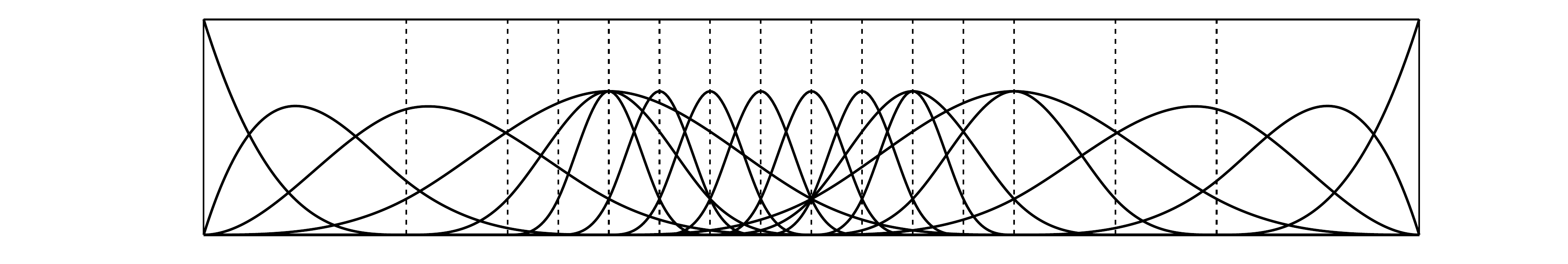}}
\subfigure[THB-splines]{
\includegraphics[width=0.49\textwidth, trim=8cm 0cm 4cm 0cm, clip]{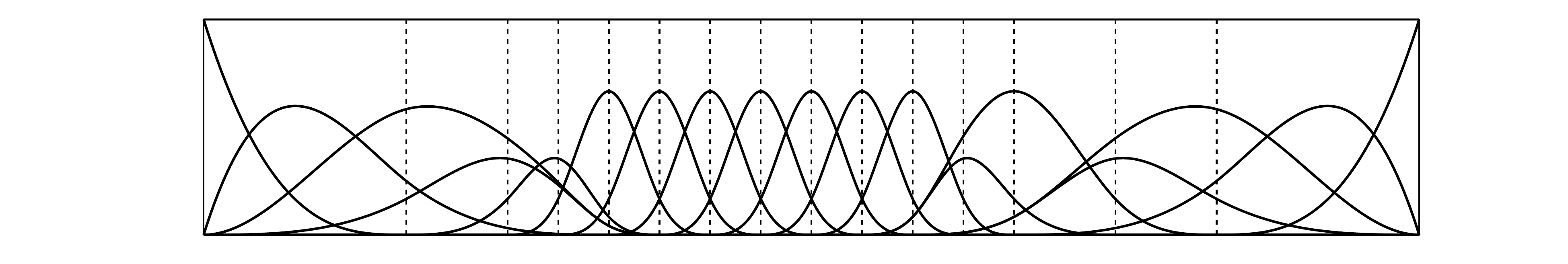}}
\caption{An example of cubic HB-splines (b) and THB-splines (c) defined on a domain hierarchy consisting of three levels (a).
All internal knots have multiplicity one.}
\label{fig:hbspline1D}
\end{figure}

The following proposition shows that $\hat\hbbasis_{\bf p}(\hat\hmesh,{\bf \kv}^{0})$ is indeed a basis for the \emph{hierarchical spline space}
\begin{align*}
\widehat{\mathbb{S}}_{\bf p}^{\rm H}(\hat\hmesh,{\bf \kv}^{0})
:=\myspan\,\hat\hbbasis_{\bf p}(\hat\hmesh,{\bf \kv}^{0}).
\end{align*}
Properties (i)--(iii) in Proposition~\ref{prop:hb properties} are proved in \cite{vgjs11,giannelli2014,sm16}. 
The characterization (iv) is taken from \cite[Section~3]{sm16}.
 
\begin{proposition}\label{prop:hb properties}
The hierarchical basis $\hat\hbbasis_{\bf p}(\hat\hmesh,{\bf \kv}^{0})$ satisfies the following properties:
\begin{itemize}
\item[\rm (i)] The HB-splines in $\hat\hbbasis_{\bf p}(\hat\hmesh,{\bf \kv}^{0})$ are nonnegative and linearly independent.
\item[\rm (ii)] The intermediate spline spaces are nested, i.e., ${\rm span} \,\hat\hbbasis^{\ell}\subseteq {\rm span}\,\hat\hbbasis^{\ell+1}$, for $\ell=0,\ldots,N-2$.
\item[\rm (iii)] Given a mesh $\hat \QQ_\fine \succeq \hat \QQ$, it holds that
$\widehat{\mathbb{S}}_{\bf p}^{\rm H}(\hat\hmesh,{\bf \kv}^{0})
\subseteq \widehat{\mathbb{S}}_{\bf p}^{\rm H}(\hat\hmesh_\fine,{\bf \kv}^{0})$.
\item[\rm (iv)]It holds the explicit characterization $\widehat{\mathbb{S}}_{\bf p}^{\rm H}(\hat\hmesh,{\bf \kv}^{0})=\set{\widehat S}{\widehat S|_{\widehat\Omega\setminus\widehat\Omega^{\ell+1}}\in \spmh{\ell}|_{\widehat\Omega\setminus\widehat\Omega^{\ell+1}}, \,\ell=0,\dots,N-1}$.
In particular, hierarchical splines are polynomials of degree ${\bf p}$ on each element $\widehat Q\in\widehat\QQ$.
\end{itemize}
\end{proposition}

The dimensions of bivariate and trivariate hierarchical B-spline spaces were investigated in  \cite{giannelli2013} and \cite{berdinsky2014b}, respectively, for the case of maximal smoothness. In \cite{mokris2014a}, a more comprehensive analysis covering also reduced regularity was presented.

\subsubsection{Truncated hierarchical B-splines}\label{subsec:thb}

The HB-spline basis is composed by B-splines defined on grids of different resolution which interact with each other on refined elements. Thanks to the refinable nature of the B-spline model, it is possible to reduce the overlapping of B-splines introduced at successive levels with the coarser ones by exploiting a \emph{truncation} mechanism \cite{gjs12}.

By recalling the nested nature of the sequence of spline spaces in \eqref{eq:nestedspaces}, let $\widehat S\in \spmh{\ell}\subset\spmh{\ell+1}$ be a spline of level $\ell$ expressed in terms of B-splines of level $\ell+1$ as
\begin{equation}\label{eq:s}
\widehat S = \sum_{ \hat{B}_{{\bf i},{\bf p}}^{\ell+1} \in {\hat{\cal B}}^{\ell+1}} 
c_{{\bf i},{\bf p}}^{\ell+1}(\widehat S) \hat{B}_{{\bf i},{\bf p}}^{\ell+1}.
\end{equation}
The truncation of $\widehat S$ with respect to level $\ell+1$ is defined as
\[
\trunc^{\ell+1} \widehat S = \sum_{{\hat{B}_{{\bf i},{\bf p}}^{\ell+1} \in {\hat{\cal B}}^{\ell+1} \setminus  \hat\hbbasis_B^{\ell+1}}}
c_{{\bf i},{\bf p}}^{\ell+1}(\widehat S) \hat{B}_{{\bf i},{\bf p}}^{\ell+1},
\]
and leads to a truncated function whose support is either equal or reduced when compared to the one of function $\widehat S$, i.e., $\supp(\trunc^{\ell+1}\widehat S) \subseteq \supp(\widehat S)$, for all $\widehat S \in \spmh{\ell}$. In particular, the contribution of B-splines of level $\ell+1$ which will be included in the hierarchical basis is removed from the expression of $\widehat S$ given by \eqref{eq:s}. For $\dpa = 1$, an example of truncation applied to a quadratic univariate B-spline is shown in Figure~\ref{fig:trunc}. 

\begin{figure}
\centerline{\includegraphics[width=0.49\textwidth]{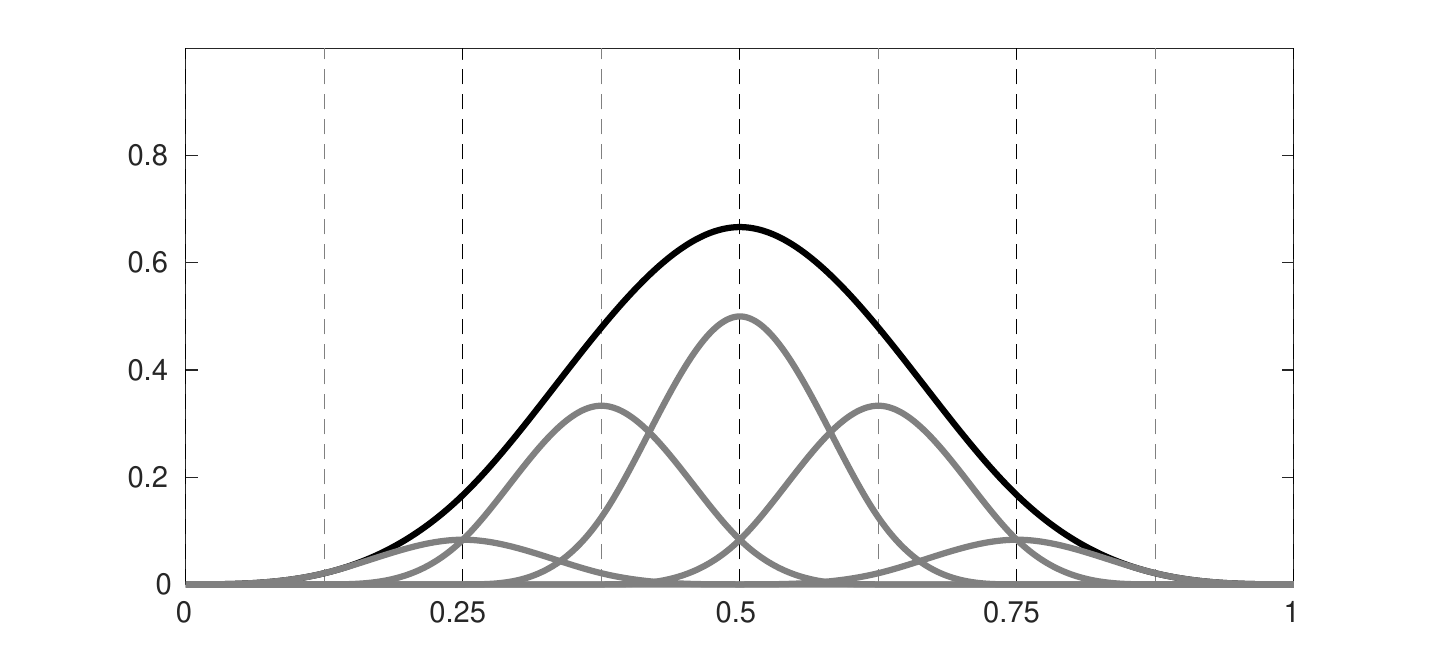}}
\centerline{\includegraphics[width=0.49\textwidth]{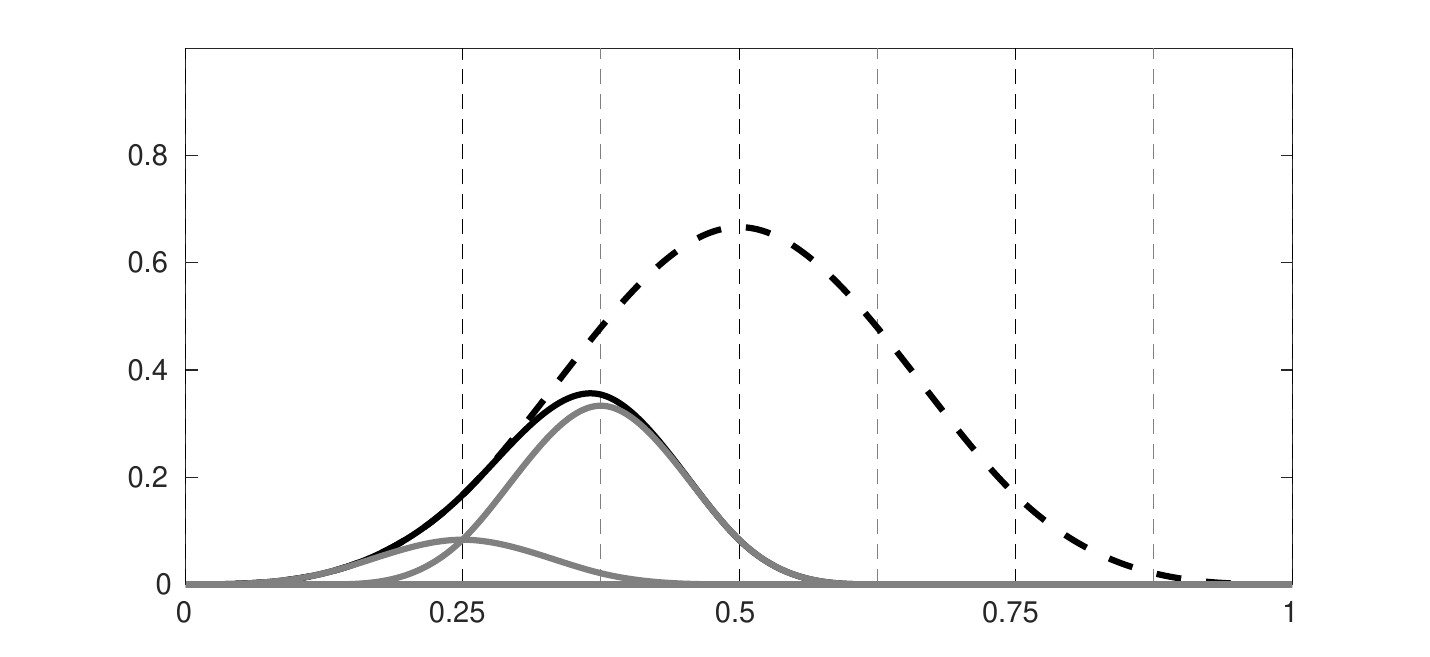}}
\caption{Top: a univariate cubic B-spline of level $\ell$ (in black) represented as linear combination of functions of level $\ell+1$ (in gray). Bottom: the original B-spline (solid dashed) and its truncated version (black solid line) by considering $\Omega^{\ell+1}=[0.25,1]$.}\label{fig:trunc}
\end{figure}

Analogously to the HB-spline case, given a hierarchical mesh $\hat\hmesh$, the set of \emph{THB-splines} $\hat\thbbasis_{\bf p}(\hat\hmesh,{\bf \kv^0}):=\hat\thbbasis^{N-1} $ can be constructed according to the following steps:
\begin{enumerate}
\item $\hat\thbbasis^0 := \hat\hbbasis^0$;
\item for $\ell=0,\ldots, N-2$ 
\[
\hat\thbbasis^{\ell+1} := \hat\thbbasis_A^0 \cup \hat\thbbasis_B^{\ell+1},
\]
\end{enumerate}
where
\begin{align*}
\hat\thbbasis_A^{\ell+1} &:= \left\{\trunc^{\ell+1}(\hat{T}_{{\bf i},{\bf p}}^\ell): 
\right.\\ & 
 \left. \qquad \qquad 
\hat{T}_{{\bf i},{\bf p}}^\ell\in \hat\thbbasis^{\ell}  
\wedge\;  \supp (\hat{T}_{{\bf i},{\bf p}}^\ell) \not\subseteq\hat\Omega^{\ell+1}\right\},\\
\hat\thbbasis_B^{\ell+1} &:= \hat\hbbasis_B^{\ell+1}.
\end{align*}

In this case, the two steps of the constructions define a selection mechanism which does not only activate and deactivate but also truncates B-splines of different levels by taking into account the hierarchical domain configuration.
After initializing the set of THB-splines with the (H)B-splines of level $0$, for any subsequent level $\ell$, the set of THB-splines of level $\ell+1$ ($\hat\thbbasis^{\ell+1}$) includes
\begin{itemize}
\item truncated B-splines of coarser levels whose support is not contained in $\hat\Omega^{\ell+1}$ ($\hat\thbbasis_A^{\ell+1}$);
\item B-splines of level $\ell+1$ whose support is contained in $\hat\Omega^{\ell+1}$ ($\hat\thbbasis_B^{\ell+1}$).
\end{itemize}

By defining the successive truncation of a B-spline of level $\ell$ as
\[
\Trunc^{\ell+1}(\hat{B}_{{\bf i},{\bf p}}^\ell) : = \trunc^{N-1}
\left(
	\ldots\left(
		\trunc^{\ell+1}(\hat{B}_{{\bf i},{\bf p}}^\ell)
	\right)\ldots\right),
\]
and $\Trunc^{N}(\hat{B}_{{\bf i},{\bf p}}^{N-1}):= \hat{B}_{{\bf i},{\bf p}}^{N-1}$, we can also define the THB-spline basis as follows: 
\begin{align*}
\hat\thbbasis_{\bf p}(\hat\hmesh,{\bf \kv^0})  = 
\big\{ \Trunc^{\ell+1}(\hat{B}_{{\bf i},{\bf p}}^\ell) :&  
 \hat{B}_{{\bf i},{\bf p}}^\ell \in{\hat{\cal B}}^\ell \cap \hat\hbbasis_{\bf p}(\hat\hmesh,{\bf \kv}^{0}),\\
&\quad\qquad \ell =0,\ldots,N-1
\big\}.
\end{align*}
Any HB-spline $\hat{B}_{{\bf i},{\bf p}}^\ell {\in \hat\hbbasis_{\bf p}(\hat\hmesh,{\bf \kv}^{0})}$ generates a corresponding THB-spline $\hat{T}_{{\bf i},{\bf p}}^\ell:=\Trunc^{\ell+1}(\hat{B}_{{\bf i},{\bf p}}^\ell) \in \hat\thbbasis_{\bf p}(\hat\hmesh,{\bf \kv^0}) $, for $\ell =0, ..., N-1$ and it is denoted as the \emph{mother} B-spline of $\hat{T}_{{\bf i},{\bf p}}^\ell$, namely
\begin{align}\label{eq:mother}
\mot\hat{T}_{{\bf i},{\bf p}}^\ell := \hat{B}_{{\bf i},{\bf p}}^\ell.
\end{align}
Note that, being defined in terms of the successive application of the truncation mechanism, each THB-spline is characterized by a support that is either equal or smaller than the one of its mother B-spline. 
However, as for finite elements defined in meshes with hanging nodes, the support of THB-splines is more complicated and in general not even convex or connected.
Figure~\ref{fig:hbspline1D}(c) and \ref{fig:thb2D} show examples of THB-splines for $\dpa=1$ and $\dpa=2$, respectively.

\begin{figure}[!t]
\centerline{\includegraphics[trim=0 0cm 0 0cm, clip, scale=.25]{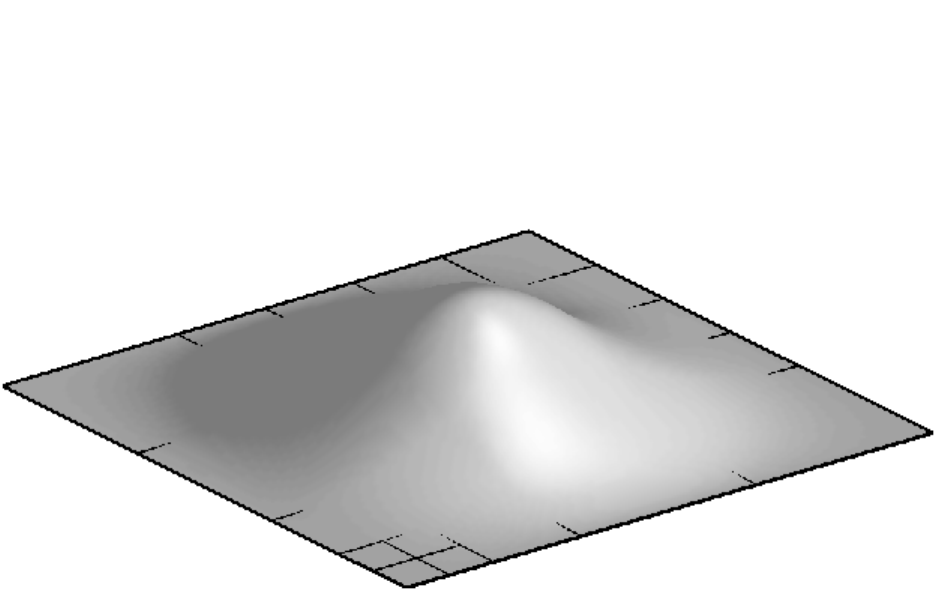}
\includegraphics[trim=0 0cm 0 0cm, clip, scale=.25]{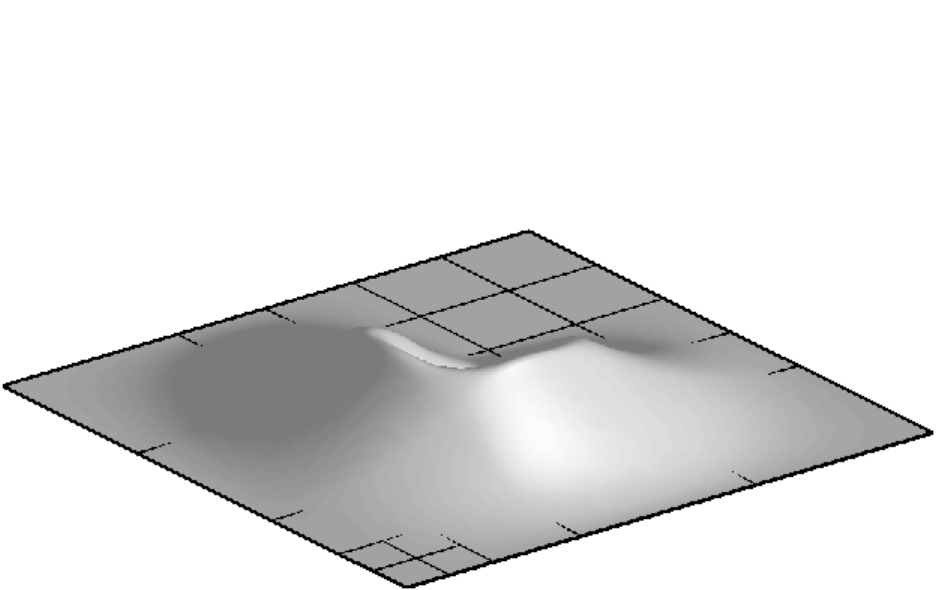}}
\centerline{\includegraphics[trim=0 0cm 0 0cm, clip, scale=.25]{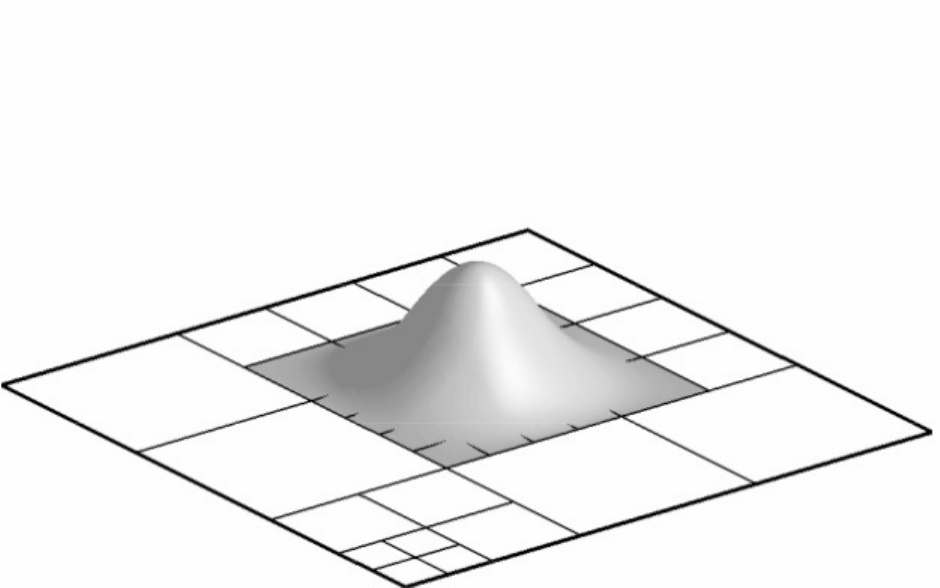}
\includegraphics[trim=0 0cm 0 0cm, clip, scale=.25]{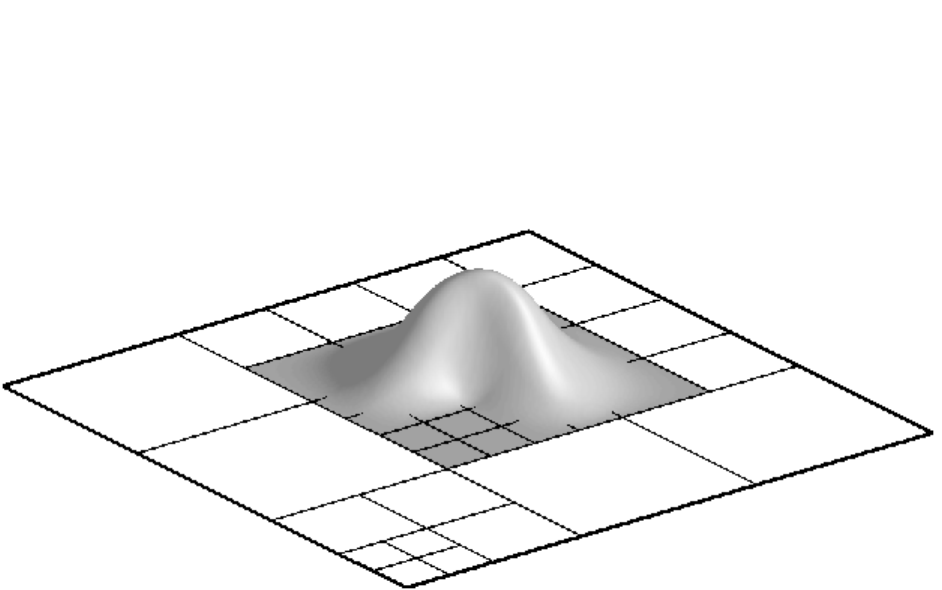}}
\vspace*{.45cm}
\centerline{\begin{tikzpicture}[scale=1]
\draw (0,0) grid (3,3);
\foreach \a in {0,1}
	\draw (0,\a/2) -- (1,\a/2);
\foreach \a in {3,5}
	\draw (1,\a/2) -- (3,\a/2);
\foreach \a in {1}
	\draw (\a/2,0) -- (\a/2,1);
\foreach \a in {3,5}
	\draw (\a/2,1) -- (\a/2,3);
\foreach \a in {0,0.5}
	\draw (0,\a/2) -- (0.5,\a/2);
\foreach \a in {1,2}
	\draw (1,1.25) -- (2,1.25);
	\draw (1,1.75) -- (2,1.75);
\draw (0.25,0) -- (0.25,0.5);
\draw (1.25,1) -- (1.25,2);
\draw (1.75,1) -- (1.75,2);
\end{tikzpicture}}
\caption{Two bi-quadratic mother B-splines (left) and corresponding THB splines (right) defined on a hierarchical mesh with three levels (bottom).
All internal knots have multiplicity one.}
\label{fig:thb2D}
\end{figure}
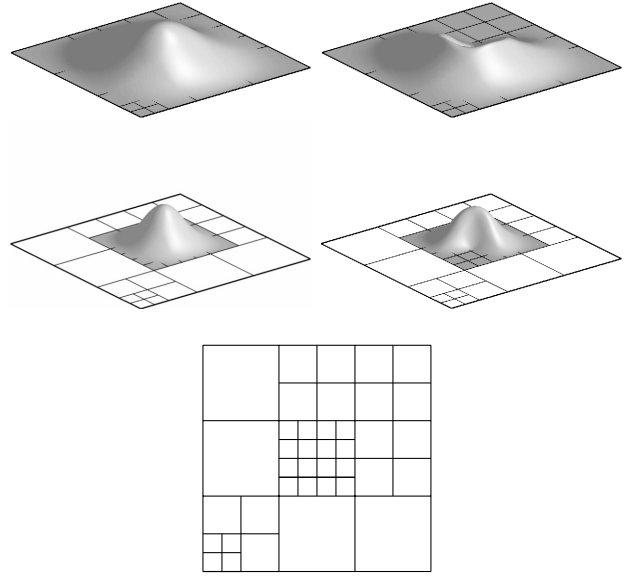

The following properties hold according to \cite{gjs12,giannelli2014}.
\begin{proposition}\label{prop:thb properties}
The truncated hierarchical basis $\hat\thbbasis_{\bf p}(\hat\hmesh,{\bf \kv^0})$ satisfies the following properties: 
\begin{itemize}
\item[\rm (i)] The THB-splines in $\hat\thbbasis_{\bf p}(\hat\hmesh,{\bf \kv^0})$ are nonnegative, linearly independent, and form a partition of unity.
\item[\rm (ii)] The intermediate spline spaces are nested, namely ${\rm span} \,\hat\thbbasis^{\ell}\subseteq {\rm span}\,\hat\thbbasis^{\ell+1}$.
\item[\rm (iii)] It holds that ${\rm span} \,\hat\thbbasis^{\ell} = {\rm span} \,\hat\hbbasis^{\ell}$, for $\ell=0,\dots,N-1$, and ${\rm span}\,\hat\thbbasis_{\bf p}(\hat\hmesh,{\bf \kv^0}) =\widehat{\mathbb{S}}_{\bf p}^{\rm H}(\hat\hmesh,{\bf \kv}^{0})$.
\end{itemize}
\end{proposition}
We also note that, in contrast to tensor-product B-splines, THB-splines and HB-splines are not locally linearly independent. In particular, their restriction to a single element can be linearly dependent.

Applications of THB-splines for adaptive CAD model reconstruction were presented in \cite{kiss2014b,bracco2018}. The truncation approach was also considered to define truncated decoupled hierarchical B-splines \cite{mokris2014b}, hierarchies of spaces spanned by generating systems \cite{zore2014}, (extended) truncated hierarchical Catmull-Clark subdivision \cite{wei2015,wei2016}, truncated hierarchical box splines \cite{kanduc2017,giannelli2019}, and truncated T-splines \cite{wzlh17}. 

\subsubsection{Refinement strategies}
\label{sec:hierarchical refine}

We have introduced above the concept of mesh refinement in the sense that $\QQ \preceq \QQ_\fine$. However, the theoretical analysis of adaptive isogeometric methods requires to impose some grading conditions on how the local refinement should be performed. To obtain hierarchical mesh configurations suitable for the theoretical analysis, 
we follow the approach originally introduced in \cite{bg16} for THB-splines and in \cite{morgenstern17} for HB-splines, and further elaborated in \cite{bgv18} by introducing a general framework for the design and implementation of refinement algorithms with (T)HB-splines.
The refinement rule for HB-splines limited to two-level interaction was already presented in \cite{ghp17}.
 These refinement rules control the interaction of hierarchical basis functions of different levels and generate suitably graded meshes for the considered hierarchical basis. Note that the effect of the truncation can be suitably exploited to generate less refined meshes for THB-splines than the ones obtained for HB-splines, while simultaneously guaranteeing limited interaction between hierarchical basis functions of different levels. 
However, THB-splines additionally require the truncation procedure and have a more complicated, although smaller, support than HB-splines. 

{\color{black}
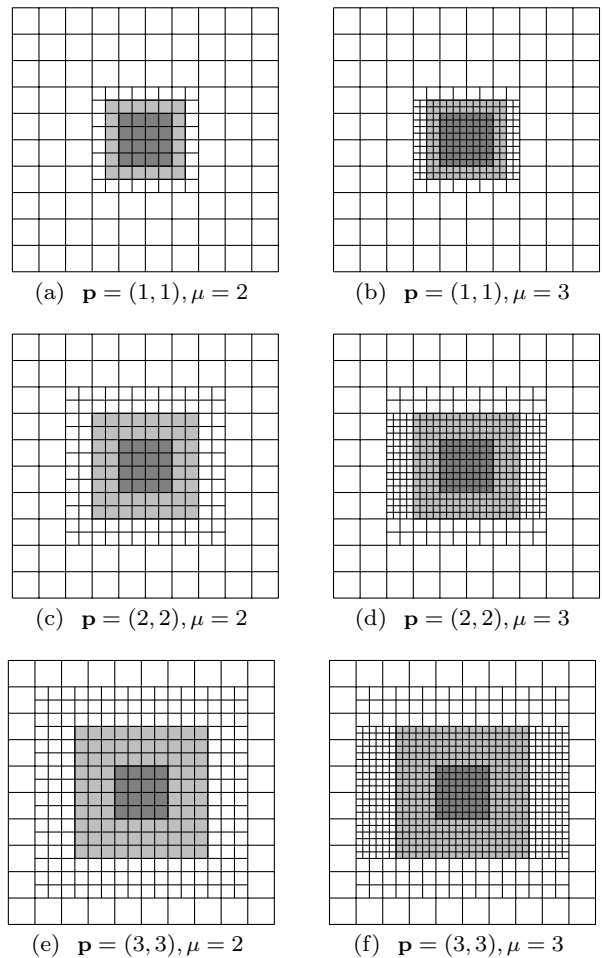
\begin{figure}[!t]\centering
\subfigure[~$\mathbf{p}=(1,1), \mu=2$ ]{\begin{tikzpicture}[scale=0.35]
\draw [fill=lightgray] (3.5,3.5) rectangle (6.5,6.5);
\draw [fill=gray] (4,4) rectangle (6,6);
\draw (0,0) grid (10,10);
\foreach \a in {7,9,11,13}
	\draw (3,\a/2) -- (7,\a/2);
\foreach \a in {7,9,11,13}
	\draw (\a/2,3) -- (\a/2,7);
\end{tikzpicture}}
\hspace*{.5cm}
\subfigure[~$\mathbf{p}=(1,1), \mu=3$ ]{\begin{tikzpicture}[scale=0.35]
\draw [fill=lightgray] (3.5,3.5) rectangle (6.5,6.5);
\draw [fill=gray] (4,4) rectangle (6,6);
\draw (0,0) grid (10,10);
\foreach \a in {7,9,11,13}
	\draw (3,\a/2) -- (7,\a/2);
\foreach \a in {7,9,11,13}
	\draw (\a/2,3) -- (\a/2,7);
\foreach \a in {15,17,19,21,23,25}
	\draw (3,\a/4) -- (7,\a/4);
\foreach \a in {13,15,17,19,21,23,25,27}
	\draw (\a/4,3.5) -- (\a/4,6.5);
\end{tikzpicture}}\\
\subfigure[~$\mathbf{p}=(2,2), \mu=2$ ]{\begin{tikzpicture}[scale=0.35]
\draw [fill=lightgray] (3,3) rectangle (7,7);
\draw [fill=gray] (4,4) rectangle (6,6);
\draw (0,0) grid (10,10);
\foreach \a in {5,7,9,11,13,15}
	\draw (2,\a/2) -- (8,\a/2);
\foreach \a in {5,7,9,11,13,15}
	\draw (\a/2,2) -- (\a/2,8);
\end{tikzpicture}}
\hspace*{.5cm}
\subfigure[~$\mathbf{p}=(2,2), \mu=3$ ]{\begin{tikzpicture}[scale=0.35]
\draw [fill=lightgray] (3,3) rectangle (7,7);
\draw [fill=gray] (4,4) rectangle (6,6);
\draw (0,0) grid (10,10);
\foreach \a in {5,7,9,11,13,15}
	\draw (2,\a/2) -- (8,\a/2);
\foreach \a in {5,7,9,11,13,15}
	\draw (\a/2,2) -- (\a/2,8);
\foreach \a in {13,15,17,19,21,23,25,27}
	\draw (2,\a/4) -- (8,\a/4);
\foreach \a in {9,11,13,15,17,19,21,23,25,27,29,31}
	\draw (\a/4,3) -- (\a/4,7);
\end{tikzpicture}}\\
\subfigure[~$\mathbf{p}=(3,3), \mu=2$ ]{\begin{tikzpicture}[scale=0.35]
\draw [fill=lightgray] (2.5,2.5) rectangle (7.5,7.5);
\draw [fill=gray] (4,4) rectangle (6,6);
\draw (0,0) grid (10,10);
\foreach \a in {3,5,7,9,11,13,15,17}
	\draw (1,\a/2) -- (9,\a/2);
\foreach \a in {3,5,7,9,11,13,15,17}
	\draw (\a/2,1) -- (\a/2,9);
\end{tikzpicture}}
\hspace*{.5cm}
\subfigure[~$\mathbf{p}=(3,3), \mu=3$ ]{\begin{tikzpicture}[scale=0.35]
\draw [fill=lightgray] (2.5,2.5) rectangle (7.5,7.5);
\draw [fill=gray] (4,4) rectangle (6,6);
\draw (0,0) grid (10,10);
\foreach \a in {3,5,7,9,11,13,15,17}
	\draw (1,\a/2) -- (9,\a/2);
\foreach \a in {3,5,7,9,11,13,15,17}
	\draw (\a/2,1) -- (\a/2,9);
\foreach \a in {11,13,15,17,19,21,23,25,27,29}
	\draw (1,\a/4) -- (9,\a/4);
\foreach \a in {5,7,9,11,13,15,17,19,21,23,25,27,29,31,33,35}
	\draw (\a/4,2.5) -- (\a/4,7.5);
\end{tikzpicture}}
\caption{Examples of the domains $\widehat\omega_{\cal H}^1$ (dark gray) and $\widehat\omega_{\cal T}^1$ (light gray) for~different degrees and mesh configurations.
All internal knots have multiplicity one.} \label{fig:omegas}
\end{figure}
}

The first notion we need to introduce extends the concept of support extension introduced in \eqref{eq:multivariatese} for the multivariate tensor-product case to the hierarchical setting. The \emph{multilevel support extension} of an element $\hat{Q}\in\cal{\hat{\hmesh}^\ell}$ with respect to level $k$, with $0\le k\le \ell$,
is defined as 
\begin{align*}
\sext{\hat{Q},k} &:= \sext{\hat{Q}'}, \; \text{ with } \hat{Q}' \in \hat{\hmesh}^k \text{ and }  \hat{Q} \subseteq \hat{Q}',
\end{align*}
where $\sext{\hat{Q}'}$ is the support extension of \eqref{eq:multivariatese} corresponding to the mesh $\hat\hmesh^k$.

The concept of \emph{admissible} hierarchical meshes is based on the auxiliary domains
\begin{subequations}\label{eq:HB-THB-omega}
\begin{align}
{\hat{\omega}}^{\ell}_{\cal H} &:=\bigcup\left\{ \overline{\hat{Q}} \,:\, \hat{Q} \in \hat{\cal Q}^{\ell} \,\wedge\, \sext{\hat{Q},{\ell}-1}\subseteq \hat{\Omega}^{\ell} \right\}, \\
{\hat{\omega}}^{\ell}_{\cal T} &:=\bigcup\left\{ \overline{\hat{Q}} \,:\, \hat{Q} \in \hat{\cal Q}^{\ell} \,\wedge\, \sext{\hat{Q},{\ell}}\subseteq \hat{\Omega}^{\ell} \right\},
\end{align}
\end{subequations}
for $\ell = 0, \ldots, N-1$, with ${\hat{\omega}}^{0}_{\cal H} := \hat\Omega^0$. The domain ${\hat{\omega}}^{\ell}_{\cal H}$ represents the region of $\hat\Omega^\ell$ where all the active basis functions of level $\ell-1$, namely functions in $\hat{\cal H}_{\bf p}(\hat\hmesh,{\bf \kv}^0) \cap \hat{\cal B}^{\ell-1}$, are zero. A similar property is valid for the domain ${\hat{\omega}}^{\ell}_{\cal T}$: all the basis functions of level $\ell-1$ truncated with respect to level $\ell$, i.e., functions in $\hat{\cal T}^\ell$ such that their mother is in $\hat{\cal B}^{\ell-1}$, vanish in $\hat\omega^\ell_{\cal T}$. By definition, it holds that $\hat\omega^\ell_{\cal H} \subseteq \hat\omega^\ell_{\cal T}$ (see also Figure~\ref{fig:omegas}).


A mesh $\hat{\cal Q}$ is \emph{${\cal H}$-admissible} (respectively, \emph{${\cal T}$-admissible}) \emph{of class} $\mu$ if it holds that 
\begin{equation}\label{eq:sameshes}
\hat{\Omega}^\ell\subseteq \hat{\omega}^{\ell-\mu+1}_{\cal H}, \quad (\text{resp. } \hat{\Omega}^\ell\subseteq \hat{\omega}^{\ell-\mu+1}_{\cal T}),
\end{equation}
for all $\ell=\mu,\mu+1,\ldots,N-1$. 
By definition, it holds that $\hat\omega^\ell_{\cal H} \subseteq \hat\omega^\ell_{\cal T}$, which immediately yields that any ${\cal H}$-admissible mesh of class $\mu$ is also ${\cal T}$-admissible of class $\mu$.
Admissibility of a mesh guarantees the following proposition, see  \cite[Definition~3]{bgv18}. 

\begin{proposition}\label{eq:former admissible} 
If $\widehat\hmesh$ is an ${\cal H}$-admissible (respectively ${\cal T}$-admissible) mesh of class $\mu$, with $\mu\ge 2$, then, the basis functions in $\hat{\cal H}_{\bf p}(\hat\hmesh,{\bf T}^0)$ (resp. $\hat{\cal T}_{\bf p}(\hat\hmesh,{\bf T}^0)$) that take non-zero values over any element $\widehat Q\in\hat\hmesh$ can only be of $\mu$ successive levels.
\end{proposition}

\begin{remark}
Note that Proposition~\ref{eq:former admissible} is not true for HB-splines $\cal{\hat{H}}_{\bf p}(\hat\hmesh,{\bf T}^0)$ on ${\cal T}$-admissible meshes (instead of ${\cal H}$-admissible meshes), see Figure~\ref{fig:rmk} for a simple example of this kind.
\end{remark}

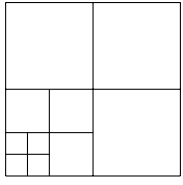
\begin{figure}[!t]\centering
\begin{tikzpicture}[scale=1.15]
\draw (0,0) grid (2,2);
\foreach \a in {1}
	\draw (0,\a/2) -- (1,\a/2);
\foreach \a in {1}
	\draw (\a/2,0) -- (\a/2,1);
\foreach \a in {1}
	\draw (0,\a/4) -- (1/2,\a/4);
\foreach \a in {1}
	\draw (\a/4,0) -- (\a/4,1/2);
\end{tikzpicture}
\caption{A {${\cal T}$-admissible} mesh for $\mathbf{p}=(1,1)$ and $\mu=2$ with three levels: HB-splines of level 0, 1, 2 are non zero on the element of the finest level in the bottom left corner.  THB-splines of only levels 1, 2 are non zero on the same element.
}
\label{fig:rmk}
\end{figure}

\begin{remark}\label{rem:strictly admissible}
Note that in \cite{bg16,bg17,bracco2019} a ${\cal T}$-admissible mesh was denoted \emph{strictly admissible}.
Reference \cite{bgv18} introduced ${\cal H}$-admissible meshes, which were there called strictly ${\cal H}$-admissible.
Similarly, ${\cal T}$-admissible meshes were called strictly ${\cal T}$-admissible there.
Instead, the property of  Proposition~\ref{eq:former admissible}  was referred to as admissible in these references.
We also mention that these references even prove that 
the basis functions that take non-zero values over $\hat Q$ can indeed only be of levels $\levelT{\widehat Q}-\mu+1, \ldots, \levelT{\widehat Q}$.
\end{remark}

The structure of admissible hierarchical configurations guarantees, first, a suitable grading of the mesh, and, second, that differences between the levels of neighboring elements 
are always bounded, as stated in the following proposition. 

\begin{proposition}
\label{prop:lqiHB-adjacent-support}
Let $\Tmeshp$ be an ${\cal H}$-admissible (resp. ${\cal T}$-admissible) hierarchical mesh of class $\mu$. For any $\elemp, \elemp' \in \Tmeshp$, let $\ell := \min\{ \levelT{\elemp},\levelT{\elemp'}\}$. If there exists $\hat \beta\in \hat{\cal B}^{\ell}$, (resp. $\hat \beta \in \hat{\cal B}^{\ell+1}$), such that $\supp (\hat \beta) \cap \elemp \not = \emptyset$ and $\supp (\hat \beta) \cap \elemp' \not = \emptyset$, then it holds that
\[
| \levelT{\elemp} - \levelT{\elemp'} | < \mu.
\]
\end{proposition}
\begin{proof}
We proceed by contradiction. Let us assume there exist $\elemp, \elemp' \in \Tmeshp$ as in the hypothesis such that $| \levelT{\elemp} - \levelT{\elemp'} | \ge \mu$. We assume without loss of generality that $\ell= \levelT{\elemp} < \levelT{\elemp'} =: \ell'$ and thus 
\begin{equation}\label{eq:aux-levels2}
\ell \le \ell' - \mu.
\end{equation}
Let $\elemp''$ be the ancestor of $\elemp'$ of level $\ell'-\mu+1$. By the assumptions on $\hat \beta$, it clearly satisfies that $\elemp \cap \sext{\elemp'',k} \not = \emptyset$ for $k=\ell'-\mu$ (respectively $k=\ell'-\mu+1$). 
As a consequence, we get with \eqref{eq:aux-levels2} and \eqref{eq:hmesh} that $\sext{\elemp'',k} \not \subseteq \widehat\Omega^{\ell'-\mu+1}$. 
We conclude from the definition in \eqref{eq:HB-THB-omega} that $\elemp' \not \subseteq {\hat{\omega}}^{\ell'-\mu+1}_{\cal H}$ (respectively ${\hat{\omega}}^{\ell'-\mu+1}_{\cal T}$), which contradicts the definition of ${\cal H}$-admissibility, and of ${\cal T}$-admissibility, see \eqref{eq:sameshes}. \qed
\end{proof}

As an immediate consequence, we have an analogous result for adjacent elements if the interior multiplicities in all knot vectors $\kv_i^\ell$ are less or equal than $p_i$ so that all B-splines are at least continuous. 
\begin{corollary}
\label{prop:lqiHB-adjacent}
Suppose that the interior multiplicities in all knot vectors $\kv_i^\ell$, $i=1,\dots\dpa$, $\ell=0,\dots,N-1$, are less or equal than $p_i$. 
Let $\Tmeshp$ be  a hierarchical mesh which is either ${\cal H}$-admissible or  ${\cal T}$-admissible  of class $\mu$. For any $\elemp, \elemp' \in \Tmeshp$ with $\overline{\elemp} \cap \overline{\elemp'} \not = \emptyset$, it holds that
\[
| \levelT{\elemp} - \levelT{\elemp'} | < \mu.
\]
\end{corollary}

The refinement algorithms to generate suitable admissible meshes recursively refine all the elements in a certain \emph{neighborhood} of any marked element to produce the refined mesh for the next step of the adaptive loop, while simultaneously preserving a fixed class of admissibility.

Given an element $\hat{Q}\in \hat{\cal Q}$ with $\levelT{\widehat Q}=:\ell$, its \emph{${\cal H}$-neighborhood} and its \emph{${\cal T}$-neighborhood} with respect to $\mu$ are defined as 
\begin{align*}
{\cal N}_{\cal H}(\hat{Q},\mu) \,&:=\, \left\{\hat{Q}'\in
\hat{\cal Q} \,\cap\,{\cal\hat{Q}}^{\ell-\mu+1}:\right.\\
&\left.\qquad\;\hat{Q}' \cap \sext{\hat{Q},\ell-\mu+1}\neq\emptyset \right\},\\
{\cal N}_{\cal T}(\hat{Q},\mu) \,&:=\, \left\{\hat{Q}'\in
\hat{\cal Q} \,\cap\,{\cal\hat{Q}}^{\ell-\mu+1}: \right.\\
&\left.\qquad\;
\hat{Q}' \cap\sext{\hat{Q},\ell-\mu+2}\neq\emptyset \right\},
\end{align*}
respectively, when $\ell-\mu+1 \ge 0$, and ${\cal N}_{\cal H}(\hat{Q},\mu) := {\cal N}_{\cal T}(\hat{Q},\mu) := \emptyset\,$ for $\ell-\mu+1 < 0$. 
Recall that we consider open elements, whereas the support extension is a closed set.
The conditions in the two sets are thus equivalent to $\hat{Q}' \subseteq \sext{\hat{Q},\ell-\mu+1}$ and $\exists\, \hat{Q}'' \in\widehat\hmesh^{\ell-\mu+2}$ with $\hat Q''\subseteq \sext{\hat{Q},\ell-\mu+2}, \hat{Q}''\subseteq \hat{Q}' $, respectively.
An example of {${\cal H}$-neighborhood} and the {${\cal T}$-neighborhood} for $\mathbf{p}=(2,2)$ and $\mu=2$ is shown in Figure~\ref{fig:hneig}.

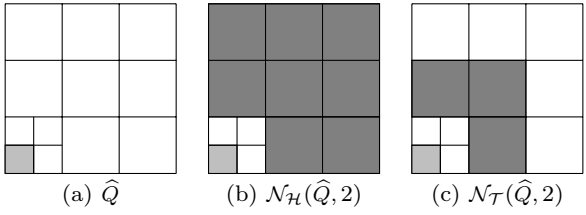
\begin{figure}[!t]\centering
\subfigure[$\hat{Q}$]{\begin{tikzpicture}[scale=.75]
\draw (0,0) grid (3,3);
\draw [fill=lightgray] (0,0) rectangle (1/2,1/2);
\foreach \a in {1}
	\draw (0,\a/2) -- (1,\a/2);
\foreach \a in {1}
	\draw (\a/2,0) -- (\a/2,1);
\end{tikzpicture}}
\hspace*{.2cm}
\subfigure[${\cal N}_{\cal H}(\hat{Q},2)$]{\begin{tikzpicture}[scale=.75]
\draw [fill=gray] (0,1) rectangle (3,3);
\draw [fill=gray] (1,0) rectangle (3,1);
\draw (0,0) grid (3,3);
\draw [fill=lightgray] (0,0) rectangle (1/2,1/2);
\foreach \a in {1}
	\draw (0,\a/2) -- (1,\a/2);
\foreach \a in {1}
	\draw (\a/2,0) -- (\a/2,1);
\end{tikzpicture}}
\hspace*{.2cm}
\subfigure[${\cal N}_{\cal T}(\hat{Q},2)$]{\begin{tikzpicture}[scale=.75]
\draw [fill=gray] (0,1) rectangle (2,2);
\draw [fill=gray] (1,0) rectangle (2,1);
\draw (0,0) grid (3,3);
\draw [fill=lightgray] (0,0) rectangle (1/2,1/2);
\foreach \a in {1}
	\draw (0,\a/2) -- (1,\a/2);
\foreach \a in {1}
	\draw (\a/2,0) -- (\a/2,1);
\end{tikzpicture}}
\caption{For the light gray element $\hat{Q}$ (a), we plot in dark gray its {${\cal H}$-neighborhood} (b) and {${\cal T}$-neighborhood} (c), for $\mathbf{p}=(2,2)$ and $\mu=2$.
All internal knots have multiplicity one.}
\label{fig:hneig}
\end{figure}

By exploiting the neighborhoods to define the refinement patch associated to each set of a marked element, we can generate admissible meshes and encapsulate a certain structure naturally connected with the support of hierarchical basis functions. 
Algorithm~\ref{alg:hrefine} and~\ref{alg:trefine} present the admissible refinement procedure for HB-splines and THB-splines, respectively. In both algorithms, given a set of marked (active) elements, we iteratively also mark the elements in the ${\cal H}$-neighborhood (Algorithm~\ref{alg:hrefine}) or ${\cal T}$-neighborhood (Algorithm~\ref{alg:trefine}) of the marked ones until these neighborhood sets are empty (and no additional elements are marked). Then, we refine the hierarchical mesh  by replacing the set of marked elements with its children. Note that the difference between the two algorithms only affects the computation of the neighborhood.  The output of the two algorithms coincides with the output of the recursive refinement modules introduced in \cite{bg16} and \cite{bgv18} for ${\cal T}$-admissible and ${\cal H}$-admissible meshes, respectively. ${\cal H}$-admissible refinements were also considered in \cite[Algorithm~3.1]{ghp17} and \cite{morgenstern17} for $\mu=2$ and $\mu\ge2$, respectively. Details for the implementation of the two refinement algorithms can be found in \cite{bgv18}.

%

\begin{algorithm}[!ht]
\caption{\texttt{refine} (${\cal{H}}$-admissible refinement)}
\label{alg:hrefine}
\begin{algorithmic}
\Require ${\cal{H}}$-admissible mesh $\hat\hmesh$, marked elements $\hat{\cal M} \subseteq \hat\hmesh$, and admissibility integer $\mu$ 
\Repeat
\State set $\displaystyle \hat{\mathcal{U}} = \bigcup_{\hat{Q} \in \hat{\cal M}} {\cal N}_{\cal H}(\hat{Q},\mu)\setminus \hat{\cal M}$
\State set $\hat{\cal M} = \hat{\cal M}\cup \hat{\mathcal{U}}$
\Until {$\hat{\mathcal{U}} = \emptyset$}
\State update $\hat{\hmesh}$ by replacing the elements in $\hat{\cal M}$ by their children\\
\Ensure refined ${\cal{H}}$-admissible mesh $\hat\hmesh$
\end{algorithmic}
\end{algorithm}

\begin{algorithm}[!h]
\caption{\texttt{refine} (${\cal{T}}$-admissible refinement)}
\label{alg:trefine}
\begin{algorithmic}
\Require ${\cal{T}}$-admissible mesh $\hat\hmesh$, marked elements $\hat{\cal M} \subseteq \hat\hmesh$, and admissibility integer $\mu$ 
\Repeat
\State set $\displaystyle \hat{\mathcal{U}} = \bigcup_{\hat{Q} \in \hat{\cal M}} {\cal N}_{\cal T}(\hat{Q},\mu)\setminus\hat{\cal M}$
\State set $\hat{\cal M} = \hat{\cal M}\cup \hat{\mathcal{U}}$
\Until {$\hat{\mathcal{U}} = \emptyset$}
\State update $\hat{\hmesh}$ by replacing the elements in $\hat{\cal{M}}$ by their children\\
\Ensure refined ${\cal{T}}$-admissible mesh $\hat\hmesh$
\end{algorithmic}
\end{algorithm}

A selection of meshes generated by the two algorithms when the finest element in the bottom left corner of the current mesh is marked for refinement is shown in Figure~\ref{fig:adm-ref} for $\mathbf{p}=(1,1)$ and $\mu=2$. Note that at each refinement step, the ${\cal T}$-neighborhood is always empty and, consequently, only the marked element is refined. A more significative comparison is shown in Figure~\ref{fig:adm-ref2}, where a diagonal refinement of the unit square is considered for $\mu=3$ and $\mathbf{p}=(2,2)$, $\mathbf{p}=(3,3)$, $\mathbf{p}=(4,4)$ after six refinement levels, see also \cite[Section~5.1]{bgv18} for different values of $\mu$. 

\begin{figure}[!t]\centering
\subfigure[initial mesh and marked element at step 0]{
\begin{tikzpicture}[scale=1.15]
\fill[white!40!white] (0,0) rectangle (2,2);
\end{tikzpicture}
\hspace*{.2cm}
\begin{tikzpicture}[scale=1.15]
\draw (0,0) grid (2,2);
\draw [fill=lightgray] (0,0) rectangle (1/2,1/2);
\foreach \a in {1}
	\draw (0,\a/2) -- (1,\a/2);
\foreach \a in {1}
	\draw (\a/2,0) -- (\a/2,1);
\end{tikzpicture}
\hspace*{.2cm}
\begin{tikzpicture}[scale=1.15]
\fill[white!40!white] (0,0) rectangle (2,2);
\end{tikzpicture}}\\
\subfigure[${\cal H}-$admissible meshes at step 1, 2, 3]{\begin{tikzpicture}[scale=1.15]
\draw [fill=gray] (0,1) rectangle (2,2);
\draw [fill=gray] (1,0) rectangle (2,1);
\draw (0,0) grid (2,2);
\foreach \a in {1,3}
	\draw (0,\a/2) -- (2,\a/2);
\foreach \a in {1,3}
	\draw (\a/2,0) -- (\a/2,2);
\foreach \a in {1}
	\draw (0,\a/4) -- (\a/2,\a/4);
\foreach \a in {1}
	\draw (\a/4,0) -- (\a/4,\a/2);
\end{tikzpicture}
\hspace*{.2cm}
\begin{tikzpicture}[scale=1.15]
\draw [fill=gray] (0,1/2) rectangle (1,1);
\draw [fill=gray] (1/2,0) rectangle (1,1/2);
\draw (0,0) grid (2,2);
\foreach \a in {1,3}
	\draw (0,\a/2) -- (2,\a/2);
\foreach \a in {1,3}
	\draw (\a/2,0) -- (\a/2,2);
\foreach \a in {1,3}
	\draw (0,\a/4) -- (1,\a/4);
\foreach \a in {1,3}
	\draw (\a/4,0) -- (\a/4,1);
\foreach \a in {1}
	\draw (0,\a/8) -- (1/4,\a/8);
\foreach \a in {1}
	\draw (\a/8,0) -- (\a/8,1/4);
\end{tikzpicture}
\hspace*{.2cm}
\begin{tikzpicture}[scale=1.15]
\draw [fill=gray] (0,1/4) rectangle (1/2,1/2);
\draw [fill=gray] (1/4,0) rectangle (1/2,1/4);
\draw (0,0) grid (2,2);
\foreach \a in {1,3}
	\draw (0,\a/2) -- (2,\a/2);
\foreach \a in {1,3}
	\draw (\a/2,0) -- (\a/2,2);
\foreach \a in {1,3}
	\draw (0,\a/4) -- (1,\a/4);
\foreach \a in {1,3}
	\draw (\a/4,0) -- (\a/4,1);
\foreach \a in {1,3}
	\draw (0,\a/8) -- (1/2,\a/8);
\foreach \a in {1,3}
	\draw (\a/8,0) -- (\a/8,1/2);
\foreach \a in {1}
	\draw (0,\a/16) -- (1/8,\a/16);
\foreach \a in {1}
	\draw (\a/16,0) -- (\a/16,1/8);
\end{tikzpicture}
}\\
\subfigure[${\cal T}-$admissible meshes at step 1, 2, 3]{\label{fig:T-admissible,notH}
\begin{tikzpicture}[scale=1.15]
\draw (0,0) grid (2,2);
\foreach \a in {1}
	\draw (0,\a/2) -- (1,\a/2);
\foreach \a in {1}
	\draw (\a/2,0) -- (\a/2,1);
\foreach \a in {1}
	\draw (0,\a/4) -- (1/2,\a/4);
\foreach \a in {1}
	\draw (\a/4,0) -- (\a/4,1/2);
\end{tikzpicture}
\hspace*{.2cm}
%
%
\begin{tikzpicture}[scale=1.15]
\draw (0,0) grid (2,2);
\foreach \a in {1}
	\draw (0,\a/2) -- (1,\a/2);
\foreach \a in {1}
	\draw (\a/2,0) -- (\a/2,1);
\foreach \a in {1}
	\draw (0,\a/4) -- (1/2,\a/4);
\foreach \a in {1}
	\draw (\a/4,0) -- (\a/4,1/2);
\foreach \a in {1}
	\draw (0,\a/8) -- (1/4,\a/8);
\foreach \a in {1}
	\draw (\a/8,0) -- (\a/8,1/4);
\end{tikzpicture}
%
%
\hspace*{.2cm}
\begin{tikzpicture}[scale=1.15]
\draw (0,0) grid (2,2);
\foreach \a in {1}
	\draw (0,\a/2) -- (1,\a/2);
\foreach \a in {1}
	\draw (\a/2,0) -- (\a/2,1);
\foreach \a in {1}
	\draw (0,\a/4) -- (1/2,\a/4);
\foreach \a in {1}
	\draw (\a/4,0) -- (\a/4,1/2);
\foreach \a in {1}
	\draw (0,\a/8) -- (1/4,\a/8);
\foreach \a in {1}
	\draw (\a/8,0) -- (\a/8,1/4);
\foreach \a in {1}
	\draw (0,\a/16) -- (1/8,\a/16);
\foreach \a in {1}
	\draw (\a/16,0) -- (\a/16,1/8);
\end{tikzpicture}}
\caption{{${\cal H}$-admissible} (b) and {${\cal T}$-admissible} (c) meshes generated by Algorithm~\ref{alg:hrefine} and \ref{alg:trefine}, respectively, by refining three times the finest element in the bottom left corner of the mesh with $\mathbf{p}=(1,1)$ and $\mu=2$. The initial mesh and a marked element at step 0 are  shown as well (a). At each step, the dark gray elements appear by refinement of the neighborhood of the previous marked element.
All internal knots have multiplicity one.}
\label{fig:adm-ref}
\end{figure}
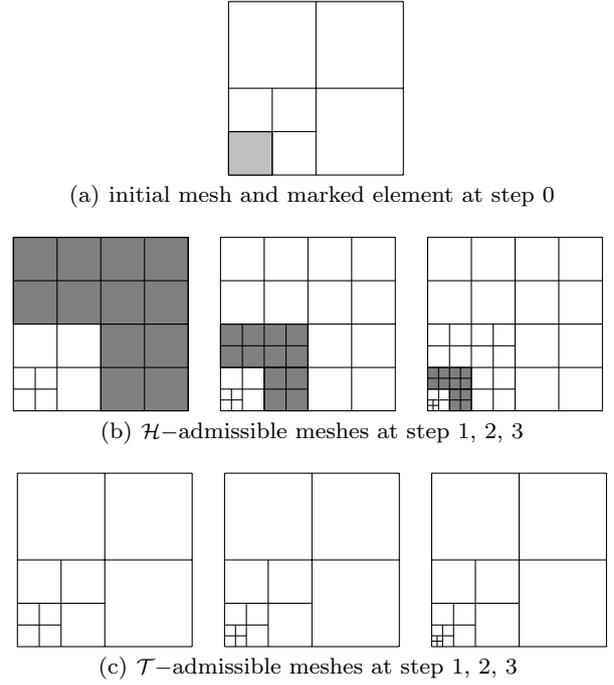

\begin{figure}[!t]\centering
\subfigure[$\mathbf{p}=(2,2)$]{
\hspace*{-0.7cm}
\includegraphics[scale=.175]{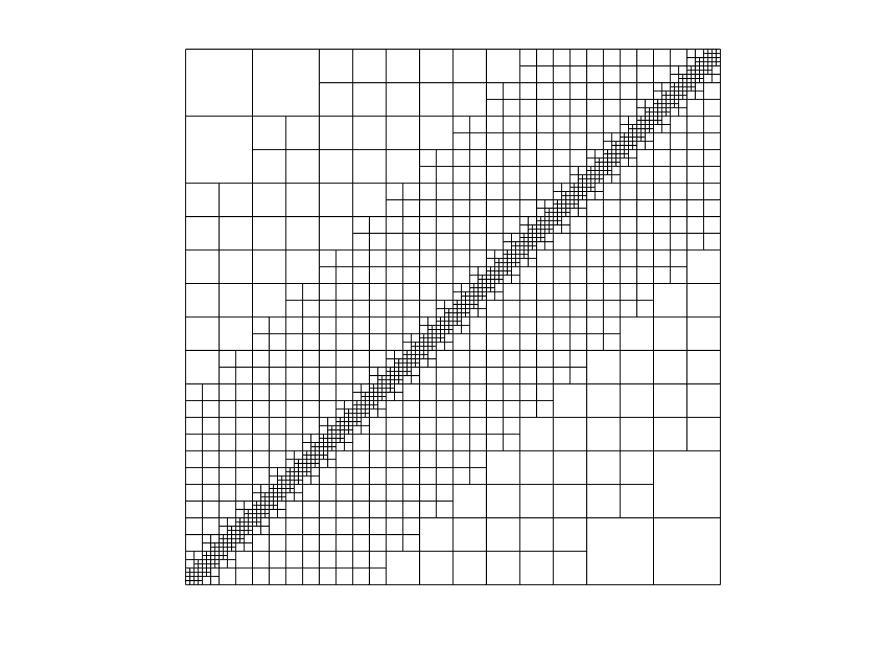}
\hspace*{-1.1cm}
\includegraphics[scale=.175]{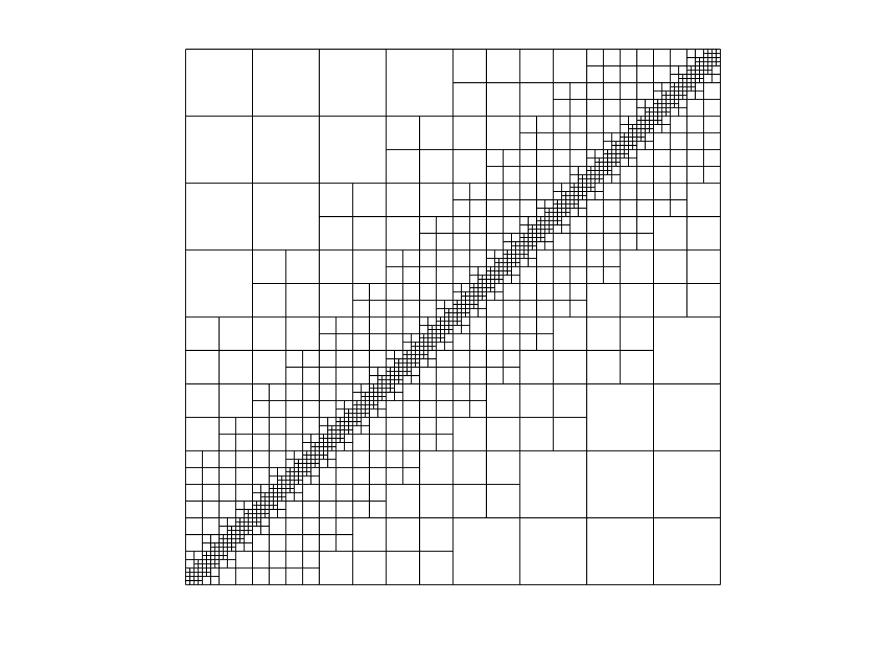}
}
\subfigure[$\mathbf{p}=(3,3)$]{
\hspace*{-0.7cm}
\includegraphics[scale=.175]{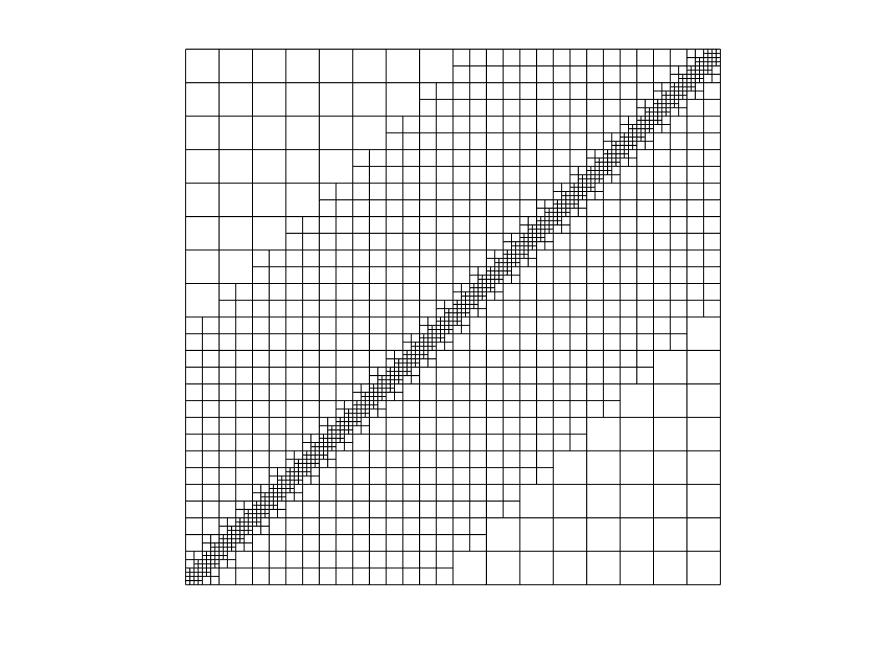}
\hspace*{-1.1cm}
\includegraphics[scale=.175]{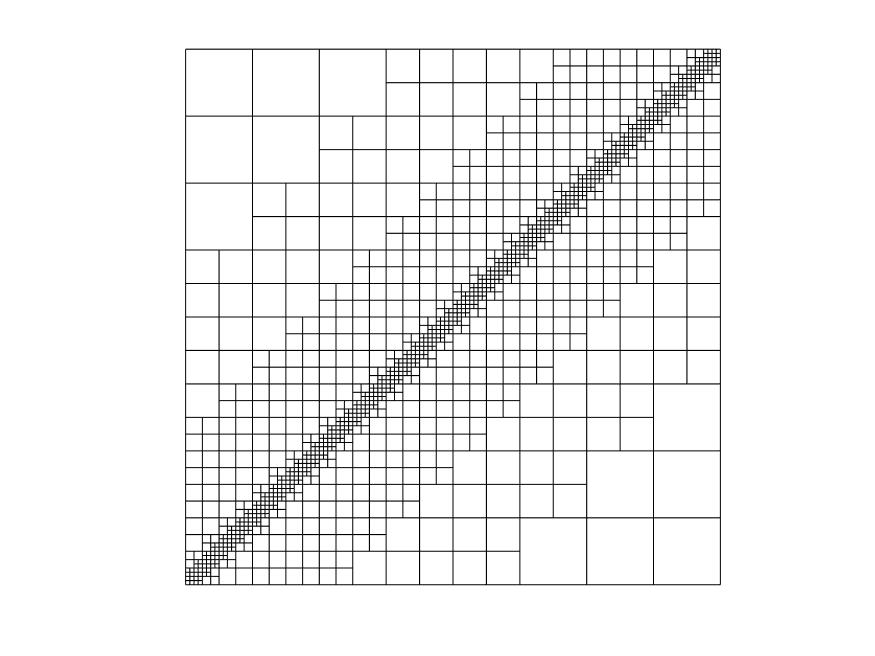}
}
\subfigure[$\mathbf{p}=(4,4)$]{
\hspace*{-0.7cm}
\includegraphics[scale=.175]{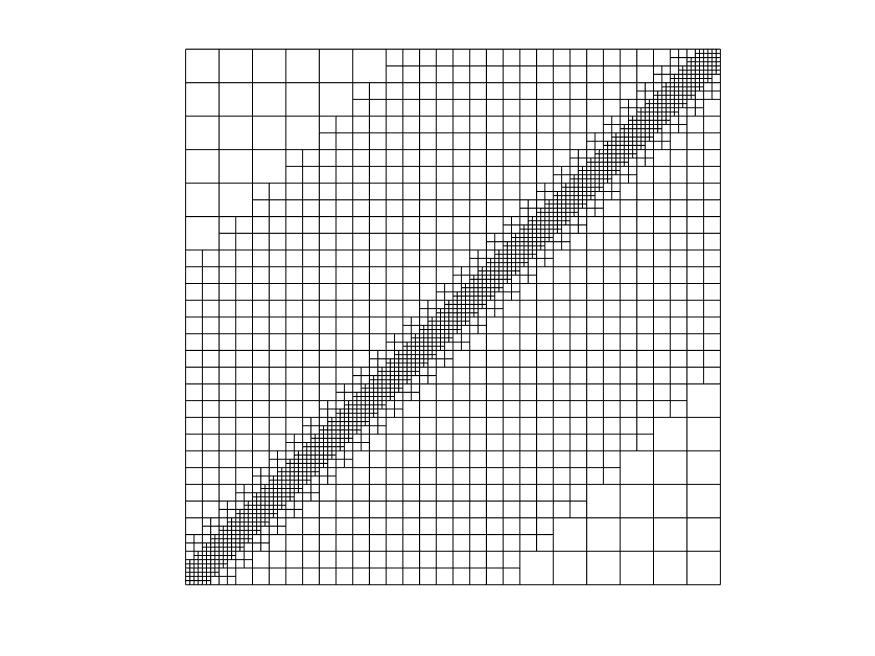}
\hspace*{-1.1cm}
\includegraphics[scale=.175]{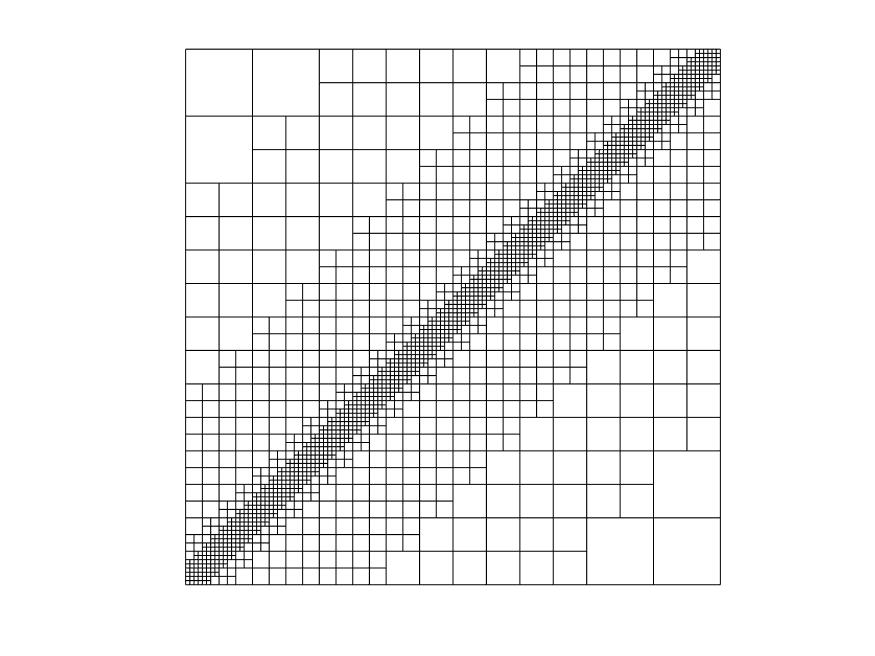}
}
\caption{Diagonal refinement of the unit square, starting from a uniform $4\times4$ mesh, after six refinement steps: ${\cal H}$-admissible (left) and ${\cal T}$-admissible (right) meshes generated by Algorithm~\ref{alg:hrefine} and \ref{alg:trefine}, respectively. Results for $\mu=3$ and $\mathbf{p}=(2,2)$, $\mathbf{p}=(3,3)$, $\mathbf{p}=(4,4)$. At each refinement step, we mark a strip of $2\lceil\frac{p+1}{2}\rceil$ cells centered at the diagonal. This naturally guarantees that in each step functions of the finest level are activated.
All internal knots have multiplicity one.}
\label{fig:adm-ref2}
\end{figure}



The properties of the refine modules were detailed in \cite{bg16,bgmp16} and \cite{ghp17,morgenstern17} for ${\cal T}$- and ${\cal H}$-admissible meshes, respectively, and subsequently revisited in \cite{bgv18} in a unified framework. Proposition~\ref{prn:htadmissible} guarantees that Algorithm~\ref{alg:hrefine} and Algorithm~\ref{alg:trefine} generate a refined hierarchical mesh characterized by the same admissibility properties of the input mesh, see \cite[Proposition~2]{bgv18}.

\begin{proposition}\label{prn:htadmissible}
Let $\hat{\cal Q}$, $\hat{\cal M}$, and $\mu$ be the input arguments of Algorithm~\ref{alg:hrefine} (respectively Algorithm~\ref{alg:trefine}), where $\hat{\cal Q}$ is ${\cal H}$-admissible (respectively ${\cal T}$-admissible) of class $\mu$. Then, the considered algorithm returns a refined hierarchical mesh, $\hat{\cal Q}_\fine \succeq \hat{\QQ}$, which is ${\cal H}$-admissible (respectively ${\cal T}$-admissible) of class $\mu$.
\end{proposition}


For fixed $\mu$ and fixed ${\cal{H}}$-admissible or ${\cal{T}}$-admissible refinement, we abbreviate $\hat\Q:= \refine(\hat\QQ_0)$ as the set of meshes that can be obtained by iterative application of admissible refinement to the initial mesh $\hat\QQ_0:=\hat\QQ^0$. In fact, $\refine(\hat\QQ_0)$ coincides with the whole set of admissible meshes that are obtained by refinement of $\hat\QQ_0$, see \cite[Propositions 3.1.8 and 4.2.3]{morgenstern17} for the ${\cal H}$-admissible and ${\cal T}$-admissible meshes, respectively. See also \cite[Prop.~5.1]{ghp17} for the proof in the case of ${\cal H}$-admissible meshes with $\mu=2$.

The following 
proposition provides a bound on the possible overrefinement of the algorithm to preserve admissibility. 
It was proved in \cite[Theorem~13]{bgmp16} and \cite[Theorem~3.1.12]{morgenstern17} for ${\cal T}$- and $\HH$-admissible refinement algorithms, respectively. 
The case $\mu=2$ for HB-splines was also addressed in \cite[Section 5.4]{ghp17}.
The original versions for triangular meshes go back to \cite{bdd04} and \cite{stevenson07} .

\begin{proposition}\label{prop:hierarchical lincomp}
There exists a uniform constant $C>0$ such that for arbitrary sequences $(\QQ_\k)_{\k\in\N_0}$ in $\Q$ with  $\QQ_{\k+1}=\refine(\QQ_\k,\MM_\k)$ for some $\MM_\k\subseteq \QQ_\k$ and all $\k\in\N_0$, it holds that
\begin{align*}
\# \QQ_\k-\#\QQ_0\le C \sum_{j=0}^{\k-1}\#\MM_j \quad \text{for all }k\in\N_0.
\end{align*}
The constant $C$ depends only on the dimension $\dpa$, the degrees $p_i$, and the initial mesh $\hat{\cal Q}_0$.
\end{proposition}

\subsubsection{Hierarchical quasi-interpolation}\label{sec:hierarchical interpolation}

The THB-spline property of \emph{preservation of coefficients} \cite{giannelli2014} enables the definition of hierarchical quasi-in\-ter\-po\-lation operators that do not require additional computations with respect to the tensor-product case \cite{sm16}. 

For each level $\ell=0,\ldots,N-1$, we consider the quasi-interpolant into the B-spline space of level $\ell$
\begin{equation} \label{eq:qi-local-projection}
\hat J_{{\bf p},{\bf T}^\ell}:L^2(\widehat\Omega)\to \widehat{\mathbb{S}}^{\ell}_{\bf p}({\bf \kv}^{\ell}),\quad \hat v\mapsto \sum_{{\bf i}\in{\cal I}^\ell} \hat\lambda_{{\bf i},{\bf p}}^\ell(\hat v) \hat{B}_{{\bf i},{\bf p}}^\ell,
\end{equation}
with ${\cal I}^\ell := \{ {\bf i}: \hat{B}_{{\bf i},{\bf p}}^\ell \in \spbasis^\ell \}$, and
each functional $\hat\lambda_{{\bf i},{\bf p}}^\ell(\hat v) $ is defined via a local projection onto one element that belongs to the support of the corresponding B-spline 
as described in Section~\ref{sec:qi-2d}, see also \cite{bggs16}. 

By construction of (T)HB-splines, it is easy to see that for each (T)HB-spline of level $\ell$ there exists within its support an element in $\hat\hmesh$ of the same level, which is contained in $\hat\Omega^{\ell} \setminus \hat\Omega^{\ell+1}$ (i.e., it is in $\hat \hmesh \cap \hat \hmesh^\ell$). Its size is obviously  equivalent to the size of the support, in the sense that their ratio is uniformly bounded.
With this choice of the element, the hierarchical quasi-interpolant $\projH{\mathbf{p}}$ can then be defined as 
\[
\projH{\mathbf{p}}:L^2(\widehat\Omega)\to \widehat{\mathbb{S}}_{\bf p}^{\rm H}(\hat\hmesh,{\bf \kv}^{0}),\,\, \hat v\mapsto \sum_{\ell=0}^{N-1} \sum_{{\bf i}\in{\cal I}_{\hat{\cal Q}}^\ell}^{}
\hat\lambda_{{\bf i},{\bf p}}^\ell(\hat v) \hat{T}_{{\bf i},{\bf p}}^\ell,
\]
where ${\cal I}_{\hat{\cal Q}}^\ell$ is the set of indices corresponding to active basis functions of level $\ell$, namely
\begin{equation}\label{eq:qiindices}
{\cal I}_{\hat{\cal Q}}^\ell := \left\{ {\bf i}: 
\hat{B}_{{\bf i},{\bf p}}^\ell\in \hat {\cal B}^\ell \cap\hat{\cal H}_{\bf p}(\hat\hmesh,{\bf T}^0)\right\},
\end{equation}
with $\hat{B}_{{\bf i},{\bf p}}^\ell= \mot\hat{T}_{{\bf i},{\bf p}}^\ell$. 
According to \cite[Theorem~4]{sm16}, 
the quasi-interpolant is in fact a projector as stated in the following proposition. 

\begin{proposition}\label{prop:hb projection}
For an arbitrary hierarchical (not necessarily admissible) mesh $\hat\hmesh$, it holds that
\begin{equation*} 
\projH{\mathbf{p}} \hat S = \hat S \quad \text{ for all } \hat S \in \widehat{\mathbb{S}}_{\bf p}^{\rm H}(\hat\hmesh,{\bf \kv}^{0}).
\end{equation*}
\end{proposition}

As a simple corollary and from the definition of the dual functionals, the quasi-interpolant is also a local projector. Let us define for $\hat{Q}\in\hat\hmesh^\ell$ a modified support extension, given by
\begin{align*}
& \sextTHB{\hat{Q}} := \\
&   \quad\bigcup \left\{
\extsuppTHB : \hat{T}_{{\bf i},{\bf p}}^\ell \in\hat{\cal T}_{\bf p}(\hat\hmesh,{\bf T}^0):\ 
\extsuppTHB \cap \hat Q\ne\emptyset  \right\}, \notag
  \end{align*}
where
\[
\extsuppTHB := \supp({\trunc}^{\ell+1} \hat{B}_{{\bf i},{\bf p}}^\ell )
\]
identifies the extended support of the THB-spline $\hat{T}_{{\bf i},{\bf p}}^\ell$, i.e., the support when only the first level of truncation has been applied.
The THB-splines on a $\TT$-admissible mesh considered in the definition of $\sextTHB{\hat{Q}}$ vary from level $\max(0,\ell-\mu+1)$ to $\ell$, see Remark~\ref{rem:strictly admissible}.

\begin{corollary}\label{cor:local projection}
For any $\elemp \in \hat\hmesh$ it holds that 
\[
(\projH{\mathbf{p}} \hat S)|_{\elemp} = \hat S|_{\elemp}\quad \text {if } \hat S|_{\sextTHB{\hat{Q}}} \in \widehat{\mathbb{S}}_{\bf p}^{\rm H}(\hat\hmesh,{\bf \kv}^{0})|_{\sextTHB{\hat{Q}}}.
\]
\end{corollary}

\begin{remark}
Actually, the locality result is also valid for a set of elements smaller than $\sextTHB{\hat{Q}}$, where the extended support $\hat{\omega}_{{\bf i},{\bf p}}^\ell$ is replaced by  $\supp({\hat{T}_{{\bf i},{\bf p}}^\ell})$. 
\end{remark}

The next property can be found in \cite[Theorem~4]{buffa2019}, and it implies that the number of active elements contained in $\sextTHB{\hat{Q}}$ is uniformly bounded.
\begin{proposition} \label{prop:Sext-star}
Let $\hat \QQ$ be a ${\cal T}$-admissible mesh of class $\mu$, and let $\hat Q \in \hat \QQ \cap \hat \QQ^\ell$. Then, the set $\sextTHB{\hat{Q}}$ is connected. Moreover, for any $\hat Q' \in \hat \QQ$ with $\hat Q' \subseteq \sextTHB{\hat{Q}}$ it holds that 
\[
|\hat Q'| \simeq | \sextTHB{\hat{Q}}|,
\]
where the hidden constants depend on the degrees $p_i$, the admissibility class $\mu$, the dimension $\dpa$, and the initial mesh $\hat\QQ_0$.
\end{proposition}

The next result is a stability property analogous to Proposition~\ref{prop:schumaker bezier in nd} in the tensor-product case. The proof can be found in \cite{buffa2019} for a slightly modified operator onto the space $\widehat{\mathbb{S}}_{\bf p}^{\rm H}(\hat\hmesh,{\bf \kv}^{0})\cap H_0^1(\widehat\Omega)$, but works the same in our case.
On ${\cal H}$-admissible meshes of class $\mu=2$, a similar result is also given in~\cite{ghp17}.
\begin{proposition}\label{prop:hqi}
Let $\hat\hmesh$ be either an ${\cal H}$-admissible or ${\cal T}$-admissible mesh of class $\mu$. 
There exists a constant $C$
such that for any element $\hat{Q}\in\hat\hmesh$ it holds that
\begin{align*}
|| \projH{\mathbf{p}}\hat v ||_{L^2(\hat{Q})} &\le C
||\hat v||_{L^2(\sextTHB{\hat{Q}})}, 
\end{align*}
for all $\hat v\in L^2(\hat \Omega)$.  
The constant $C
$ depends only on the dimension $\dpa$, the degrees $p_i$, and the initial mesh $\hat{\cal Q}_0$.
\end{proposition}


Finally, with the help of the local projector, we prove a result regarding a local characterization for refined spaces. We note that this result does not require admissibility. We start introducing some notation. Let $\hat \QQ$ be a hierarchical mesh, and $\hat \QQ_\fine$ another hierarchical mesh obtained by refinement, i.e., $\hat\QQ \preceq \hat \QQ_\fine$. 
We recall the definition of the sets of indices ${\cal I}_{\hat{\cal Q}}^\ell$ given by \eqref{eq:qiindices}, and define analogously the sets ${\cal I}_{\hat{\cal Q}_\fine}^\ell$ for levels $\ell = 0, \ldots, N_\fine-1$. This allows us to introduce their splitting in disjoint index sets as follows
\[
{\cal I}_{\hat{\QQ}}^\ell := 
{\cal I}_{\hat{\QQ}\rightarrow {\hat\QQ_\fine}}^{\ell,\rm fix} \cup {\cal I}_{\hat{\QQ}}^{\ell,\rm old},\qquad 
{\cal I}_{\hat{\QQ}_\fine}^\ell := 
{\cal I}_{\hat{\QQ} \rightarrow {\hat\QQ_\fine}}^{\ell,\rm fix} \cup {\cal I}_{\hat{\QQ}_\fine}^{\ell,\rm new},
\]
for $\ell=0,\ldots,N-1$ and $\ell=0,\ldots,N_\fine-1$, respectively, with
\begin{equation}\label{eq:qi-indices1}
{\cal I}_{\hat{\QQ}\rightarrow {\hat\QQ_\fine}}^{\ell,\rm fix} := {\cal I}_{\hat{\QQ}}^\ell \cap {\cal I}_{\hat{\QQ}_\fine}^\ell,
\end{equation}
i.e., indices related to functions which are active in both meshes, although they may differ by truncation, and
\begin{equation}\label{eq:qi-indices2}
{\cal I}_{\hat{\QQ}}^{\ell,\rm old} := {\cal I}_{\hat{\QQ}}^\ell \setminus
{\cal I}_{\hat{\QQ}\rightarrow {\hat\QQ_\fine}}^{\ell,\rm fix},\quad
{\cal I}_{\hat{\QQ}_\fine}^{\ell,\rm new} := {\cal I}_{\hat{\QQ}_\fine}^\ell
\setminus
{\cal I}_{\hat{\QQ}\rightarrow {\hat\QQ_\fine}}^{\ell,\rm fix},
\end{equation} 
indices of basis functions that are respectively removed or added after refinement.
We also introduce the set of elements in the support of the new functions, the domain they cover, and its (closed) complementary, respectively denoted by
\begin{align} 
& {\cal \hat R_\fine} : = \{\elemp \in\hat\QQ_\fine : \exists \ell \in \{0, \ldots, N_\fine-1\} \, \exists {\bf i} \in {\cal I}_{\hat{\QQ}_\fine}^{\ell,\rm new} \nonumber \\
& \qquad \text{ such that }\elemp \subset \supp(\hat{T}_{\fine,{\bf i},{\bf p}}^\ell)\},  \label{eq:Rfine} \\
& \hat \Omega_{\cal \hat R_\fine} := \bigcup \big\{ \overline \elemp : \elemp \in {\cal \hat R}_\fine\big\}, \quad \hat \Omega_{\cal \hat Q} := \overline{\Omegap \setminus \hat \Omega_{\cal \hat R_\fine}}.\nonumber
\end{align}
\begin{proposition} \label{prop:hb-refined-projection}
Let $\hat\QQ, \hat \QQ_\fine$ be two hierarchical meshes with $\hat \QQ \preceq \hat \QQ_\fine$. 
Moreover, for every ${\bf i} \in {\cal I}_{\hat{\QQ}\rightarrow {\hat\QQ_\fine}}^{\ell,\rm fix}$, $\ell = 0, \ldots, N-1$, we choose the  element for the coefficients in \eqref{eq:qi-local-projection} to be in $\hat \hmesh \cap \hat \hmesh_\fine \cap \hat \hmesh^\ell$. 
 Then it holds that
\[
(\hat J_{\mathbf{p},\hat{\hmesh}}^{\,\rm H} \hat S)|_{\hat \Omega_{\cal \hat Q}} = \hat S|_{\hat \Omega_{\cal \hat Q}}
\quad \text{for all }
\hat S \in \widehat{\mathbb{S}}_{\bf p}^{\rm H}(\hat\hmesh_\fine,{\bf \kv}^{0}).
\]
\end{proposition}
\begin{proof}
Let 
\begin{align}\label{eq:qi-1}
\projH{\mathbf{p}} \hat S &= 
\sum_{\ell=0}^{N-1} \sum_{{\bf i}\in{\cal I}_{\hat{\cal Q}}^\ell}^{}
\hat\lambda_{{\bf i},{\bf p}}^\ell(\hat S) \hat{T}_{{\bf i},{\bf p}}^\ell, 
\end{align}
and
\begin{align}
\label{eq:qi-2}
\projHfine{\mathbf{p}} \hat S &= 
\sum_{\ell=0}^{N_\fine-1} \sum_{{\bf i}\in{\cal I}_{\hat\QQ_\fine}^\ell}^{}
\hat\lambda_{\fine,{\bf i},{\bf p}}^\ell(\hat S) \hat{T}_{\fine,{\bf i},{\bf p}}^{\ell},
\end{align}
where $\hat{B}_{{\bf i},{\bf p}}^\ell= \mot\hat{T}_{{\bf i},{\bf p}}^\ell$ or $\hat{B}_{{\bf i},{\bf p}}^{\ell}= \mot\hat{T}_{\fine,{\bf i},{\bf p}}^{\ell}$,
be the two hierarchical quasi-interpolants expressed in terms of the truncated bases. 
For the operator~\eqref{eq:qi-2}, we choose for every ${\bf i} \in {\cal I}_{\hat{\QQ}\rightarrow {\hat\QQ_\fine}}^{\ell,\rm fix}$ the same  element as for \eqref{eq:qi-1}.
By using the index sets introduced in \eqref{eq:qi-indices1} and \eqref{eq:qi-indices2}, we can rewrite the inner sums in \eqref{eq:qi-1} and \eqref{eq:qi-2} respectively as 
\begin{align*}
&\sum_{{\bf i}\in{\cal I}_{\hat{\QQ}\rightarrow {\hat\QQ_\fine}}^{\ell,\rm fix}}
\hat\lambda_{{\bf i},{\bf p}}^\ell(\hat S) \hat{T}_{{\bf i},{\bf p}}^\ell
+ \sum_{{\bf i}\in{\cal I}_{\hat{\cal Q}}^{\ell,\rm old}}^{}
\hat\lambda_{{\bf i},{\bf p}}^\ell(\hat S) \hat{T}_{{\bf i},{\bf p}}^\ell
\end{align*}
and
\begin{align*}
&\sum_{{\bf i}\in
{\cal I}_{\hat{\QQ}\rightarrow {\hat\QQ_\fine}}^{\ell,\rm fix}}
\hat\lambda_{\fine,{\bf i},{\bf p}}^\ell(\hat S) \hat{T}_{\fine,{\bf i},{\bf p}}^{\ell}+
\sum_{{\bf i}\in{\cal I}_{\hat\QQ_\fine}^{\ell,\rm new}}^{}
\hat\lambda_{\fine,{\bf i},{\bf p}}^\ell(\hat S) \hat{T}_{\fine,{\bf i},{\bf p}}^{\ell}.
\end{align*}
For any index ${\bf i}\in {\cal I}_{\hat{\QQ}\rightarrow {\hat\QQ_\fine}}^{\ell,\rm fix}$, the definition of $\hat \Omega_{\hat{\cal Q}}$ and our assumptions for the coefficients of the quasi\--interpolants show that 
\[
\hat{T}_{{\bf i},{\bf p}}^{\ell}|_{\hat \Omega_{\cal \hat Q}} = 
\hat{T}_{\fine, {\bf i},{\bf p}}^{\ell}|_{\hat \Omega_{\cal \hat Q}},\qquad
\lambda_{{\bf i},{\bf p}}^\ell(\hat S) 
=
\lambda_{\fine,{\bf i},{\bf p}}^\ell(\hat S). 
\]
For any index ${\bf i}\in {\cal I}_{\hat{\QQ}}^{\ell,\rm old}$ or ${\bf i}\in {\cal I}_{\hat{\QQ}}^{\ell,\rm new}$, we have
\[
\hat{T}_{{\bf i},{\bf p}}^{\ell}|_{\hat \Omega_{\cal \hat Q}} = 0,
\quad
\hat{T}_{\fine,{\bf i},{\bf p}}^{\ell}|_{\hat \Omega_{\cal \hat Q}} = 0.
\]
Consequently, since $\projHfine{\mathbf{p}}$ is a projector onto the space $\widehat{\mathbb{S}}_{\bf p}^{\rm H}(\hat\hmesh_\fine,{\bf \kv}^{0})$,  we obtain that
\begin{align*}
(\hat J_{\mathbf{p},\hat{\hmesh}}^{\,\rm H} \hat S)|_{\hat \Omega_{\cal \hat Q}} 
& = \sum_{\ell=0}^{N-1}
\sum_{{\bf i}\in{\cal I}_{\hat{\QQ}\rightarrow {\hat\QQ_\fine}}^{\ell,\rm fix}}
\hat\lambda_{{\bf i},{\bf p}}^\ell(\hat S) \hat{T}_{{\bf i},{\bf p}}^\ell|_{\hat \Omega_{\cal \hat Q}} 
\\
& = \sum_{\ell=0}^{N-1}
\sum_{{\bf i}\in{\cal I}_{\hat{\QQ}\rightarrow {\hat\QQ_\fine}}^{\ell,\rm fix}}
\hat\lambda_{\fine,{\bf i},{\bf p}}^\ell(\hat S) \hat{T}_{\fine,{\bf i},{\bf p}}^\ell|_{\hat \Omega_{\cal \hat Q}} 
\\
& = (\hat J_{\mathbf{p},\hat{\hmesh}_\fine}^{\,\rm H} \hat S)|_{\hat \Omega_{\cal \hat Q}}  = \hat S|_{\hat \Omega_{\cal \hat Q}}.
\end{align*}
This concludes the proof. \hfill$\square$
\end{proof}

\begin{corollary} \label{corol:characterization}
Let $\hat\QQ, \hat \QQ_\fine$ be two hierarchical meshes with $\hat \QQ \preceq \hat \QQ_\fine$.
Then, their associated spaces of hierarchical splines coincide in $\hat \Omega_{\cal \hat Q}$, i.e., it holds that
\[
\widehat{\mathbb{S}}_{\bf p}^{\rm H}(\hat\hmesh,{\bf \kv}^{0}) |_{\hat \Omega_{\cal \hat Q}} = \widehat{\mathbb{S}}_{\bf p}^{\rm H}(\hat\hmesh_\fine,{\bf \kv}^{0}) |_{\hat \Omega_{\cal \hat Q}}.
\]
\end{corollary}

Finally, we remark that although our construction is based on \cite{sm16,bggs16}, hierarchical quasi-interpolation with HB-splines and THB-splines was also studied in \cite{kraft1997} and \cite{bracco2016d,speleers2017}, respectively.






\subsubsection{Hierarchical splines refined by functions}\label{subsec:simple hb}

An alternative viewpoint for the construction of hierarchical splines is to consider a refinement algorithm which does not refine the elements, but instead it refines basis functions. In the context of adaptive methods for PDEs, this idea can be traced back at least to \cite{CHARMS,Krysl-IJNME}, and it was recently improved by M. Sabin in \cite{Sabin17} to easily deal with possible linear dependencies of basis functions. The basic idea is to replace any marked basis function with their children, that are the basis functions of the next level appearing in \eqref{eq:s}. 

In terms of the analysis of adaptive methods, refinement by functions is studied in \cite{actis2018}. To avoid possible linear dependencies of basis functions, the authors suggest to use what they call \emph{absorbing generators}, which in fact are equivalent to the \emph{simplified hierarchical splines} in \cite{bgarau16_1}.
In the latter, basis functions are refined (deactivated) according to the elements in their support, but only children of refined basis functions can be activated. Their definition can be done with a recursive algorithm similar to the one of HB-splines in Section~\ref{subsec:def hb}:
\begin{enumerate}
\item $\hat\hbbasis_s^0 := \spbasis^0$;
\item for $\ell=0,\ldots, N-2$ 
\[
\hat\hbbasis_s^{\ell+1} := \hat\hbbasis_{s,A}^{\ell+1} \cup \hat\hbbasis_{s,B}^{\ell+1},
\]
\end{enumerate}
where $\hat\hbbasis_{s,A}^{\ell+1}$ is defined analogously to $\hat\hbbasis_{A}^{\ell+1}$ in the HB-splines case, while $\hat\hbbasis_{s,B}^{\ell+1}$ is given by the sets of children 
\[
\hat\hbbasis_{s,B}^{\ell+1}:= \bigcup_{\stackrel{\widehat\beta \in \hat \hbbasis_s^{\ell+1}}{\supp (\widehat\beta) \subset \Omega^{\ell+1}}} 
\left\{\hat{B}_{{\bf i},{\bf p}}^{\ell+1} \in \spbasis^{\ell+1} : c^{\ell+1}_{\bf i, \bf p}(\widehat\beta) \not = 0\right\},
\]
and the coefficients $c_{\bf i, \bf p}^{\ell+1}(\widehat\beta)$ are defined in \eqref{eq:s}.

In \cite{bgarau16_1} it is proved that, for every $\ell= 0, \ldots, N-1$, the set of simplified HB-splines is contained in the set of HB-splines, namely $\hat \hbbasis_{s}^{\ell} \subseteq \hat \hbbasis^{\ell}$.
Indeed, with the unique nonnegative coefficients from $\sum_{\widehat\beta\in\hat \hbbasis^{\ell}} c_{\widehat\beta} \widehat\beta=1$, it even holds that $\hat\hbbasis_{s}^{\ell}=\set{\widehat\beta\in\hat \hbbasis^{\ell}}{c_{\widehat\beta}>0}$.
Moreover, the first three properties of Proposition~\ref{prop:hb properties} are proved for these simplified hierarchical splines in the same paper. 
In \cite{actis2018}, the authors introduce a refinement algorithm by functions that, analogously to the admissibility property presented above, prevents the interaction of coarse and fine functions, and they prove that the algorithm has linear complexity with respect to the number of marked basis functions, see Proposition~\ref{prop:hierarchical lincomp} 
 for the analogous result in an element-wise version. 

The definition of a multilevel quasi-interpolant for simplified hierarchical splines is also given in \cite{bgarau16_1}. This quasi-interpolant generalizes the one introduced by Kraft for HB-splines in \cite{kraft1997} to general knot vectors.
However, the quasi-interpolant is not a projector.

Regarding \textsl{a posteriori} error estimation, an estimator based on basis functions was introduced in \cite{bgarau16_1}, although as far as we know only the upper bound of the error has been proved. 


\subsection{T-splines}\label{subsec:tsplines}

An alternative for the development of adaptive isogeometric methods is the use of \textit{T-splines}, which were introduced for CAD and computer graphics by T.~Sederberg et al. in  \cite{sederberg2003,sederberg2004}. They were soon recognized as an interesting tool to develop adaptivity in IGA \cite{bazilevs2010,doerfel2010}. A sound mathematical theory for approximation with T-splines was missing at that time and in \cite{bcs10} the first counterexample of linearly dependent T-splines was presented along with preliminary results about linear independence. The mathematical analysis of T-splines made a big step forward with the introduction of \emph{analysis suitable T-splines} in \cite{li2012} and the equivalent concept (under the mild assumption that facing T-junctions do not exist) of \emph{dual-compatible T-splines}
in \cite{beirao2012}, for which it was possible to construct a dual basis, and consequently to prove linear independence. While these concepts were first restricted to cubic T-splines, they were generalized to arbitrary degree in \cite{bbsv13}, and equivalence was proved under the same assumption. The characterization of the space and some other important properties were analyzed in \cite{ls14,bbs15}. These works are mostly restricted to the two-dimensional case, although the definition of dual-com\-patible T-splines extends to three-dimensional one.

Algorithms for automatic refinement with T-splines were first studied in \cite{sederberg2004}  and for analysis suitable T-splines in \cite{slsh12}. Refinement algorithms that guarantee the dual compatibility property by alternating the direction of refinement were introduced in \cite{mp15} and were later generalized to the trivariate case (with odd degree) in \cite{morgenstern16,morgenstern17}. The advantage of these refinement algorithms over previous ones is that they guarantee linear complexity and also shape-regularity of the mesh avoiding the presence of undesired anisotropic elements.

In this section we present the definition of T-splines, focusing on dual-compatible T-splines  and describe the concept of admissible T-meshes and the refinement algorithms introduced in \cite{mp15,morgenstern16,morgenstern17}. For the presentation we mainly follow the survey \cite{bbsv14}, which collects results from previous papers, and \cite{morgenstern16}. We will restrict ourselves to the case of T-splines of odd degree, because the analysis of trivariate T-splines in \cite{morgenstern16} has not been extended to arbitrary degree so far.

\subsubsection{The basic idea of T-splines}
T-splines are a generalization of B-splines, where the functions are defined from a mesh of rectangular elements with T-junctions, the so-called \textit{T-mesh}, see Figure~\ref{fig:Tmesh}. The lines of the T-mesh play a similar role as the knot indices in the tensor-product case and a knot value is associated to each of these lines. A T-spline function is then associated to each vertex of the T-mesh (or to each element if the degree is even), in what we call the \emph{anchors}. Each of these functions is defined analogously to a B-spline, and the local knot vector in the $j$-th direction is obtained by tracing a line from the anchor in this direction, and considering the intersections of this line with the T-mesh. For example, in Figure~\ref{fig:Tmesh} with degree ${\bf p} = (5,3)$, the function anchored at the node with indices $(9,8)$ has local knot vectors $\displaystyle \left(\frac{1}{16},\frac{1}{8},\frac{1}{4},\frac{3}{8},\frac{1}{2},\frac{9}{16},\frac{5}{8} \right)$ and $\displaystyle \left(\frac{5}{16},\frac{3}{8},\frac{1}{2},\frac{5}{8},\frac{3}{4}\right)$.
\begin{figure}[ht] 
\begin{center}
\includegraphics[width=0.5\textwidth]{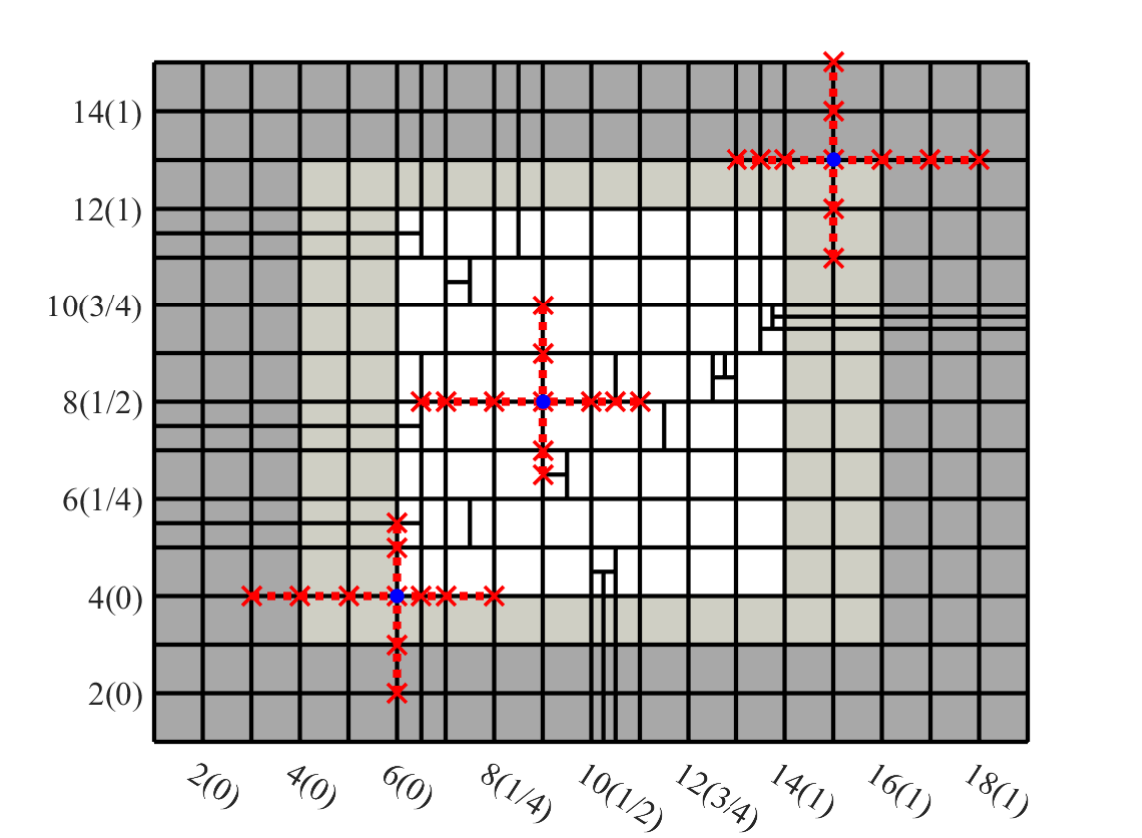}
\end{center}
\caption{
A two-dimensional T-mesh with degree $(p_1,p_2)=(5,3)$. 
For the three (blue) nodes ${\bf z}\in\{(6,4),(9,8),(15,13)\}$, their corresponding local knot vectors are indicated by red crosses. In the axes we indicate the indices in $\indices_j^0$ and, between parentheses, the value of the corresponding knots.}
\label{fig:Tmesh}
\end{figure}

We notice that, due to knot repetitions in the open knot vectors, the first $p$ knot spans in each direction have zero length. These elements with zero measure are colored in gray in Figure~\ref{fig:Tmesh}, and the white region will be called the \textit{index/parametric region}. Moreover, we recall that for B-splines, a knot vector of $n+p+1$ knots defines $n$ functions. To take this into account, the anchors are limited to what is called the \emph{region of active anchors}, given by the white and light gray elements in Figure~\ref{fig:Tmesh}.

Although the idea of T-splines is not complex, the rigorous definition and analysis of their properties require some involved notation, which we present in the following subsections.

\subsubsection{T-meshes refined by bisection}\label{sec:T-meshes defined}
For simplicity, we will limit ourselves to T-meshes where the elements are refined by bisection. We also restrict ourselves to dimension $\dpa=2,3$, as this is the state of the art, although some advances for arbitrary dimension were introduced in \cite{morgenstern17}. Moreover, as we said above we also limit the presentation to odd polynomial degrees. 
For the definition and properties of even and mixed degree T-splines, we refer to \cite{bbsv13,mp15}, see also Remark~\ref{remark:degree} below.

For the ease of reading, we repeat here some of the definitions of the tensor-product case. Let us introduce for $ 1 \le j \le \dpa$, an \emph{odd degree} $p_j \ge 3$, the number of univariate functions $n_j$, the set of indices $\indices_j^0 = \{1,2, \ldots, n_j + p_j + 1 \}$, and the \emph{open} knot vector $\kv_j^0$ $= (\knot_{j,1}, \ldots, \knot_{j,n_j+p_j+1})$ $=$ $(\knot_{j,i})_{i \in \indices_j^0}$, with $\knot_{j,p_j+1} = 0$ and $\knot_{j,n_j+1} = 1$. We also assume that internal knots are not repeated, that is
\[
\knot_{j,i}  < \knot_{j,i+1} \quad \text{ for } i = p_j+1, \ldots, n_j.
\]
Again, we abbreviate ${\bf p}:=(p_1,\dots,p_\dpa)$ as well as ${\bf \kv}^0:=(\kv^0_1,\dots,\kv^0_\dpa)$.

The starting point is the Cartesian grid
\begin{align*}
\Tmesh_0 := & \big\{(l_1, l_1+1) \times \ldots \times (l_\dpa, l_\dpa+1): \\
& \quad 1 \le l_j \le n_j + p_j \text{ for } 1\le j\le \dpa\big\},
\end{align*}
which is a uniform partition of the \emph{index domain} 
\[\Omindex := \Pi_{j=1}^\dpa (1,n_j+p_j+1).
\]
We also introduce the \emph{index/parametric domain} \[
\Omip := \Pi_{j=1}^\dpa (p_j+1,n_j+1).
\]
In Figure~\ref{fig:Tmesh}, $\Omip$ is formed by white elements, while $\Omindex$ is given by all the elements of the mesh. We also define $\Omegap := (0,1)^\dpa$.

For any integer $k >0 $, we define the set of \emph{rational indices}
\begin{align*}
\indices_j^k := \II_j^0 \cup \bigg\{ & i+r: i \in \II_j^0,\, p_j+1 \le i \le n_j,\, \\
& r \in \Big\{\frac{1}{2^k}, \ldots, \frac{2^k-1}{2^k} \Big \} \bigg \}.
\end{align*}
With this, we can define for $ 1 \le j \le \dpa$ and for $k \ge 0$, the ordered \emph{knot vectors at stage $k$}
\[
\kv_j^k := (\knot_{j,r})_{r \in \indices_j^k }
\]
in a recursive way: starting from $\kv_j^0$, for $k > 0$ and for any new index $r \in \indices_j^k \setminus \indices_j^{k-1}$, we define
\[
\knot_{j,r} := \frac{1}{2} \left(\knot_{j,r-\frac{1}{2^k}} + \knot_{j,r+\frac{1}{2^k}}\right),
\]
which is well defined because $r-\frac{1}{2^k}, r+\frac{1}{2^k} \in \indices_j^{k-1}$. Note that we are not inserting new knots between the repeated knots of the open knot vector.

We also define, for an arbitrary hyperrectangular element in the index domain $\elemi = \Pi_{i=1}^\dpa (a_i,b_i)$, the \emph{bisection operator in the $j$-th direction} (compare with \cite[Definition~2.5]{morgenstern16} and \cite[Section~4.1]{cv18})
\begin{equation}\label{eq:bisection_operator}
\bisect_j(\elemi) := \left\{
\begin{array}{ll}
\{ \elemi_j^1, \elemi_j^2 \}
 & \text{ if } \knot_{j,a_j} \not = \knot_{j,b_j}, \\
\{\elemi\} & \text{ if } \knot_{j,a_j} = \knot_{j,b_j},
\end{array}
\right.
\end{equation}
where
\begin{align*}
\elemi_j^1 = \Pi_{i=1}^{j-1}(a_i,b_i) \times \Big(a_j,\frac{a_j+b_j}{2}\Big) \times \Pi_{i=j+1}^{\dpa}(a_i,b_i), \\
\elemi_j^2 = \Pi_{i=1}^{j-1}(a_i,b_i) \times \Big(\frac{a_j+b_j}{2},b_j\Big) \times \Pi_{i=j+1}^{\dpa}(a_i,b_i).
\end{align*}
That is, the element is bisected in the $j$-th direction only if the corresponding knots in this direction are different, otherwise it is left unchanged. In particular, due to the presence of the open knot vector, the first and last $p_j$ ``columns'' of elements in the $j$-th direction are never bisected in this direction.

Setting the level of the elements of the starting mesh $\elemi \in \Tmesh_0$ equal to zero, $\levelT{\elemi} := 0$, we associate to each level $\ell \in \mathbb{N}_0$ the \emph{direction of bisection}
\[
\dirb{\ell}:= \ell+1-\underline{\ell},
\]
with $\underline{\ell} := \lfloor \ell/\dpa \rfloor \dpa$, and we also set the \emph{level} of the elements obtained by bisection
\[
\levelT{\elemi'} := \levelT{\elemi} + 1, \elemi' \in \bisect_{\dirb{\levelT{\elemi}}}(\elemi).
\]
With this choice, the elements will be split into two in alternating directions determined by their level, see the examples in Figure~\ref{fig:bisect2d} and Figure~\ref{fig:bisect3d}.
Note that if an element $\elemi$ is unchanged via bisection because its indices refer to repeated knots, see \eqref{eq:bisection_operator}, we still implicitly distinguish $\elemi$ and $\elemi' := \bisect(\elemi)$ by equipping the latter with a different level.

\begin{figure}
\centering
\tdplotsetmaincoords{0}{0}
\begin{tikzpicture}[scale=2,tdplot_main_coords]
    \coordinate (O) at (0,0,0);
    \tdplotsetcoord{P}{10}{10}{10}
    \draw[line width=0.5mm, color=black] (O) -- (1,0,0);
    \draw[line width=0.5mm, color=black] (O) -- (0,0,1);
    \draw[line width=0.5mm, color=black] (1,0,0) -- (1,0,1);
    \draw[line width=0.5mm, color=black] (1,0,1) -- (1,1,1);
    \draw[line width=0.5mm, color=black] (0,0,1) -- (1,0,1);
    \draw[line width=0.5mm, color=black] (1,0,0) -- (1,1,0);
    \draw[line width=0.5mm, color=black] (1,1,0) -- (1,1,1);
    \draw[line width=0.5mm, color=black] (0,1,1) -- (1,1,1);
    \draw[line width=0.5mm, color=black] (0,0,1) -- (0,1,1);        
\end{tikzpicture}
\qquad
\begin{tikzpicture}[scale=2,tdplot_main_coords]
    \coordinate (O) at (0,0,0);
    \tdplotsetcoord{P}{10}{10}{10}
    \draw[line width=0.5mm, color=black] (O) -- (1,0,0);
    \draw[line width=0.5mm, color=black] (O) -- (0,0,1);
    \draw[line width=0.5mm, color=black] (1,0,0) -- (1,0,1);
    \draw[line width=0.5mm, color=black] (1,0,1) -- (1,1,1);
    \draw[line width=0.5mm, color=black] (0,0,1) -- (1,0,1);
    \draw[line width=0.5mm, color=black] (1,0,0) -- (1,1,0);
    \draw[line width=0.5mm, color=black] (1,1,0) -- (1,1,1);
    \draw[line width=0.5mm, color=black] (0,1,1) -- (1,1,1);
    \draw[line width=0.5mm, color=black] (0,0,1) -- (0,1,1);
    \draw[line width=0.5mm, color=black] (0.5,0,0) -- (0.5,0,1);
    \draw[line width=0.5mm, color=black] (0.5,0,1) -- (0.5,1,1);    
\end{tikzpicture}
\qquad
\begin{tikzpicture}[scale=2,tdplot_main_coords]
    \coordinate (O) at (0,0,0);
    \tdplotsetcoord{P}{10}{10}{10}
    \draw[line width=0.5mm, color=black] (O) -- (1,0,0);
    \draw[line width=0.5mm, color=black] (O) -- (0,0,1);
    \draw[line width=0.5mm, color=black] (1,0,0) -- (1,0,1);
    \draw[line width=0.5mm, color=black] (1,0,1) -- (1,1,1);
    \draw[line width=0.5mm, color=black] (0,0,1) -- (1,0,1);
    \draw[line width=0.5mm, color=black] (1,0,0) -- (1,1,0);
    \draw[line width=0.5mm, color=black] (1,1,0) -- (1,1,1);
    \draw[line width=0.5mm, color=black] (0,1,1) -- (1,1,1);
    \draw[line width=0.5mm, color=black] (0,0,1) -- (0,1,1);
    \draw[line width=0.5mm, color=black] (0.5,0,0) -- (0.5,0,1);
    \draw[line width=0.5mm, color=black] (0.5,0,1) -- (0.5,1,1);    
    \draw[line width=0.5mm, color=black] (1,0.5,0) -- (1,0.5,1);    
    \draw[line width=0.5mm, color=black] (1,0.5,1) -- (0,0.5,1);    
\end{tikzpicture}
\caption{
Example of bisection for $\dpa=2$. The initial element of level 0 is bisected in the $x$-direction ($\dirb{0}=1$) to obtain two elements of level 1. These are then bisected in the $y$-direction ($\dirb{1} = 2$) to obtain the four elements of level 2.}
\label{fig:bisect2d}
\end{figure}
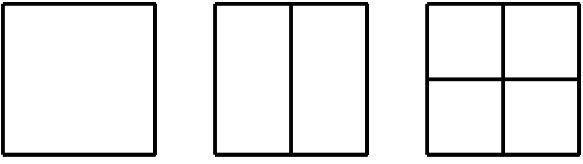

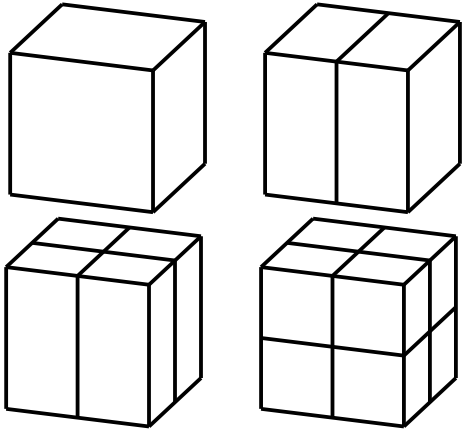
\begin{figure}
\centering
\tdplotsetmaincoords{70}{20}
\begin{tikzpicture}[scale=2,tdplot_main_coords]
    \coordinate (O) at (0,0,0);
    \tdplotsetcoord{P}{10}{10}{10}
    \draw[line width=0.5mm, color=black] (O) -- (1,0,0);
    \draw[line width=0.5mm, color=black] (O) -- (0,0,1);
    \draw[line width=0.5mm, color=black] (1,0,0) -- (1,0,1);
    \draw[line width=0.5mm, color=black] (1,0,1) -- (1,1,1);
    \draw[line width=0.5mm, color=black] (0,0,1) -- (1,0,1);
    \draw[line width=0.5mm, color=black] (1,0,0) -- (1,1,0);
    \draw[line width=0.5mm, color=black] (1,1,0) -- (1,1,1);
    \draw[line width=0.5mm, color=black] (0,1,1) -- (1,1,1);
    \draw[line width=0.5mm, color=black] (0,0,1) -- (0,1,1);        
\end{tikzpicture}
\qquad
\begin{tikzpicture}[scale=2,tdplot_main_coords]
    \coordinate (O) at (0,0,0);
    \tdplotsetcoord{P}{10}{10}{10}
    \draw[line width=0.5mm, color=black] (O) -- (1,0,0);
    \draw[line width=0.5mm, color=black] (O) -- (0,0,1);
    \draw[line width=0.5mm, color=black] (1,0,0) -- (1,0,1);
    \draw[line width=0.5mm, color=black] (1,0,1) -- (1,1,1);
    \draw[line width=0.5mm, color=black] (0,0,1) -- (1,0,1);
    \draw[line width=0.5mm, color=black] (1,0,0) -- (1,1,0);
    \draw[line width=0.5mm, color=black] (1,1,0) -- (1,1,1);
    \draw[line width=0.5mm, color=black] (0,1,1) -- (1,1,1);
    \draw[line width=0.5mm, color=black] (0,0,1) -- (0,1,1);
    \draw[line width=0.5mm, color=black] (0.5,0,0) -- (0.5,0,1);
    \draw[line width=0.5mm, color=black] (0.5,0,1) -- (0.5,1,1);    
\end{tikzpicture}
\qquad
\begin{tikzpicture}[scale=2,tdplot_main_coords]
    \coordinate (O) at (0,0,0);
    \tdplotsetcoord{P}{10}{10}{10}
    \draw[line width=0.5mm, color=black] (O) -- (1,0,0);
    \draw[line width=0.5mm, color=black] (O) -- (0,0,1);
    \draw[line width=0.5mm, color=black] (1,0,0) -- (1,0,1);
    \draw[line width=0.5mm, color=black] (1,0,1) -- (1,1,1);
    \draw[line width=0.5mm, color=black] (0,0,1) -- (1,0,1);
    \draw[line width=0.5mm, color=black] (1,0,0) -- (1,1,0);
    \draw[line width=0.5mm, color=black] (1,1,0) -- (1,1,1);
    \draw[line width=0.5mm, color=black] (0,1,1) -- (1,1,1);
    \draw[line width=0.5mm, color=black] (0,0,1) -- (0,1,1);
    \draw[line width=0.5mm, color=black] (0.5,0,0) -- (0.5,0,1);
    \draw[line width=0.5mm, color=black] (0.5,0,1) -- (0.5,1,1);    
    \draw[line width=0.5mm, color=black] (1,0.5,0) -- (1,0.5,1);    
    \draw[line width=0.5mm, color=black] (1,0.5,1) -- (0,0.5,1);    
\end{tikzpicture}
\qquad
\begin{tikzpicture}[scale=2,tdplot_main_coords]
    \coordinate (O) at (0,0,0);
    \tdplotsetcoord{P}{10}{10}{10}
    \draw[line width=0.5mm, color=black] (O) -- (1,0,0);
    \draw[line width=0.5mm, color=black] (O) -- (0,0,1);
    \draw[line width=0.5mm, color=black] (1,0,0) -- (1,0,1);
    \draw[line width=0.5mm, color=black] (1,0,1) -- (1,1,1);
    \draw[line width=0.5mm, color=black] (0,0,1) -- (1,0,1);
    \draw[line width=0.5mm, color=black] (1,0,0) -- (1,1,0);
    \draw[line width=0.5mm, color=black] (1,1,0) -- (1,1,1);
    \draw[line width=0.5mm, color=black] (0,1,1) -- (1,1,1);
    \draw[line width=0.5mm, color=black] (0,0,1) -- (0,1,1);     
    \draw[line width=0.5mm, color=black] (0.5,0,0) -- (0.5,0,1);
    \draw[line width=0.5mm, color=black] (0.5,0,1) -- (0.5,1,1);    
    \draw[line width=0.5mm, color=black] (1,0.5,0) -- (1,0.5,1);    
    \draw[line width=0.5mm, color=black] (1,0.5,1) -- (0,0.5,1);    
    \draw[line width=0.5mm, color=black] (0,0,0.5) -- (1,0,0.5);    
    \draw[line width=0.5mm, color=black] (1,0,0.5) -- (1,1,0.5);    
\end{tikzpicture}
\caption{Example of bisection for $\dpa=3$. The initial element of level 0 is bisected in the $x$-direction ($\dirb{0}=1$) to obtain two elements of level 1. These are then bisected in the $y$-direction ($\dirb{1} = 2$) to obtain four elements of level 2, which are bisected in the $z$-direction ($\dirb{2} = 3$) to get eight elements of level 3.}
\label{fig:bisect3d}
\end{figure}

A \emph{T-mesh in the index domain}, or simply \emph{index T-mesh}, is defined as $\Tmesh := \Tmesh_N$ by successively applying bisection for $k = 0, \ldots, N-1$, in the form
\begin{equation*}
\Tmesh_{k+1} := (\Tmesh_k \setminus \{\elemi_k\}) \, \cup \, \bisect_{\dirb{\levelT{\elemi_k}}}(\elemi_k), 
\end{equation*}
with $\elemi_k \in \Tmesh_k$. The index T-mesh defines a partition of the index domain $\Omindex$ into disjoint hyperrectangles. 
Noting that bisection is applied alternating the direction of refinement, it is easy to see that any element $\elemi \in \Tmesh$ can be written as $\Pi_{j=1}^\dpa (a_j,b_j)$, with $a_j, b_j \in \II_j^{k_j(\elemi)}$ and 
\[
k_j(\elemi) := \lfloor (\levelT{\elemi} + \dpa - j)/\dpa \rfloor.
\]
Thus, we can define its \emph{parametric image} as 
\[
\indtpar{\elemi} := \Pi_{j=1}^\dpa (\knot_{j,a_j}, \knot_{j,b_j}),
\]
where $\knot_{j,a_j}, \knot_{j,b_j} \in \kv_j^{k_j(\elemi)}$. 
With this definition, from the index T-mesh $\Tmesh$ we can infer a \emph{T-mesh in the parametric domain}, or \emph{parametric T-mesh}, which is given by the parametric images with non-zero measure of the elements in the index T-mesh
\begin{equation}
\Tmeshp := \{ \elemp = \indtpar{\elemi} : \elemi \in \Tmesh \text{ and } |\elemp| \not = 0\}. \label{eq:param-Tmesh}
\end{equation}
We plot in Figure~\ref{fig:param-Tmesh} the parametric T-mesh associated to the index T-mesh of Figure~\ref{fig:Tmesh}. Notice that any element in the parametric T-mesh has a corresponding element in the index T-mesh, while the opposite is not true. Therefore, we can define for any $\elemp \in \Tmeshp$ its \emph{index preimage}
\[
\partind{\elemp} := \elemi, \, \text{ with } \elemp = \indtpar{\elemi},
\]
and we can set the \emph{level} of $\elemp$ as the level of its index preimage, $\levelT{\elemp} := \levelT{\partind{\elemp}}$. Note that it always holds that $\partind{\elemp} \subset \Omip$.

\subsubsection{T-spline blending functions}
\label{sec:T-splines defined}
After defining the T-mesh, it remains to define the T-spline blending functions. We start by defining the \emph{region of active anchors}, which is usually called the \emph{active region}, as
\[
\AR := \Pi_{j=1}^\dpa (\lceil p_j/2 \rceil +1,  n_j+p_j+1-\lceil p_j/2 \rceil),
\]
and the set of \emph{anchors}, sometimes also called \emph{nodes},
\[
\nodes_{\bf p}(\Tmesh,{\bf T}^0) := \{ \anchor \in \overline\AR : \anchor \text{ vertex of some } \elemi \in \Tmesh \}.
\]
In Figure~\ref{fig:Tmesh} the region of active anchors is given by the white and light gray elements.
\begin{remark} \label{remark:degree}
We are restricting ourselves to the case of odd degree. For even degree, the anchors are associated to elements, while for mixed degree they are associated to either vertical or horizontal edges in the two-dimensional case, and to edges (two odd, one even degree) or faces (one odd, two even degrees) in the three-dimensional case, see, e.g., \cite{bbsv14} for details.
\end{remark}

For each element $\elemi = \Pi_{i=1}^\dpa (a_i,b_i) \in \Tmesh$, we define its \emph{skeleton} in the $j$-th direction, for $j= 1, \ldots, \dpa$, as
\[
\sk_{j}(\elemi) := \Pi_{i=1}^{j-1}[a_i,b_i] \times \{a_j,b_j\} \times \Pi_{i=j+1}^{\dpa}[a_i,b_i], 
\]
and the \emph{skeleton of a T-mesh} in the $j$-th direction as 
\[
\sk_{j}(\Tmesh) := \bigcup_{\elemi \in \Tmesh} \sk_{j}(\elemi).
\]
Then, to each anchor $\anchor = (z_1, \ldots, z_\dpa) \in \nodes_{\bf p}(\Tmesh,{\bf T}^0)$ and to each direction $j \in \{1,\ldots, \dpa\}$, we associate the corresponding ordered \emph{global index vector}
\[
\begin{array}{l}
\GIV_{j}(\anchor,\Tmesh) := \{ s \in [1,n_j+p_j+1] : \\
\qquad (z_1, \ldots, z_{j-1}, s, z_{j+1}, \ldots, z_\dpa) \in \sk_j(\Tmesh) \},
\end{array}
\]
and the \emph{local index vector} $\LIV_{j}(\anchor,\Tmesh) \subset \mathbb{R}^{p_j+2}$ being the vector of $p_j+2$ consecutive indices of $\GIV_{j}(\anchor,\Tmesh)$ having $z_j$ as its middle entry, in the $((p_j+3)/2)$-th position. This is equivalent to trace a line from the anchor parallel to the $j$-th axis, and considering in each direction the first $\lfloor (p_j+2)/2 \rfloor$ intersections with the skeleton of the T-mesh as in Figure~\ref{fig:Tmesh}. Once we have defined the local index vector, we define the \emph{local knot vector} in the $j$-th direction as
\[
\kv_j(\anchor,\Tmesh) := (\knot_{j,r})_{r \in \LIV_j(\anchor,\Tmesh)},
\]
and we remark that there exists an integer $k$ such that $\kv_j(\anchor,\Tmesh)$ is a subvector of $\kv_j^k$, not necessarily with consecutive indices. 
Recalling the notation from \eqref{eq:spline-1d}, we can now define the \emph{T-spline blending function} associated to each anchor as
\begin{equation} \label{eq:T-spline}
\hat B_{\anchor,\mathbf{p}}(\bparam) := \hat B[\kv_1(\anchor,\Tmesh)](\param_1) \ldots \hat B[\kv_\dpa(\anchor,\Tmesh)](\param_\dpa).
\end{equation}
Note that, in general, the restriction of a T-spline function to an element of $\Tmeshp$ is not a polynomial. 

Following \cite[Def.~7.5]{bbsv14}, we define the \emph{B\'ezier mesh} as the collection of maximal open sets $\elemp_{B} \subset \Omegap$ such that each function $\hat B_{\anchor,\mathbf{p}}$ restricted to $\elemp_{B}$ is a polynomial of degree $\mathbf{p}$. The elements of the B\'ezier mesh are called \emph{B\'ezier elements}. We note that, in the two-dimensional case, the B\'ezier mesh can be obtained by applying suitable extensions to the elements of the parametric T-mesh, see an example in Figure~\ref{fig:Bezier-mesh}, and \cite[Sect.~5.2]{sederberg2003} as well as \cite[Sect~7.3]{bbsv14} for more details. Although a similar construction is not available in the three-dimensional case, the B\'ezier mesh can be obtained from the local knot vectors of all the blending functions or, in the case of dual-compatible T-splines (see below), by using the perturbed regions defined in \cite[Sect.~5]{morgenstern16}. Note that the B\'ezier mesh is always finer than the parametric T-mesh.

\begin{figure}[ht] 
\begin{center}
\subfigure[Parametric T-mesh]{
\includegraphics[width=0.225\textwidth,trim=3cm 8mm 3cm 8mm, clip]{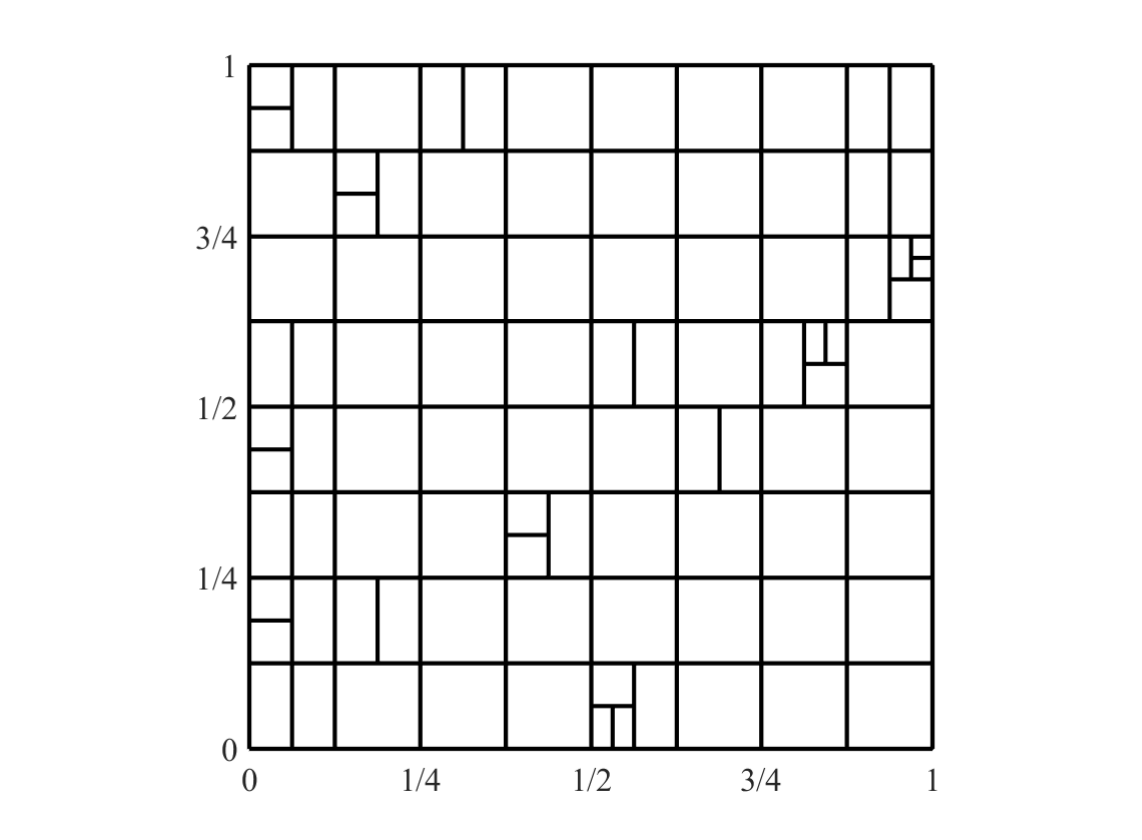} \label{fig:param-Tmesh}
}
\subfigure[B\'ezier mesh]{
\includegraphics[width=0.225\textwidth,trim=3cm 8mm 3cm 8mm, clip]{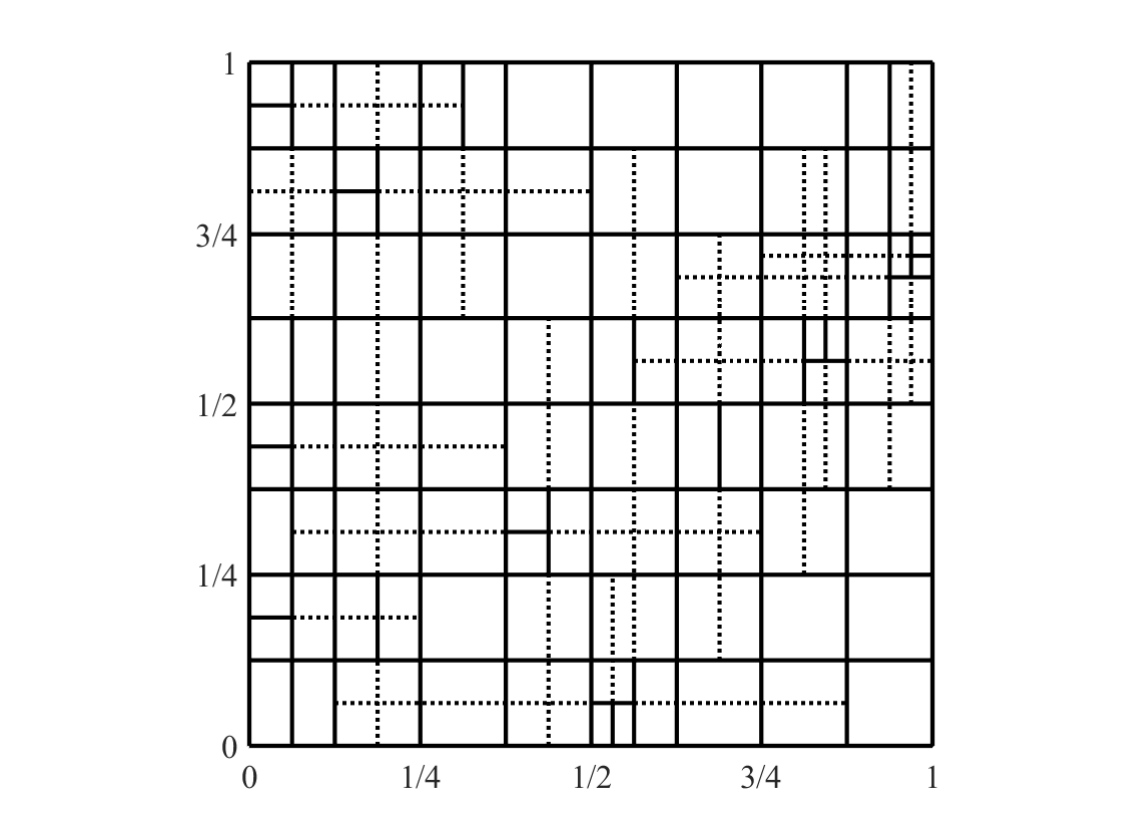} \label{fig:Bezier-mesh}
}
\end{center}
\caption{Parametric T-mesh and corresponding B\'ezier mesh, for the index T-mesh in Figure~\ref{fig:Tmesh} and degree ${\bf p} = (5,3)$.} \label{fig:T-mesh-Bezier}
\end{figure}

Finally, we define the \emph{T-spline space} as the space spanned by the T-spline blending functions, 
\[
\spT := \mathrm{span} \{ \hat B_{\anchor,\mathbf{p}} : \anchor \in \nodes_{\bf p}(\Tmesh,{\bf T}^0) \}.
\]
We notice that, in general, it is not guaranteed that the T-spline blending functions are linearly independent \cite{bcs10}. For this reason they are called \textit{blending} functions and not \textit{basis} functions. This gives the motivation to introduce dual-compatible T-splines.

\subsubsection{Dual-compatible T-splines}\label{sec:dual-compatible}
In order to obtain linearly independent T-spline blending functions, and to define a quasi-interpolant, we rely on the concept of \emph{dual-compatibility} as presented in \cite{bbsv14}. We start with the definition of overlap, see \cite[Def.~7.1]{bbsv14} and \cite[Proposition~6.1]{morgenstern16}. We remark that this definition is slightly different from the one in \cite{beirao2012,bbsv13}, which uses the local index vectors instead of the local knot vectors.

We say that two local knot vectors $\kv' = (\knot'_1, \ldots, \knot'_{p+2})$ and $\kv'' = (\knot''_1, \ldots, \knot''_{p+2})$ \emph{overlap} if they are both sub-vectors, with consecutive indices, of the same knot vector. That is, there exists a knot vector $\kv = (\knot_1, \ldots, \knot_s)$ and two integers $s', s''$ such that
\[
\knot'_i = \knot_{i+s'}, \quad 
\knot''_i = \knot_{i+s''}, \quad \text{ for all }  i = 1, \ldots, p+2.
\]
Furthermore, we say that the index T-mesh $\Tmesh$, along with the knot vectors $\kv_j^k$, is a \emph{dual-compatible T-mesh} if for every $\anchor', \anchor'' \in \nodes_{\bf p}(\Tmesh,{\bf T}^0)$ with $\anchor' \not = \anchor''$, there exists a direction $j$ such that the local knot vectors $\kv_j(\anchor',\Tmesh)$ and $\kv_j(\anchor'',\Tmesh)$ are different and overlap, cf.~\cite[Def.~7.2]{bbsv14}. We say that it is a \emph{strongly dual-compatible T-mesh} if for every $\anchor', \anchor'' \in \nodes_{\bf p}(\Tmesh,{\bf T}^0)$ with $\anchor' \not = \anchor''$, their local knot vectors overlap in $\dpa-1$ directions, cf. \cite[Def.~6.4]{morgenstern16}. For $\dpa=2$ both conditions are equivalent, while for $\dpa=3$ any strongly dual-compatible T-mesh is also dual-compatible, see the remark in \cite[Section~6]{morgenstern16}.

Dual-compatible T-meshes take their name from the fact that they allow the construction of a dual basis. Using the notation introduced in \eqref{eq:dual-basis-kv}, we define, for each anchor $\anchor \in \nodes_{\bf p}(\Tmesh,{\bf T}^0)$, the dual functional
\begin{equation} \label{eq:T-dual-basis}
\hat \lambda_{\anchor,\mathbf{p}} := 
\hat \lambda^{\rm dB}[\kv_1(\anchor,\Tmesh)] \otimes \ldots \otimes \hat \lambda^{\rm dB}[\kv_\dpa(\anchor,\Tmesh)].
\end{equation}
It can be shown that the dual functionals \eqref{eq:T-dual-basis} form a dual basis, see \cite[Propoposition~7.3]{bbsv14}. The dual basis allows to prove that dual-compatible T-splines are a partition of unity, linearly independent \cite[Proposition~7.4]{bbsv14}, and also locally linearly independent \cite[Theorem~3.2]{Li2015}.

\begin{proposition}\label{prop:Tsplines basis}
Let $\Tmesh$ be dual-compatible. 
Then, the functions $\{\hat B_{\anchor,\mathbf{p}} : \anchor \in\nodes_{\bf p}(\Tmesh,{\bf T}^0) \}$ are linearly independent, and also locally linearly independent, i.e., they are linearly independent on any open set $O \subset \Omegap$.
Moreover, if the constant functions are contained in $\spT$,  these functions form a partition of unity. 
\end{proposition}


The dual basis also allows to prove the following result, closely related to local linear independence, see \cite[Proposition~7.6]{bbsv14} and \cite[Theorem~4.1]{Li2015}.
\begin{lemma} \label{lem:lli}
Let $\Tmesh$ be a dual-compatible T-mesh. Then, for any B\'ezier element $\elemp_{B}$ there are at most $(p_1+1) \ldots (p_\dpa+1)$ basis functions that do not vanish in $\elemp_{B}$.
\end{lemma}

Moreover, from the dual basis we can define the quasi-interpolant 
\begin{equation} \label{eq:proj-Tsp}
\projT{\mathbf{p}} : L^2(\Omegap) \rightarrow \spT,  \,\hat v \mapsto \sum_{\anchor \in \nodes_{\bf p}(\Tmesh,{\bf T}^0)} \hat \lambda_{\anchor,\mathbf{p}}(\hat v) \hat B_{\anchor,\mathbf{p}}.
\end{equation}
According to \cite[Proposition~7.3]{bbsv14}, this operator is a projector.
\begin{proposition}\label{prop:projector}
Let $\Tmesh$ be a dual-compatible T-mesh. Then, the functionals \eqref{eq:T-dual-basis} form a dual basis, and the operator \eqref{eq:proj-Tsp} is a projector in the sense that
\[
\projT{\mathbf{p}}\hat v = \hat v \; \text{ for all } \hat v \in \spT.
\]
\end{proposition}

For each element $\elemp$ of the parametric T-mesh $\Tmeshp$, we define the set of anchors such that their corresponding basis functions do not vanish in $\elemp$, as
\begin{align*}
& \nodes(\elemp) := \{ \anchor \in \nodes_{\bf p}(\Tmesh,{\bf T}^0): \supp (\hat B_{\anchor,\mathbf{p}}) \cap \elemp \not = \emptyset \}.
\end{align*}
Analogously to the definition in \eqref{eq:supp_ext} for the B-spline case, we define the \emph{support extension} as the union of supports of basis functions that do not vanish on $\elemp$ 
i.e.,
\begin{align*}
& \sext{\elemp} := \bigcup_{\anchor \in \nodes(\elemp)} \supp (\hat B_{\anchor,\mathbf{p}}), 
\end{align*}
For a B\'ezier element $\elemp_B$, we define analogously $\nodes(\elemp_B)$ and $\sext{\elemp_B}$, by simply replacing $\elemp$ by $\elemp_B$ in the definitions.

Then, from the definition of the dual functionals, as an immediate corollary of Proposition~\ref{prop:projector}, the quasi-interpolant is a local projector. 
\begin{corollary}\label{corol:local-projector-Tsplines}
Let $\Tmesh$ be a dual-compatible T-mesh, and $\Tmeshp$ its associated parametric T-mesh. For any $\elemp \in \Tmeshp$, it holds that 
\[
(\projT{\mathbf{p}} \hat v)|_{\elemp} = \hat v|_{\elemp} \quad \text{if } \hat v|_{\sext{\elemp}} \in \spT|_{\sext{\elemp}}.
\]
\end{corollary}

Moreover, we have the following stability result, which is proved in \cite[Proposition~7.7] {bbsv14}.

\begin{proposition} \label{prop:tsp-proj-bezier}
Let $\Tmesh$ be a dual-compatible T-mesh, and let $\Tmeshp$ be the corresponding T-mesh in the parametric domain. Then, for all B\'ezier element $\elemp_B$ and all $\hat v \in L^2(\hat\Omega)$, we have that
\[
\| \projT{\mathbf{p}}\hat v \|_{L^2(\elemp_B)} \le C \| \hat v \|_{L^2(\sext{\elemp_B})},
\]
where $C > 0$ depends only on the dimension $\dpa$, the degrees $p_j$, and the coarsest knot vectors $\kv_j^0$.
\end{proposition}

Now, the main issue is to define a refinement strategy that delivers dual-compatible T-splines and such that the size of any element $\elemp$ and any B\'ezier element $\elemp_B$ is comparable to the size of its support extension.

\subsubsection{Refinement strategy: admissible T-meshes}
\label{sec:T refine}
We now introduce the refinement algorithm that derives from \cite{mp15,morgenstern16}. As we mentioned above, the reason to use this algorithm instead of \cite{slsh12} is that it guarantees linear complexity, and it also preserves the shape-regularity of the mesh. To proceed, we need to define some concepts related to the index T-mesh.

For any element $\elemi$ in the index T-mesh with $\check Q\subseteq\Omip$, we denote its middle point as $\mathbf{x}_{\elemi}$, and define the set of its \emph{generalized neighbors}
\begin{align*}
\neig^{\rm gen}(\elemi) &:= \{ \elemi' \in \Tmesh : \elemi' \subseteq \Omip \wedge \exists \mathbf{x}=(x_1, \ldots, x_\dpa) \in \elemi' \\
\text{ with } & |x_j - (\mathbf{x}_\elemi)_j| < (\mathbf{D}_{\mathbf{p}}(\levelT{\elemi}))_j, \; \text{ for } j=1, \ldots, \dpa\},
\end{align*}
where the vector $\mathbf{D}_{\mathbf{p}}(k)$ is defined differently for the two-dimensional case \cite[Def.~2.4]{mp15}
\[
\mathbf{D}_{\mathbf{p}}(k) := \left \{
\begin{array}{ll}
\frac{1}{2^{k/2}} (\lfloor \frac{p_1}{2} \rfloor + \frac{1}{2}, \lceil \frac{p_2}{2} \rceil + \frac{1}{2}) & k=0 {\ \rm mod \ } 2,\\
\frac{1}{2^{(k+1)/2}} (\lceil \frac{p_1}{2} \rceil + \frac{1}{2}, 2\lfloor \frac{p_2}{2} \rfloor + 1) & k=1 {\ \rm mod \ } 2,
\end{array}
\right.
\]
and for the three-dimensional case \cite[Def.~2.4]{morgenstern16}
\[
\mathbf{D}_{\mathbf{p}}(k) := \left \{
\begin{array}{ll}
\frac{1}{2^{k/3}} (p_1 + \frac{3}{2}, p_2 + \frac{3}{2}, p_3 + \frac{3}{2}) & k=0 {\ \rm mod \ } 3, \\
\frac{1}{2^{(k-1)/3}} (\frac{p_1+3/2}{2}, p_2 + \frac{3}{2}, p_3 + \frac{3}{2}) & k=1 {\ \rm mod \ } 3, \\
\frac{1}{2^{(k-2)/3}} (\frac{p_1+3/2}{2}, \frac{p_2 + 3/2}{2}, p_3 + \frac{3}{2}) & k=2 {\ \rm mod \ } 3,
\end{array}
\right.
\]
see some examples for uniform meshes in Figure~\ref{fig:tsplines_neighbors_uniform} and for a non-uniform mesh in Figure~\ref{fig:tsplines_neighbors}.
\begin{figure}
\centering
\subfigure[${\bf p} = (3,3)$, level 0]{
\includegraphics[width=0.22\textwidth,trim=3cm 1cm 2.5cm 0.5cm, clip]{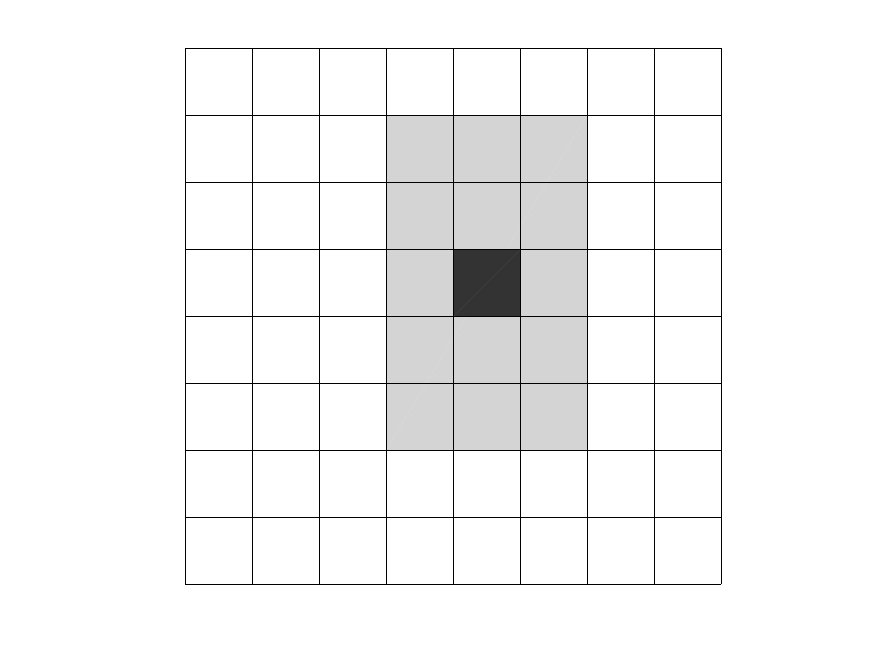}
}
\subfigure[${\bf p} = (3,3)$, level 1]{
\includegraphics[width=0.22\textwidth,trim=3cm 1cm 2.5cm 0.5cm, clip]{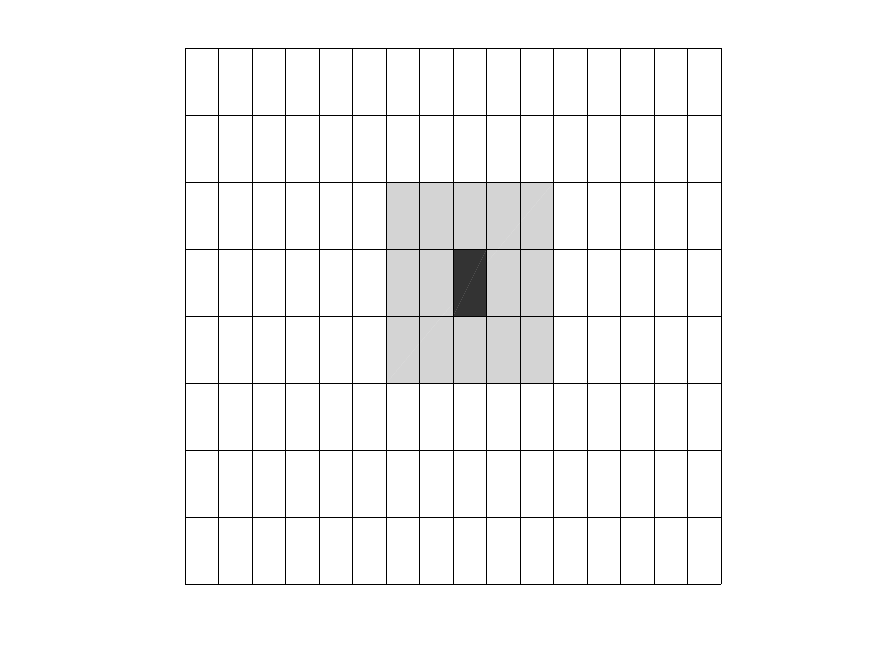}
}
\subfigure[${\bf p} = (5,5)$, level 0]{
\includegraphics[width=0.22\textwidth,trim=3cm 1cm 2.5cm 0.5cm, clip]{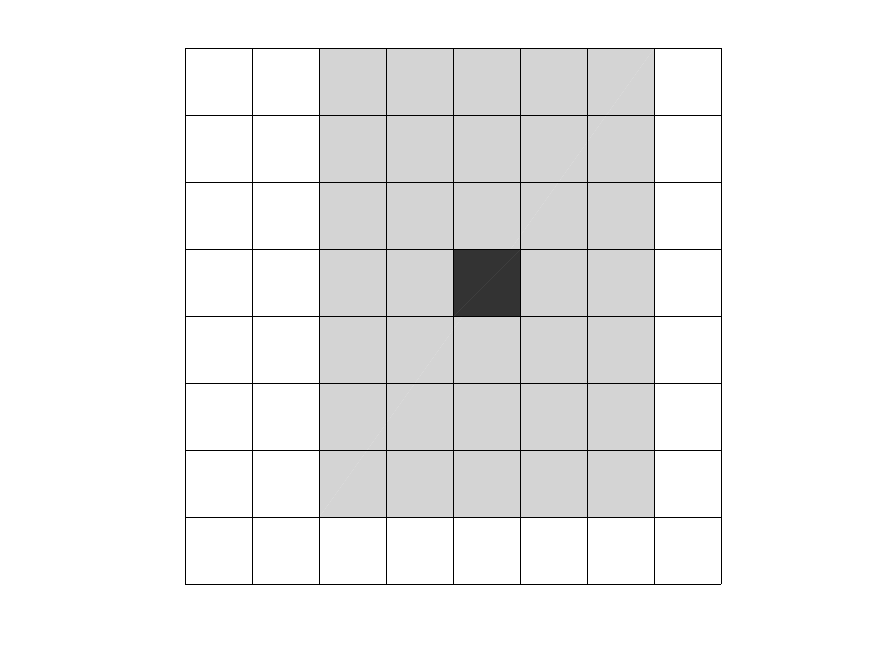}
}
\subfigure[${\bf p} = (5,5)$, level 1]{
\includegraphics[width=0.22\textwidth,trim=3cm 1cm 2.5cm 0.5cm, clip]{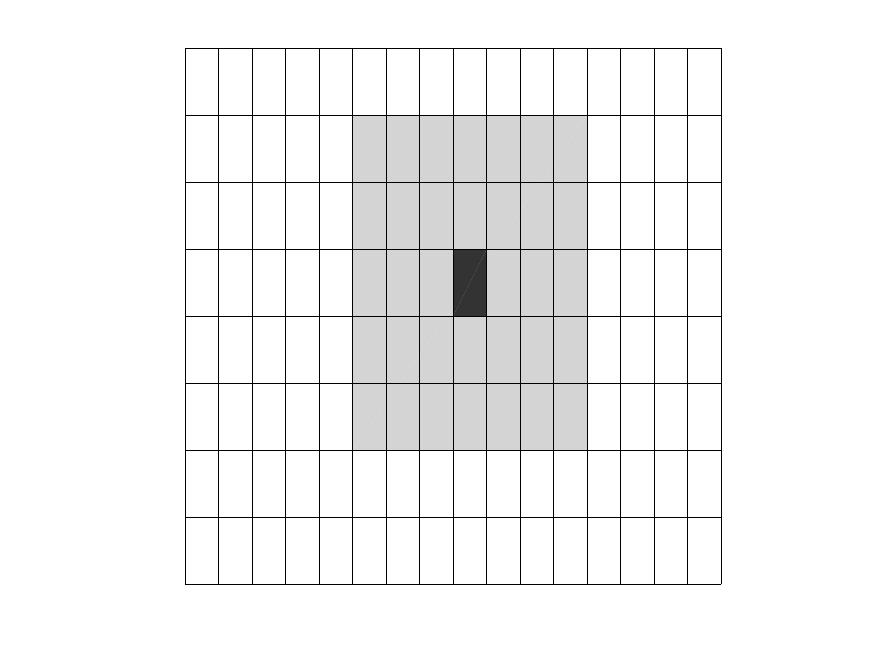}
}
\caption{Visualization of the generalized neighborhood on uniform leveled meshes, for simplicity represented in $\Omip$, and for different degrees. For the element $\elemi$ in dark gray, its generalized neighborhood $\neig^{\rm gen}(\elemi)$ is formed by all the gray elements.}
\label{fig:tsplines_neighbors_uniform}
\end{figure}
\begin{remark}
In the two-dimensional case, for a uniform even-leveled mesh, $\neig^{\rm gen}(\elemi)$ is obtained by extending $\elemi$ by $(p-1)/2$ elements to the left and right, and by $(p+1)/2$ elements above and below\footnote{This is respectively the length of edge and face extensions of T-junctions, see \cite{li2012} or \cite[Sect.~7.3]{bbsv14}.}, while for a uniform odd-leveled mesh, we have to extend by $(p+1)/2$ elements to the left and right, and by $(p-1)/2$ elements above and below, which corresponds to the gray area in Figure~\ref{fig:tsplines_neighbors_uniform}. For non-uniform meshes, $\neig^{\rm gen}(\elemi)$ is formed by elements which intersect the same area. Similar considerations apply in the three-dimensional case.
\end{remark}

We also define the set of \emph{neighbors} 
\[
\neig(\elemi) := \{ \elemi' \in \neig^{\rm gen}(\elemi) : \levelT{\elemi'} < \levelT{\elemi} \}.
\]
With a slight abuse of notation, we define the set of \emph{(generalized) neighbors} for a parametric element $\elemp \in \Tmeshp$ as
\begin{align*}
\neig^{\rm gen}(\elemp) &:= \{ \elemp' \in \Tmeshp : \partind{\elemp'} \in \neig^{\rm gen}(\partind{\elemp}) \},
\\
\neig(\elemp) &:= \{ \elemp' \in \Tmeshp : \partind{\elemp'} \in \neig(\partind{\elemp}) \}.
\end{align*}
An example of these definitions is given in Figure~\ref{fig:tsplines_neighbors}.
\begin{figure}[ht]
\centering
\includegraphics[width=0.3\textwidth,trim=2cm 1cm 2cm 0.5cm, clip]{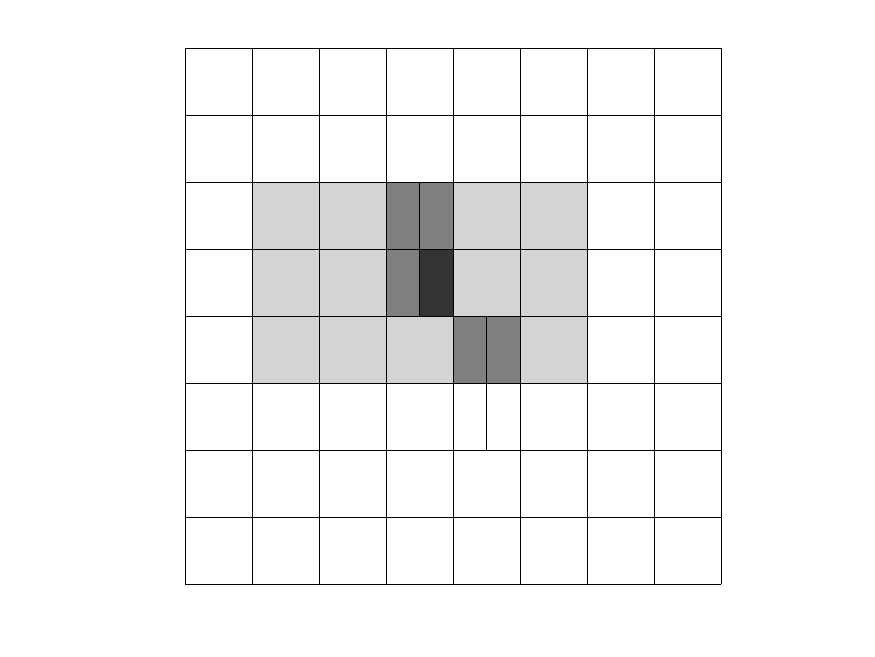}
\caption{Visualization of the generalized neighborhood for degree ${\bf p} = (5,3)$ in $\Omegap$. For the element $\elemp$ in dark gray, its generalized neighborhood $\neig^{\rm gen}(\elemp)$ is formed by all the gray elements, while the neighborhood $\neig(\elemp)$ is constituted only by the light gray elements.}
\label{fig:tsplines_neighbors}
\end{figure}

\begin{remark}\label{rem:neighbors touch}
As an immediate consequence of these definitions, and because we assume that $p_j \ge 2$ for $j = 1, \ldots, \dpa$, for any $\elemp \in \Tmeshp$ it holds that
\[
\{\elemp' \in \Tmeshp : \overline \elemp \cap \overline{\elemp'} \not = \emptyset \} \subseteq \neig^{\rm gen}(\elemp).
\]
\end{remark}

\begin{remark}
In the refinement algorithm, the neighbors will play the same role as the ${\cal H}$-neighborhood and ${\cal T}$-neighborhood of Section~\ref{sec:hierarchical refine} for (T)HB-splines.
\end{remark}

For any point $\mathbf{x} = (x_1, \ldots, x_\dpa) \in \overline{\Omindex}$ we define its projection into the index/parametric domain as $\tilde{\mathbf{x}} := (\tilde{x}_1, \ldots, \tilde{x}_\dpa)$, where $\tilde{x}_j = \min (\max (x_j, p_j+1), n_j+1)$. Then, for any element $\elemi \in \Tmesh, \elemi \subseteq \Omip$, we define its \emph{boundary prolongation} in the index T-mesh as the set of elements
\[
\bext{\elemi}{\Tmesh} := \{\elemi' \in \Tmesh : \tilde{\mathbf{x}}_{\elemi'} \in \partial \elemi \},
\]
and for any set of elements $\markedT \subseteq \Tmesh$ we will also denote 
\[
\bext{\markedT}{\Tmesh} : = \bigcup_{\elemi \in \markedT} \bext{\elemi}{\Tmesh}.
\]
Several examples of boundary prolongations are shown in Figure~\ref{fig:prolongation}.
\begin{figure}[ht] 
\begin{center}
\subfigure[Examples of boundary prolongation]{
\includegraphics[width=0.225\textwidth,trim=0mm 0mm 0mm 0mm, clip]{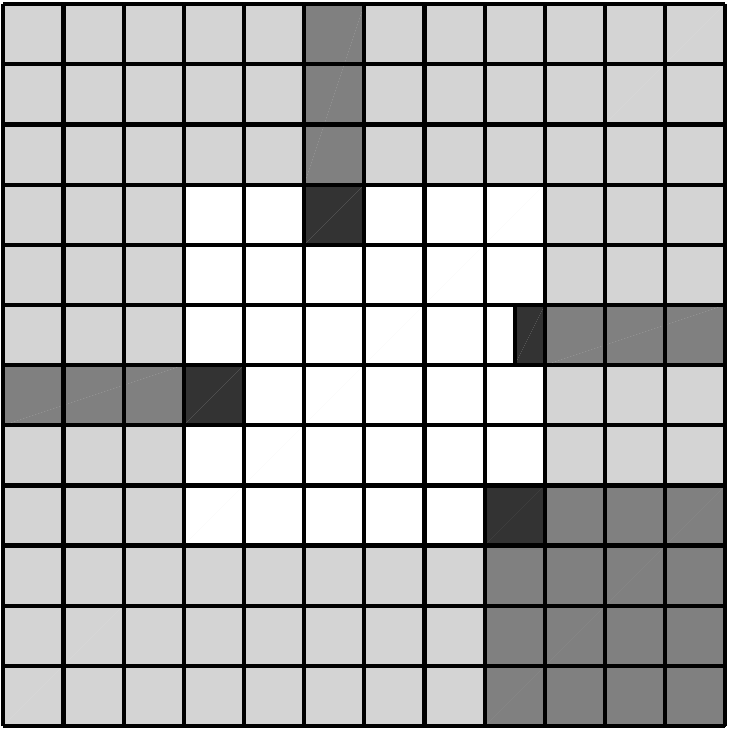} \label{fig:prolongation}
}
\subfigure[T-mesh after refinement]{
\includegraphics[width=0.225\textwidth,trim=0mm 0mm 0mm 0mm, clip]{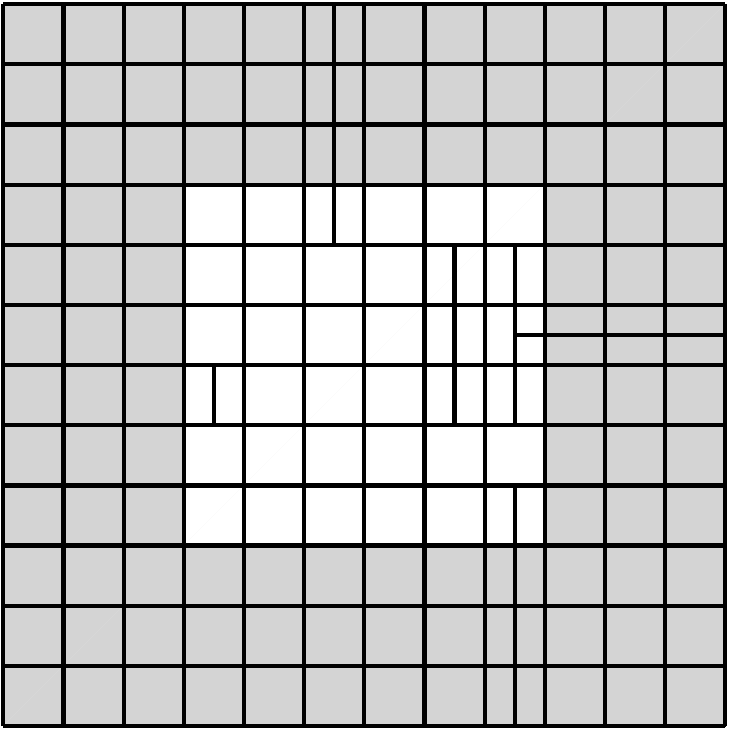} \label{fig:prolongation_ref}
}
\end{center}
\caption{The left figure shows the boundary prolongations of the dark gray elements, which are given by the gray elements. The right figure shows the result of applying Algorithm~\ref{alg:index-Tmesh}, after marking the dark gray elements on the left figure. The degree is $p_1 = p_2 = 3$. Light gray elements are outside $\Omip$.}
\end{figure}

With this notation, we are now in the position to introduce our refinement algorithm, which is based on \cite[Algorithm~2.9]{mp15} and \cite[Algorithm~2.9]{morgenstern16}, with the difference of the bisection of elements outside $\Omip$,  see Remark~\ref{rem:eq-algos} below. First, given an index T-mesh $\Tmesh$ and $\elemi \in \Tmesh$, we say that the bisection of $\elemi$ is \emph{admissible} if $\neig(\elemi) = \emptyset$, cf. \cite[Def.~2.11]{mp15} and \cite[Def.~3.1]{morgenstern16}.

Algorithm~\ref{alg:index-Tmesh} provides a refinement algorithm for index T-meshes  such that the bisections in the last step can be performed in such an order that each one is admissible, see \cite[Proposition~2.13]{mp15} and \cite[Theorem~3.3]{morgenstern16}. 
Given an index T-mesh $\Tmesh$ and a set of elements $\markedT$ to be refined with $\bigcup\markedT \subseteq \Omip$, we apply Algorithm~\ref{alg:index-Tmesh} to obtain a refined index T-mesh, which we denote by $\refineind(\Tmesh, \markedT)$. The algorithm recursively marks all the neighbors of marked elements that are contained in $\Omip$. To avoid the appearance of undesired T-junctions outside $\Omip$, the boundary prolongation of marked elements is also marked, which is equivalent to extend any T-junction from the boundary of the index/parametric domain $\Omip$ to the boundary of the index domain $\Omindex$. An example is shown in Figure~\ref{fig:prolongation_ref}. In this example, the neighbors of marked elements are marked, resulting in the refinement of other elements in the white region $\Omip$. Then, also boundary prolongations of marked elements are marked, resulting in the bisection of elements in the gray region outside $\Omip$. Some gray elements are marked but not bisected. Their level has been implicitly increased by one, and they might be bisected the next time they are marked.

\begin{algorithm}[ht]
\caption{\refineind \ (Admissible refinement for the index T-mesh)} \label{alg:index-Tmesh}
\begin{algorithmic}
\Require index T-mesh $\Tmesh$, marked elements $\markedT \subseteq \Tmesh$
with $\bigcup\markedT \subseteq \Omip$
\Repeat
\State set $\displaystyle \widecheck{\mathcal{U}} = \bigcup_{\elemi \in \markedT} \{\elemi' \in \Tmesh \setminus \markedT : \elemi' \in \neig(\elemi)\}$
\State set $\markedT = \markedT \cup \widecheck{\mathcal{U}}$
\Until {$\widecheck{\mathcal{U}} = \emptyset$}
\State set $\markedT = \markedT \cup \bext{\mathcal{\markedT}}{\Tmesh}$
\State set $\Tmesh = (\Tmesh \setminus \markedT)  \cup \left(\bigcup_{\elemi \in \markedT} \left(\bisect_{\dirb{\levelT{\elemi}}}(\elemi) \right) \right) $
\Ensure refined index T-mesh $\Tmesh$
\end{algorithmic}
\end{algorithm}


Algorithm~\ref{alg:index-Tmesh} refines the index T-mesh, but in practice the marked elements will be given in the parametric T-mesh. For this reason we need to introduce a second algorithm. Given a parametric T-mesh $\Tmeshp$, its corresponding index T-mesh $\Tmesh$, and a list of marked elements $\markedTp \subseteq \Tmeshp$, we apply Algorithm~\ref{alg:param-Tmesh} to obtain a refined parametric T-mesh, which we denote by $\refine(\Tmeshp, \markedTp)$. Note that Algorithm~\ref{alg:param-Tmesh} passes the marked elements to their index preimage, 
then it applies Algorithm~\ref{alg:index-Tmesh} to obtain the refined index T-mesh, and finally returns its parametric image.
\begin{algorithm}[ht]
\caption{\refine \ (Admissible refinement for the parametric T-mesh)}\label{alg:param-Tmesh}
\begin{algorithmic}
\Require parametric T-mesh $\Tmeshp$, the corresponding index T-mesh $\Tmesh$, and marked elements $\markedTp \subseteq \Tmeshp$
\State set $\markedT = \partind{\markedTp}$ 
\State set $\Tmesh = \refineind(\Tmesh, \markedT)$ \Comment{Algorithm~\ref{alg:index-Tmesh}}
\State set $\Tmeshp = \indtpar{\Tmesh}$
\Ensure refined parametric T-mesh $\Tmeshp$ (and refined index T-mesh $\Tmesh$)
\end{algorithmic}
\end{algorithm}

An example of the application of the refinement algorithm is shown in Figure~\ref{fig:algorithm-Tsplines}, starting from a uniform parametric T-mesh of $4\times 4$ elements, and marking always the element in the bottom left corner.
\begin{figure}
\centering
\includegraphics[width=0.23\textwidth]{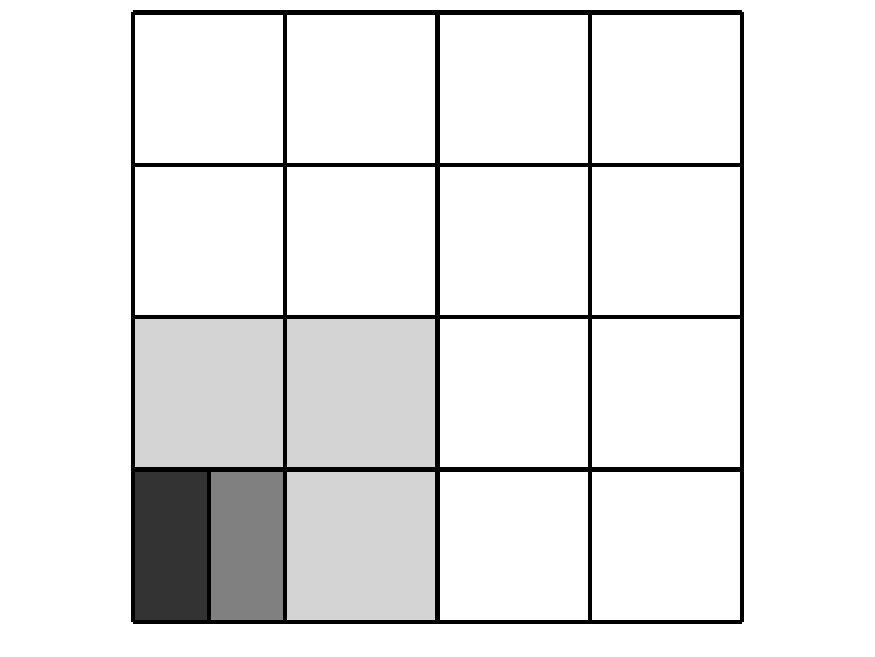}
\includegraphics[width=0.23\textwidth]{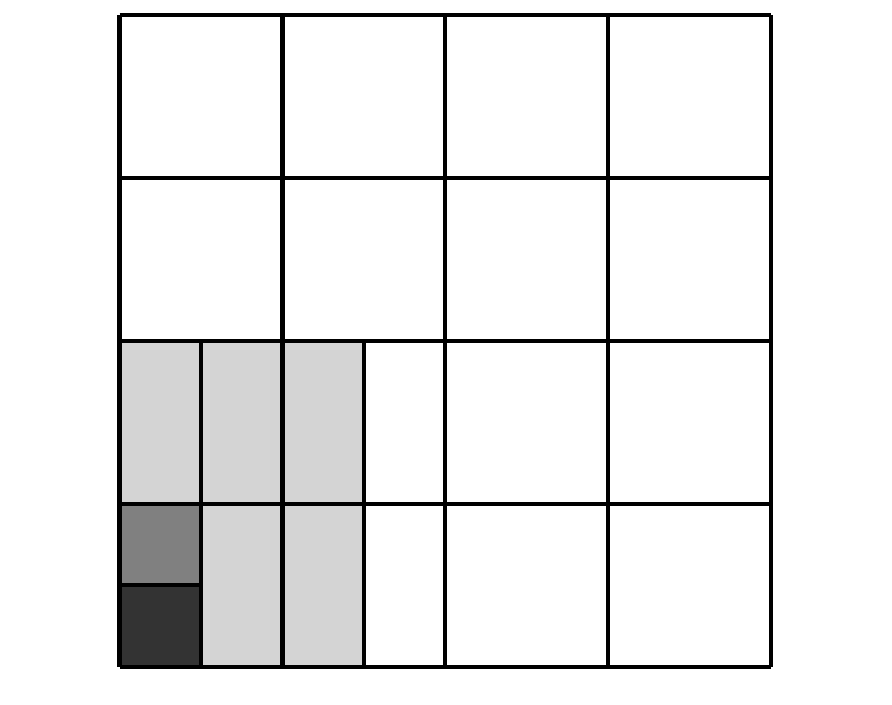}
\includegraphics[width=0.23\textwidth]{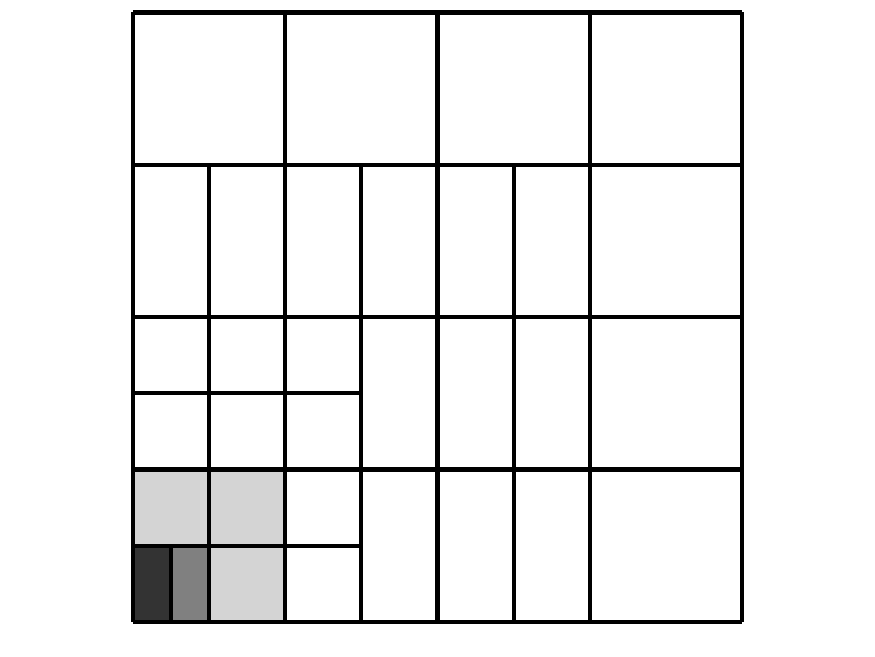}
\includegraphics[width=0.23\textwidth]{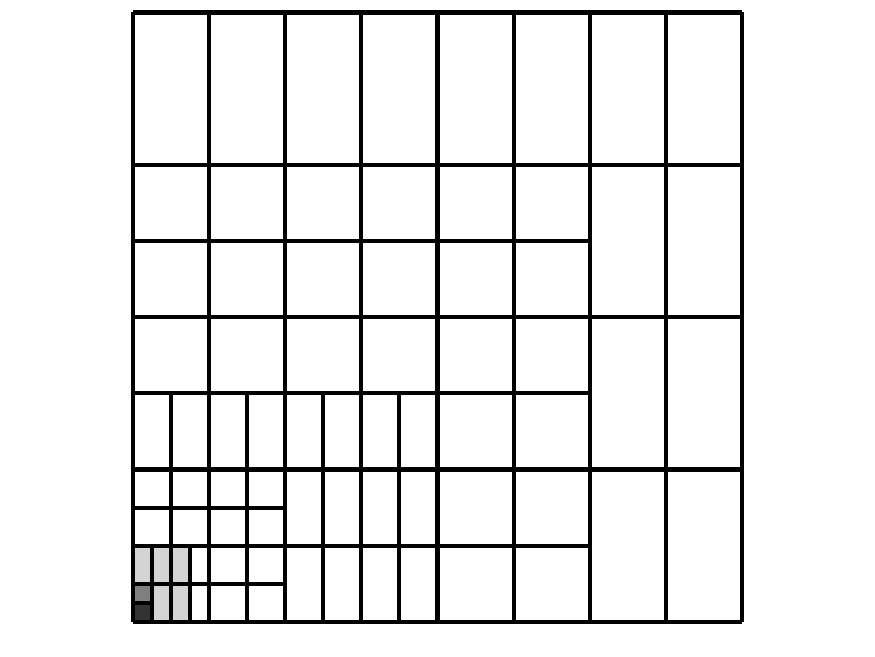}
\caption{Application of Algorithm~\ref{alg:param-Tmesh} starting from a $4 \times 4$ parametric T-mesh, with degree ${\bf p} = (5,3)$, and marking always the element in the bottom left corner. The plot shows the refined parametric T-meshes after 1, 2, 3, and 6 refinement steps. The marked element $\elemp$ is highlighted in dark gray, while all the elements in gray belong to its generalized neighborhood $\neig^{\rm gen}(\elemp)$, and the elements in light gray belong to its neighborhood $\neig(\elemp)$, and therefore are marked by the refinement algorithm. Note that also the neighbors of these elements, which we do not highlight, are marked for refinement by the algorithm.}
\label{fig:algorithm-Tsplines}
\end{figure}


We define $\refine(\Tmeshp)$ as the set of all meshes that can be obtained via iterative application of $\refine$ to $\Tmeshp$. Moreover, denoting by $\Tmeshp_0$ the parametric image of $\Tmesh_0$, which is obtained as in \eqref{eq:param-Tmesh}, we define the set of \emph{admissible parametric T-meshes}
\[
\admp := \refine(\Tmeshp_0).
\]

\begin{remark}\label{rem:eq-algos}
Unfortunately, the admissible refinement in \cite{mp15,morgenstern16} does not take care of repeated knots that appear due to open knot vectors. To our knowledge, two different remedies have been proposed: in \cite{gp18}, the refinement is performed directly on the parametric domain, and the (index) T-mesh is then extended taking into account the repetitions due to the open knot vector; in \cite{cv18}, refinement is performed on the (index) T-mesh, but the algorithm does not bisect intervals with zero length. We have followed the same approach as in \cite{gp18} with a notation similar to \cite{cv18}, because we believe this notation may be useful for future research on adaptivity with smoothness control, which requires repeated internal knots.  Our refinement algorithm provides exactly the same output as \cite[Algorithm~2.1]{gp18}. 
In fact, Algorithm~\ref{alg:param-Tmesh} generates the same \emph{parametric} T-meshes as \cite[Algorithm~2.9]{mp15} for $\dpa=2$ and \cite[Algorithm~2.9]{morgenstern16} for $\dpa=3$. 
\end{remark}

\begin{remark}\label{rem:paramtric to index}
The combination of Algorithm~\ref{alg:index-Tmesh} and \ref{alg:param-Tmesh} guarantees that to each parametric T-mesh corresponds a unique index T-mesh, which is the same as in \cite{gp18}. It is worth noting that this is not true in general, and if we do not apply the refinement algorithms above, the same parametric T-mesh could be generated by two different index T-meshes due to the bisection of elements outside $\Omip$. Note that a change in the index T-meshes implies a change in the basis functions, and consequently in the discrete space.
\end{remark}

In the following, we present the most important theoretical results that derive from the refinement algorithm and that were mainly proved in \cite{mp15,morgenstern16,morgenstern17}. 
The first result states dual-compatibility of admissible meshes.
It follows from \cite[Theorem~3.6]{mp15} for $\dpa=2$ and from \cite[Theorem 6.6]{morgenstern16} for $\dpa=3$,  
see also Remark~\ref{rem:eq-algos}.
\begin{proposition}
\label{lemma:dc}
Let $\Tmeshp \in \admp$, and let $\Tmesh$ be its corresponding admissible index T-mesh. Then, $\Tmesh$ is strongly dual-compatible, and thus it is dual-compatible.
\end{proposition}

The next result states nestedness of the spaces obtained by the refinement algorithm, which is highly non-trivial and not necessarily satisfied by general T-splines.
It is proved in \cite[Corollary~5.8]{mp15} for $\dpa=2$ and in \cite{morgenstern17} for $\dpa=3$. 
Note that ${\hat{\mathbb{S}}^{\rm T}_{\mathbf{p}}(\Tmesh_0,\mathbf{\kv}^0)}$ coincides with the usual spline space ${\hat{\mathbb{S}}_{\mathbf{p}}(\mathbf{\kv}^0)}$ and the next result in conjunction with Proposition~\ref{prop:Tsplines basis} thus implies that the T-spline blending functions form a partition of unity.

\begin{proposition} \label{prop:T-nestedness}
Let $\Tmeshp \in \admp$ and $\Tmeshp_\fine \in \refine(\Tmeshp)$, and let $\Tmesh$ and $\Tmesh_\fine$ be their associated index T-meshes. Then, 
\[
\spT \subseteq \spTargs{\Tmesh_\fine}{\boldsymbol{\kv}^0}.
\]
\end{proposition}


The  next proposition provides local quasi-uniformity of admissible meshes. 
Making use of the equivalence of the algorithms mentioned in Remark~\ref{rem:eq-algos}, the assertion follows from \cite[Lemma~2.14]{mp15} for $\dpa=2$ and from \cite[Lemma~3.5]{morgenstern16} for $\dpa=3$, where the same result was proved for the index T-mesh.
\begin{proposition}
\label{prop:lqiT}
Let $\Tmeshp \in \admp$. For any $\elemp \in \Tmeshp$, it holds that
\[
|\levelT{\elemp} - \levelT{\elemp'}| \le 1 \quad \text{for all } \elemp' \in \neig^{\rm gen}(\elemp).
\]
\end{proposition}

The following result is new. 
It relates B\'ezier elements to the elements of the considered admissible T-mesh. The proof is rather technical and is thus postponed to Section~\ref{sec:proof-2beziers}.
Without providing an explicit bound, the fact that the number of B\'ezier elements on an element is uniformly bounded and that the B\'ezier elements are of comparable size also follows easily from \cite[Lemma~2.5]{gp18}.
\begin{lemma} \label{lemma:2beziers}
Let $\Tmeshp \in \admp$ and $\elemp \in \Tmeshp$. Then, $\elemp$ consists either of one B\'ezier element equal to $\elemp$, or two B\'ezier elements of measure $|\elemp|/2$.
\end{lemma}

As an immediate consequence of the previous lemma, we obtain a result analogous to Proposition~\ref{prop:tsp-proj-bezier} for elements in the parametric T-mesh.
\begin{proposition} \label{prop:tsp-proj}
Let $\Tmeshp \in \admp$ and $\elemp \in \Tmeshp$. Then, for all $\hat v \in L^2(\hat\Omega)$, we have that
\[
\| \projT{\mathbf{p}}\hat v \|_{L^2(\elemp)} \le C \| \hat v \|_{L^2(\sext{\elemp})},
\]
where $C > 0$ depends only on the dimension $\dpa$, the degrees $p_j$, and the coarsest knot vectors $\kv_j^0$.
\end{proposition}

The next proposition bounds the number as well as the support of T-spline basis functions that live on a given element.
The first assertion is an immediate consequence of Lemma~\ref{lem:lli} and Lemma~\ref{lemma:2beziers}. 
The second assertion is already proved in \cite[Lemma~2.5]{gp18}. A similar result is also given in \cite[Proposition~4.9]{cv18} for $\dpa=2$. 

\begin{proposition}\label{prop:tsp-sext}
Let $\Tmeshp \in \admp$ and $\Tmesh$ its corresponding index T-mesh. For any $\elemp \in \Tmeshp$, there exist at most $2(p_1+1) \ldots (p_\dpa+1)$ anchors $\anchor \in \nodes_{\bf p}(\Tmesh,{\bf T}^0)$ such that $\elemp \cap \supp (\hat B_{\anchor,\mathbf{p}}) \ne \emptyset$. Moreover, there exists $q \in \mathbb{N}$ depending only on the dimension $\dpa$ and the degrees $p_j$ such that there exists $\hat{\mathcal{S}}\subseteq\hat\QQ$ with $\sext{\elemp}\subseteq \bigcup \set{\overline{\hat Q'}}{\hat Q'\in\hat{\mathcal{S}}}$, $\#\hat{\mathcal{S}}\le q$, and $\bigcup \set{\overline{\hat Q'}}{\hat Q'\in\hat{\mathcal{S}}}$ is connected.  
\end{proposition}

We conclude this section with a 
proposition from \cite[Section~6]{mp15}  for $\dph=2$ and from \cite[Section~7]{morgenstern16} for $\dph=3$, respectively, which states that the possible overrefinement of Algorithm~\ref{alg:param-Tmesh} to preserve admissibility is bounded up to some uniform constant by the number of marked elements. 
Indeed, these references even provide explicit upper bounds for the constant along with numerical experiments on the quality of these bounds. 

\begin{proposition}\label{prop:T-spline lincomp}
There exists a uniform constant $C>0$ such that for arbitrary sequences $(\QQ_\k)_{\k\in\N_0}$ in $\Q$ with  $\QQ_{\k+1}=\refine(\QQ_\k,\MM_\k)$ for some $\MM_\k\subseteq \QQ_\k$ and all $\k\in\N_0$, it holds that
\begin{align*}
\# \QQ_\k-\#\QQ_0\le C \sum_{j=0}^{\k-1}\#\MM_j \quad \text{for all }k\in\N_0.
\end{align*}
The constant $C$ depends only on the dimension $\dpa$, the degrees $p_j$, and the coarsest knot vectors $\kv_j^0$.
\end{proposition}

\subsubsection{The role of the B\'ezier mesh}\label{subsec:bezier}
The results of the previous section were presented considering the elements of the parametric T-mesh. However, the implementation of isogeometric methods with T-splines is usually based on the B\'ezier mesh. Indeed, numerical integration is usually performed on B\'ezier elements, since they are the maximal sets where the restriction of the T-spline functions are polynomials. Moreover, the evaluation of T-spline functions on the local B\'ezier element can be made through B\'ezier extraction \cite{SBVSH11}, a local change of basis to represent T-splines as linear combinations of Bernstein polynomials. The B\'ezier mesh and the matrix of the B\'ezier extraction operators have been also used in \cite{mvdB15} to analyze the linear dependence of T-splines and also to develop refinement algorithms for T-splines \cite{cdB2018b}.

We remark that, thanks to Lemma~\ref{lemma:2beziers}, for admissible meshes it is easy to pass from the B\'ezier mesh to the parametric T-mesh and vice versa. Therefore, the refinement algorithm could be easily adapted to take as input marked elements on the B\'ezier mesh. In practice, numerical quadrature must be computed on B\'ezier elements, so it may be natural to compute the estimator directly on B\'ezier elements.

\subsubsection{Recent developments on T-splines}\label{subsec:extend tsplines}
Here we give some details about other sets of T-meshes that have appeared in recent years, which relax the constraints of dual-compatible T-meshes. We stress that all these works are restricted to the two-dimensional case.

Bracco and Cho introduced in \cite{bracco2014} a generalization of dual-compatible T-meshes. They replace the concept of overlap by a certain \emph{shifting} of the anchors, which is then used to introduce the class of \emph{weakly dual-compatible} T-meshes. They prove that any dual-compatible T-mesh in the sense of \cite{beirao2012,bbsv13} is also weakly dual-compatible. However, it is important to remark that this does not hold true with the definition of dual-compatibility in \cite{bbsv14}, which we are considering, and there are examples of weakly dual-compatible T-meshes that are not dual-compatible and vice versa.

Wei et al. introduced in \cite{wzlh17} a refinement strategy with similar ideas to the one in \cite{mp15}, limited to bicubic degree. Marked elements are split into four subelements, with their level increased by one, and to obtain linear independence the refinement is propagated to other elements in just one direction, with their level increased by one half. Making use of the concept of truncated T-splines (which resembles the one for THB-splines), they prove that for T-meshes constructed with their refinement strategy, linear independence holds if the face extensions of the T-mesh, which determine the B\'ezier mesh, do not intersect.

Different and more involved constraints are introduced by Li and Zhang in \cite{lz18} to define AS++ T-splines, for which it is possible to prove linear independence and to construct a dual basis. A refinement algorithm for the set of AS++ T-splines is presented in \cite{zl18}. Although the presentation is limited to bicubic T-splines, the authors plan to generalize their approach to arbitrary degree.

It is important to note that, although the aforementioned works introduce interesting refinement algorithms for T-splines, none of them presents a rigorous analysis of the algorithm's complexity as in 
Proposition~\ref{prop:T-spline lincomp}, 
which is necessary to develop the mathematical theory of adaptivity. For this reason, we have decided to focus on the dual-compatible T-splines studied in \cite{mp15,morgenstern16}.






\subsubsection{Relation between an admissible T-mesh and its B\'ezier mesh}
\label{sec:proof-2beziers}
We now give the detailed proof of Lemma~\ref{lemma:2beziers}, which states that any element of an admissible T-mesh contains at most two B\'ezier elements of equal size.
\begin{proof}
The proof is rather technical, although the main idea is not complex. Let $\elemi = \partind{\elemp}$. Suppose by contradiction that $\elemp$ contains more than two B\'ezier elements.
These must appear after the bisection of an element $\elemi'$ which is finer, in terms of the level, than $\elemi$. On the one hand, since the bisection of $\elemi'$ affects the B\'ezier elements in $\elemp$, it must be sufficiently close to $\elemi$. On the other hand, since the mesh is admissible, $\elemi$ cannot be in the neighborhood of $\elemi'$. That is, the two elements must be at the same time sufficiently close and far from each other, and we arrive at a contradiction. Let us now begin with the technical part.

Let $\Tmesh$ be the associated index T-mesh, let $\elemi = \partind{\elemp} = \Pi_{j=1}^\dpa(a_j,b_j)$, and let us suppose that $\elemp$ contains more than two B\'ezier elements to arrive at a contradiction. Since we refine by bisection and alternate the refinement directions, it is clear that the bisection of elements of the same level of $\elemp$ will not split it in more than two B\'ezier elements. Therefore, there must exist an element $\elemi' = \Pi_{j=1}^\dpa(a'_j,b'_j)$, with $k = \levelT{\elemi'} > \levelT{\elemi}$, that has been bisected in direction $s = \dirb{k}$, such that $\bisect(\elemi') \subset \Tmesh$. This element is translated with respect to $\elemi$ only in a direction different from $s$, in the sense that
\[
a_j \le a_j'< b_j' \le b_j, \, \text{ for all } j \not = \tilde s, \text{ for one } \tilde s \not = s,
\]
see a two-dimensional example in Figure~\ref{fig:twobeziers}.
\begin{figure}[t]
\centering
\includegraphics[width=0.3\textwidth,trim=3cm 1cm 2cm 0.5cm, clip]{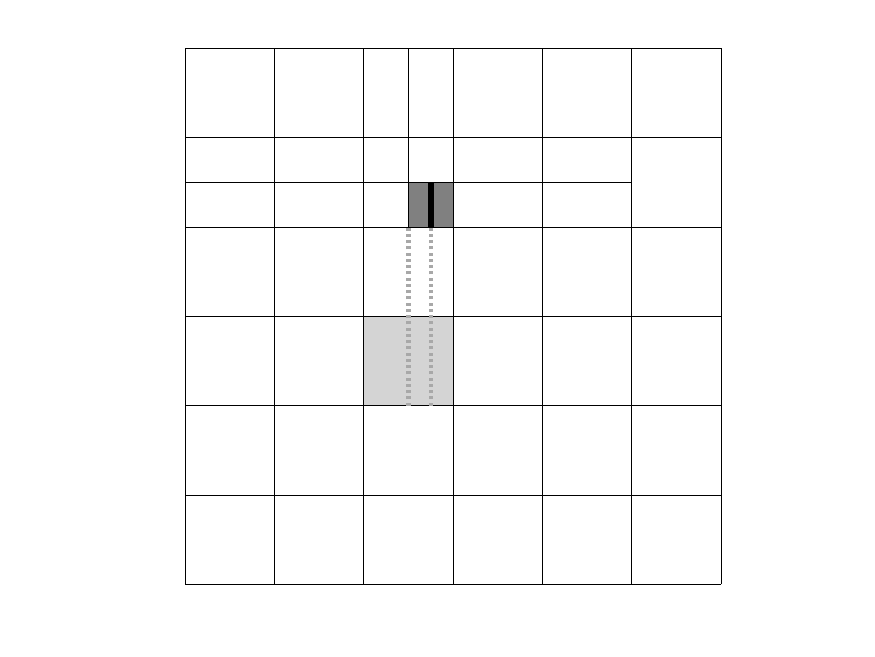}
\caption{For degree ${\bf p} = (3,3)$, the element $\elemi$ (light gray) is bisected in more than two B\'ezier elements after the bisection of $\elemi'$ (dark gray). The element $\elemi'$ is bisected by the thick black line in direction $s=1$, and it is translated with respect to $\elemi$ in direction $\tilde s = 2$.}
\label{fig:twobeziers}
\end{figure}

Moreover, there exists an anchor such that its local index vector depends on the bisection of $\elemi'$, and the support of the associated function intersects $\elemp$. Putting it rigorously, there exists $\anchor \in \nodes(\Tmesh)$ with $\supp (\hat B_{\anchor,\mathbf{p}}) \cap \elemp \not = \emptyset$, such that $\frac{a'_s+b'_s}{2} \in \LIV_{s}(\anchor,\Tmesh)$ and $z_j \in (a_j',b_j')$ for all $j \not = s$. 

Without loss of generality, we can assume that $\elemi' \subset \Omip$. Moreover, since $\Tmesh$ is admissible we can also assume that the bisection of $\elemi'$ is admissible, i.e., there exists an admissible T-mesh $\Tmesh_\coarse$ such that 
\[
\Tmesh = \refineind(\Tmesh_\coarse, \elemi'),
\]
and only $\elemi'$, and eventually $\bext{\elemi'}{\Tmesh_\coarse}$, have been bisected\footnote{If this was not the case, we would consider $\Tmesh_\coarse$ as the mesh obtained after refining all the elements in the neighborhood of $\elemi'$, but not $\elemi'$ itself.}.
%
Denoting by $\neig^{\rm gen}_\coarse(\elemi')$ the generalized neighborhood in $\Tmesh_\coarse$, since the bisection was admissible and $\levelT{\elemi'} > \levelT{\elemi}$, we know from Proposition~\ref{prop:lqiT} that $\elemi \not \in \neig^{\rm gen}_\coarse(\elemi')$, otherwise it would also be a neighbor. As a consequence, $\elemi'$ and $\elemi$ must be far from each other, and in particular
\begin{equation} \label{eq:not-neig}
|x_{\tilde s} - ({\bf x}_{\elemi'})_{\tilde s}| > ({\bf D_p}(k))_{\tilde s} \, \text{ for any } {\bf x} \in \elemi.
\end{equation}
Taking into account the relation between $k$ and $s = \dirb{k}$, and that $\tilde s \not = s$, a careful (and tedious) computation gives that
\[
({\bf D_p}(k))_{s} = 
\left\{
\begin{array}{ll}
\left(\frac{1}{2}\right)^{(k-s+1)/2} (p_s/2) & \text{ for } \dpa=2, \\
\left(\frac{1}{2}\right)^{(k-s+1)/3} (p_s + 3/2) & \text{ for } \dpa=3.
\end{array}
\right.
\]

\[
({\bf D_p}(k))_{\tilde s} = 
\left\{
\begin{array}{ll}
\left(\frac{1}{2}\right)^{(k-\tilde s+2)/2} (p_{\tilde s}/2 + 1) & \text{ for } \dpa=2, \\
\left(\frac{1}{2}\right)^{\lfloor (k-\tilde s+3)/3 \rfloor} (p_{\tilde s} + 3/2) & \text{ for } \dpa=3,
\end{array}
\right.
\]

It is readily seen that the length of an element of level $k$ in the $j$-th direction is exactly $(1/2)^{\lfloor (k-j+\dpa)/\dpa \rfloor}$. 
Note that $\LIV_{s}(\anchor,\Tmesh)$ consists of $p_s + 2$ indices. 
Since we consider odd degrees, $z_s$ is in the middle position, which means that we only need to check $(p_s+3)/2$ indices.
Using the relation between $k$ and $s = \dirb{k}$, the value of $({\bf D_p}(k))_{s}$, the length of the elements of level $k$ in direction $s$, and the number of indices, a careful check shows that $\anchor \in \partial \elemi''$ for some element $\elemi'' \in \neig^{\rm gen}_\coarse(\elemi')$.


As the bisection was admissible,  Proposition~\ref{prop:lqiT} shows that all elements in $\neig^{\rm gen}_\coarse(\elemi')$ have level at least equal to $k$. A similar check in the $\tilde s$ direction and the fact that $\anchor$ is in one generalized neighbor show that $\LIV_{\tilde s}(\anchor,\Tmesh)$ is built using only indices corresponding to elements in $\neig^{\rm gen}_\coarse(\elemi')$. As a consequence, the fact that $\supp (\hat B_{\anchor,\mathbf{p}}) \cap \elemp \not = \emptyset$ is in contradiction with \eqref{eq:not-neig}. 

Finally, the result of the measure of the B\'ezier elements holds because the elements are refined by bisection. \qed

%
%
%
%
%
%
\end{proof}

\subsection{Other splines for adaptive methods}\label{subsec:others}

Our focus in this section has been on (T)HB-splines and analysis-suitable (or dual-compatible) T-splines because the mathematical theory of adaptive isogeometric methods based on these functions is the most advanced one. However, there are other kinds of spline spaces with local refinement capabilities which are successfully used in IGA, especially in the engineering literature, but for which the mathematical theory, especially the convergence theory, has not been studied yet.
For completeness, we mention here the most popular ones and address the reader to the cited references for the details.

Probably, the most popular alternative is the one given by locally refined-splines, or \emph{LR-splines} \cite{dlp13,bressan13}. They are similar to T-splines, but instead of being defined from the T-mesh, they are directly defined from the B\'ezier mesh by associating a certain continuity to each edge (or face in 3D). They have been used for IGA for the first time in \cite{jkd14}, and after that they have appeared in several papers, see for instance \cite{kumar2015,kumar2017}. A refinement algorithm that alternates the refinement direction, similar to the one we detailed for T-splines, has been introduced in \cite{BrJu15}. 
Recently, \cite{patrizi2020} introduced another refinement algorithm which preserves (local) linear independence of LR-splines. 
A comparison of LR-splines and THB-splines can be found in \cite{jrk15}.


A simpler construction is the one of polynomial splines over hierarchical T-meshes (\emph{PHT-splines}) \cite{DeChFe06}. The starting point is a T-mesh, as for LR-splines, but in this case the continuity is the same for all edges. Assuming that the continuity is lower than one half of the degree, it is possible to determine the dimension of the space of piecewise polynomials over the T-mesh \cite{DeChFe06}, and to construct a basis for computations \cite{DeChLi08,LiDeCh10}. PHT-splines have been used in isogeometric methods for the first time in \cite{WaXuDeCh11,NgNgBoRa11}. The main drawback of the reduced smoothness of PHT-splines is that it increases the number of degrees of freedom, while its advantage is that basis functions are more localized, and their implementation and analysis is more similar to standard FEM. We refer to the survey \cite{LiChKaDe16} for more details about PHT-splines, including a complete list of references.

Another interesting approach for refinement is given by \emph{hierarchical T-splines} \cite{evans2015,cdB2018a}, where the initial mesh is defined by analysis-suitable T-meshes, and the refinement is done by applying the algorithm of hierarchical splines, replacing in the construction of $\hat{\cal H}^{N-1}$ in Section~\ref{subsec:hb} the B-splines of each level by analysis-suitable T-splines of different levels. The main difficulty of this approach is that, to build the T-splines of different levels, the initial T-mesh must be refined globally, but maintaining T-junctions away from each other in such a way that the T-mesh of the next level remains analysis-suitable, see \cite{evans2015} for details.

Finally, we remark that one of the drawbacks of (T)HB-splines compared to T-splines or LR-splines is that it is not allowed to perform anisotropic refinement, since the refinement direction at each level is determined by the (global) refinement between levels\footnote{This is also true for the T-splines refinement in \cite{morgenstern17} presented in Section~\ref{sec:T refine}}. This constraint is alleviated in the construction of \emph{patchwork B-splines} in \cite{EnJu17}, which combines different refinement directions for different regions of the domain, and even different degrees and smoothness.

\newpage


\section{Adaptivity: abstract framework}
\label{sec:abstract}


In this section, we consider an abstract adaptive algorithm of the form
\begin{align} \label{eq:box-algorithm}
 \boxed{\texttt{solve}}
 \longrightarrow
 \boxed{\texttt{estimate}}
 \longrightarrow
 \boxed{\texttt{mark}}
\longrightarrow
 \boxed{\texttt{refine}}
\end{align}
See Algorithm~\ref{alg:abstract algorithm} below for the formal statement.
First, in Section~\ref{sec:axioms}, we give general properties, the so-called \emph{axioms of adaptivity} from \cite{cfpp14}, that guarantee  convergence of the involved error estimator at optimal algebraic rate.
In Section~\ref{sec:afem} and \ref{sec:abem}, we consider Algorithm~\ref{alg:abstract algorithm} in the frame of FEM and BEM, respectively. 
These sections provide more concrete properties for the meshes, the refinement, and the ansatz spaces which ensure the axioms of adaptivity and thus guarantee optimal convergence. 
In Section~\ref{sec:adaptive igafem} below, we will show that adaptive IGAFEM and IGABEM fit into the framework of  Section~\ref{sec:afem} and \ref{sec:abem}, respectively.

\subsection{Axioms of adaptivity}
\label{sec:axioms}

We provide a set of sufficient properties for the error estimator as well as for the mesh refinement so that  Algorithm~\ref{alg:abstract algorithm} below  guarantees convergence of the estimator at optimal algebraic rate.
These properties are known as \textit{axioms of adaptivity} and have been introduced in \cite{cfpp14}.
In one way or another, the axioms arose  over the years in various works throughout the literature.
In \cite[Section~3.2]{cfpp14}, a historical overview on  their development can be found.
We especially highlight the milestones on rate-optimality \cite{bdd04,stevenson07,ckns08,cn12,ffp14}.

This section is essentially a summary of the results from \cite{cfpp14}. 
As in \cite{cfpp14}, we mainly focus  on the error estimator.
This is motivated by the fact that the adaptive algorithm has no other information than the error estimator to steer the mesh refinement.
However, at least for FEM, we will show that the corresponding error estimator is equivalent to the so-called \textit{total error} (which is the sum of  error plus data oscillations).

\subsubsection{Admissible meshes}
\label{sec:meshes_axioms}

Let $\Q$ be a set of finite sets $\QQ\in\Q$, which we refer to as \textit{admissible meshes}. 
Concretely, we will later consider quadrilateral  meshes of some Lipschitz-domain $\Omega$ or its boundary $\Gamma:=\partial\Omega$, where admissibility will describe a certain grading property, see also Section~\ref{sec:hierarchical refine} and Section~\ref{sec:T refine} for details.
Let $\refine(\cdot,\cdot)$ be a fixed refinement strategy such that, for $\QQ_\coarse\in\Q$ and marked $\MM_\coarse\subseteq\QQ_\coarse$, it holds that $\QQ_{\fine}=\refine(\QQ_\coarse,\MM_\coarse)\in\Q$ with $\MM_\coarse\subseteq\QQ_\coarse \setminus\QQ_{\fine}$, i.e., at least the marked elements $\MM_\coarse$ are refined, and $\refine(\QQ_\coarse,\emptyset)=\QQ_\coarse$.
Note that in practice, one cannot expect that only the marked elements are refined. Indeed, to preserve admissibility of our considered quadrilateral meshes, additional elements have to be refined.
For arbitrary $\QQ_\coarse,\QQ_{\fine}\in\Q$, we write $\QQ_{\fine}\in\refine(\QQ_\coarse)$, if $\QQ_{\fine}$ is obtained by iterative application of $\refine$ and we note that $\QQ_\coarse\in\refine(\QQ_\coarse)$.
Moreover, we assume that each admissible mesh $\QQ\in\Q$ can be reached via refinement starting from a fixed initial mesh $\QQ_0\in\Q$, i.e.,    $\refine(\QQ_0)=\Q$.
We suppose that 
$\#\QQ_\coarse<\# \QQ_\fine$ for all  $\QQ_\coarse\in\Q$ and all $\QQ_\fine\in\refine(\QQ_\coarse)$  with $\QQ_\coarse\neq \QQ_\fine$. In practice, the latter property is trivially satisfied, but it has to be explicitly assumed within the abstract framework.

\subsubsection{Adaptive algorithm}
\label{sec:abstract adaptive algorithm}

On each mesh $\QQ_\coarse\in\Q$, we want to compute an associated quantity $U_\coarse$, think of, e.g., a Galerkin approximation of some PDE solution $u$.
We suppose that we are given an  \textit{error estimator} associated to each mesh $\QQ_\coarse\in\Q$, i.e.,  a function $\eta_\coarse:\QQ_\coarse\to[0,\infty)$.
At least heuristically, this estimator shall estimate the difference $\norm{u-U_\coarse}{}$.
By abuse of notation, we also write  $\eta_\coarse:=\eta_\coarse(\QQ_\coarse)$, where $\eta_\coarse(\SS):=(\sum_{Q\in\SS}\eta_\coarse(Q)^2)^{1/2}$ for all $\SS\subseteq\QQ_\coarse$.
Based on this error estimator, we consider the adaptive Algorithm~\ref{alg:abstract algorithm} of the form~\eqref{eq:box-algorithm}. 

\begin{algorithm}
\caption{(Abstract adaptive algorithm)} \label{alg:abstract algorithm}
\begin{algorithmic}
\Require  initial mesh ${\cal Q}_0$, marking parameter $\theta\in(0,1]$, marking constant $\const{min}\in [1,\infty]$
\For{$k=0,1,2,\dots$}
 \LineComment{compute quantity $U_k$}
\State set $U_k=$ \texttt{solve}(${\cal Q}_k$) 
 \LineComment{compute refinement indicators $\eta_\k({Q})$ for all $Q \in {\cal Q}_k$}
\State set $\eta_k =$ \texttt{estimate}(${\cal Q}_k$, $U_k$)
\LineComment{determine $\const{min}$-minimal set of elements with \eqref{eq:Doerfler}}  
\State set $\MM_k =$ \texttt{mark}($\eta_k,{\cal Q}_k$)
\LineComment{generate refined mesh}
\State set $\QQ_{\k+1} =$ \texttt{refine}$({\cal Q}_k,{\cal M}_k)$ 
\EndFor
\Ensure refined meshes $\QQ_\k$,  
quantities $U_k$, 
 estimators $\eta_\k$ for all $\k \in\N_0$
\end{algorithmic}
\end{algorithm}

In the module \texttt{solve} and \texttt{estimate}, we compute the quantity $U_k$ and the refinement indicators $\eta_k(Q)$ of all elements $Q$ in the current mesh $\QQ_k$, respectively.
In the module  \texttt{mark}, we determine up to a multiplicative constant $\const{min}$ a minimal set of elements $\MM_k\subseteq\QQ_k$ that satisfies the D\"orfler marking \cite{MR1393904}
\begin{align}\label{eq:Doerfler}
\theta\,\eta_\k^2 \le \eta_\k(\MM_\k)^2.
\end{align}
This means that $\#\MM_\k\le\const{min} \#\SS$ for all sets $\SS\subseteq\QQ_\k$ with $\theta\eta_\k^2\le \eta_\k(\SS)^2$.
If $\const{min}=\infty$, this is always satisfied and allows for uniform refinement, where $\MM_\k=\QQ_\k$. 
We note that a naive implementation of the D\"orfler marking~\eqref{eq:Doerfler} with $\const{min}=1$, which gives the truly minimal set $\MM_k$, especially requires sorting of the error indicators, which leads to a log-linear effort.
To overcome this disadvantage, \cite[Section~5]{stevenson07} proposed an algorithm to realize it with $\const{min}=2$ in linear complexity. 
Only recently, \cite{pp19} showed that linear complexity can also be attained for $\const{min}=1$.  
Based on the marked elements $\MM_k$, the mesh $\QQ_k$ is refined in the module \texttt{refine}.

We stress that an actual implementation of Algorithm~\ref{alg:abstract algorithm} will also have some kind of stopping criterion, e.g., $k$ is greater than some given bound $K\in\N$ or if $\eta_k$ is smaller than some given tolerance $\tau>0$. 
Moreover, in practice $\eta_k$ and $\QQ_k$ are not saved but overwritten by $\eta_{k+1}$ and $\QQ_{k+1}$, respectively. 
However, the given form of the algorithm allows to present convergence results in a simple way.

\subsubsection{The axioms}
\label{sec:the axioms}

We suppose that we are given some fixed \textit{perturbations}
$\dist(\QQ_\coarse,\QQ_\fine)\ge 0$ for all $\QQ_\coarse\in\Q$, $\QQ_{\fine}\in\refine(\QQ_\coarse)$,  and constants $\const{stab}$, $\const{red}$, $\const{ref}$, $\const{drel} >0$, 
 and $0\le\ro{red}<1$ such that there hold the following  estimator properties \eqref{item:stability}--\eqref{item:discrete reliability} for all  $\QQ_\coarse\in\Q$ and all $\QQ_{\fine}\in\refine(\QQ_\coarse)$:
\begin{enumerate}[(1)]
\renewcommand{\theenumi}{E\arabic{enumi}}
\bf\item\rm\label{item:stability}\textbf{Stability on non-refined elements:}  
It holds that
\begin{align*}
|\eta_{\fine}(\QQ_\coarse\cap\QQ_{\fine})-\eta_\coarse(\QQ_\coarse\cap\QQ_{\fine})|\le\const{stab}\dist(\QQ_\coarse,\QQ_\fine).
\end{align*}
\bf\item\rm\label{item:reduction}\textbf{Reduction on refined elements:} 
It holds that
\begin{align*}
\eta_{\fine}(\QQ_{\fine}\setminus\QQ_\coarse)^2&\le\ro{red}\eta_\coarse( \QQ_\coarse\setminus\QQ_{\fine})^2\\
&\quad+\const{red}\dist(\QQ_\coarse,\QQ_\fine)^2.
\end{align*}
\bf\item\rm \label{item:discrete reliability}\textbf{Discrete reliability:} 
There exists a set $\QQ_\coarse\setminus\QQ_{\fine}$ $\subseteq$ $\RR(\QQ_\coarse, \QQ_\fine)$
$\subseteq$ $\QQ_\coarse$, with
 $\#\RR(\QQ_\coarse, \QQ_\fine)\le\const{ref}\big(\#\QQ_{\fine}-\#\QQ_\coarse\big)$ such that
\begin{align*}
\dist(\QQ_\coarse,\QQ_\fine)^2\le
\const{drel}^2\eta_\coarse(\RR\big(\QQ_\coarse, \QQ_\fine)\big)^2,
\end{align*}
i.e., the perturbations are essentially controlled by the estimator on the refined elements.
\end{enumerate}
Moreover,  with the D\"orfler parameter $0<\theta\le1$  of Algorithm~\ref{alg:abstract algorithm}, let $\const{qo}>0$ and $0\le \varepsilon_{\rm qo}<1$ satisfy the following property \eqref{item:orthogonality} for the sequence $(\QQ_\k)_{\k\in\N_0}$ from Algorithm~\ref{alg:abstract algorithm}:
\begin{enumerate}[(1)]
\renewcommand{\theenumi}{E\arabic{enumi}}
\setcounter{enumi}{3}
\bf\item\rm\label{item:orthogonality} \textbf{General quasi-orthogonality:} It holds that
\begin{align*}
0\le\varepsilon_{\rm qo}<\sup_{\delta>0}\frac{1-(1+\delta)(1-(1-\ro{red})\theta)}{\const{red}+(1+\delta^{-1})\const{stab}^2},
\end{align*}
and  for all $\k,N\in\N_0$ that
\begin{align*}
\sum_{j=\k}^{\k+N}(\dist(\QQ_j,\QQ_{j+1})^2-\varepsilon_{\rm qo}\eta_j^2)\le\const{qo} \eta_\k^2,
\end{align*}
i.e., the sum of perturbations (minus some minor estimator terms) is controlled by the estimator. 
\end{enumerate}

\begin{remark}
Later, in a more concrete setting, $\dist(\QQ_\coarse,\QQ_\fine)$ will always be the error $\norm{U_\fine-U_\coarse}{}$ between the two Galerkin solutions $U_\coarse$ and $U_\fine$ corresponding to the meshes $\QQ_\coarse$ and $\QQ_\fine$, respectively.
If the involved bilinear form is symmetric, \eqref{item:orthogonality} even with $\varepsilon_{\rm qo}=0$ follows directly from the Pythagoras identity $\norm{U_{j+1}-U_j}{}^2 =\norm{u-U_j}{}^2- \norm{u-U_{j+1}}{}^2$ in the energy norm and reliability $\norm{u-U_k}{}\lesssim \eta_k$ of the estimator, see also Remark~\ref{rem:symmetric orthogonality}.
\end{remark}

Moreover, we suppose that we are given constants $\const{child}$, $\const{clos}\ge1$ such that there hold the following   refinement properties \eqref{R:childs}--\eqref{R:overlay}:
\begin{enumerate}[(1)]
\renewcommand{\theenumi}{R\arabic{enumi}}
\bf\item\rm \label{R:childs} 
\textbf{Child estimate:}
For all $\QQ_\coarse\in\Q$, $\MM_\coarse\subseteq\QQ_\coarse$ and $\QQ_\fine:=\refine(\QQ_\coarse,\MM_\coarse)$, it holds that 
\begin{align*}\#\QQ_\fine \le \const{child}\,\#\QQ_\coarse,
\end{align*} i.e., one step of refinement leads to a bounded increase of elements. 
\bf\item\rm\label{R:closure} \textbf{Closure estimate:}
Let $(\QQ_\k)_{\k\in\N_0}$ be an arbitrary sequence in $\Q$ such that  $\QQ_{\k+1}=\refine(\QQ_\k,\MM_\k)$ with some $\MM_\k\subseteq \QQ_\k$ for all $\k\in\N_0$.
Then, for all $\k\in\N_0$, it holds  that 
\begin{align*}
\# \QQ_\k-\#\QQ_0\le \const{clos}\sum_{j=0}^{\k-1}\#\MM_j.
\end{align*}
This inequality is trivially satisfied if only marked elements are refined. However, in practice, to preserve admissibility of the meshes, additional elements have to be refined. Then, \eqref{R:closure} states that the overall number of elements $\#\QQ_k$ can be controlled by $\#\QQ_0$ plus the number of marked elements.
\bf\item\rm\label{R:overlay}
\textbf{Overlay property:}
For all meshes $\QQ_\coarse,\QQ_\meshidx\in\Q$, there exists a common refinement $\QQ_\fine\in\refine(\QQ_\coarse)$ $\cap$ $\refine(\QQ_\meshidx)$ which satisfies the overlay estimate
\begin{align*}
\#\QQ_\fine \le \#\QQ_\coarse + \#\QQ_\meshidx - \#\QQ_0.
\end{align*}
\end{enumerate}

\subsubsection{Optimal convergence for the error estimator}
\label{sec:abstract main}
The following theorem is the main result of Section~\ref{sec:axioms}.
It was already proved 
in \cite[Theorem~4.1 and Corollary~4.8]{cfpp14}.
For arbitrary $s>0$, we set
\begin{align}\label{eq:const apx}
\const{apx}(s):=\sup_{N\ge\#\QQ_0} \min_{\QQ_\coarse\in\Q(N)}(N^s\eta_\coarse)\in[0,\infty]
\end{align}
with $\Q(N):=\set{\QQ_{\coarse}\in\Q}{\#\QQ_{\coarse}\le N}$.
By definition, it holds that $\const{apx}(s)<\infty$ if and only if the error estimator converges as $\eta_\coarse=\OO((\#\QQ_\coarse)^{-s})$ if the optimal admissible meshes are chosen.
Consequently, an adaptive algorithm is called \textit{optimal} if the sequence of adaptively generated meshes leads to $\eta_\ell=\OO((\#\QQ_\ell)^{-s})$ for all $s>0$ with $\const{apx}(s)<\infty$. 

\begin{theorem}\label{thm:abstract main}
Let $\QQ_0, \theta \in (0,1]$, and $\const{min} \in [1, \infty]$ be the input arguments of Algorithm~\ref{alg:abstract algorithm}, and let $(\QQ_\k)_{\k\in\N_0}$ and $(\eta_\k)_{\k\in\N_0}$ be the meshes and estimators generated by Algorithm~\ref{alg:abstract algorithm}.
Then, there hold: 
\begin{enumerate}[\rm (i)]
\item \label{item:abstract main 1} Suppose that the axioms {\rm\eqref{item:stability}}--\eqref{item:reduction} hold true
at least for $\QQ_{k+1}\in\refine(\QQ_k)$ and all $k\in\N_0$,
and assume that $\lim_{\k\to\infty}\dist(\QQ_\k,\QQ_{\k+1})=0$.
Then, for all $0<\theta\le 1$ and all $\const{min}\in[1,\infty]$, the estimator converges, i.e.,
\begin{align}\label{eq:estimator convergence}
\lim_{\k\to \infty}\eta_\k=0.
\end{align}
\item \label{item:abstract main 2}
Suppose that the axioms \eqref{item:stability}--\eqref{item:reduction} hold true
at least for $\QQ_{k+1}\in\refine(\QQ_k)$ and all $k\in\N_0$ and \eqref{item:orthogonality} holds true as well.
 Then, for all $0<\theta\le 1$ and all $\const{min}\in[1,\infty]$, the estimator converges linearly, i.e., there exist constants $0<\ro{lin}<1$ and $\const{lin}\ge1$ such that
\begin{align}\label{eq:linear convergence}
\eta_{\k+j}^2\le \const{lin} \ro{lin}^j\eta_\k^2\quad\text{for all }j,\k\in\N_0.
\end{align}
\item \label{item:abstract main 3}
Suppose that the axioms {\rm\eqref{item:stability}--\eqref{item:orthogonality}} as well as \eqref{R:childs}--\eqref{R:overlay} hold true. Then, for all $0<\theta<\theta_{\rm opt}:=(1+\const{stab}^2 \const{drel}^2)^{-1}$ and all $\const{min}\in[1,\infty)$, the estimator converges at  optimal rate, i.e.,  for all $s>0$ there exist constants $c_{\rm opt},\const{opt}>0$ such that 
\begin{align}\label{eq:optimal convergence}
\hspace{-1mm}c_{\rm opt} \const{apx}(s)&\le\sup_{\k\in\N_0}{(\#\QQ_\k)^{s}}{\eta_\k}\le C_{\rm opt} \const{apx}(s),
\end{align}
where the lower bound relies only on  \eqref{R:childs}.
\end{enumerate}
The constants $\const{lin},\ro{lin}$ depend  only on $\ro{red}, \const{qo}, \varepsilon_{\rm qo}$, and on $\theta$.
The constant $C_{\rm opt}$ depends additionally on $\const{min},\const{ref},\const{drel}, \varepsilon_{\rm drel},\const{clos},\const{over},  \#\QQ_0$, and on $s$, while $c_{\rm opt}$ depends only on $\const{child}, \#\QQ_0$,  $s$, and if there exists $\k_0$ with $\eta_{\k_0}=0$ also on $\k_0$. 
\end{theorem}
\begin{proof}
In the following, we only give a sketch of the proof.
For details, we refer to \cite[Theorem~4.1 and Corollary~4.8]{cfpp14}.

\textbf{Sketch of \eqref{item:abstract main 1}.}
Elementary calculations show that the axioms \eqref{item:stability}--\eqref{item:reduction} and the fact that $\MM_k\subseteq\QQ_k\setminus\QQ_{k+1}$ in combination with D\"orfler marking~\eqref{eq:Doerfler} lead to estimator reduction:
There exist constants $0<\ro{est}<1$ and $\const{est}>0$ such that
\begin{align}\label{eq:estimator reduction}
0\le\eta_{k+1}^2\le \ro{est}\eta_k^2+\const{est}\dist(\QQ_\k,\QQ_{\k+1})^2 \text{ for all }k\in\N_0.
\end{align}
Due to the assumption $\lim_{k\to\infty}\dist(\QQ_\k,\QQ_{\k+1})=0$, 
  basic calculus proves \eqref{eq:estimator convergence}. 

\textbf{Sketch of \eqref{item:abstract main 2}.}
Linear convergence \eqref{eq:linear convergence} can  be equivalently  reformulated as 
\begin{align*}
\sum_{j=k+1}^\infty \eta_j^2 \lesssim \eta_k^2\quad\text{for all }k\in\N_0.
\end{align*}
The latter follows from estimator reduction \eqref{eq:estimator reduction} and general quasi-orthogonality \eqref{item:orthogonality}. 

\textbf{Sketch of \eqref{item:abstract main 3}.}
The lower estimate in \eqref{eq:optimal convergence} follows elementarily from the child estimate \eqref{R:childs}. 
The upper bound is more involved. 
Let $j\in\N_0$. 
Stability \eqref{item:stability} plus discrete reliability \eqref{item:discrete reliability} elementarily yield the existence of some constant $0<q(\theta)<1$ such that  any refinement $\QQ_{\fine(j)}\in\refine(\QQ_j)$ with $\eta_{\fine(j)}^2\le q(\theta)\eta_j^2$ satisfies the D\"orfler marking 
\begin{align}\label{eq:optimal doerfler}
\theta\eta_j^2\le \eta_j\big(\RR(\QQ_j,\QQ_{\fine(j)})\big)^2.
\end{align}

The heart of the proof is that there exists $\QQ_{\fine(j)}\in\refine(\QQ_j)$ with $\eta_{\fine(j)}^2\le q(\theta)\eta_j^2$, which additionally satisfies that 
\begin{align}\label{eq:comparison}
\#\QQ_{\fine(j)}-\#\QQ_j\lesssim \eta_j^{-1/s}. 
\end{align}
This follows from the definition of $\const{apx}(s)$, the overlay property \eqref{R:overlay}, and quasi-monotonicity $\eta_{\fine(j)}\lesssim\eta_\coarse$ for any $\QQ_\coarse\in\Q$ with $\QQ_{\fine(j)}\in\refine(\QQ_\coarse)$.
The latter is a consequence of \eqref{item:stability}, \eqref{item:reduction}, and  \eqref{item:discrete reliability}.
Since $\MM_j\subseteq\QQ_j$ is an essentially minimal set that satisfies the D\"orfler marking \eqref{eq:Doerfler}, \eqref{eq:optimal doerfler} gives that 
\begin{align*}
\#\MM_j\lesssim\# \RR(\QQ_j,\QQ_{\fine(j)}).
\end{align*}
The closure estimate \eqref{R:closure}, discrete reliability \eqref{item:discrete reliability}, and \eqref{eq:comparison} imply that 
\begin{eqnarray*}
\# \QQ_k-\#\QQ_0
&\lesssim&\sum_{j=0}^{\k-1}\#\MM_j
\lesssim\sum_{j=0}^{\k-1}\# \RR(\QQ_j,\QQ_{\fine(j)})\\
&\lesssim&\sum_{j=0}^{\k-1} (\#\QQ_{\fine(j)}-\#\QQ_j)
\lesssim
\sum_{j=0}^{\k-1} \eta_j^{-1/s}.
\end{eqnarray*}
Finally, linear convergence \eqref{eq:linear convergence} elementarily shows that $\sum_{j=0}^{\k-1} \eta_j^{-1/s}\lesssim \eta_k^{-1/s}$.
This concludes the proof. \qed
\end{proof}

\begin{remark}
The upper bound in \eqref{eq:optimal convergence} states that the estimator sequence $\eta_\k$ of Algorithm~\ref{alg:abstract algorithm}  converges   with algebraic rate $s$ if $\const{apx}(s)<\infty$.
This means that  if a decay with rate $s$ is possible for optimally chosen admissible meshes, the same decay is realized  by the adaptive algorithm. 
Together with the upper bound, the lower bound in \eqref{eq:optimal convergence} states that the convergence rate of the estimator sequence characterizes the theoretically optimal convergence rate.
\end{remark}

\subsection{Abstract adaptive FEM}
\label{sec:afem}

This section summarizes the results of the recent own works \cite{ghp17,gantner17}. 
For the model problem~\eqref{eq:problem} of Section~\ref{sec:FEM problem},  i.e., 
\begin{align*}
\mathscr{P}u&=f\quad \text{in }\Omega,\\
u&=0\quad\text{on }\Gamma:=\partial\Omega,
\end{align*}
we consider Algorithm~\ref{alg:abstract algorithm} in the  context of conforming FEM discretizations on a multi-patch  geometry $\Omega$ as in Section~\ref{sec:parametrization_assumptions}, where adaptivity is driven by the  \textit{weighted-{residual {\sl a~posteriori} error estimator}}  \eqref{eq:eta}, 
which reads
\begin{align*}
 \eta(Q)^2:=h_Q^{2} \norm{f-\mathscr{P}U}{L^2(Q)}^2+h_Q \norm{[\mathscr{D}_{\normal} U]}{L^2(\partial Q\cap \Omega)}^2.
\end{align*}
We identify the crucial properties of the underlying meshes, the mesh refinement, and the finite element spaces, which ensure that the weighted-residual error estimator fits into the general framework of Section~\ref{sec:axioms} and which hence guarantee optimal convergence behavior of the adaptive algorithm in the sense of Theorem~\ref{thm:abstract main}.
The main result of this section is Theorem~\ref{thm:abstract}. 
In Section~\ref{sec:adaptive igafem}, we will see that it is applicable to hierarchical splines as well as T-splines.

\subsubsection{Axioms of adaptivity (revisited)}
\label{sec:abstract setting fem}

\paragraph{Meshes}
Throughout this section, $\QQ_\coarse$ is a \textit{mesh} of the bounded Lipschitz domain $\Omega\subset \R^\dph$ in the following sense:
\begin{itemize}
\item $\QQ_\coarse$ is a finite set of 
transformed open hyperrectangles, i.e., each element $Q$ has the form $Q=\F_m(\widehat Q)$ for some $\F_m$ from Section~\ref{sec:parametrization_assumptions}, where $\widehat Q=\prod_{i=1}^\dph(a_i,b_i)$ is an open $\dph$-dimensional hyperrectangle; 
\item for all $Q,Q'\in\QQ_\coarse$ with $Q\neq Q'$, the intersection  is empty, i.e., $Q\cap Q'=\emptyset$;
\item $\overline\Omega = \bigcup_{Q\in\QQ_\coarse}{\overline Q}$, i.e., $\QQ_\coarse$ is a partition of $\Omega$.
\end{itemize}
Let $\Q$ be a set of such meshes.
These are referred to as \textit{admissible}.
In order to ease notation, we introduce for $\QQ_\coarse\in\Q$ the corresponding \textit{mesh-width function} 
\begin{align*}
h_\coarse\in L^\infty(\Omega),  \,h_\coarse|_Q:=h_Q:=|Q|^{1/\dph}\text{ for all }Q\in\QQ_\coarse.
\end{align*}

For $\omega\subseteq\overline\Omega$, we define the \textit{element-patches\footnote{Do not confuse the element-patches with the patches of Section~\ref{sec:multi-patch} that are used to describe the geometry.} $\pi_\coarse^q(\omega)\subseteq\overline\Omega$ of order} $q\in\N_0$ inductively by
\begin{align}\label{eq:patch1}
\begin{split}
 \pi_\coarse^0(\omega) &:=\overline \omega,\\
 \pi_\coarse^{q+1}(\omega) &:= \bigcup\set{\overline Q}{ Q\in \QQ_\coarse,\, {\overline Q}\cap \pi_\coarse^{q}(\omega)\neq \emptyset}.
 \end{split}
\end{align}
The corresponding set of elements is defined as
\begin{align}\label{eq:patch2}
 \Pi_\coarse^q(\omega) := \set{Q\in\QQ_\coarse}{ {Q} \subseteq \pi_\coarse^q(\omega)} \quad\text{for } q>0,
\end{align}
i.e., $\pi_\coarse^q(\omega) = \overline{\bigcup\Pi_\coarse^q(\omega)}$.
To abbreviate notation, we set $\pi_\coarse(\omega) := \pi_\coarse^1(\omega)$ and $\Pi_\coarse(\omega) := \Pi_\coarse^1(\omega)$.
For $\SS\subseteq\QQ_\coarse$, we define $\pi_\coarse^q(\SS):=\pi_\coarse^q(\bigcup\SS)$
and $\Pi_\coarse^q(\SS):=\Pi_\coarse^q(\bigcup\SS)$. 

We suppose that there exist  constants $\const{shape}$, $\const{locuni}$ $>$ $0$ such that all meshes $\QQ_\coarse\in\Q$ satisfy the following  two mesh properties~\eqref{M:shape}--\eqref{M:locuni}:
\begin{enumerate}[(i)]
\renewcommand{\theenumi}{M\arabic{enumi}}
\bf\item\rm\label{M:shape}
{\bf Shape-regularity:} 
It holds that\footnote{Recall the definition ${\rm diam}(S):=\sup\set{|{\bf x}-{\bf y}|}{{\bf x},{\bf y}\in S}$ of the diameter of an arbitrary non-empty set $S\subseteq\R^{\dph}$.}
\begin{align*}
{\diam(Q)}/{h_Q}\le \const{shape} \quad \text{for all } Q\in\QQ_\coarse.
\end{align*}
Since there always holds that $h_Q\le \diam(Q)$, this implies  that 
$h_Q \simeq \diam(Q)$.
\bf\item\rm\label{M:locuni}
\textbf{Local  quasi-uniformity:}
It holds that 
\begin{align*}
 h_Q/h_{Q'} \le \const{locuni}\quad\text{for all }Q\in\QQ_\coarse, Q'\in\Pi_\coarse(Q),
\end{align*} i.e., neighboring elements have comparable size.
\end{enumerate}

\smallskip\paragraph{Mesh refinement}
We suppose that we are given a refinement strategy $\refine(\cdot,\cdot)$ as in Section~\ref{sec:meshes_axioms}  and Section~\ref{sec:the axioms}.
In particular, we suppose the existence of some initial mesh $\QQ_0$ with $\refine(\QQ_0)=\Q$  and the refinement axioms \eqref{R:childs}--\eqref{R:overlay} hold true.
Moreover, we assume that for all $\QQ_\coarse\in\Q$ and arbitrary marked elements $\MM_\coarse\subseteq\QQ_\coarse$ with  refinement $\QQ_\fine$ $:=$ $\refine(\QQ_\coarse,\MM_\coarse)$, it holds that
\begin{align}\label{eq:union}
\overline Q=\bigcup\set{\overline {Q'}}{Q'\in\QQ_\fine ,  Q'\subseteq Q}\text{ for all }Q\in\QQ_\coarse,
\end{align}
 i.e., each element $Q$ is the union of its successors.

\smallskip\paragraph{Finite element space}
With each $\QQ_\coarse\in\Q$, we associate a finite dimensional space 
\begin{align*}
\begin{split}
\mathbb{S}_\coarse \subset \big\{v\in H^1_0(\Omega): \, &v|_Q\in H^2(Q)\text{ for all }Q\in\QQ_\coarse\big\}.
\end{split}
\end{align*}
Let $U_\coarse\in\mathbb{S}_\coarse$ be the corresponding Galerkin approximation, defined via the variational formulation \eqref{eq:Galerkin fem}, to the solution $u\in H^1_0(\Omega)$ of problem~\eqref{eq:weak fem}. 

We suppose that there exist constants  $\q{loc},\q{proj} \in\N_0$ and 
for all $\QQ_\coarse\in\Q$ a Scott--Zhang-type projector $\JJ_\coarse:H^1_0(\Omega)\to\mathbb{S}_\coarse$ 
such that the following space properties~\eqref{S:nestedness}--\eqref{S:proj} hold for all $\QQ_\coarse\in\Q$ and all refinements $\QQ_\fine\in\refine(\QQ_\coarse)$:
\begin{enumerate}[(i)]
\renewcommand{\theenumi}{S\arabic{enumi}}
\bf\item\rm\label{S:nestedness}
\textbf{Nestedness:}
 It holds that 
\begin{align*}\mathbb{S}_\coarse\subseteq\mathbb{S}_\fine.
\end{align*}
\bf\item\rm\label{S:local}
\textbf{Local domain of definition:}
For all
 $Q$ $\in$ $\QQ_\coarse\setminus \Pi_\coarse^{\q{loc}}( \QQ_\coarse\setminus\QQ_\fine)\subseteq\QQ_\coarse\cap\QQ_\fine$ (i.e., $Q$ is in a certain sense far away from the refined elements $\QQ_\coarse\setminus\QQ_\fine$)  and for all $V_\fine\in\mathbb{S}_\fine$, it holds that 
\begin{align*}V_\fine|_{\pi_\coarse^{\q{proj}}(Q)} \in \set{V_\coarse|_{\pi_\coarse^{\q{proj}}(Q)}}{V_\coarse\in\mathbb{S}_\coarse}.
\end{align*}
\bf\item\rm\label{S:proj}
\textbf{Local projection property:}
For all $v\in H^1_0(\Omega)$ and $Q\in\QQ_\coarse$, it holds that 
\begin{align*}
&(\JJ_\coarse v)|_Q = v|_Q\\
&\quad\text{if } v|_{\pi_\coarse^{\q{proj}}(Q)} \in \set{V_\coarse|_{\pi_\coarse^{\q{proj}}(Q)}}{V_\coarse\in\mathbb{S}_\coarse}.
\end{align*}
\end{enumerate}

Besides \eqref{S:nestedness}--\eqref{S:proj}, which are also required in the following Section~\ref{sec:abem} on abstract adaptive BEM, we suppose the existence of   constants $\const{inv},\const{sz}>0$ and $\q{sz}\in\N_0$  such that the following FEM properties \eqref{S:inverse}--\eqref{S:grad} hold for all $\QQ_\coarse\in\Q$:
\begin{enumerate}[(i)]
\renewcommand{\theenumi}{F\arabic{enumi}}
\bf\item\rm\label{S:inverse}
\textbf{Inverse inequality:}
For  all $j,k\in\{0,1,2\}$ with $k\le j$, all $V_\coarse\in\mathbb{S}_\coarse$, and all $Q\in\QQ_\coarse$, it holds that 
\begin{align*}
h_Q^{(j-k)} \norm{V_\coarse}{H^j(Q)}\le \const{inv} \, \norm{V_\coarse}{H^{k}(Q)}.
\end{align*}
\bf\item\rm\label{S:app}
\textbf{Local $\boldsymbol{L^2}$-approximation property:}
For all $Q\in\QQ_\coarse$ and all $v\in H_0^1(\Omega)$, it holds that 
\begin{align*}
\norm{(1-\JJ_\coarse)v}{L^2(Q)}\le \const{sz} \,h_Q\,\norm{v}{H^1(\pi_\coarse^{\q{sz}}(Q))}.
\end{align*}
\bf\item\rm\label{S:grad}
\textbf{Local $\boldsymbol{H^1}$-stability:}
For all $Q\in\QQ_\coarse$ and $v\in H_0^1(\Omega)$, it holds that 
\begin{align*}\norm{\nabla \JJ_\coarse v}{L^2(Q)}\le \const{sz} \norm{v}{H^1(\pi_\coarse^{\q{sz}}(Q))}.
\end{align*}
\end{enumerate}

\subsubsection{Data oscillations}\label{sec:oscillations}
The definition of the data oscillations corresponding to the residual error estimator \eqref{eq:eta} requires some further notation. 
Let $\mathbb{P}(\Omega)$ be the set of all (transformed) tensor polynomials of some fixed degree $(p',\dots,p')$ on $\Omega$, i.e.,  with the patches $\Omega_m$ from Section~\ref{sec:multi-patch}, 
\begin{align*}
\begin{split}
\mathbb{P}(\Omega):=\big\{W:\, W|_{\Omega_m}\circ\F_m \text{ polynomial of degree } \\
(p',\dots,p') \text{ for all }m\in\{1,\dots,M\}\big\}.
\end{split}
\end{align*}
For $\QQ_\coarse\in\Q$
and $Q\in\QQ_\coarse$, let  $P_{Q}:L^2(Q)\to \set{{W}|_{Q}}{{W}\in\mathbb{P}(\Omega)}$ be the $L^2$-orthogonal projection, 
i.e.,
\begin{align*}
\norm{v-P_Q v}{L^2(Q)} = \min_{W\in \mathbb{P}(\Omega)} \norm{v-W}{L^2(Q)}
\end{align*}
for all $v\in L^2(Q)$.
For an interior edge in 2D or face in 3D,
$E\in\mathcal{E}_{Q}:=\set{\overline{Q}\cap \overline{Q}'}{Q'\in\QQ_\coarse,  {\rm dim}(\overline Q\cap \overline Q')=\dph-1}$, where $\dim(\cdot)$ denotes the  dimension, we define the $L^2$-orthogonal projection  $P_{E}:L^2(E)\to \set{{W}|_{E}}{{W}\in\mathbb{P}(\Omega)}$. Finally,
for $V_\coarse\in\mathbb{S}_\coarse$, we define  the corresponding \textit{oscillations}
\begin{subequations}\label{eq:osc}
\begin{align}
 &\osc_\coarse(V_\coarse) := \osc_\coarse(V_\coarse,\QQ_\coarse)\\
 &\quad\text{with } 
 \osc_\coarse({{V_\coarse}},\SS)^2:=\sum_{Q\in\SS} \osc_\coarse({{V_\coarse}}, Q)^2
 \text{ for all }\SS\subseteq\QQ_\coarse,\notag
\end{align}
where, for all $Q\in\QQ_\coarse$, the local oscillations read
\begin{align}
\osc_\coarse({{V_\coarse}},Q)^2&:=h_Q^2 \norm{(1-P_{Q})(f-\mathscr{P}U_\coarse)}{L^2(Q)}^2\\&\quad+\sum_{E\in\mathcal{E}_{Q}}h_Q\norm{(1-P_{E})[\mathscr{D}_\normal U_\coarse]}{L^2(E)}^2.\notag
\end{align}
\end{subequations}

\begin{remark}
For the analysis of oscillations in the frame of standard FEM with piecewise polynomials of fixed order, we refer, e.g., to~\cite{nv11}.
\end{remark}

\begin{remark}\label{rem:C12}
If $\mathbb{S}_\coarse\subset C^1(\overline\Omega)$, then the jump contributions in~\eqref{eq:osc} vanish and $\osc_\coarse(V_\coarse,Q)$ consists only of the volume oscillations. 
\end{remark}%

\subsubsection{Optimal convergence}
\label{sec:FEM axioms}
Recall the definition \eqref{eq:const apx} of the approximation constant $\const{apx}(s)$.
We say that the solution $u\in H_0^1(\Omega)$ belongs to the \textit{approximation class $s$ with respect to the estimator} \eqref{eq:eta}, if
\begin{align*}
\const{apx}(s)<\infty.
\end{align*}
Further, we say that it belongs to the \textit{approximation class $s$ with respect to the minimal total error} if 
\begin{align}\label{eq:total class}
\notag \const{apx}^{\rm tot}(s)
 &:=
\sup_{N\ge\#\QQ_0} \min_{\QQ_\coarse\in\Q(N)}(N^s\inf_{V_\coarse\in\mathbb{S}_\coarse}\big(\norm{u-V_\coarse}{H^1(\Omega)}\\
&\hspace{35mm}+\osc_\coarse(V_\coarse))<\infty.
\end{align}
Note that both approximation classes depend on the considered ansatz spaces, the underlying meshes, and the corresponding refinement strategy.

By definition, $\const{apx}(s)<\infty$ (resp.\ $\const{apx}^{\rm tot}(s)<\infty$) implies that the error estimator $\eta_\coarse$ (resp.\ the \textit{minimal total error}) on the optimal meshes $\QQ_\coarse$ decays at least with rate $\OO\big((\#\QQ_\coarse)^{-s}\big)$. The following main theorem states that each possible rate $s>0$ is in fact realized by Algorithm~\ref{alg:abstract algorithm}.
It stems from \cite[Theorem~2.1]{ghp17} and essentially follows from its abstract counterpart Theorem~\ref{thm:abstract main} by verifying the axioms of Section~\ref{sec:the axioms} for the perturbations 
\begin{align*}
\dist(\QQ_\coarse,\QQ_\fine):=\norm{U_\fine-U_\coarse}{H^1(\Omega)}
&\text{ for all }\QQ_\coarse\in\Q, \\
&\text{ and }\QQ_{\fine}\in\refine(\QQ_\coarse).
\end{align*}
For piecewise polynomials on shape-regular triangulations of a polyhedral domain $\Omega$, optimal convergence was already proved in \cite{ckns08} for symmetric $\mathscr{P}$ and in \cite{cn12,ffp14} for non-symmetric $\mathscr{P}$.

\begin{theorem}\label{thm:abstract}
Let $(\QQ_\k)_{\k\in\N_0}$ be the sequence of meshes generated by Algorithm~\ref{alg:abstract algorithm} with Galerkin approximations $U_k\in\mathbb{S}_k$.
Then, there hold:
\begin{enumerate}[\rm (i)]
\item\label{item:qabstract reliable fem}
Suppose that \eqref{M:shape}--\eqref{M:locuni} and \eqref{S:app}--\eqref{S:grad} hold true.
Then, 
the residual error estimator  satisfies reliability, i.e., there exists $\const{rel}>0$ such that for all $\QQ_\coarse\in\Q$, 
\begin{align}\label{eq:reliable}
 \norm{u-U_\coarse}{H^1(\Omega)}+\osc_\coarse(U_\coarse)\le \const{rel}\eta_\coarse.
\end{align}
\item
\label{item:qabstract efficient fem}
Suppose that \eqref{M:shape}--\eqref{M:locuni} and \eqref{S:inverse} hold true. 
Then, 
the residual error estimator satisfies efficiency, i.e., there  exists  $\const{eff}>0$ such  that for all $\QQ_\coarse\in\Q$, 
\begin{align}\label{eq:efficient}
\const{eff}^{-1}\eta_\coarse\le \inf_{V_\coarse\in\mathbb{S}_\coarse}\big(& \norm{u-V_\coarse}{H^1(\Omega)}+\osc_\coarse(V_\coarse)\big).
\end{align}
\item 
Suppose that \eqref{M:shape}--\eqref{M:locuni}, \eqref{S:nestedness} and \eqref{S:inverse} hold true.
Then, the axioms \eqref{item:stability}--\eqref{item:reduction}  as well as the convergence of the perturbations $\lim_{k\to\infty}\dist(\QQ_\k,\QQ_{\k+1})=0$ are satisfied. 
These are exactly the assumptions of Theorem~\ref{thm:abstract main}~\eqref{item:abstract main 1}, which implies  convergence~\eqref{eq:estimator convergence} of the estimator. 
\item\label{item:qabstract linear convergence fem}
Suppose that  \eqref{M:shape}--\eqref{M:locuni}, \eqref{S:nestedness} and  \eqref{S:inverse}--\eqref{S:grad} hold true.
Then, the axioms \eqref{item:stability}--\eqref{item:reduction} and \eqref{item:orthogonality} are satisfied.
These are exactly the assumptions of Theorem~\ref{thm:abstract main}~\eqref{item:abstract main 2}, which implies linear convergence~\eqref{eq:linear convergence} of the estimator.
\item\label{item:qabstract optimal convergence fem}
Suppose \eqref{M:shape}--\eqref{M:locuni},  \eqref{R:childs}--\eqref{R:overlay}, \eqref{S:nestedness}--\eqref{S:proj} and \eqref{S:inverse}--\eqref{S:grad} hold true.
Then, the axioms {\rm\eqref{item:stability}--\eqref{item:orthogonality}} as well as \eqref{R:childs}--\eqref{R:overlay} are satisfied.
These are exactly the assumptions of Theorem~\ref{thm:abstract main}~\eqref{item:abstract main 3}, which implies optimal convergence \eqref{eq:optimal convergence} of the estimator.
\end{enumerate}
\noindent 
All involved constants $\const{rel},\const{eff},\const{lin},\ro{lin},\theta_{\rm opt}$, and $\const{opt}$  (of Theorem~\ref{thm:abstract main}) depend only on the assumptions made, the dimension $\dph$, the coefficients of the differential operator $\mathscr{P}$, $\diam(\Omega)$, and the parametrization constant $C_\F$ from Section~\ref{sec:parametrization_assumptions},  where $\const{lin},\ro{lin}$ depend additionally on $\theta$ and the sequence $(U_\k)_{\k\in\N_0}$ (see also Remark~\ref{rem:symmetric orthogonality}), and $\const{opt}$ depends furthermore on $\const{min}$  and $s$.
The constant $c_{\rm opt}$ depends only on $\const{child}, \#\QQ_0$,  $s$, and if there exists $\k_0$ with $\eta_{\k_0}=0$ also on $\k_0$ and $\eta_0$.
\end{theorem}
\begin{proof}
The proof is found in \cite[Section~4]{ghp17} and details are elaborated in \cite[Section~4.5]{gantner17}.
Indeed, these works consider even more general meshes.
They assume additional abstract mesh and refinement properties, given by~\eqref{eq:bounded patch}--\eqref{eq:reduction} below. 
Moreover, the analysis of \cite{ghp17,gantner17} requires certain properties of the space $\mathbb{P}(\Omega)$, which have only been proved for single-patch domains in \cite[Section~5.11 and 5.12]{ghp17}, but the proof easily extends to the considered multi-patch case.
Therefore, in the remainder of the proof, we only verify that the assumptions \eqref{eq:bounded patch}--\eqref{eq:reduction} are automatically satisfied in our setting, and the result follows from \cite[Section~4]{ghp17}.

Let $\QQ_\coarse\in\Q$.
The properties \eqref{M:shape}--\eqref{M:locuni} especially imply the uniform boundedness of the number of elements within an element-patch, i.e., 
\begin{align}\label{eq:bounded patch}
\#\Pi_\coarse(Q)\lesssim 1 \quad\text{for all }Q\in\QQ_\coarse.
\end{align} 
To see this, we note the elementary inequality $|\Pi_\coarse(Q)|\le \diam(\Pi_\coarse(Q))^\dph$. 
Then, 
\eqref{M:shape}--\eqref{M:locuni} show that $\diam(\Pi_\coarse(Q))$ $\lesssim$ $\diam(Q)\simeq |Q|^{1/\dph}$.
On the one hand, we see that $|Q|\le|\Pi_\coarse(Q)|\lesssim |Q|$, i.e., $|\Pi_\coarse(Q)|\simeq |Q|$. 
On the other hand, \eqref{M:locuni} implies that $|\Pi_\coarse(Q)|$ $\simeq$ $\#\Pi_\coarse(Q)|Q|$. This concludes the proof of  \eqref{eq:bounded patch}. 

Moreover, it is easy to see that our assumptions on $C_\F$ of \eqref{eq:Cgamma} along with \eqref{M:shape} yield that $[0,1]^\dph$ is a reference element in the sense that for all $Q\in\QQ_\coarse\in\Q$, there exists a bi-Lipschitz mapping $\widetilde\F_Q:[0,h_Q]^\dph\to Q$ with uniform Lipschitz constants that depend only on $C_\F$ and \eqref{M:shape}.
In particular, one obtains the trace inequality 
\begin{align}\label{eq:trace inequality}
\begin{split}
\norm{v}{L^2(\partial Q)}^2\lesssim \big(h_Q^{-1}\norm{v}{L^2(Q)}^2+ \norm{v}{L^2(Q)}\norm{ \nabla v}{L^2(Q)}\big)\\
\text{for all }v\in H^1(\Omega),
\end{split}
\end{align}
and the local Poincar\'e estimate
\begin{align}\label{eq:dual estimate}
 h_Q^{-1} \norm{w}{H^{-1}(Q)}\lesssim\norm{w}{L^2(Q)}\quad\text{for all }w\in L^2(\Omega)
\end{align}
for the dual norm $\norm{{w}}{H^{-1}(Q)}= \sup\{\int_Q{w}  v\,\d\xx:$ $v\in H_0^1(Q),$  $\norm{ v}{H^1(Q)}=1\}$.
The constants hidden in \eqref{eq:trace inequality}--\eqref{eq:dual estimate} depend only on $C_\F$ and \eqref{M:shape}, see, e.g.,  \cite[Proposition~4.2.2 and Proposition~4.2.3]{gantner17} for a proof.
The latter inequality \eqref{eq:dual estimate} is a simple consequence of the classical Poincar\'e inequality.

The identities \eqref{eq:union} and \eqref{M:locuni} imply the existence of a uniform constant $0<\ro{}<1$ such that
\begin{align}\label{eq:reduction}
\hspace{-1mm}|Q'| \le \ro{}|Q|\text{ for all }Q\in\QQ_\coarse, Q'\in\QQ_\fine\text{ with } Q'\subsetneqq Q,
\end{align}
 i.e., children are uniformly smaller than their parents.
To see this, one can argue by contradiction.
If the assertion is wrong,
then there exist  sequences $(Q'_n)_{n\in\N}$ and $(Q_n)_{n\in\N}$ of such elements with 
\begin{align*} 
\lim_{n\to\infty} \frac{|Q_n|-|Q_n'|}{|Q_n|}=\lim_{n\to\infty} \frac{|Q_n\setminus Q_n'|}{|Q_n|}=0.
\end{align*}
However, \eqref{eq:union} and \eqref{M:locuni} show the existence of $Q_n''\subset Q_n\setminus Q_n'$ with $|Q_n''|\simeq |Q_n'|$, which contradicts the latter equality. 
\qed
\end{proof}

\begin{remark}
If the assumptions of Theorem~\ref{thm:abstract} \eqref{item:qabstract reliable fem}--\eqref{item:qabstract efficient fem} are  satisfied, there holds in particular that 
\begin{align*}
\const{eff}^{-1}\norm{u}{\mathbb{A}^{\rm est}_s}\le\norm{u}{\mathbb{A}^{\rm tot}_s}\le \const{rel}\norm{u}{\mathbb{A}^{\rm est}_s} \text{ for all }s>0.
\end{align*}
This also shows that the optimality results in \cite{stevenson07,ckns08,cn12} coincide with that of \cite{ffp14,cfpp14}.
\end{remark}

\begin{remark}\label{rem:symmetric orthogonality}
Only general quasi-orthogonality~\eqref{item:orthogonality} depends on the sequence $(U_\k)_{\k\in\N_0}$.
If the bilinear form $\dual{\cdot}{\cdot}_{\mathscr{P}}$ is symmetric, then 
 \eqref{item:orthogonality}  follows from the Pythagoras identity $\norm{u-U_{j+1}}{\mathscr{P}}^2+\norm{U_{j+1}-U_j}{\mathscr{P}}^2=\norm{u-U_j}{\mathscr{P}}^2$ in the $\mathscr{P}$-induced energy norm $\norm{v}{\mathscr{P}}^2:=\dual{v}{v}_{\mathscr{P}}$ and norm equivalence
\begin{align*}
 &\sum_{j=k}^{k+N}\norm{U_{j+1}-U_j}{H^1(\Omega)}^2
 \simeq \sum_{j=k}^{k+N}\norm{U_{j+1}-U_j}{\mathscr{P}}^2
 \\
 &\quad= \norm{u-U_k}{\mathscr{P}}^2-\norm{u-U_{k+N+1}}{\mathscr{P}}^2 
 \lesssim \norm{u-U_k}{H^1(\Omega)}^2.
\end{align*}
Together with reliability~\eqref{eq:reliable}, this proves~\eqref{item:orthogonality} even for $\varepsilon_{\rm qo}=0$, and $\const{qo}$ is independent of the sequence $(U_k)_{k\in\N_0}$.
In this case, the constants $\const{lin}, \ro{lin}$ and $\const{opt}$ in Theorem~\ref{thm:abstract} are independent of $(U_k)_{k\in\N}$.
In the general case, a compactness argument and {\sl a priori} convergence $\norm{u-u_k}{H^1(\Omega)}\to 0$ as $k\to\infty$ guarantee that $\const{lin},\const{opt}>0$ and $0<\ro{lin}<1$ exist, but their size may depend on the possibly slow convergence in the preasymptotic regime of Algorithm~\ref{alg:abstract algorithm}.
\end{remark}

\begin{remark}
Under the assumption that $\norm{h_{\k}}{L^\infty(\Omega)}\to0$ as $\k\to\infty$
(which can be easily guaranteed by marking additional elements), one can show that
$\mathbb{S}_\infty:=\overline{\bigcup_{\k\in\N_0}\mathbb{S}_\k}=H_0^1(\Omega)$, see \cite[Remark~2.7]{ghp17}.
This allows to follow the ideas of~\cite{bhp17} and to prove Theorem~\ref{thm:abstract} if the bilinear form $\dual\cdot\cdot_\mathscr{P}$ is only elliptic up to some compact perturbation, provided that the continuous problem is well-posed. This includes, e.g., adaptive FEM for the Helmholtz equation. For details, the reader is referred  to~\cite{bhp17}.
\end{remark}

\subsection{Abstract adaptive BEM}
\label{sec:abem}

This section summarizes the results of the recent own works \cite{gantner17,gp20}. 
Given the setting of Section~\ref{sec:model problem bem}, we consider Algorithm~\ref{alg:abstract algorithm} (with $U_k$ replaced by $\Phi_k$)
in the  context of conforming BEM discretizations of our model problem~\eqref{eq:strong}, i.e., 
\begin{align*}
 \mathscr{V}\phi = f 
\end{align*}
on a multi-patch  geometry $\Gamma$ as in Section~\ref{sec:parametrization_assumptions}, where adaptivity is driven by the \textit{weighted-residual {\sl a~posteriori} error estimator} from \eqref{eq:eta bem}, which reads
\begin{align*}
\eta(Q)^2:= h_Q \seminorm{f-\mathscr{V}\Phi}{H^1(Q)}^2 \quad\text{for all } Q\in\QQ\in\Q.
\end{align*}
We identify the crucial properties of the underlying meshes, the mesh refinement, and the boundary element spaces which ensure that the weighted-residual error estimator fits into the general framework of Section~\ref{sec:axioms} and which hence guarantee optimal convergence behavior of the adaptive Algorithm~\ref{alg:abstract algorithm} in the sense of Theorem~\ref{thm:abstract main}.
The main result of this section is Theorem~\ref{thm:abstract bem}.

\subsubsection{Axioms of adaptivity (revisited)}\label{sec:axioms revisited bem}
\label{sec:abstract setting bem}

\paragraph{Meshes}
Throughout this section, $\QQ_\coarse$ is a \textit{mesh} of the boundary $\Gamma=\partial\Omega$ of the bounded Lipschitz domain $\Omega\subset\R^\dph$ in the following sense:
\begin{itemize}
\item $\QQ_\coarse$ is a finite set of 
transformed hyperrectangles, i.e., each element $Q$ has the form $Q=\F_m(\widehat Q)$ for some $\F_m$ from Section~\ref{sec:parametrization_assumptions}, where $\widehat Q=\prod_{i=1}^{\dph-1}(a_i,b_i)$ is an open $(\dph-1)$-dimensional hyperrectangle;
\item for all $Q,Q'\in\QQ_\coarse$ with $Q\neq Q'$, the intersection is empty, i.e., $Q\cap Q'=\emptyset$; 
\item  $\QQ_\coarse$ is a partition of $\Gamma$, i.e., $\Gamma = \bigcup_{Q\in\QQ_\coarse}\overline{Q}$.
\end{itemize}
Let $\Q$ be a set of such meshes.
These are referred to as \textit{admissible}.
In order to ease notation, we introduce for $\QQ_\coarse\in\Q$ the corresponding \textit{mesh-width function} 
\begin{align*}
h_\coarse\in L^\infty(\Gamma),  \,h_\coarse|_Q:=h_Q:=|Q|^{1/{(\dph-1)}}\text{ for all }Q\in\QQ_\coarse.
\end{align*}
For $\omega\subseteq\Gamma$, we define the \textit{element-patches $\pi_\coarse^q(\omega)$ of order} $q\in\N_0$  and the corresponding set of elements $\Pi_\coarse^q(\omega)$ as in \eqref{eq:patch1}--\eqref{eq:patch2}, and we also use the abbreviations from there. 
As in Section~\ref{sec:abstract setting fem}, we suppose that shape regularity~\eqref{M:shape} and local quasi-uniformity~\eqref{M:locuni} are satisfied.

\smallskip\paragraph{Mesh refinement}
\label{sec:general refinement bem}
We suppose that we are given a mesh refinement strategy $\refine(\cdot,\cdot)$ as in Section~\ref{sec:abstract setting fem}.
In particular, we suppose the existence of some initial mesh $\QQ_0$ with $\refine(\QQ_0)=\Q$ and that the refinement axioms \eqref{R:childs}--\eqref{R:overlay} hold true.
Moreover, we even suppose a stronger version of \eqref{eq:union}, which ensures that there are only finitely many reference element-patches: 
For all $\QQ_\coarse\in\Q$ and arbitrary marked elements $\MM_\coarse\subseteq\QQ_\coarse$ with  refinement $\QQ_\fine$ $:=$ $\refine(\QQ_\coarse,\MM_\coarse)$, any refined element $Q\in \QQ_\coarse\setminus\QQ_\fine$ can only be uniformly bisected in any direction, i.e., for any $Q\in\QQ_\coarse$, the corresponding element in the parametric domain $\widehat Q=\prod_{i=1}^{\dph-1}(a_i,b_i)$  is split into elements in the parametric domain of the form 
\begin{align*}
\widehat Q'=\prod_{i=1}^{\dph-1}(a_i',b_i') \text{ with } a_i'=a_i+(b_i-a_i)k/n_{Q,i}
\end{align*} 
for some $n_{Q,i}\in\N$ and $k\in\{1,\dots,n_{Q,i}\}$. 
Here, $n_{Q,i}-1$ is the number of (uniform) bisections in direction $i$.
Note that \eqref{R:childs} yields boundedness of all $n_{Q,i}$. 
This stronger version is used to prove the auxiliary results~\eqref{M:patch center}--\eqref{M:poincare} below.

\smallskip\paragraph{Boundary element spaces}
With each $\QQ_\coarse\in\Q$, we associate a finite dimensional space 
\begin{align*}
\mathbb{S}_\coarse \subset L^2(\Gamma)\subset H^{-1/2}(\Gamma).
\end{align*}
Let $\Phi_\coarse\in\mathbb{S}_\coarse$ be the corresponding Galerkin approximation, defined via the variational formulation \eqref{eq:discrete}, to the solution $\phi\in H^{-1/2}(\Gamma)$ of problem~\eqref{eq:strong}.

We assume that the same space properties \eqref{S:nestedness}--\eqref{S:local} as in Section~\ref{sec:abstract setting fem} hold true. 
Additionally, we assume a slightly stronger version of \eqref{S:proj}:
For all $\QQ_\coarse\in\Q$ and all $\SS\subseteq\QQ_\coarse$, there exists a linear operator  $\JJ_{\coarsecomma \SS}:L^2(\Gamma)\to\set{\Psi_\coarse\in\mathbb{S}_\coarse}{\Psi_\coarse|_{\bigcup(\QQ_\coarse\setminus\SS)}=0}$ with the following property~\eqref{S:proj bem}:
\begin{enumerate}[(i)]
\renewcommand{\theenumi}{S\arabic{enumi}'}
\setcounter{enumi}{2}
\bf\item\rm\label{S:proj bem}
\textbf{Local projection property.}
Let $\q{loc}, \q{proj}\in\N_0$ from \eqref{S:local}.
For all $\psi\in L^2(\Gamma)$ and $Q\in\QQ_\coarse$ with $\Pi_\coarse^{\q{loc}}(Q)\subseteq\SS$, it holds  that 
\begin{align*}
&(\JJ_{\coarsecomma \SS} \psi)|_Q = \psi|_Q\\
& \quad\text{if }\psi|_{\pi_\coarse^{\q{proj}}(Q)} \in \set{\Psi_\coarse|_{\pi_\coarse^{{\q{proj}}}(Q)}}{\Psi_\coarse\in\mathbb{S}_\coarse}.
\end{align*}
\end{enumerate}
Clearly, \eqref{S:proj bem} coincides with \eqref{S:proj} if $\SS=\QQ_\coarse$.
In contrast to  \eqref{S:proj}, \eqref{S:proj bem} provides a local projection operator that can be additionally used as cut-off operator, which somehow replaces \eqref{S:app} in the proof of discrete reliability~\eqref{item:discrete reliability}, see \cite[Section~4.8]{gp20} for details.

Besides \eqref{S:nestedness}--\eqref{S:local} and \eqref{S:proj bem}, we suppose the existence of   constants $\const{inv},\const{sz}>0$, $\,\q{supp},\q{sz}\in\N_0$, and $0<\ro{unit}<1$  such that the following BEM properties \eqref{S:inverse bem}--\eqref{S:stab bem} hold for all $\QQ_\coarse\in\Q$:

\begin{enumerate}[(i)]
\renewcommand{\theenumi}{B\arabic{enumi}}
\bf\item\rm\label{S:inverse bem}
\textbf{Inverse inequality:}
For  all  $\Psi_\coarse\in\mathbb{S}_\coarse$, it holds that 
\begin{align*}
\norm{h_\coarse^{1/2}\Psi_\coarse}{L^2(\Gamma)}\le \const{inv} \, \norm{\Psi_\coarse}{H^{-1/2}(\Gamma)}.
\end{align*}
\bf\item\rm\label{S:unity bem}{\bf Local approximation of unity:} For all $Q\in\QQ_\coarse$, 
there exists  $\Psi_{\coarsecomma Q}\in\mathbb{S}_\coarse$ with  
$Q\subseteq \supp (\Psi_{\coarsecomma Q})\subseteq \pi_\coarse^{\q{supp}}(Q)$,
 and
\begin{align*}
&\norm{1-\Psi_{\coarsecomma Q}}{L^2(\supp(\Psi_{\coarsecomma Q}))} \le \ro{unit}{|\supp(\Psi_{\coarsecomma Q})|}^{1/2}.
\end{align*}
\bf\item\rm\label{S:stab bem}
\textbf{Local $\boldsymbol{L^2}$-stability.}
For all   $\psi\in L^2(\Gamma)$ and all $Q\in\QQ_\coarse$, it holds that 
\begin{align*}
\norm{ \JJ_{\coarsecomma \SS} \psi}{L^2(Q)}\le \const{sz} \norm{\psi}{L^2(\pi_\coarse^{\q{sz}}(Q))}.
\end{align*}
\end{enumerate}

\subsubsection{Optimal convergence}
\label{sec:BEM axioms}
Recall the definition \eqref{eq:const apx} of the approximation constant $\const{apx}(s)$.
With the definitions from Section~\ref{sec:meshes_axioms}, we say that the solution $\phi \in H^{-1/2}(\Gamma)$ belongs to the \textit{approximation class $s$ with respect to the estimator} \eqref{eq:eta bem} if
\begin{align*}
\const{apx}(s)<\infty.
\end{align*}
By definition, $\const{apx}(s)<\infty$  implies that the error estimator $\eta_\coarse$
 decays at least with rate $\OO\big((\#\QQ_\coarse)^{-s}\big)$ on the optimal meshes $\QQ_\coarse$. 
The following main theorem from \cite[Theorem~3.4]{gp20} states that each possible rate $s>0$ is in fact realized by Algorithm~\ref{alg:abstract algorithm}.
We note that \cite{gp20} even allows for systems of PDEs (such as the linear elasticity problem) with possibly complex coefficients.
The theorem follows essentially from its abstract counterpart Theorem~\ref{thm:abstract main} by verifying the axioms of Section~\ref{sec:the axioms} for the perturbations 
\begin{align*}
\dist(\QQ_\coarse,\QQ_\fine):=\norm{\Phi_\fine-\Phi_\coarse}{H^{-1/2}(\Gamma)}
\quad\text{for all }  \QQ_\coarse\in\Q\\
 \text{and } \QQ_{\fine}\in\refine(\QQ_\coarse).
\end{align*}
Such an optimality result was first proved in \cite{fkmp13} for the Laplace operator $\mathscr{P}=-\Delta$ on a polyhedral domain $\Omega$.
As ansatz space, \cite{fkmp13} considered piecewise constants on shape-regular triangulations.
The work \cite{ffkmp14} in combination with \cite{affkmp17} extends the assertion to piecewise polynomials on shape-regular curvilinear triangulations  of some piecewise smooth boundary $\Gamma$.
Independently, \cite{gantumur13} proved the same result for globally smooth $\Gamma$ and a large class of  symmetric and elliptic boundary integral operators.

\begin{theorem}\label{thm:abstract bem}
Let $(\QQ_\k)_{\k\in\N_0}$ be the sequence of meshes generated by  Algorithm~\ref{alg:abstract algorithm} with the corresponding Galerkin approximations $\Phi_k\in\mathbb{S}_k$.
Then, there hold:
\begin{enumerate}[\rm (i)]
\item\label{item:qabstract reliable bem}
Suppose that \eqref{M:shape}--\eqref{M:locuni}, \eqref{R:childs}, and \eqref{S:unity bem} hold true.
Then, 
the residual error estimator satisfies reliability, i.e., there exists $\const{rel}>0$ such that for all $\QQ_\coarse\in\Q$, 
\begin{align}\label{eq:reliable bem}
 \norm{\phi-\Phi_\coarse}{H^{-1/2}(\Gamma)}\le \const{rel}\eta_\coarse.
\end{align}
\item 
Suppose that \eqref{M:shape}--\eqref{M:locuni}, \eqref{R:childs}, \eqref{S:nestedness} and \eqref{S:inverse bem} hold true.
Then, the axioms \eqref{item:stability}--\eqref{item:reduction} and convergence of the perturbations $\lim_{k\to\infty}\dist(\QQ_\k,\QQ_{\k+1})$ $=$ $0$ are satisfied. 
These are exactly the assumptions of Theorem~\ref{thm:abstract main}~\eqref{item:abstract main 1}, which implies convergence \eqref{eq:estimator convergence} of the estimator. 
\item\label{item:qabstract linear convergence bem}
Suppose that \eqref{M:shape}--\eqref{M:locuni}, \eqref{R:childs}, \eqref{S:nestedness} and \eqref{S:inverse bem} hold true.
Then, the axioms \eqref{item:stability}--\eqref{item:orthogonality} are satisfied.
These are exactly the assumptions of Theorem~\ref{thm:abstract main}~\eqref{item:abstract main 2}, which implies linear convergence \eqref{eq:linear convergence} of the estimator.
\item\label{item:qabstract optimal convergence bem}
Suppose that \eqref{M:shape}--\eqref{M:locuni},  \eqref{R:childs}--\eqref{R:overlay}, \eqref{S:nestedness}--\eqref{S:local}, \eqref{S:proj bem} and \eqref{S:inverse bem}--\eqref{S:stab bem} hold true.
Then, the axioms {\rm\eqref{item:stability}--\eqref{item:discrete reliability}} as well as \eqref{R:childs}--\eqref{R:overlay} are satisfied.
These are exactly the assumptions of Theorem~\ref{thm:abstract main}~\eqref{item:abstract main 3}, which implies optimal convergence \eqref{eq:optimal convergence} of the estimator.
\end{enumerate}
\noindent All involved constants $\const{rel},\const{lin},\ro{lin},\theta_{\rm opt}$, and $\const{opt}$  (of Theorem~\ref{thm:abstract main}) depend only on the assumptions made and the dimension $\dph$, the coefficients of the differential operator $\mathscr{P}$,  the boundary $\Gamma$,  and the parametrization constants $C_{\F_m}$ from Section~\ref{sec:parametrization_assumptions}, while $\const{lin},\ro{lin}$ depend additionally on $\theta$ and the sequence $(\Phi_\k)_{\k\in\N_0}$ (see also Remark~\ref{rem:symmetric orthogonality bem}), and $\const{opt}$ depends furthermore on $\const{min}$, and $s>0$.
The constant $c_{\rm opt}$ depends only on $\const{child}, \#\QQ_0$,  $s$, and if there exists $\k_0$ with $\eta_{\k_0}=0$, then also on $\k_0$ and $\eta_0$.
\end{theorem}
\begin{proof}
The proof is found  in \cite[Section~4]{gp20}.
Indeed, that work considers even more general meshes.
It assumes additional abstract mesh and refinement properties \eqref{eq:bounded patch} and \eqref{eq:reduction},  as in the FEM case, plus \eqref{M:patch center} and \eqref{M:poincare}, which are automatically satisfied in our setting.
We note that \eqref{eq:bounded patch} and \eqref{eq:reduction} follow along the lines of the proof of Theorem~\ref{thm:abstract}. 
Thus, in the remainder of the proof we only verify that \eqref{M:patch center} and \eqref{M:poincare} are satisfied.

Let $\QQ_\coarse\in\Q$. 
The flattening assumption from Section~\ref{sec:assumption-BEM}, \eqref{M:shape}--\eqref{M:locuni},  \eqref{R:childs}, and the assumption that each element can only be uniformly bisected in any direction imply that there exist only finitely many reference element-patches of elements. 
This shows that each element lies essentially in the center of its element-patch, i.e.,  that\footnote{Recall the definition of the diameter of an arbitrary non-empty set $S\subseteq\R^{\dph-1}$, ${\rm diam}(S):=\sup\set{|{\bf x}-{\bf y}|}{{\bf x},{\bf y}\in S}$
and the definition of the distance of two arbitrary non-empty sets $S_1,S_2\subseteq \R^{\dph-1}$, ${\rm dist}(S_1,S_2):=\inf\set{|{\bf x}-{\bf y}|}{{\bf x}\in S_1,{\bf y}\in S_2}$. We use the convention ${\rm dist}(Q,\emptyset):=\diam(\Gamma)$.} 
\begin{align}\label{M:patch center}
\diam(Q)\lesssim {\rm dist}(Q,\Gamma\setminus\pi_\coarse(Q))\quad\text{for all }Q\in\QQ_\coarse;
\end{align}
see \cite[Section~5.5.4]{gantner17} for details. 
Similarly, one sees that there exist only finitely many reference element-patches of points. 
This implies for all points ${\bf z}\in\Gamma$ and $v\in H^1(\Gamma)$ the following Poincar\'e-type inequality
 \begin{align}\label{M:poincare}
|v|_{H^{1/2}(\pi_\coarse(\{{\bf z}\}))}\lesssim\diam(\pi_\coarse(\{{\bf z}\}))^{1/2}| v|_{H^1(\pi_\coarse(\{\bf z\}))},
\end{align}
see \cite[Section~5.5.4]{gantner17}.\qed
\end{proof}

\begin{remark}\label{rem:weak efficiency}
In contrast to FEM, an efficiency result analogous to \eqref{eq:efficient} for the weighted-residual error estimator $\eta_\coarse$ is an open question. 
Indeed, \cite{affkp13} is the only available result in the literature.
However, \cite{affkp13} is restricted to the two dimensional case $\Omega\subset\R^2$ with piecewise constant ansatz functions.
Moreover, additional (regularity) assumptions on the right-hand side $f$ are required.
More precisely, it then holds that 
\begin{align*}
 \eta \lesssim \norm{h^{1/2} (\phi-\Phi)}{H^{-1/2}(\Gamma)} + {\rm osc}
\end{align*}
with some higher order oscillation term ${\rm osc}$. 

We also mention that \cite{affkmp17} proves a so-called weak efficiency of the weighted-residual estimator, which states that
\begin{align*}
\eta \lesssim \norm{h^{1/2} (\phi-\Phi)}{L^2(\Gamma)}
\end{align*}
provided that the sought solution has additional regularity $\phi\in L^2(\Gamma)$.

\end{remark}

\begin{remark}\label{rem:symmetric orthogonality bem}
As in Remark~\ref{rem:symmetric orthogonality}, we mention that only general quasi-orthogonality~\eqref{item:orthogonality} depends on the sequence $(\Phi_\k)_{\k\in\N_0}$.
Along the same lines, one sees that this dependence vanishes if the bilinear form $\dual{\mathscr{V}\,\cdot}{\cdot}$ is symmetric.
\end{remark}

\begin{remark}
Let $\Gamma_0\subsetneqq\Gamma$ be an open subset of $\Gamma=\partial\Omega$ and let $\mathscr{E}_0:L^2(\Gamma_0)\to L^2(\Gamma)$ denote the extension operator that extends a function defined on $\Gamma_0$ by zero to a function on $\Gamma$.
We define the space of restrictions $H^{1/2}(\Gamma_0):=\set{v|_{\Gamma_0}}{v\in H^{1/2}(\Gamma)}$ endowed with the quotient norm $v_0\mapsto\inf\set{\norm{v}{H^{1/2}(\Gamma)}}{v|_{\Gamma_0}=v_0}$ 
 and its dual space $\widetilde H^{-1/2}(\Gamma_0):= H^{1/2}(\Gamma_0)^*$.
According to \cite[Section~2.1]{affkmp17}, $\mathscr{E}_0$ can be extended to an isometric operator $\mathscr{E}_0:\widetilde H^{-1/2}(\Gamma_0)\to  H^{-1/2}(\Gamma)$.
Then, one can consider the integral equation
\begin{align}\label{eq:screen}
(\mathscr{V}\mathscr{E}_0\phi)|_{\Gamma_0}=f|_{\Gamma_0},
\end{align}
where $(\mathscr{V}\mathscr{E}_0(\cdot))|_{\Gamma_0}:\widetilde H^{-1/2}(\Gamma_0)\to H^{1/2}(\Gamma_0)$.
In the literature, such problems are known as \emph{screen problems}, see, e.g., \cite[Section~3.5.3]{ss11}.
 Theorem~\ref{thm:abstract bem}  holds analogously for the screen problem \eqref{eq:screen}.
Indeed, the works \cite{fkmp13,ffkmp14,affkmp17,gantumur13} cover this case as well.
To ease the presentation, we  only focus  on closed boundaries $\Gamma=\partial\Omega$.
\end{remark}

\begin{remark}
Let us additionally assume that 
$\mathbb{S}_\coarse$ contains all  constant functions, i.e.,  
\begin{align*}
({\bf x}\mapsto c)\in\mathbb{S}_\coarse
\quad\text{for all }c\in\R.
\end{align*}
Then, under the assumption that $\norm{h_{\k}}{L^\infty(\Omega)}\to0$ as $\k\to\infty$ (which can be easily guaranteed by  marking additional elements), one can show  that
$\mathbb{S}_\infty:=\overline{\bigcup_{\k\in\N_0}\mathbb{S}_\k}$ $=H^{-1/2}(\Gamma)$, see \cite[Remark~5.2.9]{gantner17}.
This observation allows to follow the ideas of~\cite{bhp17} and to prove Theorem~\ref{thm:abstract bem} even if the bilinear form $\dual{\mathscr{V}\,\cdot}{\cdot}$ is only elliptic up to some compact perturbation, provided that the continuous problem is well-posed. This includes, e.g., adaptive BEM for the Helmhotz equation, see~\cite[Section~6.9]{steinbach08}. For details, the reader is referred  to~\cite{bhp17,bbhp19}. 
\end{remark}

\newpage


\section{Adaptive IGAFEM in arbitrary dimension}\label{sec:adaptive igafem}

In this section, we consider two concrete realizations of the abstract adaptive FEM framework from Section~\ref{sec:afem}, namely hierarchical splines in Section~\ref{sec:H-igafem} and T-splines in Section~\ref{sec:T-igafem}.
To ease the presentation, we focus on single-patch Lipschitz domains $\Omega\subset \R^\dph$ as in Section~\ref{sec:parametrization_assumptions}.
For hierarchical splines, the generalization to multi-patch domains is notationally more involved but straightforward and will thus only be sketched in Section~\ref{sec:H-multipatches}, where we comment on the minor changes in the proof. 
Instead, for T-splines, since the direction of bisection on admissible T-meshes is periodically changed, one cannot avoid hanging nodes at the interfaces of multi-patch domains as in Section~\ref{sec:H-multipatches}.
This complicates the generalization for T-splines. Indeed, such a generalization is not available yet. 
For hierarchical splines, the theoretical findings are underlined by numerical experiments in Section~\ref{sec:numerical igafem}.
For numerical experiments with T-splines and the considered refinement strategy, we refer to \cite{hkmp17}. 
As in Section~\ref{sec:IGAFEM-intro}, in the following we restrict ourselves to the case $\dpa = \dph$, i.e., a $d$-dimensional domain $\Omega \subset \mathbb{R}^d$.

\subsection{Adaptive IGAFEM with hierarchical splines}\label{sec:H-igafem}
For the IGAFEM setting with hierarchical B-splines, we start with the single-patch domain.
Let $\mathbf{p}:=(p_1,\dots,p_d)$ be a vector of positive polynomial degrees and $\mathbf{\kv}^0$ be a  multivariate open knot vector on $\widehat\Omega=(0,1)^\dph$ with induced initial mesh $\widehat\QQ_0:=\widehat\QQ^0$.
We assume that $\hat{\mathbb{S}}_{{\bf p}} (\mathbf{\kv}^0)$ and $\hat{\mathbb{S}}_{{\bf p}_{\F}} (\mathbf \kv_{\F})$ with ${\bf p}_\F$ and ${\bf T}_\F$ from the parametrization  ${\bf F}:\widehat\Omega\to\Omega$ (see Section~\ref{sec:parametrization_assumptions}) are compatible to each other as in Section~\ref{sec:IGA-basics}.
Note that $\hat{\mathbb{S}}_{{\bf p}} (\mathbf{\kv}^0)=\hat{\mathbb{S}}_{{\bf p}}^{\rm H}(\hat\QQ_{0},{\bf \kv}^0)$, i.e., the starting space corresponds to standard B-splines.
We fix the admissibility parameter $\mu$ (see~\eqref{eq:sameshes}) as well as the basis and the kind of meshes that we want to consider, i.e., $\HH$-admissible or $\TT$-ad\-missible meshes, 
and abbreviate the set of all corresponding admissible meshes as $\widehat\Q$ and the corresponding refinement strategy as $\refine(\cdot,\cdot)$, see Section~\ref{sec:hierarchical refine}. 
For all $\widehat\QQ_\coarse\in\widehat\Q$, let $\widehat{\mathbb{S}}_\coarse:=\widehat{\mathbb{S}}_{\bf p}^{\rm H}(\hat\QQ_\coarse,\mathbf{\kv}^{0})\cap H_0^1(\widehat\Omega)$ be the associated ansatz space in the parametric domain, see Section~\ref{subsec:def hb}. 
As in Section~\ref{sec:model}, we define the corresponding quantities in the physical domain via the parametrization ${\bf F}:\widehat\Omega\to\Omega$, i.e., the meshes are given by
\begin{align*}
&\QQ_\coarse:=\set{{\bf F}(\widehat Q)}{\widehat Q\in\widehat\QQ_\coarse}\quad\text{for all }\widehat\QQ_\coarse\in\widehat\Q,
\\
&\Q:=\set{\QQ_\coarse}{\widehat\QQ_\coarse\in\widehat\Q},
\\
&\refine(\QQ_\coarse,\MM_\coarse):=\set{{\bf F}(\widehat Q)}{\widehat Q\in\refine(\widehat\QQ_\coarse,\widehat\MM_\coarse)}
\\
&\qquad\text{for all }\QQ_\coarse\in\Q, \MM_\coarse\subseteq\QQ_\coarse 
\\
&\qquad\text{with }\widehat\MM_\coarse:=\set{{\bf F}^{-1}(Q)}{Q\in\QQ_\coarse},
\end{align*}
and the discrete space associated to $\QQ_\coarse$ is defined as
\begin{align*}
&\mathbb{S}_\coarse:=\set{\widehat V\circ{\bf F}^{-1}}{\widehat V\in\widehat{\mathbb{S}}_\coarse}.
\end{align*}
In the following lemma, we give two bases in terms of (T)HB-splines for $\widehat{\mathbb{S}}_\coarse$.
The proof is given in \cite[Corollary~3.1]{ghp17} and relies on the fact that HB-splines restricted to any $(\dph-1)$-dimensional hyperface of the unit cube are again HB-splines. 
Clearly, these bases can be transferred to the physical domain via the parametrization $\F$.
We stress that the chosen basis is theoretically irrelevant for the realization of Algorithm~\ref{alg:abstract algorithm} (in particular for the solving step), see also Section~\ref{sec:numerical igafem} for a detailed discussion.

\begin{lemma}\label{lem:homogeneous H-basis}
Let $\hat\QQ_\coarse$ be a hierarchical (not necessarily admissible) mesh in the parametric domain.
Then, the hierarchical B-splines $\hat\HH_{\bf p}(\hat\hmesh,\mathbf{\kv}^{0})\cap H_0^1(\widehat\Omega)$ and the truncated hierarchical B-splines $\hat\TT_{\bf p}(\hat\hmesh,\mathbf{\kv}^{0})\cap H_0^1(\widehat\Omega)$ are both bases of $\widehat{\mathbb{S}}_\coarse$.
\end{lemma}

The given setting fits into the abstract framework of Section~\ref{sec:afem}, and in particular the axioms for abstract adaptive FEM are satisfied, as we will see in Theorem~\ref{thm:H-igafem} below.
For $\TT$-admissible meshes, the proof of this result is implicitly given in \cite{bg16,bg17}.
Independently, the theorem has been proved for $\HH$-admissible meshes of class $\mu=2$ in \cite{ghp17}, see also \cite[Section~4.4--4.5]{gantner17} for details.
We also stress that the last assertion \eqref{eq:admissible equivalent} of the theorem is new. 
It states that the approximation class with respect to the minimal total error and admissible hierarchical meshes defined in \eqref{eq:total class} is equivalent to the one with arbitrary hierarchical meshes. 
In particular, the approximation class does neither depend on whether $\HH$-admissible or $\TT$-admissible meshes are considered nor on  the admissibility parameter $\mu$. 
For its formulation, we set for $s>0$
\begin{align*}
\notag \const{apx}^{\rm tot,H}(s)
 &:=
\sup_{N\ge\#\QQ_0} \min_{\QQ_\coarse\in\Q^{\rm H}(N)}\big(N^s\inf_{V_\coarse\in\mathbb{S}_\coarse}\big(\norm{u-V_\coarse}{H^1(\Omega)}\\
&\hspace{40mm}+\osc_\coarse(V_\coarse)\big)\big).
\end{align*}
with
\begin{align*}
\Q^{\rm H}(N):=\set{\QQ_{\coarse}\text{ hier.\ mesh}}{\#\QQ_{\coarse}\le N}\supseteq\Q(N),
\end{align*}
where $\Q(N):=\set{\QQ_{\coarse}\in\Q}{\#\QQ_{\coarse}\le N}$.
Here, we say that $\QQ_{\coarse}$ is a hierarchical mesh if the corresponding set $\widehat \QQ_\coarse$ defined via $\F$ is a hierarchical mesh in the parametric domain obtained by arbitrary bisections of the initial mesh $\widehat\QQ_0$.
We also note that $\mathbb{S}_\coarse$ as well as $\osc_\coarse$ have actually only been defined for admissible meshes, but the definitions can be extended in an obvious way.

\begin{theorem}\label{thm:H-igafem}
Hierarchical splines on admissible meshes satisfy the mesh properties \eqref{M:shape}--\eqref{M:locuni},  the refinement properties \eqref{R:childs}--\eqref{R:overlay}, and the space properties \eqref{S:nestedness}--\eqref{S:proj} and \eqref{S:inverse}--\eqref{S:grad}.
The involved constants depend only on the dimension $\dph$, the parametrization constant $C_\F$ of Section~\ref{sec:parametrization_assumptions}, the degree $\mathbf{p}$, the initial knot vector ${\bf T}^0$, and the admissibility parameter $\mu$.
In particular, Theorem~\ref{thm:abstract} is applicable.
Together with Theorem~\ref{thm:abstract main}, this yields reliability~\eqref{eq:reliable}, efficiency~\eqref{eq:efficient}, and  linear convergence at optimal rate~\eqref{eq:linear convergence}--\eqref{eq:optimal convergence} of the residual error estimator~\eqref{eq:eta}, when the adaptive Algorithm~\ref{alg:abstract algorithm} is employed.
Moreover, for all $s>0$, there exists $C_{\rm opt}'>0$ depending only on $\const{clos}$ from \eqref{R:closure} (and thus in particular on $\mu$) and $s$ such that 
\begin{align}\label{eq:admissible equivalent}
\const{apx}^{\rm tot,H}(s)\le \const{apx}^{\rm tot}(s) \le \const{opt}' \const{apx}^{\rm tot,H}(s).
\end{align}
\end{theorem}

The proof of Theorem~\ref{thm:H-igafem} is split over Section~\ref{sec:H-fem mesh verification}--\ref{sec:equivalent classes}, and it relies mostly on the properties that we have already introduced in Section~\ref{subsec:hb}.
Sections~\ref{sec:H-fem mesh verification}--\ref{sec:H-fem space verification} respectively focus on the verification of the mesh properties \eqref{M:shape}--\eqref{M:locuni},  the refinement properties~\eqref{R:childs}--\eqref{R:overlay}, and the space properties \eqref{S:nestedness}--\eqref{S:proj} and \eqref{S:inverse}--\eqref{S:grad}, while Section~\ref{sec:equivalent classes} provides the equivalence~\eqref{eq:admissible equivalent} of the approximation classes.

\begin{remark}
While \eqref{eq:admissible equivalent} states that the approximation class is independent of the admissibility class of the mesh, it clearly depends on the degree ${\bf p}$ of the space $\mathbb{S}$. An interesting question arises: does the approximation class depend on the continuity of the splines in $\mathbb{S}$? In \cite{BN10} it is proved that the approximation classes are equivalent for $C^0$ finite elements and discontinuous Galerkin methods. We will show in a numerical test in Section~\ref{sec:numerical igafem} that the same does not hold true for high continuity splines if elements are refined by bisection: for certain functions, the order of the approximation class to which they belong decreases when we increase the continuity.
\end{remark}

\begin{remark}\label{rem:PCG-HB-FEM}
We also mention that the works \cite{bracco2019bpx,HMS19} have recently proposed local multilevel preconditioners for the stiffness matrix of symmetric problems which lead to uniformly bounded condition numbers on admissible hierarchical meshes, see also Section~\ref{sec:numerical igafem} for some details. 
An important consequence is that the corresponding PCG solver is uniformly contractive.
It has recently been proved in~\cite{ghps20} that such a contraction is the key to prove that an adaptive algorithm which steers mesh refinement and an inexact PCG solver leads to optimal convergence not only with respect to the number of elements but also with respect to the overall computational cost (i.e., computational time).
\end{remark}

\subsubsection{Mesh properties}
\label{sec:H-fem mesh verification}

Shape regularity \eqref{M:shape} is trivially satisfied in the parametric domain, since each refined element is uniformly bisected in each direction. 
Due to the regularity of the parametrization ${\bf F}$ of Section~\ref{sec:parametrization_assumptions}, the property transfers to the physical domain. 

Local quasi-uniformity \eqref{M:locuni} in the parametric domain  follows from Corollary~\ref{prop:lqiHB-adjacent}. 
Again, the regularity of the parametrization guarantees the property also in the physical domain.

\subsubsection{Refinement properties}
\label{sec:H-fem refinement verification}

The child estimate \eqref{R:childs} is trivially satisfied with $\const{child}=2^\dph$, since each refined element in the parametric domain is uniformly bisected in each direction. 
The closure estimate \eqref{R:closure} is just the assertion of Proposition~\ref{prop:hierarchical lincomp}.

According to \cite[Section~2.2]{bgmp16} (in the case of $\TT$-admissible meshes) and \cite[Section~3.1.4]{morgenstern17} (in the case of $\HH$-admissible),  the overlay 
\begin{align*}
\hat \hmesh_\fine:=&\set{\hat Q\in\hat\hmesh_\coarse}{\exists \hat Q'\in \hat\hmesh_\meshidx\text{ with }\hat Q\subseteq \hat Q'}\\
& \cup \set{\hat Q'\in\hat\hmesh_\meshidx}{\exists \hat Q\in \hat\hmesh_\coarse\text{ with }\hat Q'\subseteq \hat Q}
\end{align*}
of two admissible meshes $\hat\QQ_\coarse,\hat\QQ_\meshidx$ in the parametric domain is again admissible (of the same class).
Obviously, this property immediately transfers to the physical domain.
For $\HH$-admissible meshes of class $\mu=2$, this result is also found in \cite[Section~5.5]{ghp17}.
Clearly, the resulting mesh $\QQ_\fine$ in the physical domain satisfies the overlay property~\eqref{R:overlay}.

\subsubsection{Space properties}
\label{sec:H-fem space verification}

\paragraph{Nestedness}
Nestedness \eqref{S:nestedness} follows immediately from Proposition~\ref{prop:hb properties}. 
The inverse inequality \eqref{S:inverse} in the parametric domain follows easily via standard scaling arguments, since each hierarchical spline is a polynomial of fixed degree $\mathbf{p}$ on each mesh element, see Proposition~\ref{prop:hb properties}~(iv). 
Due to the regularity of the parametrization ${\bf F}$ of Section~\ref{sec:parametrization_assumptions}, this property transfers to the physical domain.

\smallskip\paragraph{Local domain of definition}
We start with an auxiliary result about element-patches.

\begin{lemma} \label{lemma:patches}
Let $k_1, k_2 \in \N_0$, and let $Q, Q' \in \QQ$ be such that $Q' \in \QQ \setminus \Pi^{k_1 + k_2}(Q)$. Then,  $\Pi^{k_1}(Q') \subseteq \QQ \setminus \Pi^{k_2}(Q)$.
\end{lemma}
\begin{proof}
By contradiction, let us assume that there exists $Q'' \in \Pi^{k_1}(Q')$ such that $Q'' \in \Pi^{k_2}(Q)$. By the definition of element-patches, it is clear that $Q' \in \Pi^{k_1}(Q'')$. Moreover, for any $k, k' \in \N_0$, it holds that $\Pi^{k + k'}(Q) = \Pi^{k}(\Pi^{k'}(Q))$, and therefore $Q' \in \Pi^{k_1+k_2}(Q)$, which contradicts the hypothesis. \hfill$\square$
\end{proof}

It suffices to prove \eqref{S:local} in the parametric domain. Let $\hat\QQ_\coarse \in \hat\Q$ and $\hat\QQ_\fine \in \refine(\hat\QQ_\coarse)$, and let us recall from \eqref{eq:qi-indices2}  the set ${\cal I}_{\hat{\QQ}_\fine}^{\ell,\rm new}$ of indices associated to new functions of level $\ell$ as well as the related sets ${\cal \hat R_\fine}, \hat \Omega_{\cal \hat R_\fine}$, and $\hat \Omega_{\cal \hat Q_\coarse}$ from \eqref{eq:Rfine}. 
We now introduce the subdomain formed by the support of their mother functions, namely 
\begin{align*}
\hat \Omega_{\fine}^{\rm new} := \bigcup \left\{ \supp(\mot(\hat{T}_{\fine,{\bf i},{\bf p}}^\ell)) : {\bf i} \in {\cal I}_{\hat{\QQ}_\fine}^{\ell,\rm new} , \right. \\
 \ell = 0, \ldots, N_\fine -1 \Big\},
\end{align*}
with $\hat \Omega_{\cal \hat R_\fine} \subseteq \hat \Omega^{\rm new}_{\fine}$.

We first show that there exists $q_1 \in \mathbb{N}$ depending on the degree $\mathbf{p}$ and the admissibility class $\mu$ such that $\hat \Omega_{\cal \hat R_\fine} \subseteq \hat \Omega^{\rm new}_{\fine} \subseteq \pi_\coarse^{q_1}(\widehat\QQ_\coarse\setminus\widehat\QQ_\fine)$: 
Note that any function $\hat{T}_{\fine,{\bf i},{\bf p}}^\ell$ as in the definition of $\hat \Omega_{\fine}^{\rm new}$ was activated during refinement, hence the support of its mother function must intersect an element in the refined region $\hat Q'\in\hat \QQ_\coarse \setminus \hat \QQ_\fine$ with level $\ell' < \ell$. Then, any element $\elemp'' \in \hat \QQ_\coarse \cap \hat \QQ_\fine$ with $\hat Q''\cap\supp(\mot(\hat{T}_{\fine,{\bf i},{\bf p}}^\ell))\neq\emptyset$ has obviously level $\ell'' > \ell$.
The fact that $\hat \QQ_\coarse \in \hat\Q$ and Proposition~\ref{prop:lqiHB-adjacent-support} yield that $\ell'' \le \ell' + \mu -1$. 
We conclude that all elements in $\hat \QQ_\coarse \cap \hat \QQ_\fine $ intersecting $\supp(\hat{T}_{\fine,{\bf i},{\bf p}}^\ell)$ have comparable level, which gives the desired result.

This implies that $\hat Q \subseteq \overline{\hat \Omega \setminus \hat \Omega_{\cal \hat R_\fine}} = \hat \Omega_{\cal \hat Q_\coarse}$ for any $\hat Q \in \hat\QQ_\coarse \setminus \Pi_\coarse^{q_1}(\widehat\QQ_\coarse\setminus\widehat\QQ_\fine)$, and by Corollary~\ref{corol:characterization} (which can be proved along the same lines for the current case of homogeneous Dirichlet boundary conditions, see also~\eqref{eq:J for H-igafem} below), \eqref{S:local} holds with $\q{loc} = q_1$ and $\q{proj} = 0$. In fact, we see from Lemma~\ref{lemma:patches} that property $\eqref{S:local}$ holds for any given $\q{proj} \in \N$, with $\q{loc} = \q{proj} + q_1$.

\smallskip\paragraph{Scott--Zhang-type operator}
Due to the regularity of the parametrization ${\bf F}$ of Section~\ref{sec:parametrization_assumptions}, it is sufficient to provide for all $\hat\QQ_\coarse\in\hat\Q$ an operator $\hat J_\coarse:H_0^1(\widehat\Omega)\to \hat{\mathbb{S}}_\coarse$ with  the properties \eqref{S:proj}, \eqref{S:app}--\eqref{S:grad} in the parametric domain. 
We define this operator similarly as $\projH{\mathbf{p}}$ of Section~\ref{sec:hierarchical interpolation}, but now have to take into account the homogeneous boundary conditions 
\begin{align}\label{eq:J for H-igafem}
\widehat J_\coarse:\,&H_0^1(\widehat\Omega)\to \widehat{\mathbb{S}}_\coarse,\quad
 \hat v\mapsto \sum_{\ell=0}^{N_\coarse-1} \sum_{{\bf i}\in\widetilde{\cal I}_\coarse^\ell}
\hat\lambda_{{\bf i},{\bf p}}^\ell(\hat v) \hat{T}_{\coarsecomma {\bf i},{\bf p}}^\ell,
\end{align}
where 
\begin{align*}
\widetilde{\cal I}_\coarse^\ell := \left\{ {\bf i}: 
\hat{B}_{{\bf i},{\bf p}}^\ell\in \hat {\cal B}^\ell \cap\hat{\cal H}_{\bf p}(\hat\QQ_\coarse,{\bf T}^0)\cap H_0^1(\widehat\Omega)\right\},
\end{align*}
$\hat{B}_{{\bf i},{\bf p}}^\ell$ is the mother B-spline of the THB-spline $\hat{T}_{\coarsecomma {\bf i},{\bf p}}^\ell$
(see \eqref{eq:mother}), and $\hat\lambda_{{\bf i},{\bf p}}^\ell$ is the corresponding dual functional from Section~\ref{sec:qi-2d}.

We have seen in Proposition~\ref{prop:Sext-star} that $\sextTHB{\hat{Q}}$ is connected and the number of contained elements is uniformly bounded.
In particular, this yields the existence of a uniform constant $q_2\in\N$ such that 
for any element $\hat Q\in\hat\QQ_\coarse$ of level $\ell$,
\begin{align}\label{eq:trunc in patch}
\sextTHB{\hat{Q}}\subseteq \pi_\coarse^{q_2}(\hat Q),
\end{align}
With Corollary~\ref{cor:local projection}, this immediately gives \eqref{S:proj}.

Moreover, the local $L^2$-stability of Proposition~\ref{prop:hqi} is also valid for $\hat J_\coarse$, see \cite{buffa2019}. 
Together with the local projection property  \eqref{S:proj} and  the inverse inequality \eqref{S:inverse}, the Poincar\'e (for elements away from the boundary) as well as the Friedrichs inequality (for elements close to the boundary) readily imply for all $\hat v\in H_0^1(\widehat \Omega)$ and $\hat Q\in\hat\QQ_\coarse$ that
\begin{align*}
\norm{(1-\hat J_\coarse)\hat v}{L^2(\hat Q)}&\lesssim |\hat Q|^{1/\dph}\,\norm{\hat v}{H^1(\sextTHB{\hat{Q}})}\\
\norm{\nabla \hat J_\coarse \hat v}{L^2(\hat Q)}&\lesssim \norm{\hat v}{H^1(\sextTHB{\hat{Q}})},
\end{align*}
see \cite{buffa2019} or \cite[Section~5.10]{ghp17} for details.
We conclude  \eqref{S:app}--\eqref{S:grad} with $q_{\rm sz}=q_2$.

\subsubsection{Equivalence of approximation classes}
\label{sec:equivalent classes}

The assertion~\eqref{eq:admissible equivalent} follows easily from the closure estimate \eqref{R:closure} and the fact that the minimal total error decreases when the underlying mesh is refined, which itself is an immediate consequence of the nestedness property in Proposition~\ref{prop:hb properties}(iii).
In the following, we elaborate ideas from \cite[Appendix~C]{dfgp19}, where a similar assertion on triangular meshes is proved:
Since $\Q(N)\subseteq\Q^{\rm H}(N)$ and hence $\const{apx}^{\rm tot,H}(s)\le \const{apx}^{\rm tot}(s)$, we only have to prove the second inequality in \eqref{eq:admissible equivalent}. 
For any given mesh $\QQ_\coarse$, we abbreviate the considered error quantity
\begin{align*}
\varrho := \inf_{V_\coarse\in\mathbb{S}_\coarse}\big(\norm{u-V_\coarse}{H^1(\Omega)}+\osc_\coarse(V_\coarse)\big).
\end{align*}
Clearly, it holds that
\begin{align*}
 \varrho_\fine\le \varrho \quad\text{for all }\QQ_{\fine}\in\refine(\QQ_\coarse).
\end{align*}
Let $N\in\N_0$ be arbitrary and $\QQ_\star\in\Q^{\rm H}(N)$ with $\varrho_\star = \min_{\QQ\in\Q^{\rm H}(N)} \varrho$.
The mesh $\QQ_\star$ results from the initial mesh $\QQ_0$ via bisecting a sequence of marked elements $(\MM_j)_{j=0}^{k-1}$. 
We define a sequence of associated admissible meshes via $\QQ_j:=\refine(\QQ_{j-1},\MM_{j-1}\cap\QQ_{j-1})$ for $j=1,\dots,k$, and we define $\overline\QQ_\star:=\QQ_k$, usually called the \emph{admissible closure} of $\QQ_\star$.
Indeed, $\overline\QQ_\star$ is finer than $\QQ_\star$ and hence $\#\QQ_\coarse\le \#\overline\QQ_\star$.
The closure estimate~\eqref{R:closure} shows that 
\begin{align*}
 \# \overline\QQ_\star -\#\QQ_0 \lesssim \sum_{j=0}^{k-1} \# (\MM_j\cap\QQ_j) 
 &\le\sum_{j=0}^{k-1} \# \MM_j 
 \\
 &\lesssim \# \QQ_\star - \#\QQ_0.
\end{align*}
For $\QQ_\star\neq\QQ_0$, this implies that $ \# \overline\QQ_\star\simeq \# \QQ_\star$ and thus $\#\overline\QQ_\star \le C N$ for some uniform constant $C>0$.
It holds that 
\begin{align*}
 \min_{\QQ\in\Q(CN)} \big((CN)^s \varrho\big) \le (CN)^s \overline\varrho_\star \le C^s N^s \varrho_\star \le C^s \const{apx}^{\rm tot,H}(s).
\end{align*}
Finally, elementary estimation yields for arbitrary $M\in\N_0$ and $N:=\lfloor M/C\rfloor$ that 
\begin{align*}
\min_{\QQ\in\Q(M)}(M^s\varrho) \lesssim \min_{\QQ\in\Q(CN)} \big((CN)^s \varrho\big) \lesssim \const{apx}^{\rm tot, H}(s). 
\end{align*}
Taking the supremum over all $M\in\N_0$, we conclude the proof of \eqref{eq:admissible equivalent}.

\subsubsection{Extension to multi-patch domains}\label{sec:H-multipatches}

Let now $\Omega$ be a multi-patch domain as in Section~\ref{sec:multi-patch}.
For each $m=1,\dots,M$, let $\mathbf{p}_m$ be a vector of positive polynomial degrees and $\mathbf{\kv}^0_m$ be a  multivariate open knot vector on $\widehat\Omega=(0,1)^\dph$ with induced initial mesh $\widehat\QQ_{0,m}:=\hat\QQ^0_m$. 
We assume that $\hat{\mathbb{S}}_{{\bf p}_m} (\mathbf{\kv}^0_m)$ and $\hat{\mathbb{S}}_{{\bf p}_{\F_m}} (\mathbf \kv_{\F_m})$ with ${\bf p}_{\F_m}$ and ${\bf T}_{\F_m}$ from the parametrization  ${\bf F}_m:\widehat\Omega\to\Omega_m$ (see Section~\ref{sec:multi-patch}) are compatible to each other as in Section~\ref{sec:IGA-basics}.
Note that $\hat{\mathbb{S}}_{{\bf p}_m} (\mathbf{\kv}^0_m)=\hat{\mathbb{S}}^{\rm H}_{{\bf p}_m}(\hat\QQ_{0,m},{\bf \kv}_m^0)$.
Moreover, we assume that $\mathbf{p}_m$ and $\mathbf{\kv}^0_m$ satisfy the compatibility condition \eqref{P:conforming-basis-space} of Section~\ref{sec:IGA-basics}.
We fix the admissibility parameter $\mu$ (see \eqref{eq:sameshes}) as well as the basis and the kind of meshes that we want to consider, i.e., ${\cal H}$-admissible or ${\cal T}$-admissible meshes, and abbreviate for each $m=1,\dots,M$ the set of all corresponding admissible meshes as $\widehat\Q_m$, see Section~\ref{sec:hierarchical refine}. 
Moreover, we abbreviate $\Q_m:=\set{\QQ_{\coarsecomma m}}{\hat \QQ_{\coarsecomma m}\in\hat\Q_m}$ with $\QQ_{\coarsecomma m}:=\set{\F_m(\hat Q)}{\hat Q\in\hat\QQ_{\coarsecomma m}}$.
We define the set of all admissible meshes $\Q$ as the set of all 
\begin{align*}
\QQ_\coarse=\bigcup_{m=1}^M \QQ_{\coarsecomma m} \text{ with } \QQ_{\coarsecomma m}\in\Q_m
\end{align*} 
 such that there are no hanging nodes on any interface $\Gamma_{m,m'}=\overline{ \Omega_m} \cap \overline {\Omega_{m'}}$ with $m\neq m'$, see also \eqref{P:conforming-mesh} of Section~\ref{sec:multi-patch}.

For $\QQ_\coarse\in\Q$, the associated ansatz space is defined as
\begin{align*}
\begin{split}
\mathbb{S}_\coarse:=\tilde{\mathbb{S}}_\coarse\cap H_0^1(\Omega),
\end{split}
\end{align*}
where the multi-patch space without boundary conditions is
\begin{align*}
\tilde{\mathbb{S}}_\coarse:=\big\{ V \in C^0(\Omega) : V|_{\Omega_m} \in {\mathbb{S}}^{\rm H}_{{\bf p}_m}(\hat\QQ_{\coarsecomma m},{\bf \kv}_m^0), \quad \\
\text{ for } m = 1, \ldots, M\big\},
\end{align*}
with the space of hierarchical splines on each patch
\begin{align*}
{\mathbb{S}}^{\rm H}_{{\bf p}_m}(\hat\QQ_{\coarsecomma m},{\bf \kv}_m^0):=\set{\hat V\circ\F_m^{-1}}{\hat V\in\hat{\mathbb{S}}^{\rm H}_{{\bf p}_m}(\hat\QQ_{\coarsecomma m},{\bf \kv}_m^0)}.
\end{align*}
To obtain bases of the space $\mathbb{S}_\coarse$, we first define  
\begin{align*}
\HH_{{\bf p}_m}(\hat\QQ_{\coarsecomma m},\mathbf{\kv}^{0}_m)&:=\set{\hat\beta\circ\F_m^{-1}}{\hat\beta\in \hat\HH_{{\bf p}_m}(\hat\QQ_{\coarsecomma m},\mathbf{\kv}^{0}_m)},\\
\TT_{{\bf p}_m}(\hat\QQ_{\coarsecomma m},\mathbf{\kv}^{0}_m)&:=\set{\hat\tau\circ\F_m^{-1}}{\hat\tau\in \hat\TT_{{\bf p}_m}(\hat\QQ_{\coarsecomma m},\mathbf{\kv}^{0}_m)}.
\end{align*}
The reference \cite[Proposition~3.1]{ghp17} shows that HB-splines restricted to any $(\dph-1)$-dimensional hyperface of the unit cube are again HB-splines. 
Hence, the assumption \eqref{P:conforming-basis-space} is also satisfied if the sets  ${\cal B}_{{\bf p}_m}({\bf \kv}_{m})$ and ${\cal B}_{{\bf p}_{m'}}({\bf \kv}_{m'})$ are replaced by the sets $\HH_{{\bf p}_m}(\hat\QQ_{\coarsecomma m},\mathbf{\kv}^{0}_m)$ and $\HH_{{\bf p}_{m'}}(\hat\QQ_{\coarsecomma m'},\mathbf{\kv}^{0}_{m'})$. 
With a similar proof, one can also show the assertion of \cite[Proposition~3.1]{ghp17} for THB-splines. 
Thus, \eqref{P:conforming-basis-space} is also valid for the sets $\TT_{{\bf p}_m}(\hat\QQ_{\coarsecomma m},\mathbf{\kv}^{0}_m)$ and $\TT_{{\bf p}_{m'}}(\hat\QQ_{\coarsecomma m'},\mathbf{\kv}^{0}_{m'})$. 
This allows to construct a basis of $\tilde{\mathbb{S}}_\coarse$ similarly as in Section~\ref{sec:IGA-basics} by gluing (T)HB-splines together at interfaces.
According to  \cite[Proposition~3.1]{ghp17}, discarding all resulting functions that are non-zero on the boundary $\partial\Omega$ gives a basis of $\mathbb{S}_\coarse$.

To obtain admissible meshes starting from the initial one, we adapt the single-patch refinement strategies from Section~\ref{sec:hierarchical refine}. 
For arbitrary $\QQ_\coarse\in\Q$ and $Q\in\QQ_{\coarsecomma m}$ with corresponding element $\hat Q:=\F_m^{-1}(Q)$ in the parametric domain, let $\NN_{\coarsecomma m}(\hat Q)\subseteq\hat\QQ_{\coarsecomma m}$ either denote the corresponding $\HH$-neighborhood in the case of $\HH$-admissible meshes or the $\TT$-neighborhood in the case of $\TT$-admissible meshes, see Section~\ref{sec:hierarchical refine}.
We define the \emph{neighbors} of $Q$ as 
\begin{align*}
&\NN_\coarse(Q):=\set{Q'\in\QQ_{\coarsecomma m}}{\hat Q' \in \NN_{\coarsecomma m}(\hat Q)}\\
&\qquad\cup \bigcup_{m'\neq m} \set{Q'\in\QQ_{\coarsecomma m'}}{{\rm dim}(\overline Q\cap \overline Q')=d-1},
\end{align*}
i.e., apart from the standard neighbors within the patch, we add the adjacent elements from other patches to avoid hanging nodes.
Then, it is easy to see that Algorithm~\ref{alg:multi-patch fem} returns an admissible mesh.
Indeed, one can show that the set of all possible refinements $\refine(\QQ_0)$ even coincides with $\Q$, see \cite[Proposition~5.4.3]{gantner17} in the case of $\HH$-admissible meshes  of class $\mu=2$.

\begin{algorithm}[!ht]\label{alg:multi refinement_fem}
\caption{\texttt{refine} (Multi-patch refinement)}
\label{alg:multi-patch fem}
\begin{algorithmic}
\Require admissible mesh $\QQ_\coarse$ and marked elements $\MM\subseteq\QQ_\coarse$
\Repeat
\State set $\displaystyle \mathcal{U} = \bigcup_{Q \in \MM} \NN_\coarse(Q)\setminus \MM$
\State set $\MM = \MM\cup \mathcal{U}$
\Until {$\mathcal{U} = \emptyset$}
\State update $\QQ_\coarse$ by replacing the elements in $\MM$ by their children\\
\Ensure refined admissible mesh $\QQ_\coarse$
\end{algorithmic}
\end{algorithm}

We stress that Theorem~\ref{thm:H-igafem} holds accordingly for the given setting.
Here, the mesh properties \eqref{M:shape}--\eqref{M:locuni} and the child estimate \eqref{R:childs} are trivially satisfied.
The closure estimate \eqref{R:closure} can be proved similarly as in the single-patch case, see \cite[Section~5.5.7]{gantner17} in the case of $\HH$-admissible meshes  of class $\mu=2$.
The overlay in \eqref{R:overlay} can be built patch-wise as in Section~\ref{sec:H-fem refinement verification}. 
Nestedness~\eqref{S:nestedness} follows from Proposition~\ref{prop:hb properties}. 
The local domain of definition property~\eqref{S:local} and the inverse inequality~\eqref{S:inverse} follow similarly as in the single-patch case.
It remains to check the Scott--Zhang type properties \eqref{S:proj} and \eqref{S:app}--\eqref{S:grad}.
To construct a suitable operator $J_\coarse:H_0^1(\Omega)\to \mathbb{S}_\coarse$, one can proceed similarly as in \eqref{eq:J for H-igafem} by additionally gluing together THB-splines at interfaces and considering the average of the dual functions at interfaces. 
Then, the required properties can be seen as in Section~\ref{sec:H-fem space verification}.
Details are left to the reader.

\begin{remark}\label{rem:second multi patch}
In principle, the requirement that there are no hanging nodes on the interface can be removed. In fact, starting from a level zero mesh without hanging nodes in the multi-patch domain, it is possible to define spline functions with $C^0$ continuity with a support that may intersect different patches as in Section~\ref{sec:IGA-basics}. Then, we can define the multi-patch spaces of next levels by uniform refinement of the whole multi-patch domain. These spaces satisfy the conditions given in the abstract setting of \cite{giannelli2014}, and the recursive algorithm for the definition of hierarchical splines can be applied to construct hierarchical multi-patch basis functions, replacing the sequence of B-spline spaces by a sequence of  multi-patch spaces with conforming meshes, see \cite{Buchegger2016} and \cite[Section~3.4]{garau2018}. Most of the definitions of Section~\ref{subsec:hb}, and in particular the neighborhoods, have a seamless extension to this setting. Although quasi-interpolants for the uniform multi-patch case have been introduced in \cite{BVSB15}, the complete adaptive theory in the non-conforming case has not been analyzed yet, and  is beyond the scope of this work.
\end{remark}

\subsubsection{Numerical experiments}\label{sec:numerical igafem}

We now apply the adaptive IGAFEM with hierarchical splines, analyzed in the previous sections, to the Poisson problem. 
In particular, in \eqref{eq:problem}-\eqref{eq:defP}, the matrix $\AA$ is the identity matrix, and $\bb$ and $c$ are zero. 
Although not directly covered by our analysis,  we  also consider non-homogeneous Dirichlet--Neumann boundary conditions for some cases. In all three numerical experiments, we set the degrees $p_1 = \ldots = p_\dph =: p$. The continuity within a patch is taken to be $C^{p-1}$ across elements, also for the elements of the coarsest mesh. All the numerical tests of this section are run with THB-splines but, as we mentioned above, the computed solution of the Galerkin problem is the same independently of the basis.

\smallskip
\paragraph{Comments on the use of HB- and THB-splines}  
In spite of having the same solution for HB-splines and THB-splines, the choice of the basis will affect the sparsity pattern and the condition number of the matrix appearing in the linear system, which can also affect the performance of the method. 

In particular, the reduced support of THB-splines always gives a lower number of nonzero entries in the matrix when compared to HB-splines, but to control this number it is important to control the interaction between coarse and fine functions, for which it is necessary to use suitable admissible meshes. We recall from Proposition~\ref{eq:former admissible} that the number of HB-splines (resp. THB-splines) with support on some fixed element of an $\HH$-admissible (resp. $\TT$-admissible) is uniformly bounded, while in general this is not the case for HB-splines on $\TT$-admissible meshes. 
The examples in \cite{giannelli2016,bgv18} show that the gain in the number of nonzero entries when using THB-splines instead of HB-splines ranges between 10\% and 50\%, with the biggest gains in ${\cal T}$-admissible meshes or non-admissible ones, and the smallest ones in ${\cal H}$-admissible meshes. 
We remark that these numbers depend on the degree $p$ and the admissibility class $\mu$, but also on the kind of refinement (edge refinement, corner refinement...) required for a good approximation of the solution, see the aforementioned references for more details. We also note that, for non-admissible meshes, the number of nonzero entries in the matrix can behave as bad as $\mathcal{O}(N_{\rm dof}^2)$, with $N_{\rm dof}$ the number of degrees of freedom. For instance, this is the case for HB-splines in the meshes of Figure~\ref{fig:T-admissible,notH}. The efficient assembly of the matrix for hierarchical splines is also an important issue, and an active topic of research, as the tensor-product techniques cannot be trivially extended to the hierarchical case. In this sense, the recent work \cite{PJM21} proposes a method based on interpolation and the use of look-up tables, which allows to reduce the complexity compared to Gaussian quadrature, for bivariate HB-splines on ${\cal H}$-admissible meshes of class $\mu=2$.

Regarding the condition number, all the numerical tests in \cite{giannelli2016,bgv18} show that in any hierarchical mesh the condition number of the mass matrix is always equal or smaller for THB-splines than for HB-splines. Concerning the stiffness matrix, although in most cases the condition number is also lower for THB-splines, the property is not valid in general, and some counterexamples have been shown in the same references. The numerical results from those papers do not show a clear behavior on how the basis and the admissibility class influence the condition number of the stiffness matrix, and as for the nonzero entries, the numbers seem to strongly depend on the kind of refinement (corner refinement, edge refinement...)

Related to the condition number, multigrid solvers and preconditioners for hierarchical splines have been introduced in \cite{Hofreither2016b}, where the subspace of each level of the preconditioner coincides with the one in the HB-splines construction algorithm, i.e., it is given by ${\rm span} \hat {\cal H}^\ell$. Local variants of the preconditioner with subspaces for each level formed by functions with support in $\hat\Omega^\ell$ or its vicinity have been analyzed in \cite{bracco2019bpx,HMS19}. In these works, it has been proved that the condition number is uniformly bounded with respect to the number of levels.
The numerical results of those papers show a better behavior of the preconditioners for THB-splines than for HB-splines. Moreover, the theoretical analysis also shows that it is necessary to use HB-splines (resp. THB-splines) on ${\cal H}$-admissible (resp. ${\cal T}$-admissible) meshes to obtain a bounded condition number independent of the number of levels, see \cite{bracco2019bpx} for details.

Nevertheless, we cannot give a clear answer about which basis is better to use. 
While THB-splines improve the sparsity pattern of the matrix, and in most cases behave better with respect to the condition number, the rectangular support of HB-splines may be easier to implement.
Still, from the comments above there is one important suggestion we can make: the admissibility type (${\cal H}$- or ${\cal T}$-admissible) should be in accordance with the chosen basis.
Recall that any $\HH$-admissible mesh is also $\TT$-admissible by definition, but not vice versa.

\smallskip
\paragraph{Edge singularities on square}
In the first numerical test, we choose a problem that was already considered in \cite{bgarau16_2,bracco2019}. The domain is given by the unit square $\Omega = (0,1)^2$, in such a way that the parametrization ${\bf F}$ is the identity. We impose homogeneous Dirichlet conditions on the boundary $\partial \Omega$, while the source function $f$ is chosen such that the exact solution is given by
\[
u(x,y) = x^{2.3} (1-x) y^{2.9} (1-y),
\]
which is singular at the edges $\{0\}\times (0,1)$ and $(0,1)\times\{0\}$. In fact, it can be shown that $u \in H^{\beta-\epsilon}(\Omega)$, with $\beta = 2.3 + 1/2 = 1 + 9/5$, for every $\epsilon > 0$. Hence, the expected convergence rate for uniform refinement is ${\cal O}(h^{9/5}) = {\cal O}(N_{\rm dof}^{-9/10})$ with respect to the mesh size $h$ and to the number of degrees of freedom $N_{\rm dof}$, respectively.

For the simulation, we consider spaces of hierarchical B-splines with degree $p \in \{2,3,4,5\}$. The initial mesh $\QQ_0$ consists of $4 \times 4$ elements, 
and Algorithm~\ref{alg:abstract algorithm} is run using the residual \textsl{a posteriori} estimator \eqref{eq:eta}. For marking we use D\"orfler's strategy \eqref{eq:Doerfler} with parameter $\theta = 0.25$ and the constant $\const{\min} = 1$. For refinement, we use Algorithm~\ref{alg:trefine} for ${\cal T}$-admissible meshes, with a value of the admissibility class $\mu = 2$.

Some meshes for the four different degrees at iteration $k = 15$ are displayed in Figure~\ref{fig:edge_singularity_meshes}. It is evident that the adaptive algorithm satisfactorily refines near the edges, specially for high degree.
\begin{figure}
\centering
\subfigure[Degree 2]{
\includegraphics[width=0.22\textwidth,trim=3cm 1cm 2cm 0cm, clip]{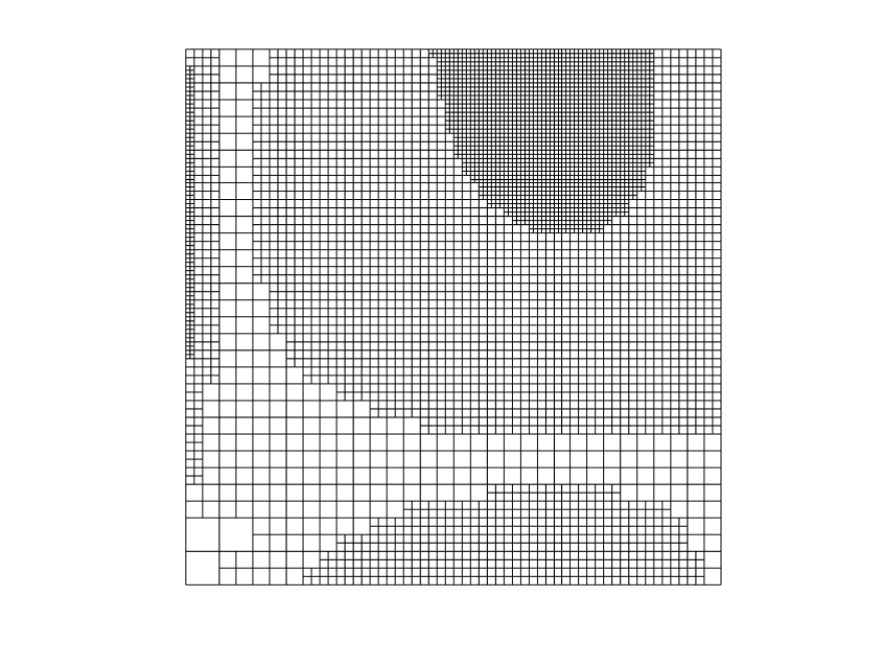}
}
\subfigure[Degree 3]{
\includegraphics[width=0.22\textwidth,trim=3cm 1cm 2cm 0cm, clip]{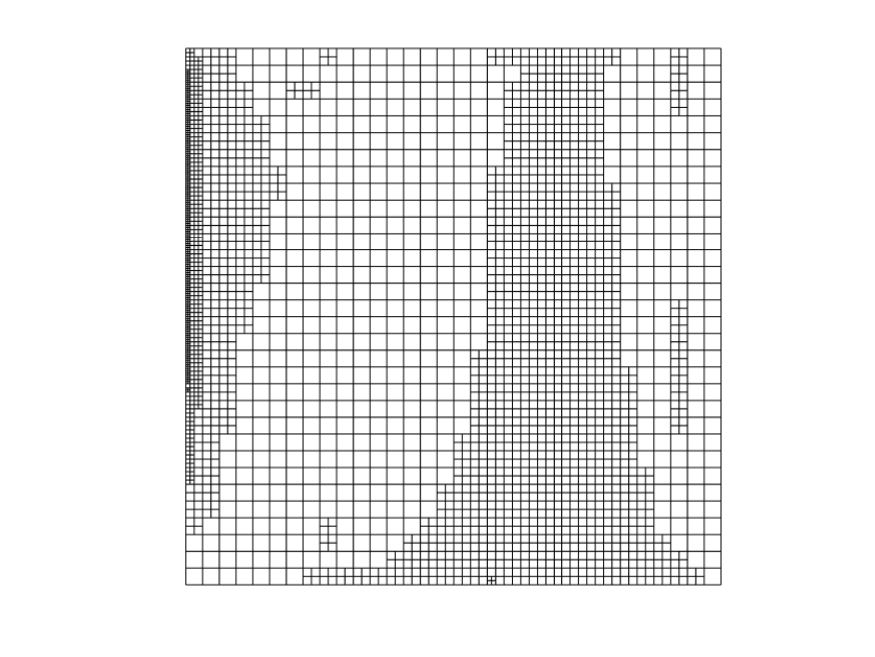}
}
\subfigure[Degree 4]{
\includegraphics[width=0.22\textwidth,trim=3cm 1cm 2cm 0cm, clip]{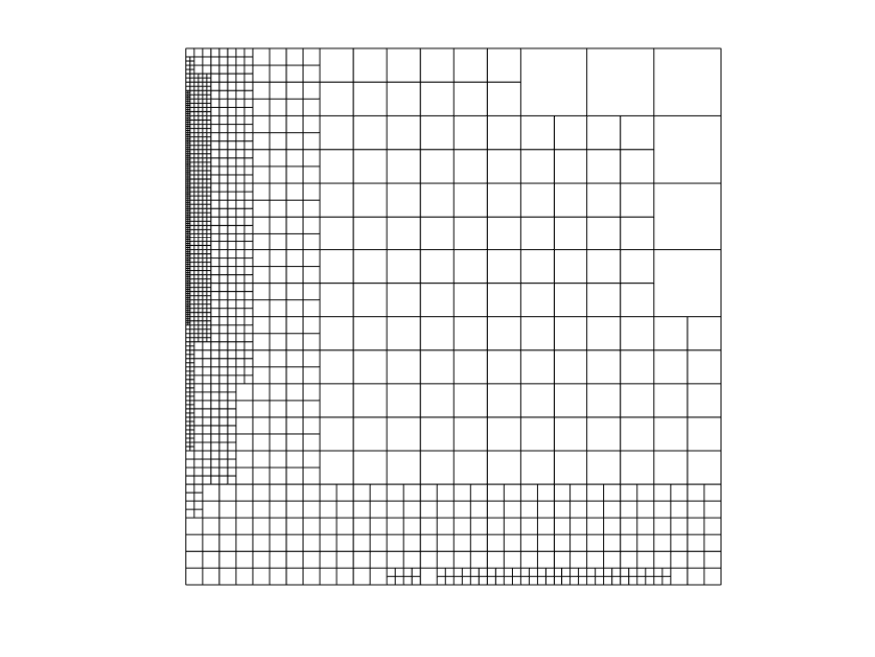}
}
\subfigure[Degree 5]{
\includegraphics[width=0.22\textwidth,trim=3cm 1cm 2cm 0cm, clip]{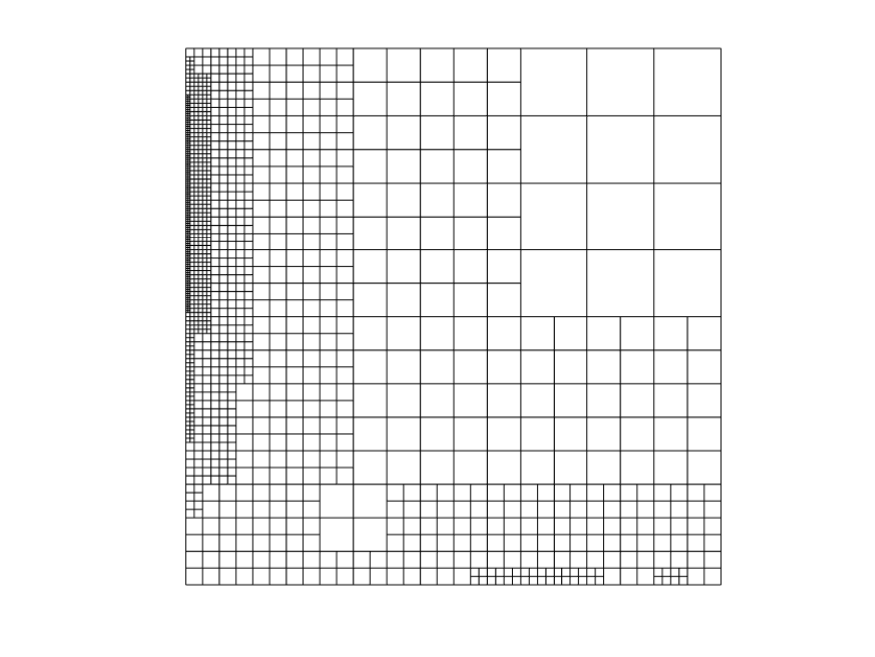}
}
\caption{Test with edge singularities: meshes obtained after 15 iterations of the adaptive algorithm for ${\cal T}$-admissible meshes and $\mu = 2$.}
\label{fig:edge_singularity_meshes}
\end{figure}

In Figure~\ref{fig:square_edge_singularity}, we show the behavior of the error in the energy norm and the estimator with respect to the number of degrees of freedom, both for the adaptive method described above and for uniform refinement. It is clearly seen that adaptivity drastically reduces the number of degrees of freedom required to achieve the same numerical error. Moreover, the error and the estimator curves always converge with the same order, as expected from the results of reliability and efficiency of the estimator.
\begin{figure}
\centering
\subfigure[Error and estimator for $p=2$]{
\includegraphics[width=0.35\textwidth]{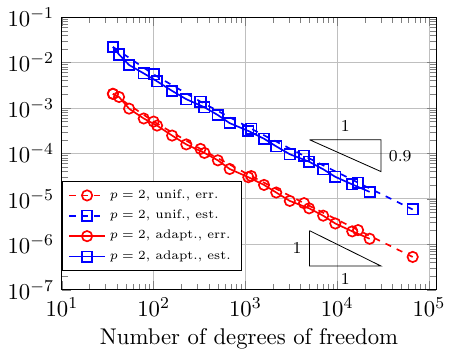}
}
\subfigure[Error and estimator for $p=3$]{
\includegraphics[width=0.35\textwidth]{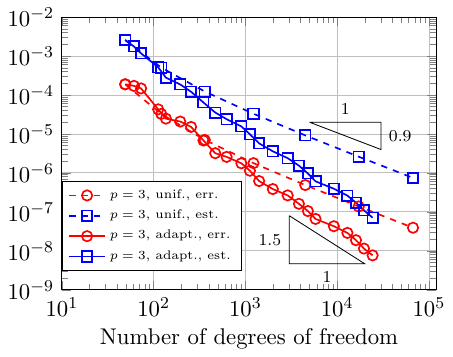}
}
\subfigure[Error and estimator for $p=4$]{
\includegraphics[width=0.35\textwidth]{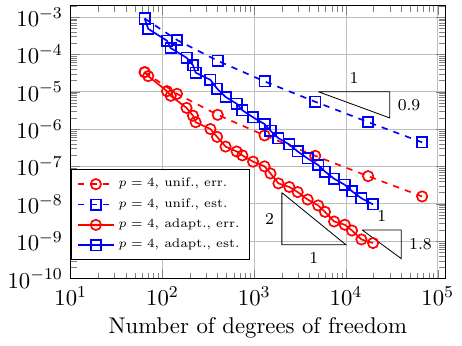}
}
\subfigure[Error and estimator for $p=5$]{
\includegraphics[width=0.35\textwidth]{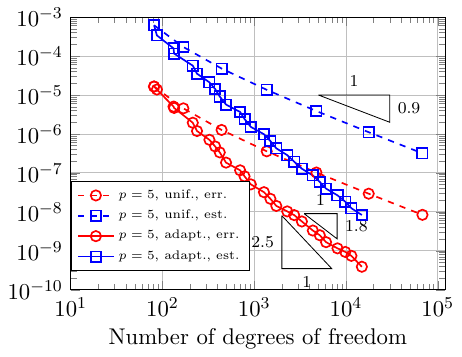}
}
\caption{Test with edge singularities: energy error $|u -U_k|_{H^1(\Omega)}$ and residual estimator for degree $p \in \{2, 3, 4, 5\}$. Comparison of uniform and adaptive refinement.}
\label{fig:square_edge_singularity}
\end{figure}

From Figure~\ref{fig:square_edge_singularity}, it can also be seen that the convergence is not equal to $p/2$ for high $p$. In fact, Figure~\ref{fig:edge_singularity_p2-5} shows the same convergence rate for degrees $p=4$ and $p=5$, which seems equal to 1.8. This behavior was analyzed with heuristic arguments in \cite[Sect.~4.6.2]{gantner17}, noting that to obtain the optimal convergence rate $s_{\rm opt} = p/2$ in the presence of edge singularities, it is necessary to consider anisotropic elements in a mesh graded towards the edges, while the bisection refinement that we consider attains at most a convergence rate equal to
\[
s = \min (2s_{\rm unif}, s_{\rm opt}),
\]
where $s_{\rm unif}$ is the convergence rate in case of uniform refinement. In this particular test, its value is $s_{\rm unif} = 0.9$, and in fact the convergence rate that we observe in Figure~\ref{fig:edge_singularity_p2-5} is twice this value.
\begin{figure}
\centering
\includegraphics[width=0.35\textwidth]{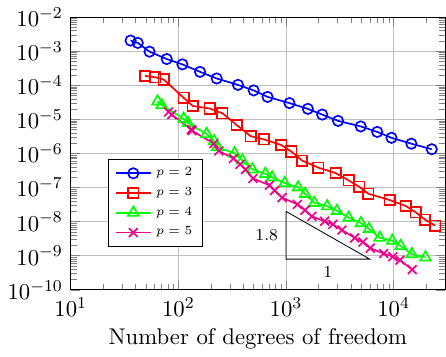}
\caption{Test with edge singularities: energy error $|u -U_k|_{H^1(\Omega)}$ for degree $p$ from 2 to 5. For high degree, the optimal convergence rate is not reached.} \label{fig:edge_singularity_p2-5}
\end{figure}

Finally, we compare the behavior of the refinement Algorithms~\ref{alg:hrefine} and~\ref{alg:trefine}, that is, for ${\cal H}$-admissible and ${\cal T}$-admissible meshes, respectively. We set the degree $p = 4$, and the admissibility class $\mu=2$. For comparison, we also include the results of refinement without ensuring admissibility, that is, refining only elements marked by the marking strategy without any addition, which we denote by $\mu = \infty$. The results presented in Figure~\ref{fig:edge_singularity_admissibility} show that the convergence order is the same in the three cases, both for the error and the estimator, as predicted by theory (for $\mu<\infty$). Moreover, in this particular case, there is small difference between the error obtained with the different admissibility types, although the use of non-admissible meshes gives slightly better results in terms of degrees of freedom.
\begin{figure}
\centering
\includegraphics[width=0.35\textwidth]{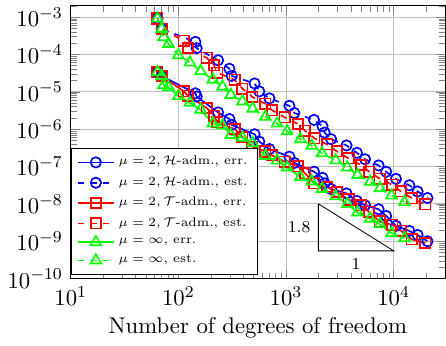}
\caption{Test with edge singularities: residual estimator and energy error $|u -U_k|_{H^1(\Omega)}$ for degree $p = 4$. Results for ${\cal H}$-admissible and ${\cal T}$-admissible meshes of class $\mu=2$, and for non-admissible meshes.} \label{fig:edge_singularity_admissibility}
\end{figure}

\smallskip
\paragraph{Corner singularity on curved L-shaped domain}
The second numerical test was presented in \cite{bgv18}. We consider the curved L-shaped domain shown in Figure~\ref{fig:curvedL-domain}, which is an affine transformation of the benchmark in \cite{BMAX}. The solution is given by
\[
u(x,y) = r^{2/3} \sin(2 \varphi / 3), 
\]
with polar coordinates $(x,y)=(r \cos(\varphi),r \sin(\varphi))$, by setting $f = 0$ and imposing non-homogeneous Dirichlet boundary conditions on $\partial \Omega$. We note that the domain is formed by three quadratic NURBS patches, and for the discretization we follow the method explained in Remark~\ref{rem:second multi patch} so that the meshes may be non-conforming on the interfaces.

We consider discrete spaces of degree $p \in \{2, 3, 4, 5\}$, with $C^{p-1}$ continuity inside each patch, and $C^0$ continuity across the interfaces. For the D\"orfler marking \eqref{eq:Doerfler}, we choose the parameter $\theta = 0.9$ and $\const{\min}$ = 1. To understand the role of the admissibility class, we consider both ${\cal H}$-admissible meshes as in Algorithm~\ref{alg:hrefine}, and ${\cal T}$-admissible meshes as in Algorithm~\ref{alg:trefine}, with the value of the admissibility class $\mu$ ranging from 2 to 4. For comparison, we also include results for non-admissible meshes, which we denote as above by $\mu = \infty$. The algorithm is run until we reach a maximum of 13 levels.
\begin{figure}
\centering
\includegraphics[width=0.4\textwidth]{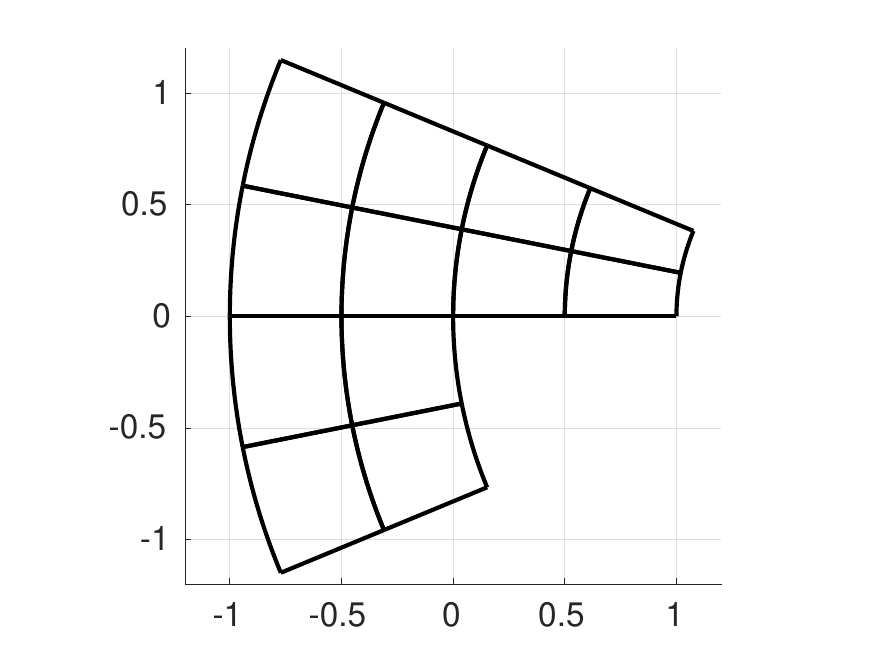}
\caption{Curved L-shaped domain: domain and initial mesh.
}
\label{fig:curvedL-domain}
\end{figure}

In Figure~\ref{fig:curvedL_adap_unif}, we show the value of the energy error and the residual estimator with respect to the number of degrees of freedom, considering degree $p=2$ with uniform refinement and with adaptive refinement for ${\cal T}$-admissible meshes of class $\mu=2$. As in the previous test, the error and the estimator converge with the same rate, which in the case of uniform refinement is equal to $1/3$, while for adaptive refinement the optimal rate of $1$ is reached.
\begin{figure}
\centering
\includegraphics[width=0.35\textwidth]{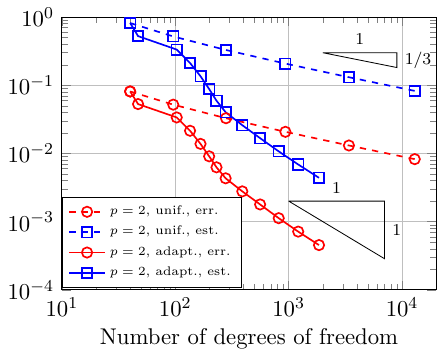}
\caption{Curved L-shaped domain: energy error and residual estimator for degree $p=2$, for uniform refinement and adaptive refinement on ${\cal T}$-admissible meshes with $\mu=2$.}\label{fig:curvedL_adap_unif}
\end{figure}

In Figure~\ref{fig:curvedL_admissibility}, we show the results of convergence of the energy error with respect to the number of degrees of freedom for the different degrees and admissibility types considered. In all the tests, non-admissible meshes show a better ratio between the error and the number of degrees of freedom than any other choice. Moreover, ${\cal T}$-admissible meshes and higher values of $\mu$ require less degrees of freedom than ${\cal H}$-admissible ones and lower values of $\mu$, respectively, to attain the same error. We note, however, that except for degree $p=2$ the asymptotic regime has not been reached. In fact, from Figures~\ref{fig:curvedL-deg2-adm} and~\ref{fig:curvedL-deg3-adm} it seems that, in the asymptotic regime, the error will be very similar for all the admissibility classes. However, ${\cal H}$-admissible meshes with low values of $\mu$ need more iterations to reach the asymptotic behavior. See also the results for the L-shaped domain in \cite{ghp17} and \cite[Section~4.6.3]{gantner17}.
\begin{figure}
\centering
\subfigure[Error for $p=2$]{
\includegraphics[width=0.35\textwidth]{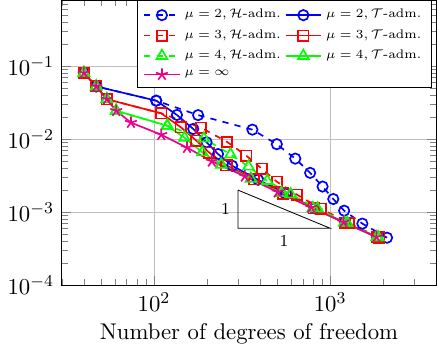}\label{fig:curvedL-deg2-adm}
}
\subfigure[Error for $p=3$]{
\includegraphics[width=0.35\textwidth]{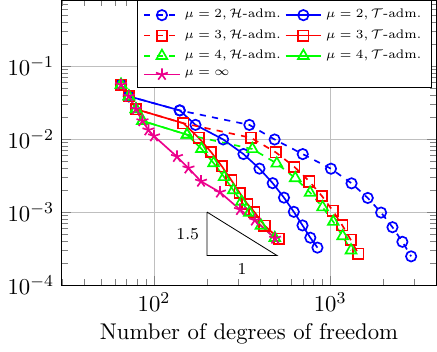} \label{fig:curvedL-deg3-adm}
}
\subfigure[Error for $p=4$]{
\includegraphics[width=0.35\textwidth]{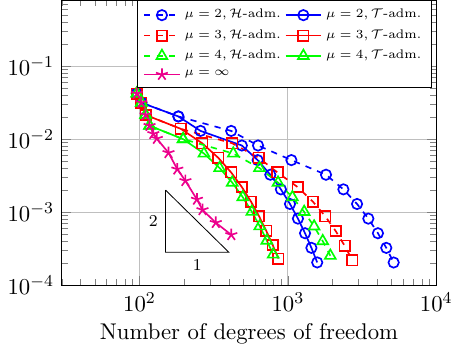}
}
\subfigure[Error for $p=5$]{
\includegraphics[width=0.35\textwidth]{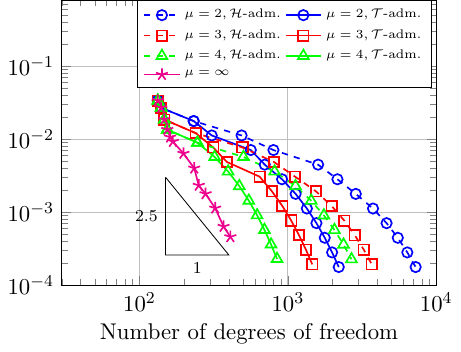}
}
\caption{Curved L-shaped domain: energy error $|u -U_k|_{H^1(\Omega)}$ for degree $p$ from 2 to 5, and for different values of the admissibility class $\mu$, both for ${\cal H}$-admissible and ${\cal T}$-admissible meshes.}
\label{fig:curvedL_admissibility}
\end{figure}

This behavior can be better understood with the help of the plots in Figure~\ref{fig:curvedL-meshes}, where we show the mesh after 8 refinement steps for degree $p=4$ and for different types of admissibility. While the estimator satisfactorily marks elements to refine the mesh towards the corner, to maintain the admissibility of the mesh, the refinement algorithm forces to refine some elements away from it. This behavior is more significant for ${\cal H}$-admissible meshes than for ${\cal T}$-admissible meshes, and also for lower values of the admissibility class $\mu$ than for higher ones.
\begin{figure}
\centering
\subfigure[$\mu=2$, ${\cal H}$-admissible]{\includegraphics[trim=35mm 10mm 35mm 10mm,clip,width=0.23\textwidth]{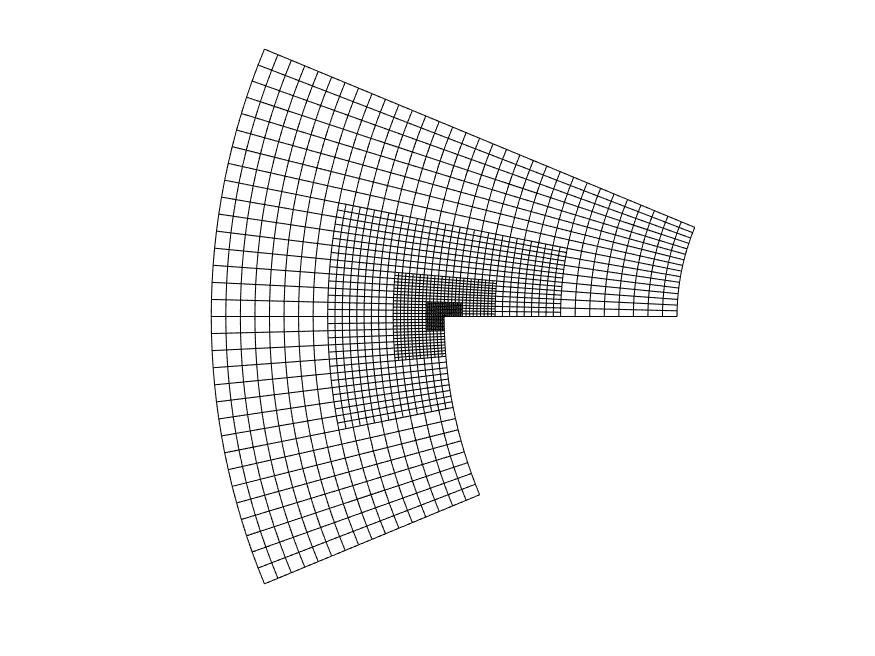}}
\subfigure[$\mu=2$, ${\cal T}$-admissible]{\includegraphics[trim=35mm 10mm 35mm 10mm,clip,width=0.23\textwidth]{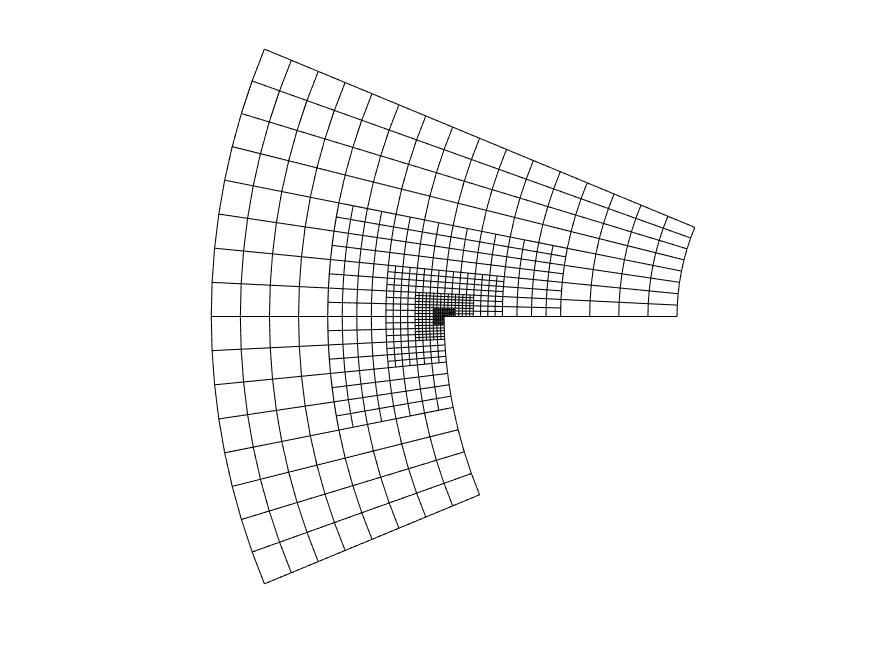}}
\subfigure[$\mu=4$, ${\cal H}$-admissible]{\includegraphics[trim=35mm 10mm 35mm 10mm,clip,width=0.23\textwidth]{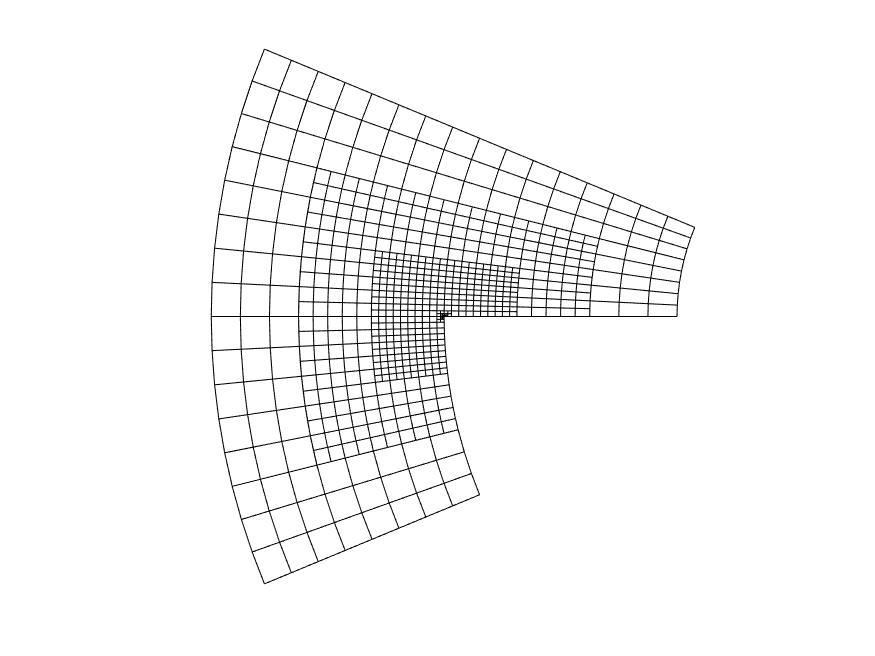}}
\subfigure[$\mu=4$, ${\cal T}$-admissible]{\includegraphics[trim=35mm 10mm 35mm 10mm,clip,width=0.23\textwidth]{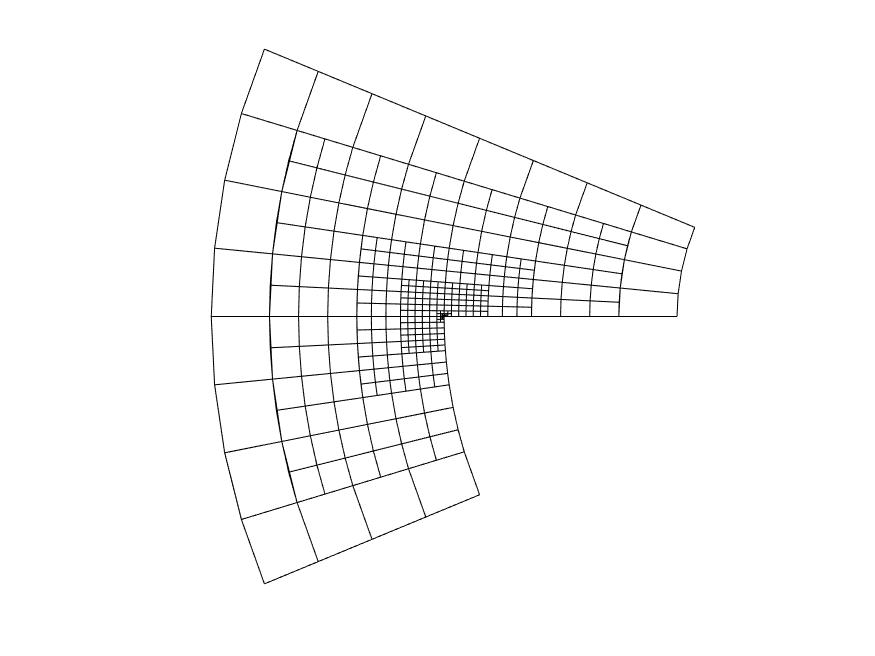}}
\caption{Curved L-shaped domain: mesh after 8 refinement steps for degree $p=4$.}
\label{fig:curvedL-meshes}
\end{figure}

\smallskip\paragraph{Test about the approximation class}
The following test shows that the approximation class depends on the continuity of the discrete spaces. 

The domain is the unit square $\Omega = (0,1)^2$, we set homogeneous Dirichlet boundary conditions, and the right-hand side is chosen such that the exact solution is given by
\[
u(x,y) = \left \{
\begin{array}{rl}
\displaystyle \sin^2\left(\frac{\pi(x-a)}{b-a}\right) \sin(\pi y), & \text{ if } a \le x \le b, \\
0, & \text{ elsewhere}. 
\end{array}
\right.
\]
The solution is smooth everywhere except at the vertical lines $x = a$ and $x=b$, where it is only $C^1$ and it has edge singularities. To understand how the approximation class depends on the continuity, we run the adaptive algorithm starting from a mesh $\QQ_0$ of $2 \times 2$ elements, with the D\"orfler marking \eqref{eq:Doerfler} with parameters $\theta=0.5$ and $\const{\rm min}=1$, for fixed degree $p=4$ and ${\cal T}$-admissibility with $\mu=4$, and we change the continuity using $C^1$, $C^2$, and $C^3$ hierarchical splines. We run two different tests: in the first one we choose $a=1/4$ and $b=3/4$, in such a way that the singularity lines coincide with lines of the mesh; in the second test we set $a=1/5$ and $b=4/5$, in such a way that, since we always refine by bisection, the singularity lines can never coincide with lines of the mesh. 
We note that in the second test we increased the number of quadrature points per element, to compute accurately the integrals on elements crossed by the singularity lines. 

The errors in the energy norm for the first test are shown in Figure~\ref{fig:class_approx_aligned}. 
Similar to the edge singularity case from above, we see that for high continuity splines the convergence rate is only $\OO({N_{\rm dof}^{-3/2}})$, while uniform refinement (not displayed) leads to $\OO(N_{\rm dof}^{-3/4})$. Instead, for $C^1$ splines we obtain the convergence rate $\OO(N_{\rm dof}^{-p/2})$, and the same rate is also obtained for uniform refinement (not displayed). This test shows with a simple example that the approximation classes depend on the continuity. Indeed, the results suggest that the solution belongs to the approximation class $-2$ for $C^1$ splines, while it only belongs to the approximation class $-3/2$ for $C^2$ and $C^3$ splines. This differs from the result in \cite{BN10}, which states that the approximation classes for $C^0$ piecewise polynomials and discontinuous Galerkin methods are identical. 
The corresponding proof exploits that the solution $u\in H_0^1(\Omega)$ has vanishing jumps across element boundaries. 
Generalizing the argument to smooth splines would likely require that also the jumps of certain derivatives of $u$ vanish, which is not the case for the currently considered $u$.

The reason why the function belongs to different classes is the fact that the singularity line coincides with a line of the mesh. Indeed, the results of the second test displayed in Figure~\ref{fig:class_approx_nonaligned} show the same convergence rate, equal to $\OO(N_{\rm dof}^{-3/2})$ independently of the continuity. It seems that, under the condition of refining by bisection, the same convergence rate as for the smooth solution can only be recovered if the singularity lines can be aligned with the mesh. Nevertheless, adaptive refinement at least doubles the convergence rate $\OO(N_{\rm dof}^{-3/4})$ for uniform refinement (not displayed).
\begin{figure}
\centering
\subfigure[Singularity line coinciding with the mesh, $a=1/4, b=3/4$]{\includegraphics[width=0.35\textwidth]{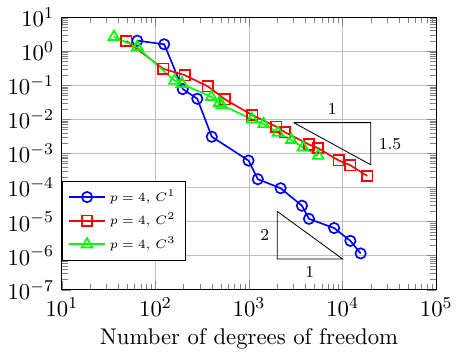} \label{fig:class_approx_aligned}}
\subfigure[Singularity line not coinciding with the mesh, $a=1/5, b=4/5$]{\includegraphics[width=0.35\textwidth]{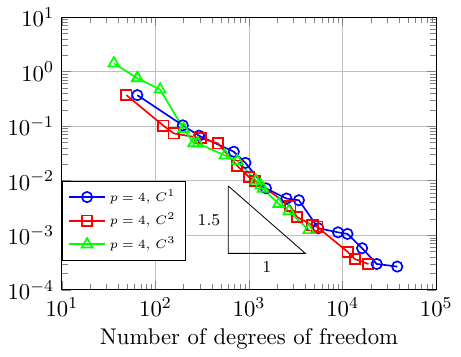} \label{fig:class_approx_nonaligned}}
\caption{Test about the approximation class: energy error $|u - U_k|_{H^1(\Omega)}$ for degree $p=4$ and ${\cal T}$-admissible meshes with $\mu=4$, with THB-splines of different continuity.}
\label{fig:approx-class}
\end{figure}

\smallskip\paragraph{Dirichlet--Neumann conditions on twisted thick ring in 3D}
The third numerical test was considered in \cite{bracco2019}. The domain $\Omega$ consists of a twisted thick ring, obtained by linear interpolation of two surfaces, where the lower one is a quarter of an annulus with inner radius equal to one and outer radius equal to two, and the upper one is the same surface rotated by 90 degrees around the $z$-axis, and translated by the vector $(0.5,0,1)$, as shown in Figure~\ref{fig:3D-domain}. We set the source term $f=0$ and impose homogeneous Dirichlet conditions everywhere, except on the upper boundary where we impose the Neumann condition $\partial u / \partial n = 1$. In this case the exact solution is not known, but we plot in Figure~\ref{fig:3D-solution} an approximate solution computed in a fine mesh and the magnitude of its gradient. It can be seen that the boundary conditions generate singularities on the edges of the upper boundary. 
\begin{figure}[ht]
\centering
\includegraphics[width=0.4\textwidth,trim=1cm 0cm 0cm 1cm, clip]{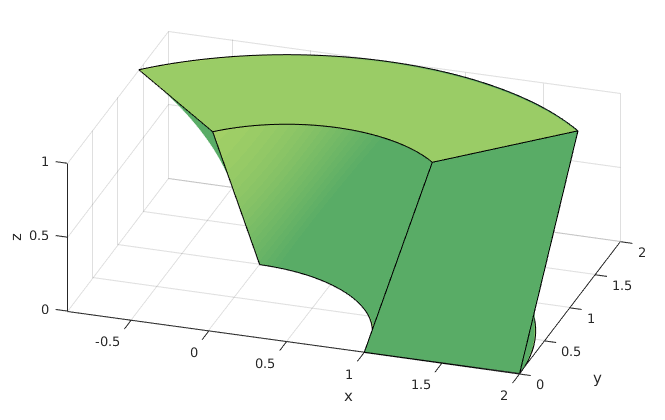}
\caption{Twisted thick ring: coordinates of the domain.} \label{fig:3D-domain}
\end{figure}

For this numerical test, and due to the presence of a Neumann condition $\partial u / \partial n = \phi_N$ on $\Gamma_N$, the weighted residual error estimator \eqref{eq:eta} is replaced by
\[
\eta_N(Q)^2 := \eta(Q)^2 + h_Q \left\| \phi_N-{\partial U}/{\partial n} \right\|^2_{L^2(\partial Q \cap \Gamma_N)},
\]
see, e.g., \cite[Section~11]{cfpp14}.
For the D\"orfler marking \eqref{eq:Doerfler}, we use the values $\theta = 0.75$ and $\const{\rm min} = 1$. Starting from an initial mesh of one single element, we run numerical tests for THB-splines of degree $p=2,3,4$, and with ${\cal T}$-admissible meshes with different admissibility classes. 

\begin{figure}[ht]
\centering
\subfigure[Value of the solution]{
\includegraphics[width=0.22\textwidth,trim=3cm 3cm 1cm 9cm, clip]{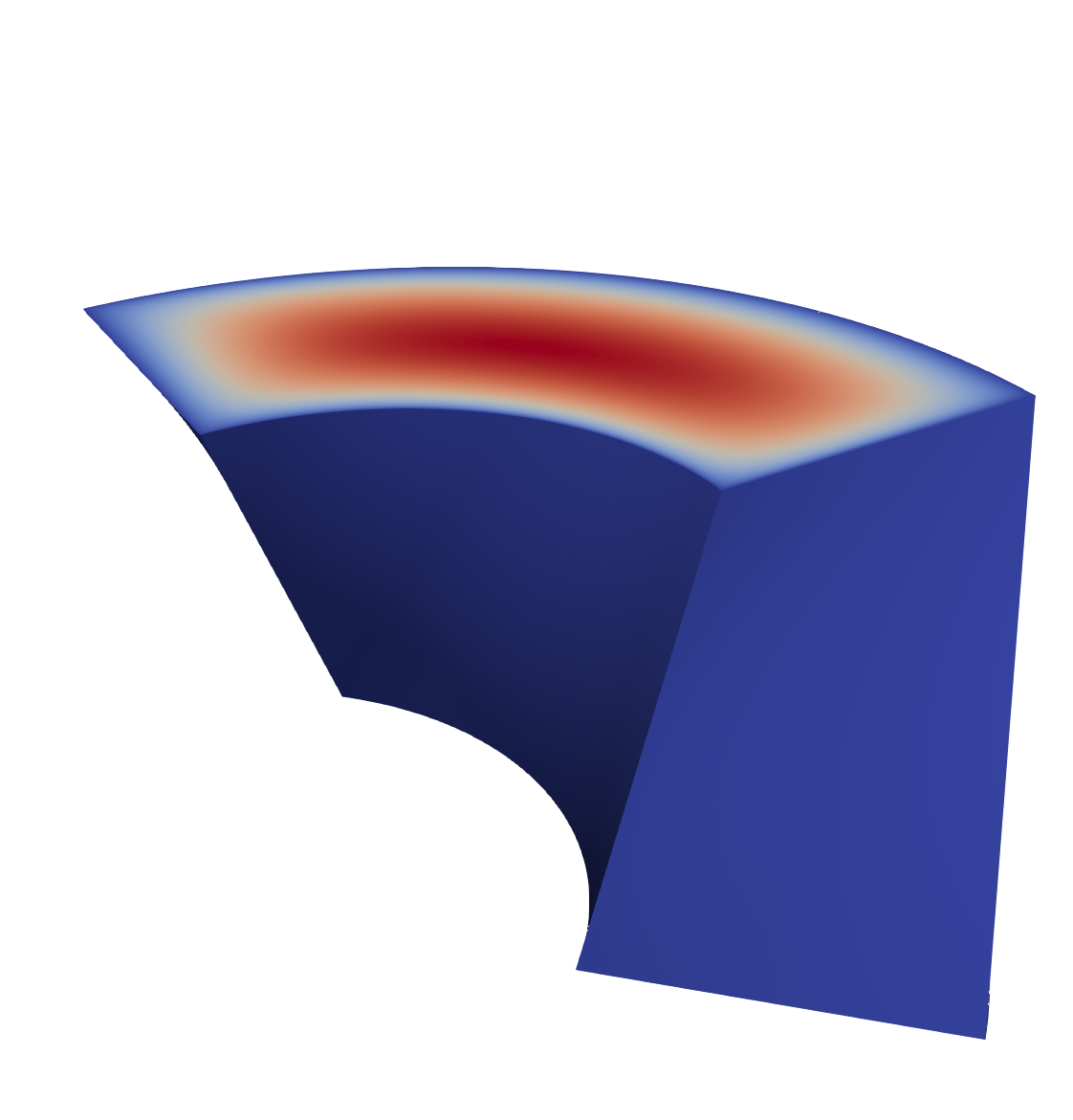}
}
\subfigure[Magnitude of the gradient]{
\includegraphics[width=0.22\textwidth,trim=3cm 3cm 1cm 9cm, clip]{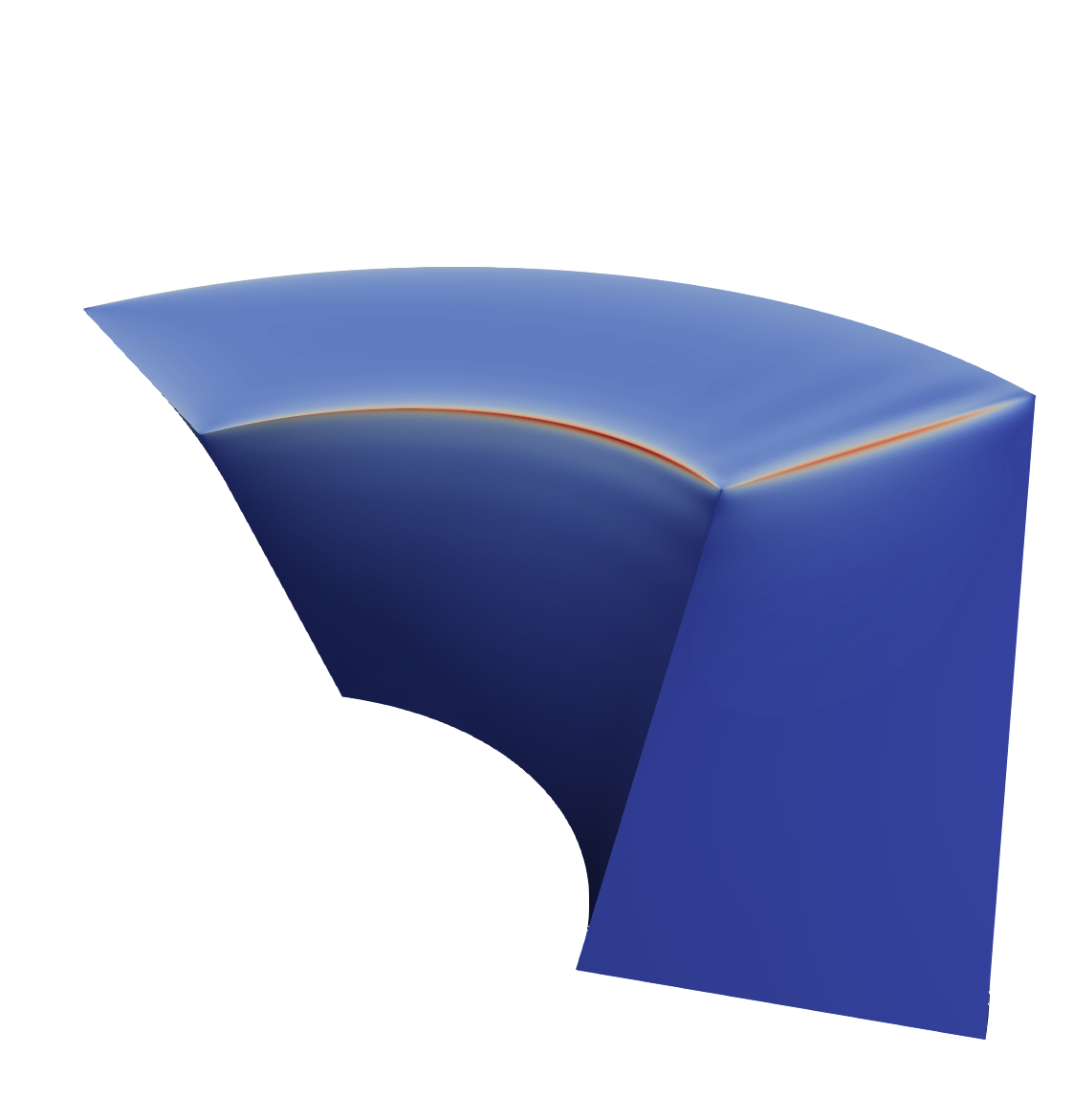}
}
\caption{Twisted thick ring: approximate solution and the magnitude of the gradient.} \label{fig:3D-solution}
\end{figure}

Since the exact solution is not known, Figure~\ref{fig:3D-estimator} shows only the values of the error estimator for different choices of the degree and the admissibility class. We also compare the results with the ones obtained for uniform refinement. As in the case of edge singularities, and since we do not allow anisotropic refinement, the optimal order of convergence, which is equal to $s_{\rm opt} = p/3$, is only reached for degree $p=2$. Heuristic arguments similar to those used in the 2D case tell us that, in general, the convergence order that we obtain for edge singularities in the 3D case is equal to
\[
s = \min(3 s_{\rm unif}, s_{\rm opt}),
\]
where again $s_{\rm unif}$ is the convergence rate for uniform refinement, which in this particular case is $s_{\rm unif} \approx 0.25$. 
Regarding the impact of the admissibility class $\mu$, as in the two-dimensional examples lower values of $\mu$ require more degrees of freedom, although the convergence rate is the same for all the admissibility classes. The plot of the meshes in Figure~\ref{fig:3D-meshes_deg3} shows that this is due to the refinement away from the singularity, which is necessary to maintain the admissibility class.
\begin{figure}
\centering
\subfigure[Error estimator for $p=2$]{
\includegraphics[trim=2mm 1mm 1mm 1mm,clip,width=0.36\textwidth]{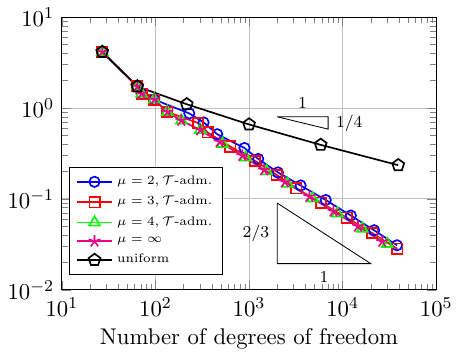}
}
\subfigure[Error estimator for $p=3$]{
\includegraphics[trim=2mm 1mm 1mm 1mm,clip,width=0.36\textwidth]{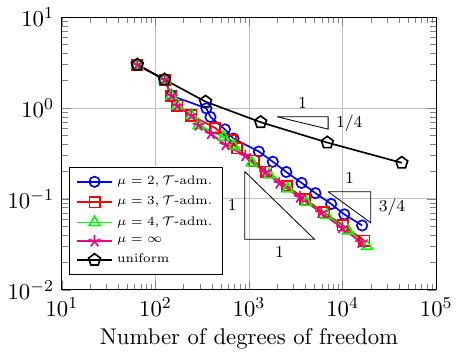}
}
\subfigure[Error estimator for $p=4$]{
\includegraphics[trim=2mm 1mm 1mm 1mm,clip,width=0.36\textwidth]{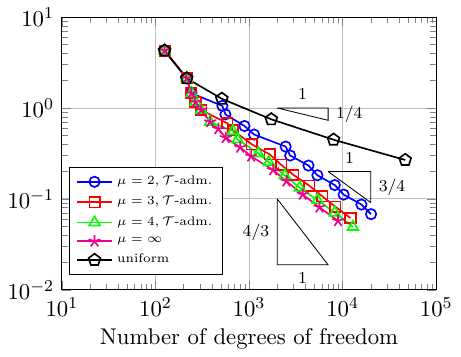}
}
\caption{Twisted thick ring: comparison of the error estimator for uniform refinement and adaptive refinement with different degree $p$ and admissibility class $\mu$.} \label{fig:3D-estimator}
\end{figure}

\begin{figure}[h!]
\centering
\subfigure[Mesh for $\mu = 2$]{\includegraphics[width=0.35\textwidth,trim=15cm 3cm 10cm 4cm, clip]{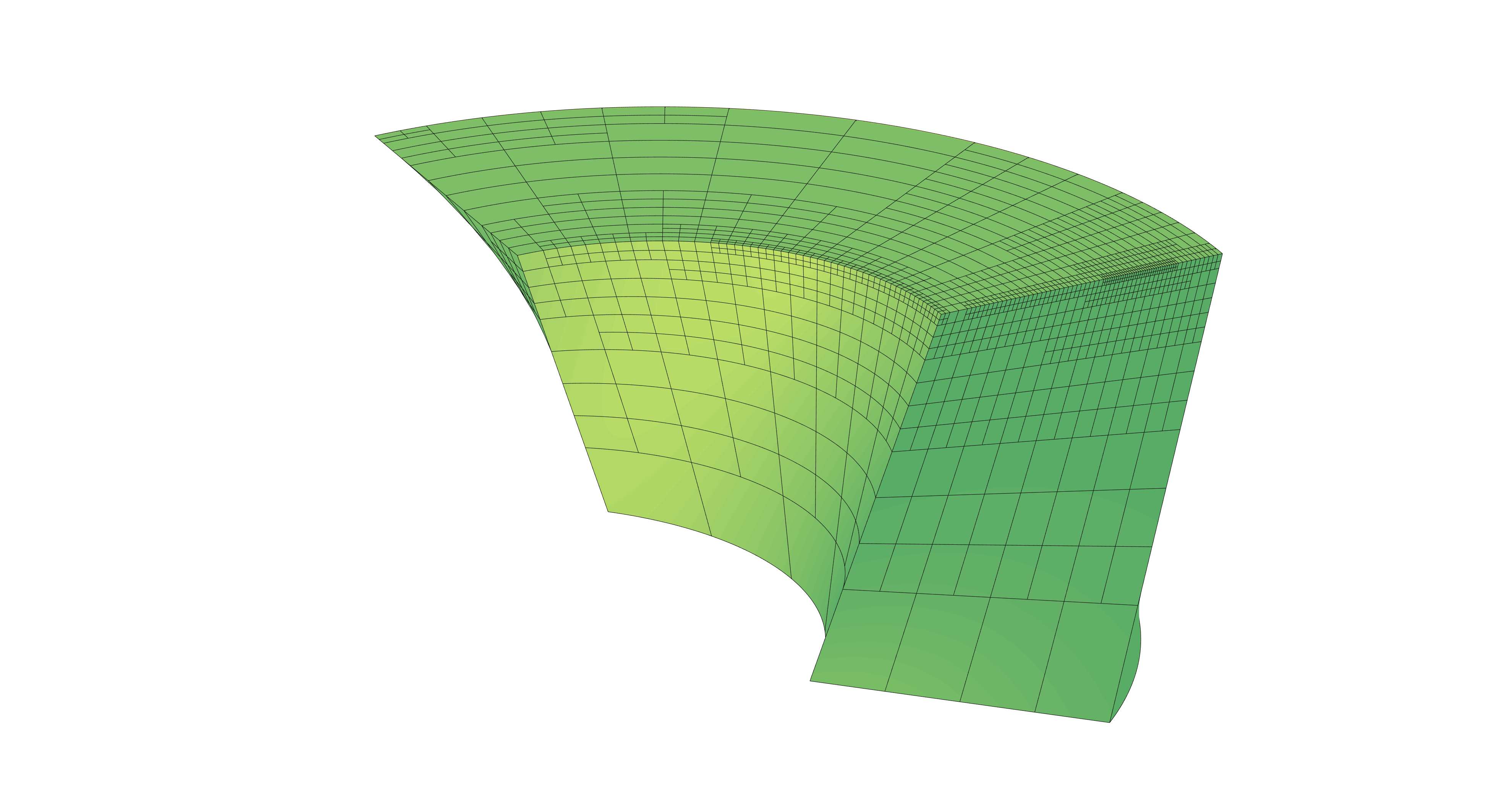}}
\subfigure[Mesh for $\mu = 3$]{\includegraphics[width=0.35\textwidth,trim=15cm 3cm 10cm 4cm, clip]{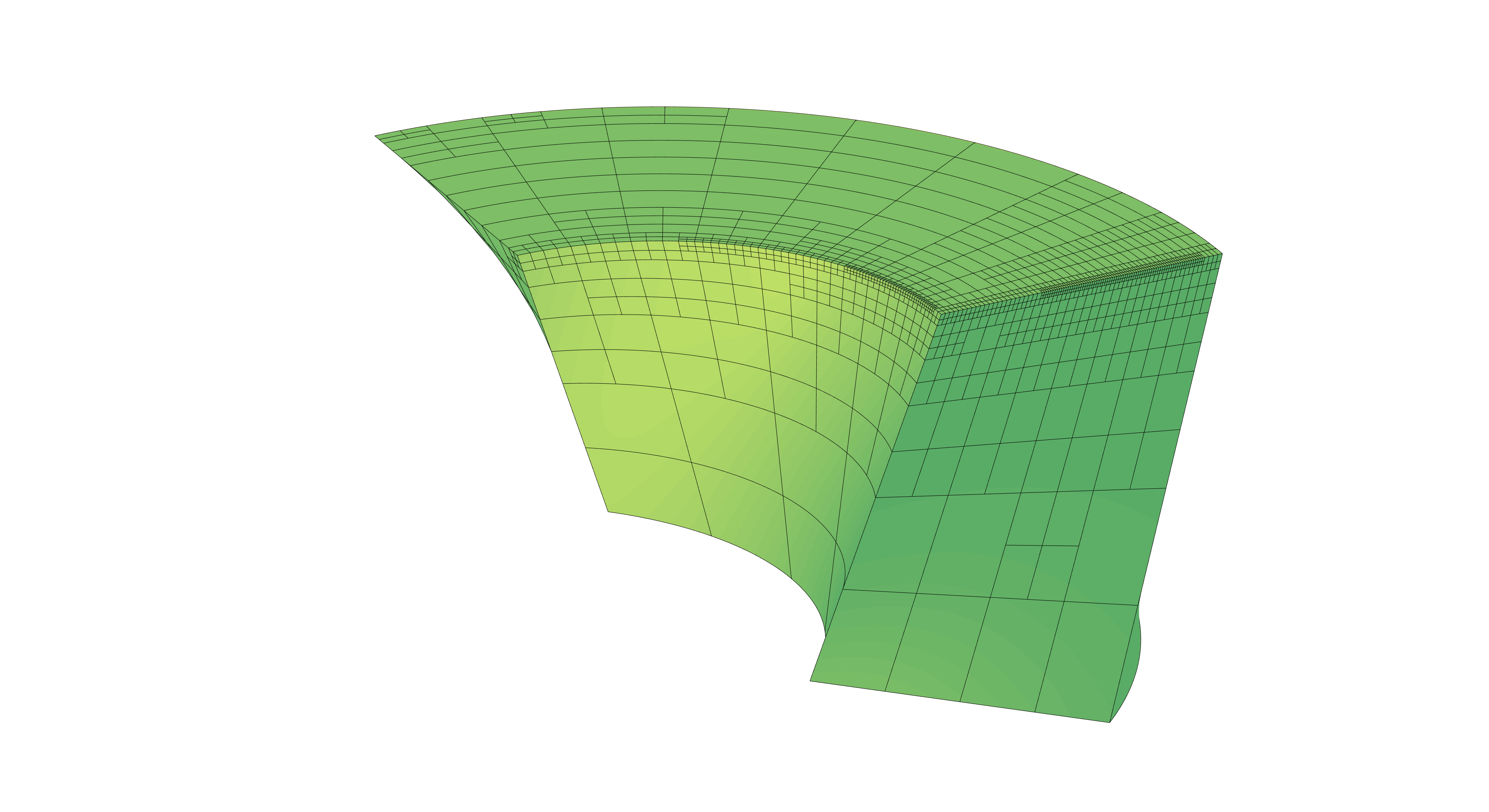}}
\subfigure[Mesh for $\mu = 4$]{\includegraphics[width=0.35\textwidth,trim=15cm 3cm 10cm 4cm, clip]{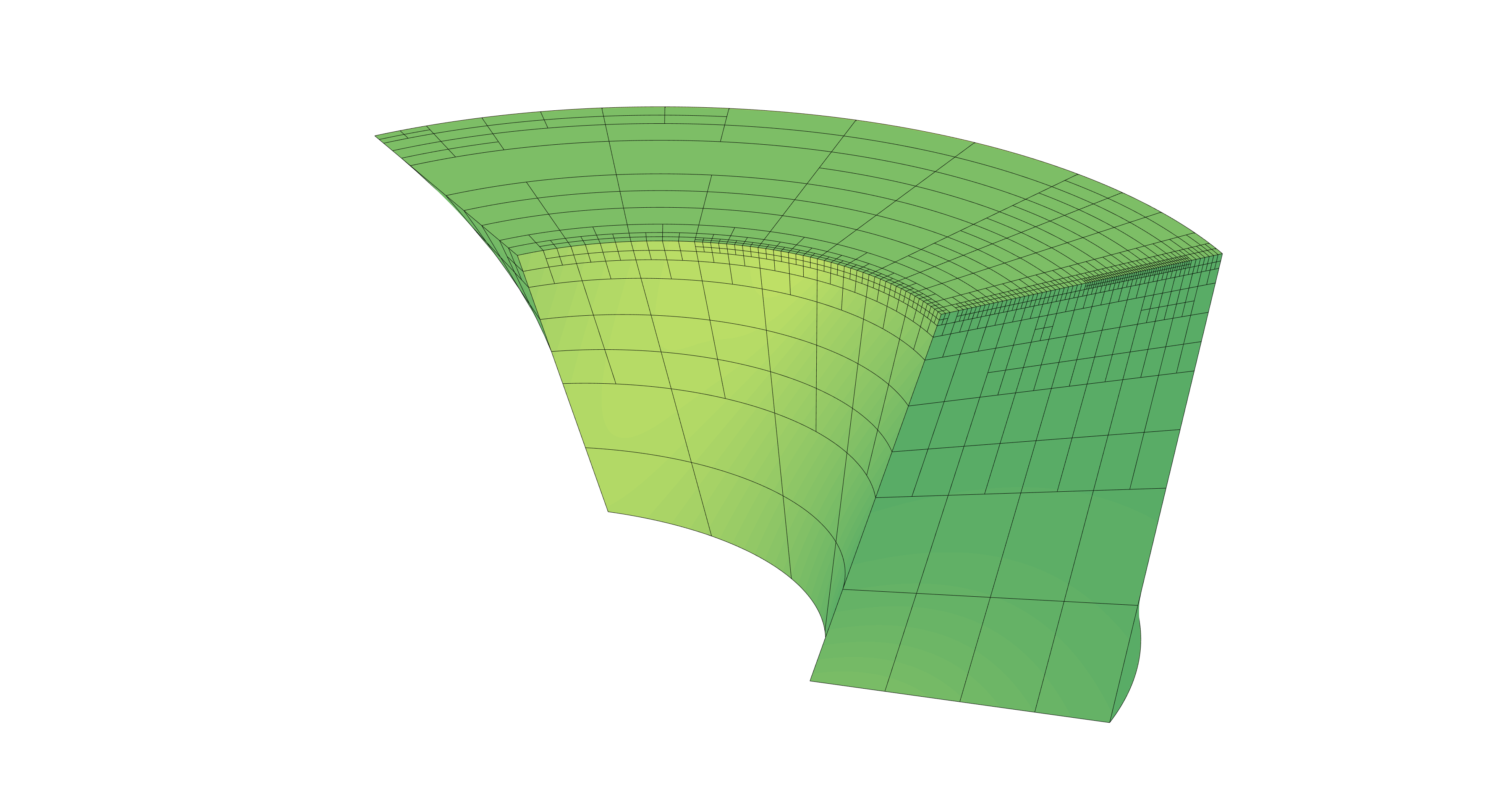}}
\subfigure[Mesh for $\mu = \infty$]{\includegraphics[width=0.35\textwidth,trim=15cm 3cm 10cm 4cm, clip]{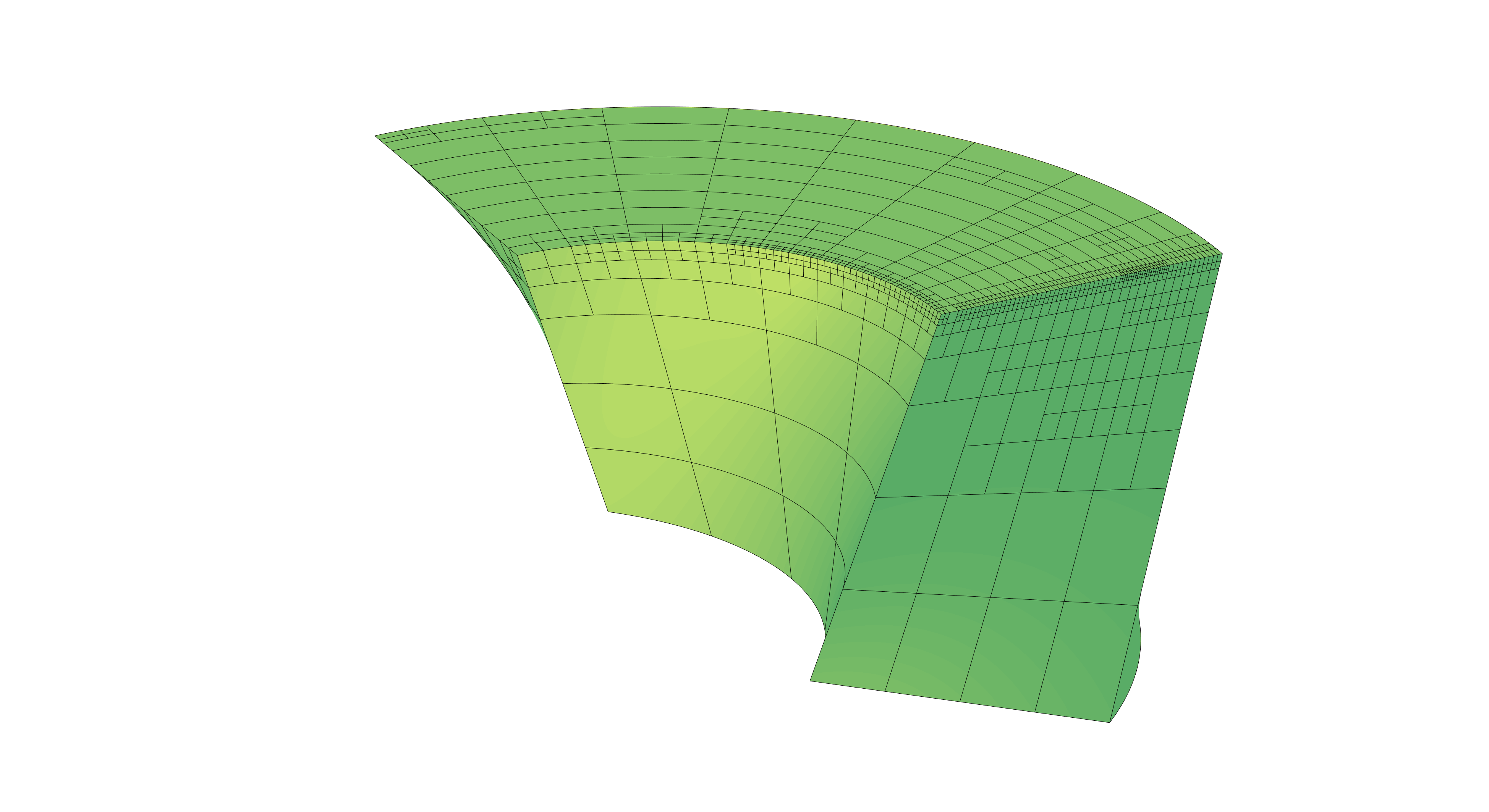}}
\caption{Twisted thick ring: meshes for degree $p=3$ and different values of the admissibility class $\mu$ after eight refinement steps.}\label{fig:3D-meshes_deg3}
\end{figure}

\subsection{Adaptive IGAFEM with T-splines}\label{sec:T-igafem}
We now apply the adaptive IGAFEM setting to T-splines on a single-patch domain.
Let $\mathbf{p}:=(p_1,\dots,p_d)$ be a vector of positive polynomial degrees and $\mathbf{\kv}^0$ be a  multivariate open knot vector on $\widehat\Omega=(0,1)^\dph$ with induced initial index T-mesh $\check\QQ_0$ (see Section~\ref{sec:T-meshes defined}).
We assume that $\hat{\mathbb{S}}_{{\bf p}} (\mathbf{\kv}^0)$ and $\hat{\mathbb{S}}_{{\bf p}_{\F}} (\mathbf \kv_{\F})$ with ${\bf p}_\F$ and ${\bf T}_\F$ from the parametrization  ${\bf F}:\widehat\Omega\to\Omega$ (see Section~\ref{subsec:def hb}) are compatible to each other as in Section~\ref{sec:IGA-basics}.
Note that $\hat{\mathbb{S}}_{{\bf p}} (\mathbf{\kv}^0)=\hat{\mathbb{S}}_{{\bf p}}^{\rm T}(\check\QQ_{0},{\bf \kv}^0)$, i.e., the initial space corresponds to tensor-product B-splines.
We choose the refinement strategy $\refine(\cdot,\cdot)$ of Section~\ref{sec:T refine}, which induces the set of all admissible meshes $\widehat\Q$.
For all $\widehat\QQ_\coarse\in\widehat\Q$ with corresponding index T-mesh $\check\QQ_\coarse$ (see Remark~\ref{rem:paramtric to index}), let $\widehat{\mathbb{S}}_\coarse:=\widehat{\mathbb{S}}_{\bf p}^{\rm T}(\check\QQ_\coarse,\mathbf{\kv}^{0})\cap H_0^1(\widehat\Omega)$ be the associated ansatz space in the parametric domain, see Section~\ref{sec:T-splines defined}. 
As in Section~\ref{sec:model}, we define the corresponding quantities in the physical domain via the parametrization ${\bf F}:\widehat\Omega\to\Omega$, i.e., 
\begin{align*}
&\QQ_\coarse:=\set{{\bf F}(\widehat Q)}{\widehat Q\in\widehat\QQ_\coarse}\quad\text{for all }\widehat\QQ_\coarse\in\widehat\Q,
\\
&\Q:=\set{\QQ_\coarse}{\widehat\QQ_\coarse\in\widehat\Q},
\end{align*}
\begin{align*}
&\refine(\QQ_\coarse,\MM_\coarse):=\set{{\bf F}(\widehat Q)}{\widehat Q\in\refine(\widehat\QQ_\coarse,\widehat\MM_\coarse)}
\\
&\qquad\text{for all }\QQ_\coarse\in\Q, \MM_\coarse\subseteq\QQ_\coarse 
\\
&\qquad\text{with }\widehat\MM_\coarse:=\set{{\bf F}^{-1}(Q)}{Q\in\QQ_\coarse}, 
\end{align*}
and the discrete space on mesh $\QQ_\coarse$ is given by
\begin{align*}
&\mathbb{S}_\coarse:=\set{\widehat V\circ{\bf F}^{-1}}{\widehat V\in\widehat{\mathbb{S}}_\coarse}.
\end{align*}
In the following lemma, we give a basis in terms of T-spline blending functions for $\widehat{\mathbb{S}}_\coarse$.
The proof is given in \cite[Lemma~3.2]{ghp17} and relies on the fact that, due to the lack of T-junctions in the frame region $\Omindex\setminus\Omip$, T-spline blending functions restricted to any $(\dph-1)$-dimensional hyperface of the unit hypercube are B-splines corresponding to the induced mesh on this hyperface. We note that the anchors corresponding to these basis functions are precisely the ones on the boundary of the region of active anchors, see Section~\ref{sec:T-splines defined}.
Clearly, this basis can be transferred to the physical domain via the parametrization $\F$.
\begin{lemma}\label{lem:homogeneous T-basis}
Given $\hat\QQ_\coarse\in\hat\Q$, the T-spline blending functions $\set{\hat B_{\coarsecomma \anchor,\mathbf{p}}}{\anchor \in \nodes_{\bf p}(\Tmesh_\coarse,{\bf T}^0)}\cap H_0^1(\widehat\Omega)$  provide a basis of $\widehat{\mathbb{S}}_\coarse$.
Here, the functions $\hat B_{\coarsecomma \anchor,\mathbf{p}}$ are defined as in \eqref{eq:T-spline}.
\end{lemma}

The given setting fits into the abstract framework of Section~\ref{sec:afem}, in particular it satisfies the assumptions of Theorem~\ref{thm:abstract bem}, which has been proved in \cite{gp18}.
We only sketch the proof in Section~\ref{sec:T-fem mesh verification}--\ref{sec:T-fem space verification}.
As already mentioned at the beginning of Section~\ref{sec:adaptive igafem}, the multi-patch case is essentially open.
For numerical experiments validating Theorem~\ref{thm:T-igafem}, we refer to \cite{hkmp17}.

\begin{theorem}\label{thm:T-igafem}
T-splines on admissible meshes satisfy the mesh properties \eqref{M:shape}--\eqref{M:locuni},  the refinement properties \eqref{R:childs}--\eqref{R:overlay}, and the space properties \eqref{S:nestedness}--\eqref{S:proj} and \eqref{S:inverse}--\eqref{S:grad}.
The involved constants depend only on the dimension $\dph$, the parametrization constant $C_\F$ of Section~\ref{sec:parametrization_assumptions}, the degree $\mathbf{p}$, and the initial knot vector $\mathbf{T}^0$.
In particular, Theorem~\ref{thm:abstract} is applicable.
In conjunction with Theorem~\ref{thm:abstract main}, this yields reliability~\eqref{eq:reliable}, efficiency~\eqref{eq:efficient}, and  linear convergence at optimal rate~\eqref{eq:linear convergence}--\eqref{eq:optimal convergence} of the residual error estimator~\eqref{eq:eta}, when the adaptive Algorithm~\ref{alg:abstract algorithm} is employed.
\end{theorem}

\begin{remark}
We also mention that \cite{cv18} has recently introduced a local multilevel preconditioner for the stiffness matrix of symmetric problems which leads to uniformly bounded condition numbers for T-splines on admissible T-meshes.
An important consequence is that the corresponding PCG solver is uniformly contractive. As for hierarchical splines, see Remark~\ref{rem:PCG-HB-FEM}, \cite{ghps20} thus allows to prove  that an adaptive algorithm which steers mesh-refinement and an inexact PCG solver leads to optimal convergence rates both with respect to the number of elements and with respect to the computational cost.
\end{remark}

\subsubsection{Mesh properties}
\label{sec:T-fem mesh verification}

Shape regularity \eqref{M:shape} is trivially satisfied in the parametric domain, since the direction in which an element is bisected periodically alternates after each refinement. 
Due to the regularity of the parametrization ${\bf F}$ of Section~\ref{sec:parametrization_assumptions}, the property transfers to the physical domain. 

Local quasi-uniformity \eqref{M:locuni} in the parametric domain  follows from Remark~\ref{rem:neighbors touch} together with Proposition~\ref{prop:lqiT}. 
Again, the regularity of the parametrization guarantees this property also in the physical domain.

\subsubsection{Refinement properties}
\label{sec:T-fem refinement verification}

The child estimate \eqref{R:childs} is trivially satisfied with $\const{child}=2$, since each refined element in the parametric domain is only bisected in one direction. 
The closure estimate \eqref{R:closure} is just the assertion of Proposition~\ref{prop:T-spline lincomp}.

For $\dph=2$, \cite[Section~5]{mp15} shows that the overlay 
\begin{align*}
\hat \hmesh_\fine:=&\set{\hat Q\in\hat\hmesh_\coarse}{\exists \hat Q'\in \hat\hmesh_\meshidx\text{ with }\hat Q\subseteq \hat Q'}\\
\cup &\set{\hat Q'\in\hat\hmesh_\meshidx}{\exists \hat Q\in \hat\hmesh_\coarse\text{ with }\hat Q'\subseteq \hat Q}
\end{align*}
of two admissible meshes $\hat\QQ_\coarse,\hat\QQ_\meshidx$ in the parametric domain is again admissible.
The proof also extends to the three-dimensional case $\dph=3$.
Obviously, this property immediately transfers to the physical domain.
Clearly, the resulting mesh $\QQ_\fine$ in the physical domain satisfies the properties in~\eqref{R:overlay}.

\subsubsection{Space properties}
\label{sec:T-fem space verification}

Nestedness \eqref{S:nestedness} follows from Proposition~\ref{prop:T-nestedness}. 
The inverse inequality \eqref{S:inverse} in the parametric domain follows easily from Lemma~\ref{lemma:2beziers} and standard scaling arguments, since each T-spline is a polynomial of fixed degree $\mathbf{p}$ on each element of the B\'ezier mesh. 
Due to the regularity of the parametrization ${\bf F}$ of Section~\ref{sec:parametrization_assumptions}, the property transfers to the physical domain.
The local domain of definition property \eqref{S:local} follows easily from Lemma~\ref{lem:homogeneous T-basis}, Proposition~\ref{prop:tsp-sext}, and the definition of T-spline blending functions, see \cite[Section~3.3]{gp18} for details.

\smallskip\paragraph{Scott--Zhang type operator}
Due to the regularity of the parametrization ${\bf F}$ of Section~\ref{sec:parametrization_assumptions}, it is sufficient to provide for all $\hat\QQ_\coarse\in\hat\Q$ an operator $\hat J_\coarse:H_0^1(\widehat\Omega)\to \widehat{\mathbb{S}}_\coarse$ satisfying the properties \eqref{S:proj} and \eqref{S:app}--\eqref{S:grad}  in the parametric domain. 
We define this operator similarly as $\projT{\mathbf{p}}$ of Section~\ref{sec:dual-compatible}, but now have to take into account the homogeneous boundary conditions
\begin{align*}
\widehat J_\coarse:\,&H_0^1(\widehat\Omega)\to \widehat{\mathbb{S}}_\coarse,
 \hat v \mapsto \sum_{\substack{\anchor \in \nodes_{\bf p}(\Tmesh,{\bf T}^0)\\\hat B_{\coarsecomma \anchor,\mathbf{p}}\in H_0^1(\hat\Omega)}} \hat \lambda_{\coarsecomma \anchor,\mathbf{p}}(\hat v) \hat B_{\coarsecomma \anchor,\mathbf{p}},
\end{align*}
where $\hat\lambda_{\coarsecomma \anchor,\mathbf{p}}$ is defined as in \eqref{eq:T-dual-basis}.

The second property of Proposition~\ref{prop:tsp-sext} particularly implies the existence of a uniform constant $q_1$ such that for all $\hat Q\in\hat\QQ$,
\begin{align*}
\sext{\elemp} \subseteq \pi_\coarse^{q_1}(\hat Q).  
\end{align*}
With Corollary~\ref{corol:local-projector-Tsplines}, this immediately gives \eqref{S:proj}.
Moreover, the local $L^2$-stability of Proposition~\ref{prop:tsp-proj} is also valid for $\hat J_\coarse$ as the corresponding proof only relies on estimates of the dual functionals. 
Together with the local projection property  \eqref{S:proj} and  the inverse inequality \eqref{S:inverse}, the Poincar\'e (for elements away from the boundary) as well as the Friedrichs inequality (for elements close to the boundary) readily imply for all $\hat v\in H_0^1(\widehat \Omega)$ and $\hat Q\in\hat\QQ_\coarse$ that
\begin{align*}
\norm{(1-\hat J_\coarse)\,\hat v}{L^2(\hat Q)}&\lesssim |\hat Q|^{1/\dph}\,\norm{\hat v}{H^1(\sext{\elemp})}\\
\norm{\nabla \hat J_\coarse\, \hat v}{L^2(\hat Q)}&\lesssim \norm{\hat v}{H^1(\sext{\elemp})},
\end{align*}
see \cite[Section~3.3]{gp18} for details.
We conclude  \eqref{S:app}--\eqref{S:grad} with $q_{\rm sz}=q_1$.

\newpage


\section{Adaptive IGABEM in arbitrary dimension}
\label{sec:igabem}

In this section, we consider two concrete realizations of the abstract adaptive Galerkin BEM framework from Section~\ref{sec:abem}. We consider hierarchical splines in Section~\ref{sec:H-igabem}, assuming that the boundary $\Gamma$ is a multi-patch domain. Convergence results in this setting are proved in \cite[Section~5.4 and 5.5]{gantner17} and \cite{gp20+} for $\HH$-admissible meshes, and leveraging on \cite{bg16,bg17}, we extend them to $\TT$-admissible meshes. The theoretical findings are underlined by numerical experiments in Section~\ref{sec:numerical igabem3d}. Then, in Section~\ref{sec:T-igabem}, we present an adaptive IGABEM  based on T-splines. 
In contrast to IGAFEM, it is easy to define a suitable refinement strategy on the multi-patch domain $\Gamma$ as we do not enforce continuity across interfaces. 
{The corresponding results are new, but mostly follow from~\cite{gp18}. 
 
Finally, in Section~\ref{sec:elementary_bem}, we also consider an adaptive IGABEM in 2D which additionally controls the smoothness of the used one-dimensional spline ansatz space as in \cite{fghp16}.
Although the theoretical results of Section~\ref{sec:abem} are not directly applicable in the setting of adaptive smoothness, optimal convergence can be proved with similar techniques.


\subsection{Adaptive IGABEM with hierarchical splines}\label{sec:H-igabem}
Hierarchical meshes on the boundary $\Gamma\subset\R^\dph$, $\dph\ge2$, can be defined similarly as in the IGAFEM setting in Section~\ref{sec:H-multipatches}: 
for each $m=1,\dots,M$, let $\mathbf{p}_m$ be a vector of positive polynomial degrees and $\mathbf{\kv}^0_m$ be a  multivariate open knot vector on $\widehat\Gamma=(0,1)^\dpa$, $\dpa = \dph-1$, with induced initial mesh $\widehat\QQ_{0,m}:=\hat\QQ^0_m$. 
We assume that $\hat{\mathbb{S}}_{{\bf p}_m} (\mathbf{\kv}^0_m)$ and $\hat{\mathbb{S}}_{{\bf p}_{\F_m}} (\mathbf \kv_{\F_m})$ with ${\bf p}_{\F_m}$ and ${\bf T}_{\F_m}$ from the parametrization  ${\bf F}_m:\widehat\Gamma\to\Gamma_m$ (see Section~\ref{sec:multi-patch}) are compatible to each other as in Section~\ref{sec:IGA-basics}. 
Note that the coarsest spaces are $\hat{\mathbb{S}}_{{\bf p}_m} (\mathbf{\kv}^0_m)=\hat{\mathbb{S}}^{\rm H}_{{\bf p}_m}(\hat\QQ_{0,m},{\bf \kv}_m^0)$, i.e., they correspond to tensor-product B-splines on each patch.
Moreover, we assume for the initial mesh $\QQ_0=\bigcup_{m=1}^M \QQ_{0,m}$ with $\QQ_{0,m}:=\set{\F_m(\hat Q)}{\hat Q\in\hat\QQ_{0, m}}$ that there are no hanging nodes between patch interfaces $\Gamma_{m,m'}=\overline{ \Gamma_m} \cap \overline {\Gamma_{m'}}$ with $m\neq m'$, see also \eqref{P:conforming-mesh} of Section~\ref{sec:multi-patch}.
We fix the admissibility parameter $\mu$ as well as the kind of mesh that we want to consider,  i.e., $\HH$-admissible or $\TT$-admissible meshes, 
 and abbreviate for each $m=1,\dots,M$ the set of all corresponding admissible meshes as $\widehat\Q_m$, see Section~\ref{sec:hierarchical refine}. 
Moreover, we abbreviate $\Q_m:=\set{\QQ_{\coarsecomma m}}{\hat \QQ_{\coarsecomma m}\in\hat\Q_m}$ with $\QQ_{\coarsecomma m}:=\set{\F_m(\hat Q)}{\hat Q\in\hat\QQ_{\coarsecomma m}}$.
We define the set of all admissible meshes $\Q$ as the set of all 
\begin{align*}
\QQ_\coarse=\bigcup_{m=1}^M \QQ_{\coarsecomma m} \text{ with }\QQ_{\coarsecomma m}\in\Q_m
\end{align*} 
 such that there are no hanging nodes on any interface $\Gamma_{m,m'}=\overline{ \Gamma_m} \cap \overline {\Gamma_{m'}}$ with $m\neq m'$.

For $\QQ_\coarse\in\Q$, the associated ansatz space is defined as
\begin{align*}
\begin{split}
\mathbb{S}_\coarse:=\big\{ V \in L^2(\Gamma) : V|_{\Gamma_m} \in {\mathbb{S}}^{\rm H}_{{\bf p}_m}(\hat\QQ_{\coarsecomma m},{\bf \kv}_m^0), \quad \\
\text{ for } m = 1, \ldots, M\big\},
\end{split}
\end{align*}
where
\begin{align*}
{\mathbb{S}}^{\rm H}_{{\bf p}_m}(\hat\QQ_{\coarsecomma m},{\bf \kv}_m^0):=\set{\hat V\circ\F_m^{-1}}{\hat V\in\hat{\mathbb{S}}^{\rm H}_{{\bf p}_m}(\hat\QQ_{\coarsecomma m},{\bf \kv}_m^0)}.
\end{align*}
To obtain bases of the space $\mathbb{S}_\coarse$, we first define  
\begin{align*}
\HH_{{\bf p}_m}(\hat\QQ_{\coarsecomma m},\mathbf{\kv}^{0}_m)&:=\set{\hat\beta\circ\F_m^{-1}}{\hat\beta\in \hat\HH_{{\bf p}_m}(\hat\QQ_{\coarsecomma m},\mathbf{\kv}^{0}_m)},\\
\TT_{{\bf p}_m}(\hat\QQ_{\coarsecomma m},\mathbf{\kv}^{0}_m)&:=\set{\hat\tau\circ\F_m^{-1}}{\hat\tau\in \hat\TT_{{\bf p}_m}(\hat\QQ_{\coarsecomma m},\mathbf{\kv}^{0}_m)}.
\end{align*}
Since the ansatz functions do not have to be continuous across interfaces, a basis of $\mathbb{S}_\coarse$ is given by 
\begin{align}\label{eq:hierarchical BEM basis}
\begin{split}
\mathbb{S}_\coarse&={\rm span}\Big( \bigcup_{m=1}^M \HH_{{\bf p}_m}(\hat\QQ_{\coarsecomma m},\mathbf{\kv}^{0}_m)\Big)\\
&= {\rm span}\Big( \bigcup_{m=1}^M \TT_{{\bf p}_m}(\hat\QQ_{\coarsecomma m},\mathbf{\kv}^{0}_m)\Big),
\end{split}
\end{align}
where we extend the involved (T)HB-splines, which actually only live on $\Gamma_m$, by zero to the whole boundary $\Gamma$.
We stress that the chosen basis is theoretically irrelevant for the realization of Algorithm~\ref{alg:abstract algorithm} (in particular for the solving step), see also Section~\ref{sec:numerical igafem} for a detailed discussion in the case of IGAFEM.

\begin{remark}\label{rem:H-adaptive Neumann}
We note that it is actually not necessary to forbid  hanging nodes at interfaces, but it would be sufficient to control the size difference between intersecting elements.
However, in contrast to weakly-singular integral equations, hypersingular integral equations, which result from Neumann problems (see,  e.g. \cite[Chapter~7]{mclean00}), require continuous trial functions. 
As in Section~\ref{sec:H-multipatches}, one sees that the conformity property \eqref{P:conforming-basis-space} of Section~\ref{sec:IGA-basics} is satisfied for (T)HB-splines on admissible meshes provided that  $\mathbf{p}_m$ and $\mathbf{\kv}^0_m$ satisfy~\eqref{P:conforming-basis-space}, which is slightly stronger than assuming that there are no hanging nodes at interfaces. 
Thus, corresponding basis functions can easily be constructed, and admissible meshes are suited for both the weakly- and the hypersingular case.
Alternatively, one can also proceed as in Remark~\ref{rem:second multi patch}, applying conformity at each level and then defining directly hierarchical multi-patch functions.
\end{remark}

To obtain admissible meshes starting from the initial one, we can essentially employ the same refinement algorithm as in Section~\ref{sec:H-multipatches}: 
For arbitrary $\QQ_\coarse\in\Q$ and $Q\in\QQ_{\coarsecomma m}$ with corresponding element $\hat Q:=\F_m^{-1}(Q)$ in the parametric domain, let $\NN_{\coarsecomma m}(\hat Q)\subseteq\hat\QQ_{\coarsecomma m}$ either denote the corresponding $\HH$-neighborhood in the case of $\HH$-admissible meshes or the $\TT$-neighborhood in the case of $\TT$-admissible meshes, see Section~\ref{sec:hierarchical refine}.
We define the \emph{neighbors} of $Q$ as
\begin{align*}
&\NN_\coarse(Q):=\set{Q'\in\QQ_{\coarsecomma m}}{\hat Q' \in \NN_{\coarsecomma m}(\hat Q)}\\
&\qquad\cup \bigcup_{m'\neq m} \set{Q'\in\QQ_{\coarsecomma m'}}{{\rm dim}(\overline Q\cap \overline Q')=d-1},
\end{align*}
i.e., as in the IGAFEM case of Section~\ref{sec:H-multipatches}, we add to the neighborhood adjacent elements from other patches.
Then, it is easy to see that Algorithm~\ref{alg:multi-patch bem} returns an admissible mesh.
Indeed, one can show that the set of all possible refinements $\refine(\QQ_0)$ even coincides with $\Q$, see \cite[Proposition~5.4.3]{gantner17} in the case of $\HH$-admissible meshes of class $\mu=2$.

\begin{algorithm}[!ht]
\caption{\texttt{refine} (Multi-patch refinement)}
\label{alg:multi-patch bem}
\begin{algorithmic}
\Require admissible mesh $\QQ_\coarse$ and marked elements $\MM\subseteq\QQ_\coarse$
\Repeat
\State set $\displaystyle \mathcal{U} = \bigcup_{Q \in \MM} \NN_\coarse(Q)\setminus \MM$
\State set $\MM = \MM\cup \mathcal{U}$
\Until {$\mathcal{U} = \emptyset$}
\State update $\QQ_\coarse$ by replacing the elements in $\MM$ by their children\\
\Ensure refined admissible mesh $\QQ_\coarse$ 
\end{algorithmic}
\end{algorithm}

The given setting fits into the abstract framework of Section~\ref{sec:abem}.
So far, this is only proved in \cite[Section~5.4 and 5.5]{gantner17} and \cite{gp20+} for $\HH$-admissible meshes of class $\mu=2$. 
However, building on \cite{bg16,bg17}, where IGAFEM on  $\TT$-admissible meshes has been considered, the generalization to arbitrary admissible meshes is indeed straightforward. 
We only sketch the proof in Section~\ref{sec:H-bem space verification}.
Note that most of the properties have already been verified in Section~\ref{sec:H-igafem} for IGAFEM-meshes.

\begin{theorem}\label{thm:H-igabem}
Hierarchical splines on admissible meshes satisfy the mesh properties  \eqref{M:shape}--\eqref{M:locuni},  the refinement properties \eqref{R:childs}--\eqref{R:overlay}, and the space properties \eqref{S:nestedness}--\eqref{S:local}, \eqref{S:proj bem}, and \eqref{S:inverse bem}--\eqref{S:stab bem}.
The involved constants depend only on the dimension $\dph$, the parametrization constants $C_{\F_m}$ of Section~\ref{sec:parametrization_assumptions}, $\mathbf{p}_m$, ${\bf T}_m^0$, and $\mu$.
In particular, Theorem~\ref{thm:abstract bem} is applicable. 
In conjunction with Theorem~\ref{thm:abstract main}, this yields reliability~\eqref{eq:reliable bem} and  linear convergence at optimal rate~\eqref{eq:linear convergence}--\eqref{eq:optimal convergence} of the residual error estimator~\eqref{eq:eta bem}, when the adaptive Algorithm~\ref{alg:abstract algorithm} is employed.
\end{theorem}



\subsubsection{Mesh, refinement, and space properties}
\label{sec:H-bem space verification}
The mesh properties  \eqref{M:shape}--\eqref{M:locuni} follow as in Section~\ref{sec:H-fem mesh verification}.
The refinement properties \eqref{R:childs}--\eqref{R:overlay} follow as in Sections~\ref{sec:H-fem refinement verification} and~\ref{sec:H-multipatches}. 
The properties \eqref{S:nestedness}--\eqref{S:local} follow as in Section~\ref{sec:H-fem space verification}, see also \cite[Section~5.5.12]{gantner17} for details. 

For $\HH$-admissible meshes of class $\mu=2$, the proof of  \eqref{S:inverse bem} is given in \cite[Section~5.5.9]{gantner17}, which itself strongly builds on a similar result for simplicial meshes \cite{dfghs04}.
However, we stress that the proof only hinges on the mesh properties \eqref{M:shape}--\eqref{M:locuni} and the fact that hierarchical splines are polynomials on all elements in the parametric domain (see Proposition~\ref{prop:hb properties}~(iv)), 
and the result hence extends to arbitrary admissible hierarchical meshes. 

The reference \cite[Proposition~5.5.5]{gantner17} states that the local approximation of unity property \eqref{S:unity bem} is satisfied if there exists a finite subset $\mathcal{B}\subset \mathbb{S}_\coarse$ whose elements 
are non-negative, local in the sense that for all $\beta\in\mathcal{B}$ there exists $Q\in\QQ_\coarse$ and a uniform constant $q\in\N$ such that $\supp(\beta)\subseteq \pi_\coarse^q(Q)$, and form a partition of unity.
According to Proposition~\ref{prop:thb properties} (i) and \eqref{eq:trunc in patch}, these assumptions are fulfilled for THB-splines 
\begin{align*}
\BB:=\bigcup_{m=1}^M \TT_{{\bf p}_m}(\hat\QQ_{\coarsecomma m},\mathbf{\kv}^{0}_m).
\end{align*}

\paragraph{Scott--Zhang type operator}
Since the ansatz functions do not have to be continuous at interfaces and due to the regularity of the parametrization ${\bf F}_m$ of Section~\ref{sec:parametrization_assumptions}, it is sufficient to provide for each patch $\Gamma_m$ and $\hat\SS_m\subseteq\hat\QQ_{\coarsecomma m}$ an operator
\begin{align*}
\hat\JJ_{\coarsecomma m,\hat\SS_m}:L^2(\hat\Gamma)\to\big\{\hat\Psi_{\coarsecomma m}\in {\mathbb{S}}^{\rm H}_{{\bf p}_m}(\hat\QQ_{\coarsecomma m},{\bf \kv}_m^0): \\
\hat\Psi_{\coarsecomma m}|_{\bigcup(\hat\QQ_{\coarsecomma m}\setminus\hat\SS_m)}=0\big\}
\end{align*}
which satisfies \eqref{S:proj bem} and \eqref{S:stab bem}.
We define this operator similarly as $\projH{\mathbf{p}}$ of Section~\ref{sec:hierarchical interpolation}, but now have to take into account that the output should only live on $\bigcup \hat\SS_m$ 
by discarding all THB-splines that have support entirely outside of this set.
Then the local projection property \eqref{S:proj bem} as well as the local $L^2$-stability \eqref{S:stab bem} for the operator $\hat\JJ_{\coarsecomma m,\hat\SS_m}$ can be shown as in Section~\ref{sec:H-fem space verification}. 
Details for hierarchical splines on $\HH$-admissible meshes of class $\mu=2$ are found in \cite[Section~5.5.14]{gantner17}.


\subsubsection{Numerical experiments}\label{sec:numerical igabem3d}
We now apply the adaptive IGABEM with HB-splines analyzed in the previous sections. 
We consider the 3D Laplace operator $\mathscr{P}:=-\Delta$ as partial differential operator, and we present two numerical experiments that were already considered in~\cite[Section~5.6]{gantner17}: a quasi-singular solution on a thick ring and an exterior problem on a cube.
For numerical experiments with (one-dimensional) hierarchical splines in 2D, we refer to \cite{fgkss19}.
The  fundamental solution of $-\Delta$ in 3D is given by
\begin{align*}
G({\bf z}):=\frac{1}{4\pi}\frac{1}{|{\bf z}|}\quad\text{for all }{\bf z}\in \R^3\setminus\{0\},
\end{align*}
and the resulting single-layer operator $\mathscr{V}:H^{-1/2}(\Gamma)\to H^{1/2}(\Gamma)$ is elliptic, see Section~\ref{sec:model problem bem}.
Throughout, we use $\HH$-admissible hierarchical meshes of class $\mu=2$ 
and the basis of (non-truncated) HB-splines given in~\eqref{eq:hierarchical BEM basis}  for the considered  ansatz spaces.
An explanation on how the involved singular integrals are computed via suitable Duffy transformations and subsequent standard tensor Gaussian quadrature is given in \cite[Section~5.6]{gantner17}, see also \cite[Chapter~5]{ss11} and \cite[Section~7.1]{karkulik12}. 
We mention that no compression techniques have been used for the dense Galerkin matrices. 
Moreover, to ease computation, the term $h_Q=|Q|^{1/2}$ in the estimator~\eqref{eq:eta bem} is replaced by the equivalent term $\diam(\Gamma)|\widehat Q|^{1/2}$ with the corresponding element $\widehat Q$ in the parametric domain.

\paragraph{Quasi-singularity on thick ring}
For given Dirichlet data $g\in {H}^{1/2}(\Gamma_{})$, we consider the interior Laplace--Dirichlet problem
\begin{align}\label{eq:Laplace interior bem}
\begin{split}
-\Delta u&=0\quad\text{in }{\Omega},\\ u&=g\quad\text{on } \Gamma,\end{split}
\end{align}
on the (quarter of a) thick ring
\begin{align*}
\Omega: = \big\{&10^{-1}(r\cos(\varphi),r\sin(\varphi),z):
\\
& r \in(1/2, 1), \varphi\in(0,{\pi}/{2}),  z\in(0,1)\big\}; 
\end{align*}
see Figure~\ref{fig:bendcube_mesh} for an illustration. The boundary of $\Omega$ is described by six patches of rational splines of degrees 1 and 2, without any internal knots, see \cite[Section~5.6.2]{gantner17} for a precise parametrization of the boundary.

Then, \eqref{eq:Laplace interior bem} can be equivalently rewritten as an integral equation in the form of \eqref{eq:Symmy interior}. 
In particular, the normal derivative $\phi:=\partial_{\boldsymbol{\nu}} u$ of the  weak solution $u$ of~\eqref{eq:Laplace interior bem} satisfies  
$\mathscr{V}\phi =(\mathscr{K}+1/2) g$.
We prescribe the exact solution of \eqref{eq:Laplace interior bem} as the shifted fundamental solution
\begin{equation*}
u({\bf x}):=G({\bf x}-{\bf y}_0)=\frac{1}{4\pi}\frac{1}{|{\bf x}-{\bf y}_0|}, 
\end{equation*}
with ${\bf y}_0:=10^{-1}(0.95\cdot 2^{-3/2},0.95\cdot 2^{-3/2},1/2)\in\R^3\setminus\overline\Omega$.
Although $u$ is smooth in $\overline\Omega$, it is nearly singular at the midpoint  $\widetilde {\bf y}_0:=10^{-1}( 2^{-3/2},2^{-3/2},1/2)$ of the front surface.
The normal derivative $\phi=\partial_{\boldsymbol{\nu}} u$ of $u$ is given by 
\begin{equation*}
\phi({\bf x})=-\frac{1}{4\pi}\frac{{\bf x}-{\bf y}_0}{|{\bf x}-{\bf y}_0|^3}\cdot{\boldsymbol{\nu}}({\bf x}).
\end{equation*}


We consider polynomial degrees $p\in\{0,1,2\}$.
For the initial ansatz space with spline degree ${\bf p}_m:=(p,p)$ for all $m\in\{1,\dots 6\}$, we choose one single element on each patch as initial mesh, and when refining we consider the maximum continuity $C^{p-1}$ within each patch.
We choose the parameters of Algorithm~\ref{alg:abstract algorithm} as $\theta=0.5$ and $\const{min}=1$.
In the lowest-order case $p=0$, we modify the refinement strategy of Algorithm~\ref{alg:multi-patch bem} by setting for all $Q\in\QQ_{\coarsecomma m}$,
\begin{align}\label{eq:lowest BEM refinement}
&\NN_\coarse(Q):=\set{Q'\in\QQ_{\coarsecomma m}}{\overline Q\cap \overline Q'\neq\emptyset \wedge \levelT{Q'} < \levelT{Q}}\notag
\\
&\quad\cup \bigcup_{m'\neq m} \set{Q'\in\QQ_{\coarsecomma m'}}{{\rm dim}(\overline Q\cap \overline Q')=d-1},
\end{align}
i.e., within the patch we mark any coarser element which intersects $Q$, and we add adjacent elements from other patches to avoid hanging nodes.
For comparison, we also consider uniform refinement, where we mark all elements in each step, i.e., $\MM_k=\QQ_k$ for all $k\in\N_0$. 
This leads to uniform bisection of all elements.
 In Figure~\ref{fig:bendcube_mesh}, some adaptively generated hierarchical meshes are depicted.
 
To (approximately) compute the energy error, we use extrapolation:
Let $\Phi_{k}\in \mathbb{S}_k$ be the Galerkin approximation of the $k$-th step with  the corresponding coefficient vector $\boldsymbol{c}_k$,
and let $\boldsymbol{V}_k$ be the Galerkin matrix. 
With  Galerkin orthogonality~\eqref{eq:galerkin bem}, which yields that $\dual{\mathscr{V}(\phi-\Phi_k)}{\Phi_k}=0$,  and the energy norm $\norm{\phi}{\mathscr{V}}^2=\dual{\mathscr{V}\phi}{\phi}$ obtained (as, e.g., in~\cite{cp06}) by Aitken's $\Delta^2$-extrapolation, we can compute the energy error as
\begin{align}\label{eq:error calc gal bem1}
\begin{split}
\norm{\phi-\Phi_k}{\mathscr{V}}^2&=\norm{\phi}{\mathscr{V}}^2-\norm{\Phi_k}{\mathscr{V}}^2=\norm{\phi}{\mathscr{V}}^2-\boldsymbol{c}_k^\top\boldsymbol{V}_k\boldsymbol{c}_k.
\end{split}
\end{align}
In  Figure~\ref{fig:bendcube_p} and Figure~\ref{fig:bendcube_pcomp}, 
we plot   the approximated energy error $\norm{\phi-\Phi_k}{\mathscr{V}}$ and the error estimator $\eta_k$ against the  number of elements $\#\QQ_k$.  
 Although we only proved reliability \eqref{eq:reliable bem} of the employed estimator, the curves (in a double-logarithmic plot) for the error and the estimator are parallel in each case, which numerically indicates reliability and efficiency, see also Remark~\ref{rem:weak efficiency} which states efficiency in a slightly weaker sense.
Since the solution $\phi$ is smooth, the uniform and the adaptive approach both lead to the optimal asymptotic convergence rate $\mathcal{O}((\#\QQ_k)^{-3/4-p/2})$, see \cite[Corollary~4.1.34]{ss11}.
However, $\phi$ is nearly singular at $\widetilde {\bf y}_0$, which is why adaptivity yields a much better multiplicative constant.

\begin{figure}[h!]
\psfrag{x1}[c][c]{\tiny $x_1$}
\psfrag{x2}[c][c]{\tiny $x_2$}
\psfrag{x3}[c][c]{\tiny $x_3$}

\begin{center}
\subfigure[Mesh $\QQ_4$.]{
\includegraphics[width=0.23\textwidth,clip=true]{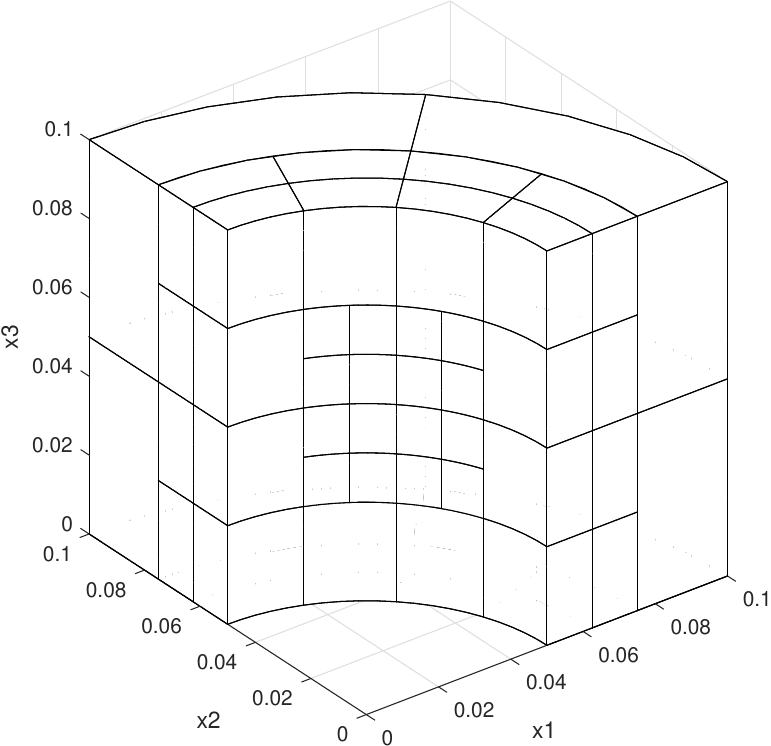}}
\subfigure[Mesh $\QQ_7$.]{
\includegraphics[width=0.23\textwidth,clip=true]{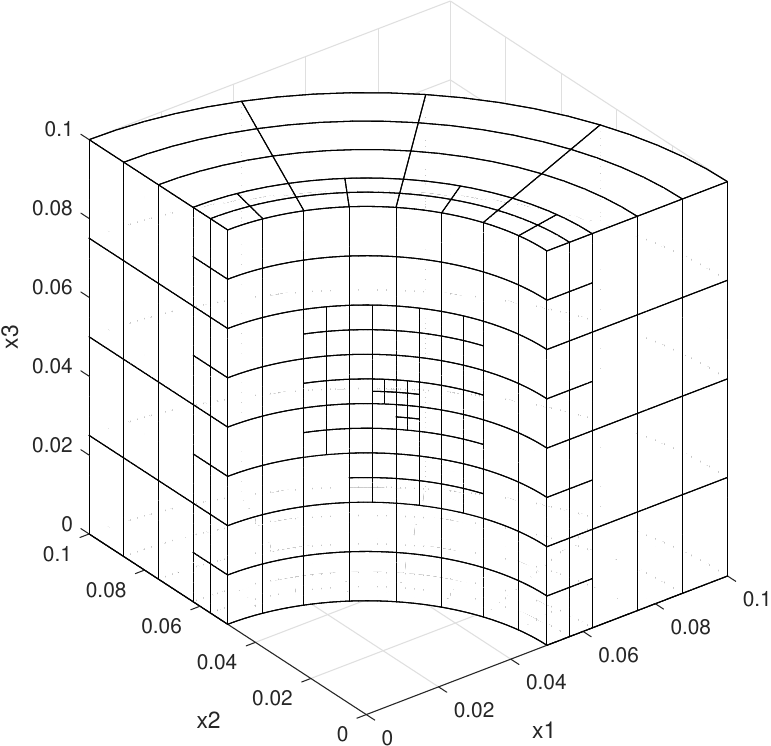}}
\subfigure[Mesh $\QQ_9$.]{
\includegraphics[width=0.23\textwidth,clip=true]{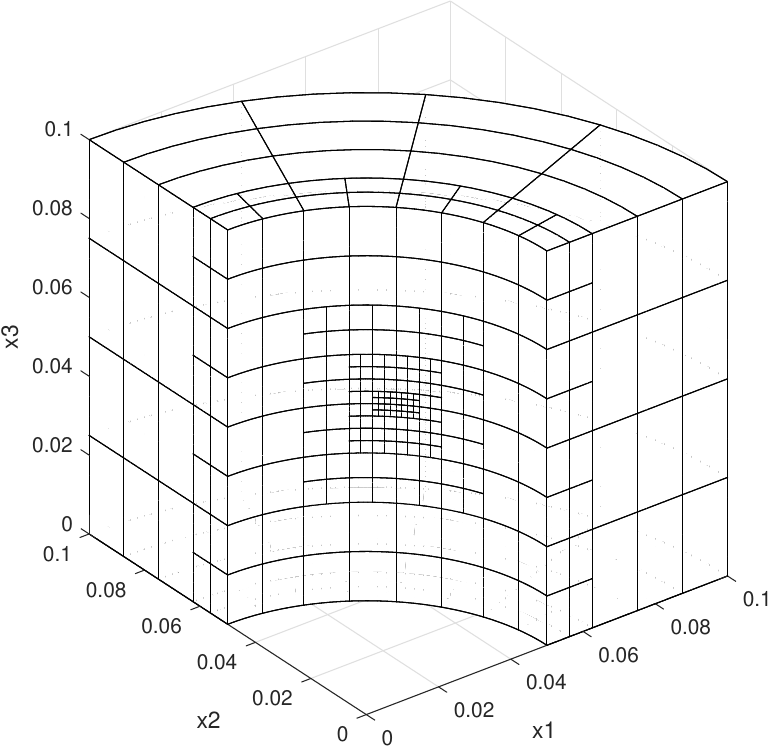}}
\subfigure[Mesh $\QQ_{10}$.]{
\includegraphics[width=0.23\textwidth,clip=true]{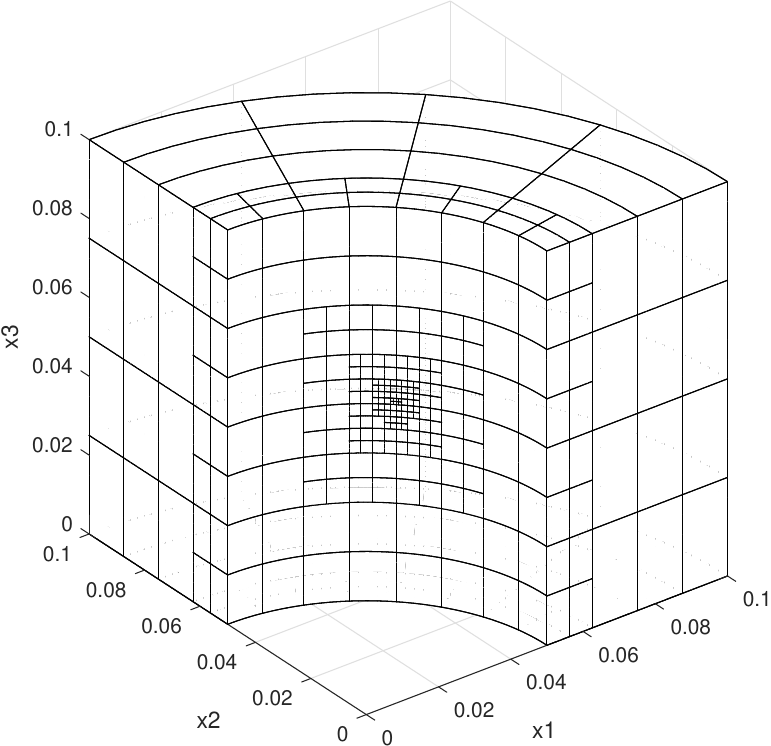}}
\end{center}
\caption{Quasi-singularity on thick ring:  
Hierarchical meshes generated by  Algorithm~\ref{alg:abstract algorithm} (with $\theta=0.5$) for hierarchical splines of degree $p=1$. }

\label{fig:bendcube_mesh}
\end{figure}

\begin{figure}

\begin{center}

\subfigure[Error and estimator for $p=0$.]{ 
\includegraphics[width=0.35\textwidth,clip=true]{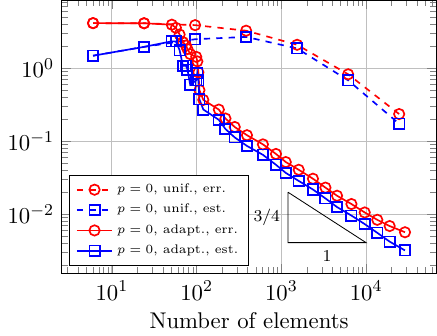}}
\subfigure[Error and estimator for $p=1$.]{ 
\includegraphics[width=0.35\textwidth,clip=true]{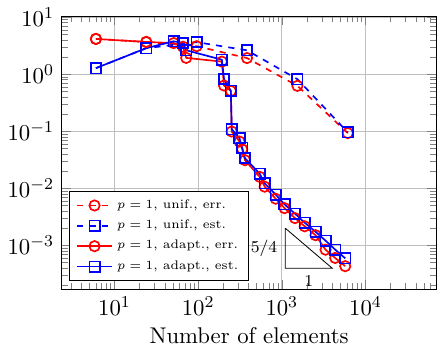}}
\subfigure[Error and estimator for $p=2$.]{ 
\includegraphics[width=0.35\textwidth,clip=true]{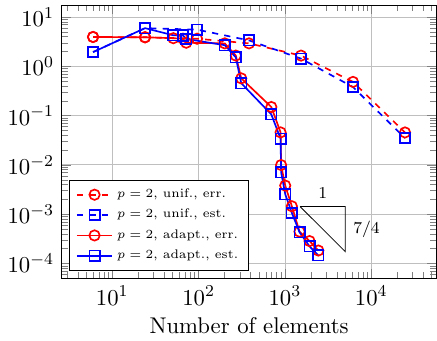}}
\end{center}

\caption{Quasi-singularity on thick ring:  
Energy error $\norm{\phi-\Phi_k}{\mathscr{V}}$ and estimator $\eta_k$ of Algorithm~\ref{alg:abstract algorithm} for hierarchical splines of degree $p$ are plotted versus the number of elements $\#\QQ_k$.
Uniform and adaptive ($\theta=0.5$) refinement is considered.}
\label{fig:bendcube_p} 
\end{figure}

\begin{figure}

\centering 
 
\includegraphics[width=0.35\textwidth]{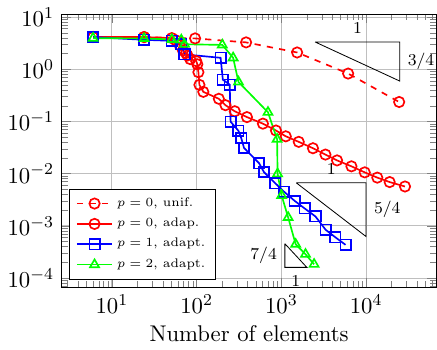}

\caption{Quasi-singularity on thick ring:  
The energy errors $\norm{\phi- \Phi_k}{\mathscr{V}}$ of Algorithm~\ref{alg:abstract algorithm} for hierarchical splines of degree $p\in\{0,1,2\}$ are plotted versus the number of elements $\#\QQ_k$.
Uniform (for $p=2$) and adaptive ($\theta=0.5$ for $p\in\{0,1,2\}$) refinement is considered.}
\label{fig:bendcube_pcomp} 
\end{figure}

\paragraph{Exterior problem on cube}
We consider the exterior Laplace--Dirichlet problem
\begin{subequations}\label{eq:Laplace exterior bem}
\begin{align}
\begin{split}
-\Delta u&=0\quad\text{in }{\R^3\setminus\overline{\Omega}},\\ u&=g\quad\text{on } \Gamma,\end{split}
\end{align}
for given Dirichlet data $g\in {H}^{1/2}(\Gamma_{})$, together with the far field radiation condition
\begin{align}
u({\bf x})=\mathcal{O}\Big(\frac{1}{|{\bf x}|}\Big)\quad\text{as }|{\bf x}|\to\infty
\end{align}
\end{subequations}
in the cube $\Omega := (0,1/10)^3$.
Then, \eqref{eq:Laplace exterior bem} is equivalent to an integral equation \eqref{eq:strong}, see, e.g., \cite[Theorem~7.15 and Theorem~8.9]{mclean00} or \cite[Section~3.4.2.2]{ss11}.
The (exterior) normal derivative $\phi:=\partial_{{\boldsymbol{\nu}}} u$ of the  weak solution $u$ of  \eqref{eq:Laplace exterior bem} satisfies~\eqref{eq:strong} with $f:=(\mathscr{K}-1/2)g$, i.e., 
\begin{equation*}
\mathscr{V}\phi =(\mathscr{K}-1/2) g,
\end{equation*}
where $\mathscr{K}$ denotes again the double-layer operator~\eqref{eq:double layer mapping}.

We choose $g:=-1$. 
Since the constant function $1$ satisfies the Laplace problem, \eqref{eq:Symmy interior} implies that   $\mathscr{K}1=-1/2$, 
and thus $f=(\mathscr{K}-1/2) g$ simplifies to $f=1$.
We expect singularities at the non-convex edges of $\R^3\setminus\overline{\Omega}$, i.e., at all edges of the cube $\Omega$.

The boundary of the cube is trivially represented by six bilinear patches. 
Again, we consider $p\in\{0,1,2\}$ and discrete spaces of splines of degree ${\bf p}_m:=(p,p)$ for all $m\in\{1,\dots 6\}$ with one single element per patch as initial mesh, and when refining we consider the maximum continuity $C^{p-1}$ across the elements within the patch. 
We choose the parameters of Algorithm~\ref{alg:abstract algorithm} as $\theta=0.5$ and $\const{min}=1$, where we use again \eqref{eq:lowest BEM refinement} in the lowest-order case $p=0$.
For comparison, we also consider uniform refinement, where we mark all elements at each step, i.e., $\MM_k=\QQ_k$ for all $k\in\N_0$. 
This leads to uniform bisection of all elements.
 
 In Figure~\ref{fig:cube_mesh}, some adaptively generated hierarchical meshes are depicted.
 In  Figure~\ref{fig:cube_p} and Figure~\ref{fig:cube_pcomp}, 
we plot   the approximated energy error $\norm{\phi-\Phi_k}{\mathscr{V}}$ (see~\eqref{eq:error calc gal bem1}) and the error estimator $\eta_k$ against the  number of elements $\#\QQ_k$.  
 In all cases, the lines of the error and the error estimator are parallel, which numerically indicates reliability 
 and efficiency.
The uniform approach always leads to the suboptimal  convergence rate $\mathcal{O}((\#\QQ_k)^{-1/3})$  due to the edge singularities.
Independently on the chosen polynomial degree $p$, the adaptive approach leads approximately to the rate $\mathcal{O}((\#\QQ_k)^{-1/2})$.
For smooth solutions $\phi$, one would expect the rate  $\mathcal{O}((\#\QQ_k)^{-3/4-p/2})$, see \cite[Corollary~4.1.34]{ss11}.
However, according to Theorem~\ref{thm:H-igabem}, the achieved rate is optimal if one uses the proposed refinement strategy and the resulting hierarchical splines.
The reduced optimal convergence rate is likely due to the edge singularites.
A similar reduced convergence behavior has also been observed in \cite{fp08}
for the lowest-order case $p=0$ and in Section~\ref{sec:numerical igafem} in case of IGAFEM.
\cite{ferraz07} additionally considers anisotropic refinement, which recovers the optimal rate $\mathcal{O}((\#\QQ_k)^{-3/4})$.

\begin{figure}[h!] 
\psfrag{x1}[c][c]{\tiny $x_1$}
\psfrag{x2}[c][c]{\tiny $x_2$}
\psfrag{x3}[c][c]{\tiny $x_3$}

\begin{center}
\subfigure[Mesh $\QQ_8$.]{ 
\includegraphics[width=0.23\textwidth,clip=true]{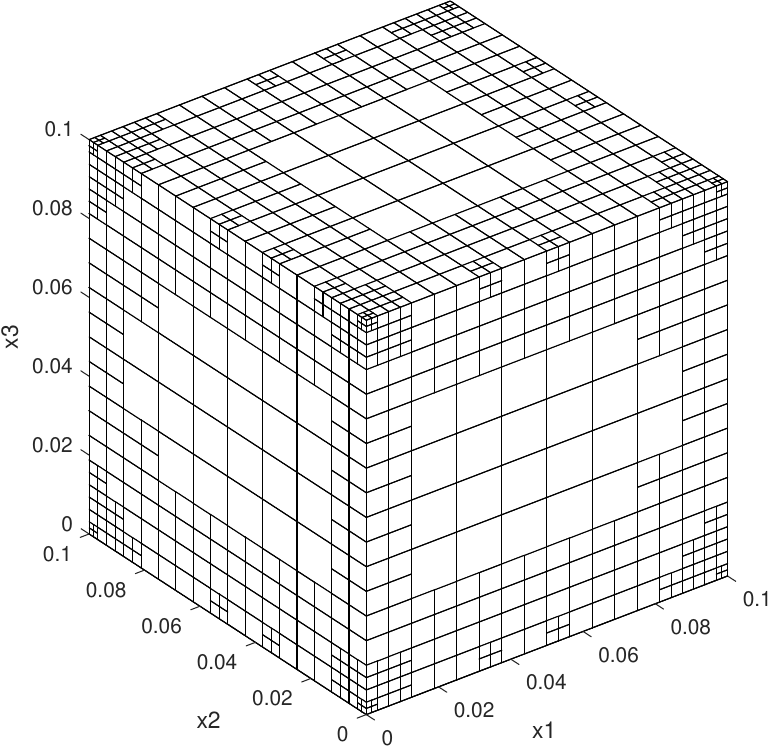}}
\subfigure[Mesh $\QQ_{10}$.]{ 
\includegraphics[width=0.23\textwidth,clip=true]{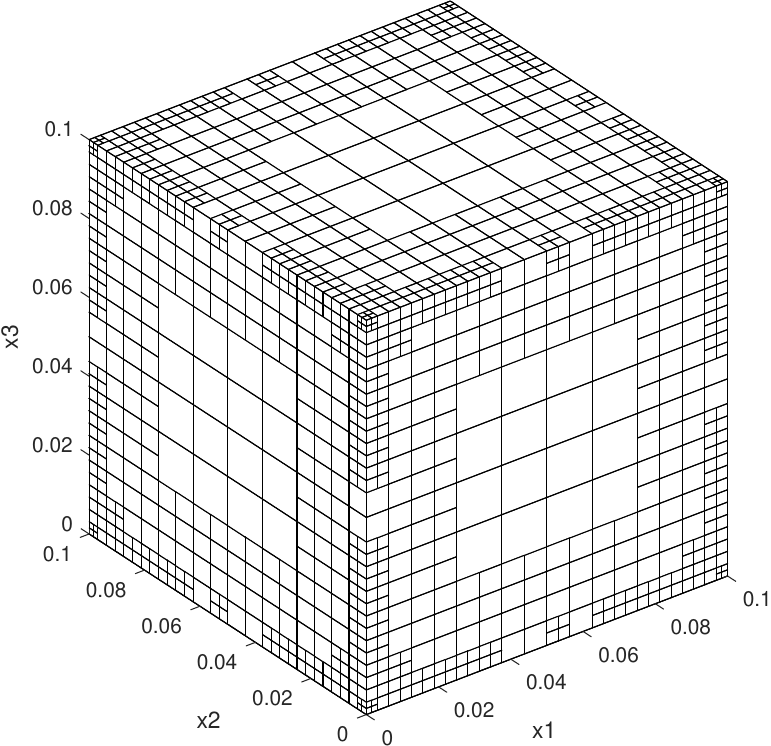}}
\subfigure[Mesh $\QQ_{11}$.]{ 
\includegraphics[width=0.23\textwidth,clip=true]{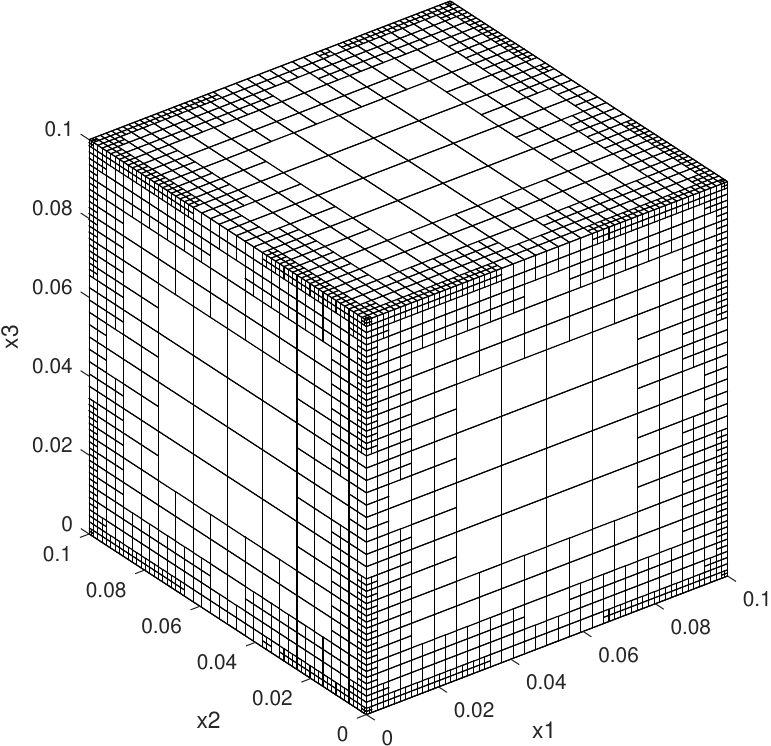}}
\subfigure[Mesh $\QQ_{13}$]{ 
\includegraphics[width=0.23\textwidth,clip=true]{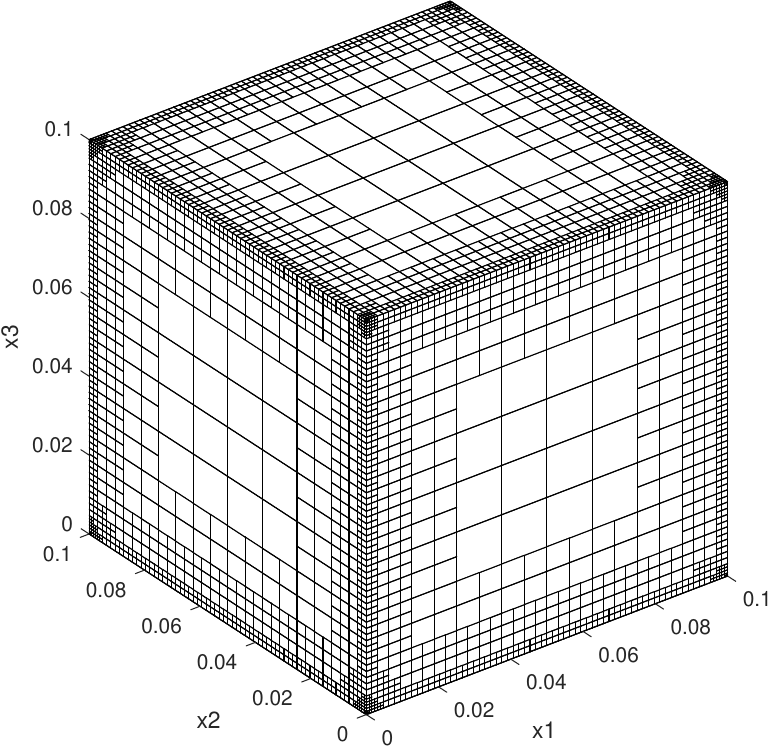}}
\end{center}
\caption{Exterior problem on cube:
Hierarchical meshes generated by  Algorithm~\ref{alg:abstract algorithm} (with $\theta=0.5$) for hierarchical splines of degree $p=1$. }

\label{fig:cube_mesh}
\end{figure}

\begin{figure}

\begin{center}
 
\subfigure[Error and estimator for $p=0$.]{ 
\includegraphics[width=0.35\textwidth,clip=true]{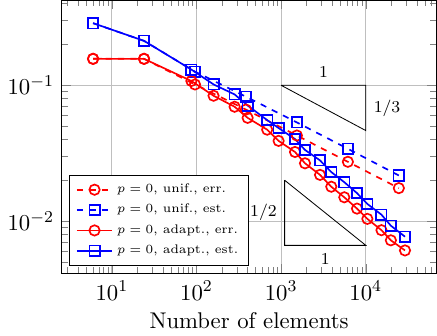}}
\subfigure[Error and estimator for $p=1$.]{ 
\includegraphics[width=0.35\textwidth,clip=true]{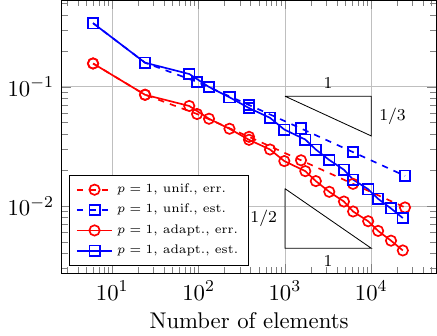}}
\subfigure[Error and estimator for $p=2$.]{ 
\includegraphics[width=0.35\textwidth,clip=true]{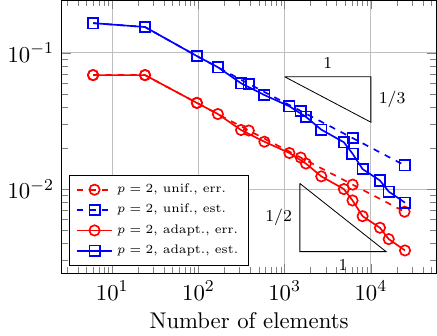}}
\end{center}

\caption{Exterior problem on cube:
Energy error $\norm{ \phi- \Phi_k}{\mathscr{V}}$ and estimator $\eta_k$ of Algorithm~\ref{alg:abstract algorithm} for hierarchical splines of degree $p$ are plotted versus the number of elements $\#\QQ_k$.
Uniform  and adaptive ($\theta=0.5$) refinement is considered.}
\label{fig:cube_p} 
\end{figure}

\begin{figure}

\centering 
 
\includegraphics[width=0.35\textwidth]{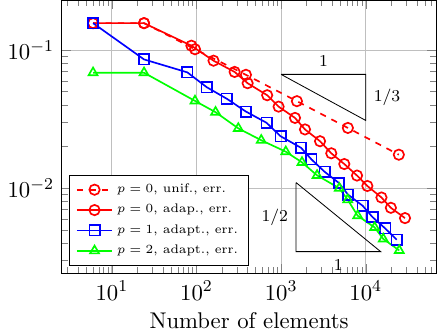}

\caption{Exterior problem on cube:
The energy errors $\norm{\phi- \Phi_k}{\mathscr{V}}$ of Algorithm~\ref{alg:abstract algorithm} for hierarchical splines of degree $p\in\{0,1,2\}$ are plotted versus the number of elements $\#\QQ_k$.
Uniform (for $p=0$) and adaptive ($\theta=0.5$ for $p\in\{0,1,2\}$) refinement is considered.}
\label{fig:cube_pcomp} 
\end{figure}

\subsection{Adaptive IGABEM with T-splines}\label{sec:T-igabem}
We start defining T-splines on the multi-patch boundary $\Gamma$. 
Note that these T-splines do not need to be continuous across interfaces as we consider the weakly-singular integral equation. 
In contrast to the common approach in the engineering literature, where T-spline functions may be smooth across patches, see, e.g., \cite{ScSiEvLiBoHuSe13,KoGiPoKa15,SiScTaThLi14}, we define them separately on each patch, see also Remark~\ref{rem:T-adaptive Neumann}.  
For each $m=1,\dots,M$, let $\mathbf{p}_m$ be a vector of positive polynomial degrees and $\mathbf{\kv}^0_m$ be a  multivariate open knot vector on $\widehat\Gamma=(0,1)^{\dpa}$, $\dpa=d-1\ge 2$, with induced initial index T-mesh $\check\QQ_{0,m}$. 
We assume that $\hat{\mathbb{S}}_{{\bf p}_m} (\mathbf{\kv}^0_m)$ and $\hat{\mathbb{S}}_{{\bf p}_{\F_m}} (\mathbf \kv_{\F_m})$ with ${\bf p}_{\F_m}$ and ${\bf T}_{\F_m}$ from the parametrization  ${\bf F}_m:\widehat\Gamma\to\Gamma_m$ (see Section~\ref{sec:multi-patch}) are compatible to each other as in Section~\ref{sec:IGA-basics}.
Note that $\hat{\mathbb{S}}_{{\bf p}_m} (\mathbf{\kv}^0_m)=\hat{\mathbb{S}}^{\rm T}_{{\bf p}_m}(\check\QQ_{0,m},{\bf \kv}_m^0)$.
Moreover, we assume for the initial mesh $\QQ_0=\bigcup_{m=1}^M \QQ_{0,m}$ with $\QQ_{0,m}:=\set{\F_m(\hat Q)}{\hat Q\in\hat\QQ_{0, m}}$ that there are no hanging nodes between patch interfaces $\Gamma_{m,m'}=\overline{ \Gamma_m} \cap \overline {\Gamma_{m'}}$ with $m\neq m'$, see also \eqref{P:conforming-mesh} of Section~\ref{sec:multi-patch}.
For each $m=1,\dots,M$, we abbreviate the set of all corresponding admissible meshes as $\widehat\Q_m$, see Section~\ref{sec:T refine}.
Moreover, we abbreviate $\Q_m:=\set{\QQ_{\coarsecomma m}}{\hat \QQ_{\coarsecomma m}\in\hat\Q_m}$ with $\QQ_{\coarsecomma m}:=\set{\F_m(\hat Q)}{\hat Q\in\hat\QQ_{\coarsecomma m}}$.
The index T-mesh corresponding to $\QQ_{\coarsecomma m}$ is denoted by $\check\QQ_{\coarsecomma m}$, see Remark~\ref{rem:paramtric to index}.
We define the set of all admissible meshes $\Q$ as the set of all 
\begin{align*}
\QQ_\coarse=\bigcup_{m=1}^M \QQ_{\coarsecomma m} \text{ with }\QQ_{\coarsecomma m}\in\Q_m
\end{align*} 
 such that $|{\rm lev}(Q)-{\rm lev}(Q')|\le1$ whenever $Q\in\QQ_{\coarsecomma m}, Q'\in\QQ_{\coarsecomma m'}$ with $m\neq m'$ and $\overline Q\cap\overline Q'\neq\emptyset$.

For $\QQ_\coarse\in\Q$, the associated ansatz space is defined as
\begin{align*}
\begin{split}
\mathbb{S}_\coarse:=\big\{ V \in L^2(\Gamma) : V|_{\Gamma_m} \in {\mathbb{S}}^{\rm T}_{{\bf p}_m}(\check\QQ_{\coarsecomma m},{\bf \kv}_m^0), \quad \\
\text{ for } m = 1, \ldots, M\big\},
\end{split}
\end{align*}
where
\begin{align*}
{\mathbb{S}}^{\rm T}_{{\bf p}_m}(\check\QQ_{\coarsecomma m},{\bf \kv}_m^0):=\set{\hat V\circ\F_m^{-1}}{\hat V\in\hat{\mathbb{S}}^{\rm T}_{{\bf p}_m}(\check\QQ_{\coarsecomma m},{\bf \kv}_m^0)}.
\end{align*}
To obtain a basis of the space $\mathbb{S}_\coarse$, we first define  
\begin{align*}
B_{\coarsecomma m,\anchor,\mathbf{p}_m}:=\hat B_{\coarsecomma m,\anchor,\mathbf{p}_m}\circ\F_m^{-1}\\
\end{align*}
for all anchors $\anchor \in \nodes_{{\bf p}_m}(\Tmesh_{\coarsecomma m},{\bf T}^0_m)$,
where $\hat B_{\coarsecomma m,\anchor,\mathbf{p}_m}$ is defined as in \eqref{eq:T-spline}.
Since the ansatz functions do not have to be continuous across interfaces, a basis of $\mathbb{S}_\coarse$ is given via
\begin{align*}
\mathbb{S}_\coarse&={\rm span}\Big( \bigcup_{m=1}^M \set{B_{\coarsecomma m,\anchor,\mathbf{p}_m}}{\anchor \in \nodes_{\mathbf{p}_m}(\Tmesh_{\coarsecomma m},{\bf T}^0_m)}\Big),
\end{align*}
where we extend the involved T-spline blending functions, which actually only live on $\Gamma_m$, by zero to the whole boundary $\Gamma$. 
%
\begin{remark}\label{rem:T-adaptive Neumann}
In contrast to weakly-singular integral equations, hypersingular integral equations, which result from Neumann problems (see e.g. \cite[Chapter~7]{mclean00}), require continuous trial functions. 
While the construction of continuous T-splines across patches has been already used in other works, see for instance \cite{ScSiEvLiBoHuSe13,KoGiPoKa15}, the extension of the refinement algorithm with admissible meshes in Section~\ref{sec:T refine} to the multi-patch case is not evident, because the alternate directions of bisection may differ from patch to patch.
%
\end{remark}

To obtain admissible meshes starting from the initial one, we adapt the single-patch refinement strategy from Section~\ref{sec:T refine}:
For arbitrary $\QQ_\coarse\in\Q$ and $Q\in\QQ_{\coarsecomma m}$ let us denote by $\hat Q:=\F_m^{-1}(Q)$ and $\check Q$ the corresponding elements in the parametric domain and in the index domain, respectively, and let $\NN_{\coarsecomma m}(\check Q)\subseteq\check\QQ_{\coarsecomma m}$ denote the corresponding neighborhood, see Section~\ref{sec:T refine}.
Recall that each element in $\NN_{\coarsecomma m}(\check Q)$ lies in the index/parametric domain.
We define the \emph{neighbors} of $Q$ as
\begin{align*}
 &\NN_\coarse(Q):=\set{Q'\in\QQ_{\coarsecomma m}}{\check Q' \in \NN_{\coarsecomma m}(\check Q)}\\
 &\cup \bigcup_{m'\neq m} \set{Q'\in\QQ_{\coarsecomma m'}}{\overline Q\cap \overline Q'\neq\emptyset\wedge{\rm lev}(Q)>{\rm lev}(Q')},
\end{align*}
i.e., apart from the standard neighbors within the patch, we add (as already suggested in Remark~\ref{rem:H-adaptive Neumann} for HB-splines) neighbor elements from other patches of a coarser level.

With this notation, we can employ Algorithm~\ref{alg:multi-patch bem} of Section~\ref{sec:H-igabem} for refinement.
Then, one can show that the set of all possible refinements $\refine(\QQ_0)$ coincides with $\Q$. 
Such a result is proved in \cite[Proposition~5.4.3]{gantner17} for the analogous case of HB-splines on $\HH$-admissible meshes of class $\mu=2$.
The proof easily extends to T-splines on admissible T-meshes.

The given setting fits into the abstract framework of Section~\ref{sec:abem}.
We stress that this result is new, but follows quite easily from \cite{gp18}, where IGAFEM with T-splines has been considered.
We only sketch the proof in Section~\ref{sec:T-bem space verification}.
Note that most of the properties have already been verified in Section~\ref{sec:T-igafem} for IGAFEM-meshes.

\begin{theorem}\label{thm:T-igabem}
T-splines on admissible meshes satisfy the mesh properties  \eqref{M:shape}--\eqref{M:locuni},  the refinement properties \eqref{R:childs}--\eqref{R:overlay}, and the space properties \eqref{S:nestedness}--\eqref{S:local}, \eqref{S:proj bem}, and \eqref{S:inverse bem}--\eqref{S:stab bem}.
The involved constants depend only on the dimension $\dph$, the parametrization constants $C_{\F_m}$ of Section~\ref{sec:parametrization_assumptions}, the degree $\mathbf{p}_m$, and the initial knot vector ${\bf T}_m^0$.
In particular, Theorem~\ref{thm:abstract bem} is applicable. 
In conjunction with Theorem~\ref{thm:abstract main}, this yields reliability~\eqref{eq:reliable bem} and  linear convergence at optimal rate~\eqref{eq:linear convergence}--\eqref{eq:optimal convergence} of the residual error estimator~\eqref{eq:eta bem}, when the adaptive Algorithm~\ref{alg:abstract algorithm} is employed.
\end{theorem}



\subsubsection{Mesh, refinement, and space properties}
\label{sec:T-bem space verification}
The mesh properties  \eqref{M:shape}--\eqref{M:locuni} follow as for IGAFEM in Section~\ref{sec:T-fem mesh verification}.
The child estimate \eqref{R:childs} is trivially satisfied.
The closure estimate \eqref{R:closure} can be proved similarly as in the single-patch case~\cite[Section~6]{mp15}.
The overlay in \eqref{R:overlay} can be built patch-wise as in Section~\ref{sec:T-fem refinement verification}. 
The properties \eqref{S:nestedness}--\eqref{S:local} follow as in Section~\ref{sec:T-fem space verification}. 

For the analogous case of hierarchical B-splines on $\HH$-admissible meshes of class $\mu=2$, the proof of  \eqref{S:inverse bem} is given in \cite[Section~5.5.9]{gantner17}, which itself strongly builds on a similar result on triangular meshes \cite{dfghs04}.
However, we stress that the proof only hinges on the mesh properties \eqref{M:shape}--\eqref{M:locuni} and the fact that hierarchical splines are polynomials on all elements in the parametric domain.  
Indeed, it only requires the considered functions to be polynomials on a rectangular subset of the same size as the element.
Since there are at most two B\'ezier elements on each element (see Lemma~\ref{lemma:2beziers}),  the result thus easily extends to T-splines on admissible T-meshes.

The reference \cite[Proposition~5.5.5]{gantner17} states that the local approximation of unity property \eqref{S:unity bem} is satisfied if there exists a finite subset $\mathcal{B}\subset \mathbb{S}_\coarse$ whose elements 
are non-negative, local in the sense that for all $\beta\in\mathcal{B}$ there exists $Q\in\QQ_\coarse$ and a uniform constant $q\in\N$ such that $\supp(\beta)\subseteq \pi_\coarse^q(Q)$, and form a partition of unity.
According to Proposition~\ref{prop:tsp-sext} and Proposition~\ref{prop:Tsplines basis} together with Proposition~\ref{prop:T-nestedness}, these assumptions are fulfilled for T-spline basis functions 
\begin{align*}
\BB:= \bigcup_{m=1}^M \set{B_{\coarsecomma m,\anchor,\mathbf{p}_m}}{\anchor \in \nodes_{\mathbf{p}_m}(\Tmesh_{\coarsecomma m},{\bf T}^0_m)},
\end{align*}
on admissible meshes.

\paragraph{Scott--Zhang type operator}
Since the ansatz functions do not have to be continuous at interfaces and due to the regularity of the parametrization ${\bf F}_m$ of Section~\ref{sec:parametrization_assumptions}, it is sufficient to provide for each patch $\Gamma_m$ and $\hat\SS_m\subseteq\hat\QQ_{\coarsecomma m}$ an operator
\begin{align*}
\hat\JJ_{\coarsecomma m,\hat\SS_m}:L^2(\hat\Gamma)\to\big\{\hat\Psi_{\coarsecomma m}\in \hat{\mathbb{S}}^{\rm T}_{{\bf p}_m}(\check\QQ_{\coarsecomma m},{\bf \kv}_m^0): \\
\hat\Psi_{\coarsecomma m}|_{\bigcup(\hat\QQ_{\coarsecomma m}\setminus\hat\SS_m)}=0\big\}
\end{align*}
satisfying \eqref{S:proj bem} and \eqref{S:stab bem}.
We define this operator similarly as $\projT{\mathbf{p}}$ in Section~\ref{sec:dual-compatible}, but now have to take into account that the output should only live on $\bigcup \hat\SS_m$ 
by discarding all T-spline blending functions that have support entirely outside of this set. 
Then the local projection property \eqref{S:proj bem} as well as the local $L^2$-stability \eqref{S:stab bem} for the operator $\hat\JJ_{\coarsecomma m,\hat\SS_m}$ can be shown as in Section~\ref{sec:T-fem space verification}. 
A detailed analogous proof is given for hierarchical splines on $\HH$-admissible meshes of class $\mu=2$ in \cite[Section~5.5.14]{gantner17}. 
Indeed, the proof could essentially be copied, replacing THB-splines and their corresponding dual functionals by T-spline basis functions and their dual functionals.

%



\subsection{Adaptive IGABEM in 2D with smoothness control}
\label{sec:elementary_bem}

Finally, we briefly summarize results from \cite{fgp15,fghp16,fghp17,gps19}, where a slightly modified adaptive IGABEM in 2D has been studied, which additionally controls the smoothness of the one-dimensional spline ansatz functions. 
This control is achieved by using $h$-refinement together with multiplicity increase of the knots, which reduces the regularity of the basis functions.
This combination allows to automatically resolve strong singularities but also simple discontinuities, which can both lead to a reduced convergence rate for uniform refinement. 
We note that in the 2D case, the boundary is  one-dimensional, and adaptive refinement can be done using standard B-splines (or NURBS). The extension to the 3D setting using HB-splines or T-splines is far from being straightforward, and it has not been studied yet.

In this work, we restrict ourselves to the weakly-singular integral equation arising from Dirichlet problems, and we only allow reduction of the smoothness by multiplicity increase, but we remark that \cite{gps19} also allows multiplicity decrease  and analyzes both the weakly-singular integral equation, which we consider here, and the hypersingular integral equation arising from Neumann problems. 
For both cases, an optimal additive Schwarz preconditioner has been introduced in \cite{fgps19} for the Laplace problem, 
i.e., it is proved that the preconditioned Galerkin systems have a uniformly bounded condition number being independent of the local mesh-refinement and the smoothness of the B-spline ansatz functions. 
An important consequence is that the PCG solver is uniformly contractive, and analogously to the FEM case with hierarchical splines explained in Remark~\ref{rem:PCG-HB-FEM}, this allows to prove that an adaptive algorithm combining adaptive refinement with an inexact PCG solver leads to optimal convergence with respect to the number of elements and also with respect to the overall computational cost, see \cite{fhps19} for details.

\subsubsection{Setting of the discrete problem}\label{sec:igabem1d setting}
Let $\Omega\subset\R^2$ be a Lipschitz domain with connected  boundary $\Gamma\subset\R^2$. 
We consider boundary integral equations as in Section~\ref{sec:model problem bem}. 
We assume that there exists a global NURBS parametrization ${\bf F}:[0,1]\to\Gamma$ (see Section~\ref{sec:parametrization}) such that ${\bf F}|_{[0,1)}$ is bijective with ${\bf F}(0)={\bf F}(1)$.
Moreover, we denote the knot vector associated to ${\bf F}$ by $\kv_{\bf F}$, and its induced mesh on $[0,1]$ by $\hat\QQ_{\bf F}$, and further assume that ${\bf F}|_{\hat Q}$ is bi-Lipschitz for all $\hat Q\in\hat\QQ_{\bf F}$.
Let  $p\in\N_0$ be a fixed polynomial degree. 
We consider $p$-open knot vectors $\kv_\coarse$ on $[0,1]$ (see Section~\ref{sec:univariate-properties})
with 
$\VV_{\bf F}\subseteq\VV_\coarse$, 
where $\VV_{\bf F}$ and $\VV_\coarse$ respectively denote the set of vertices corresponding to $\kv_{\bf F}$ and $\kv_\coarse$, which are defined as the images of all breakpoints (see~\ref{sec:univariate-properties}) under ${\bf F}$. 
We define the space of all splines  on $[0,1]$ and $\Gamma$ as 
\begin{align*}
\widehat{\mathbb{S}}_\coarse&:=\spu{p}{\kv_\coarse},
\\ 
{\mathbb{S}}_{\coarse}&:=\set{ \Psi_{\coarse}\circ{\bf F}^{-1}}{ \Psi_{\coarse}\in{\mathbb{S}}_{\coarse}}
\subset L^2(\Gamma)\subset H^{-1/2}(\Gamma).
\end{align*}
Note that the functions in $\mathbb{S}$ are allowed to be discontinuous at the initial vertex ${\bf F}(0)={\bf F}(1)$.
We consider B-splines, transformed via ${\bf F}$, as the basis of the space $\mathbb{S}_\coarse$.

\subsubsection{Refinement  of knot vectors}\label{subsec:concrete refinement bem2}

Let $\kv_0$ be a fixed initial $p$-open knot vector with $\VV_{\bf F}\subseteq\VV_0$.
With the corresponding mesh $\hat\QQ_0$ in $[0,1]$, define the initial shape-regularity constant
\begin{align*}
\widehat\gamma_0:=\max \Big\{\frac{|\widehat Q|}{|\widehat Q'|}:\widehat Q,\widehat Q'\in \widehat\QQ_0\text{ with }\overline Q\cap \overline {Q'}\neq\emptyset\Big\},
\end{align*}
where $Q={\bf F}(\widehat Q)$ and $Q'={\bf F}(\widehat Q')$. 
We recall that, for meshes $\QQ_\coarse$ on $\Gamma$ corresponding to $\kv_\coarse$ and for an element $Q \in \QQ_\coarse$, the element-patch $\Pi_\coarse(Q)$ of \eqref{eq:patch2} is given by the element itself and its adjacent neighbors. 
First, we formulate the auxiliary refinement Algorithm~\ref{alg:refinement bem21} taken from \cite{affkp13}, 
which focusses on plain $h$-refinement, but ensures shape-regularity for the refined meshes.

\begin{algorithm}[!h]
\caption{\texttt{refine\_h} ($h$-refinement of knot vector)}  
\label{alg:refinement bem21}
\begin{algorithmic}
\Require $p$-open knot vector  $\kv_\coarse$ , marked elements $\MM\subseteq\QQ_\coarse$.
\Repeat
\State set $\displaystyle {\mathcal{U}} = \bigcup_{{Q} \in \MM} \set{Q'\in\Pi_\coarse(Q)}{|\widehat Q'|>\widehat\gamma_0|\widehat Q|} \setminus\MM$
\State set $\MM = \MM\cup {\mathcal{U}}$
\Until {${\mathcal{U}} = \emptyset$}
\State update $\kv_\coarse$ by bisecting (i.e., adding a new knot of multiplicity $1$ to) all elements in ${\cal{M}}$
\Ensure refined $p$-open knot vector $\kv_\coarse$ 
\end{algorithmic}
\end{algorithm}

The refinement strategy Algorithm~\ref{alg:refinement bem22} will be used to steer a modified version of the adaptive Algorithm~\ref{alg:abstract algorithm}. 
In contrast to all refinement strategies in previous sections, it receives marked vertices instead of marked elements as input and also uses knot multiplicity increase for refinement.

For any vertex ${\bf z} \in \VV_\coarse$, we denote by ${\bf z}_r$ the vertex right to ${\bf z}$ with respect to the orientation of ${\bf F}$. 
We denote by $z:={\bf F}|_{[0,1)}^{-1}({\bf z}), z_r:={\bf F}|_{(0,1]}^{-1}({\bf z}_r)$ the corresponding breakpoints, and with some abuse of notation, by $({\bf z},{\bf z}_r):={\bf F}\big((z,z_r)\big)$ the unique element delimited by the two vertices. 
The refinement strategy in Algorithm~\ref{alg:refinement bem22} does the following: If both vertices of an element are marked, the element is marked for refinement via Algorithm~\ref{alg:refinement bem21}. For all other vertices (i.e., those that are not stored in ${\cal R}$) the multiplicity is increased if it is less than $p+1$, otherwise the neighboring elements are marked. 
Clearly, $\kv_\fine = \refine(\kv_\coarse,\MM)$ is finer than $\kv_\coarse$, in the sense that $\kv_\coarse$ is a subsequence of $\kv_\fine$ and thus $\mathbb{S}_\coarse \subseteq \mathbb{S}_\fine$. 
For any $p$-open knot vector $\kv_\coarse$, we define $\refine(\kv_\coarse)$ as the set of all $p$-open knot vectors $\kv_\fine$ that can be obtained by iterative application of $\refine$.
We define the set of all \textit{admissible} $p$-open knot vectors 
\begin{align*}
\mathbb{T}:=\refine(\kv_0).
\end{align*}
It is easy to see that $\mathbb{T}$ coincides with the set of all $p$-open knot vectors $\kv_\coarse$ which are obtained via iterative bisections in the parametric domain and arbitrary knot multiplicity increases such that  
\begin{align*}
|\widehat Q|/|\widehat Q'|\le 2\widehat\gamma_0 \quad\text{for all }  Q,Q'\in\QQ_\coarse \text{ with }Q\cap Q'\neq \emptyset.
\end{align*}
Indeed, by marking in each step both vertices of an element,  Algorithm~\ref{alg:refinement bem22} can realize Algorithm~\ref{alg:refinement bem21}, which can generate according to~\cite[Theorem~2.3]{affkp13} arbitrary bisected meshes satisfying the latter local quasi-uniformity. By marking iteratively only one vertex, it is possible to arbitrarily increase the resulting knot multiplicities.

\begin{algorithm}[!h]
\caption{ \refine \ (Refinement of knot vector) }
\label{alg:refinement bem22}
\begin{algorithmic}
\Require $p$-open knot vector  $\kv_\coarse$, marked vertices $\MM\subseteq\VV_\coarse$
\LineComment{initialize marked elements $\MM'$}
\State set $\MM' = \emptyset$, $\RR=\emptyset$ 
\For{${\bf z} \in\MM$}
\If{${\bf z}_r\in\MM$}  set $\MM' = \MM'\cup\{({\bf z},{\bf z}_r)\}$, $\RR=\RR\cup\{{\bf z}, {\bf z}_r\}$ 
\EndIf
\EndFor
\For{${\bf z} \in\MM\setminus \RR$}
\If{$\#_\coarse z< p+1$} set $\#_\coarse z = \#_\coarse z +1$
\Else{ set $\MM'=\MM'\cup\Pi_\coarse(\{{\bf z}\})$}
\EndIf
\EndFor
\State update $\kv_\coarse={\tt refine\_h}(\kv_\coarse,\MM')$ 
\Ensure refined $p$-open knot vector $\kv_\coarse$
\end{algorithmic}
\end{algorithm}

\subsubsection{Adaptive algorithm}
\label{sec:1D algorithm}
Let $\kv_\coarse\in\mathbb{T}$ with vertices $\VV_\coarse$.
We consider a vertex-based version of the weighted-residual {\sl a~posteriori} error estimator~\eqref{eq:eta bem}
\begin{subequations}\label{eq:eta bem2}
\begin{align}
\begin{split}
 &\eta_\coarse := \eta_\coarse(\VV_\coarse)
 \\
 &\quad\text{with}\quad 
 \eta_\coarse(\SS)^2:=\sum_{{\bf z}\in\SS} \eta_\coarse({\bf z})^2
 \text{ for all }\SS\subseteq\VV_\coarse,
\end{split}
\end{align}
where, for all ${\bf z}\in\VV_\coarse$, the local refinement indicators read, with $\pi_\coarse(\{{\bf z}\})=\bigcup\set{Q\in\QQ_\coarse}{{\bf z}\in \overline{Q}}$,
\begin{align}
\eta_\coarse({\bf z})^2:=|\pi_\coarse({\bf z})| \seminorm{f-\mathscr{V}\Phi_\coarse}{H^1(\pi_\coarse({\bf z}))}^2.
\end{align}
\end{subequations}
The refinement strategy in Algorithm~\ref{alg:refinement bem22} and the given vertex-based error estimator give rise to a modified version of Algorithm~\ref{alg:abstract algorithm},
namely Algorithm~\ref{alg:1D algorithm}, which uses the same solving step, but computes indicators associated to vertices instead of elements,  marks vertices via the D\"orfler criterion~\eqref{eq:Doerfler}, and refines via Algorithm~\ref{alg:refinement bem22} based on these marked vertices. 

\begin{algorithm}
\caption{(Adaptive algorithm)} \label{alg:1D algorithm}
\begin{algorithmic}
\Require  initial knot vector $T_0$, marking parameter $\theta\in(0,1]$, marking constant $\const{min}\in [1,\infty]$
\For{$k=0,1,2,\dots$}
 \LineComment{compute Galerkin approximation $\Phi_k$}
\State set $\Phi_k=$ \texttt{solve}(${\cal Q}_k$) 
 \LineComment{compute refinement indicators $\eta_\k({{\bf z}})$ for all ${\bf z} \in {\cal V}_k$}
\State set $\eta_k =$ \texttt{estimate}(${\cal V}_k$, $\Phi_k$)
\LineComment{determine $\const{min}$-minimal set of vertices with \eqref{eq:Doerfler}}  
\State set $\MM_k =$ \texttt{mark}($\eta_k,{\cal V}_k$)
\LineComment{generate refined mesh with Algorithm~\ref{alg:refinement bem22}}
\State set $T_{\k+1} =$ \texttt{refine}$(T_k,{\cal M}_k)$ 
\EndFor
\Ensure refined meshes $T_\k$,  
quantities $\Phi_k$, 
 estimators $\eta_\k$ for all $\k \in\N_0$
\end{algorithmic}
\end{algorithm}



\subsubsection{Optimal convergence for one-dimensional splines}
As in Section~\ref{sec:BEM axioms}, we say that the solution $\phi \in H^{-1/2}(\Gamma)$ lies in the \textit{approximation class $s$ with respect to the estimator}~\eqref{eq:eta bem2} if
\begin{align*}
\widetilde C_{\rm apx}(s):= \sup_{N\ge\#T_0} \min_{T_\coarse\in\mathbb{T}(N)}(N^s\eta_\coarse)<\infty,
\end{align*}
with $\mathbb{T}(N):=\set{T_{\coarse}\in\mathbb{T}}{\# T_{\coarse}\le N}$ and $\# T_{\coarse}$ is the sum of all knot multiplicities in $T_\coarse$. 
In the notation of Section~\ref{sec:splines-univariate}, it holds that $\# T_{\coarse}= n+p+1$.
By definition, $\widetilde C_{\rm apx}(s)<\infty$  implies that the error estimator $\eta_\coarse$
decays at least with rate $\OO\big((\# T_\coarse)^{-s}\big)$ on the optimal knot vectors $T_\coarse$. 
The following theorem, which mainly stems from \cite[Theorem~3.2]{fghp17}, states that each possible rate $s>0$ is in fact realized by Algorithm~\ref{alg:1D algorithm}.
Theorem~\ref{thm:abstract bem2} \eqref{item:qabstract reliable bem2}  states reliability, which was  verified for the current setting in  \cite[Theorem~4.4]{fghp16}.


\begin{theorem}\label{thm:abstract bem2}
Let $(\kv_k)_{k\in\N_0}$ be the sequence of knots generated in Algorithm~\ref{alg:1D algorithm}.
Then, there hold:
\begin{enumerate}[\rm (i)]
\item\label{item:qabstract reliable bem2}
The residual error estimator satisfies reliability, i.e., there exists a constant $\const{rel}>0$ such that
\begin{align*}
 \norm{\phi-\Phi_\coarse}{H^{-1/2}(\Gamma)}\le \const{rel}\eta_\coarse\quad\text{for all }\kv_\coarse\in\mathbb{T}.
\end{align*}
\item\label{item:qabstract linear convergence bem2}
For arbitrary $0<\theta\le1$ and $\const{min}\in[1,\infty]$, the residual error estimator converges linearly, i.e., there exist constants $0<\ro{lin}<1$ and $\const{lin}\ge1$ such that
\begin{align*}
\eta_{k+j}^2\le \const{lin}\ro{lin}^j\eta_k^2\quad\text{for all }j,k\in\N_0.
\end{align*}
\item\label{item:qabstract optimal convergence bem2}
There exists a constant $0<\theta_{\rm opt}\le1$ such that for all $0<\theta<\theta_{\rm opt}$ and $\const{min}\in[1,\infty)$, the estimator converges at optimal rate, i.e., for all $s>0$ there exist constants $c_{\rm opt},\const{opt}>0$ such that
\begin{align*} 
 c_{\rm opt} \widetilde C_{\rm apx}(s)
 \le \sup_{k\in\N_0}{(\# T_k)^{s}}\,{\eta_k}
 \le \const{opt} \widetilde C_{\rm apx}(s).
\end{align*}
\end{enumerate}
\noindent All involved constants $\const{rel},\const{lin},\ro{lin},\theta_{\rm opt}$, and $\const{opt}$ depend only on the coefficients of the differential operator $\mathscr{P}$,  the parametrization ${\bf F}$, the polynomial order $p$, and the initial knot vector $T_0$,  while $\const{lin},\ro{lin}$ depend additionally on $\theta$ and the sequence $(\Phi_k)_{k\in\N_0}$, and $\const{opt}$ depends furthermore on $\const{min}$ and $s>0$.
\end{theorem}

\begin{remark}\label{rem:standard 1D BEM}
If one uses the original Algorithm~\ref{alg:abstract algorithm} with the refinement strategy Algorithm~\ref{alg:refinement bem21} (which does not  use knot multiplicity increase) and the element-based residual error estimator $\eta_k$ of~\eqref{eq:eta bem}, the abstract framework of Section~\ref{sec:abem} is directly applicable, see \cite{fghp16} for details.  
In particular, Theorem~\ref{thm:abstract bem} applies and guarantees linear convergence of the estimator at optimal algebraic rate.
Recently, \cite[Section~A.5]{gps19} has even proved the important result that $\const{apx}(s)\simeq \tilde C_{\rm apx}(s)$ for all $s>0$, where $\const{apx}(s)$ is the approximation class of the adaptive method without smoothness control defined analogously to \eqref{eq:const apx}.
This particularly yields that the asymptotic approximation behavior of smooth splines and piecewise polynomials coincides at least in the simple case of 2D IGABEM.
The numerical example of Section~\ref{sec:numerical igafem} for 2D IGAFEM suggests that this is in general not the case for 3D IGABEM due to the possible presence of edge singularities.
\end{remark}

\begin{remark}
The adaptive algorithm introduced in \cite{gps19} allows for both multiplicity increase and decrease.
The latter converges as well at optimal algebraic rate and practically yields an even more accurate insight of the smoothness of the exact solution. 
As the algorithm is quite technical and again restricted to the 2D case, we refer to \cite{gps19} for  details.
\end{remark}

\subsubsection{Numerical experiment}
\label{sec:1D numerics}

In this section, we empirically investigate the performance of  the Algorithm~\ref{alg:1D algorithm} for 
a Laplace--Dirichlet problem
\begin{align}\label{eq:Laplace bem2}
\begin{split}
-\Delta u&=0\quad\text{in }{\Omega},\\ u&=g\quad\text{on } \Gamma,
\end{split}
\end{align}
for given Dirichlet data $g\in {H}^{1}(\Gamma_{})$, where the additional regularity $H^1(\Gamma)$ instead of $H^{1/2}(\Gamma)$ is only needed for the weighted-residual error estimator.
The following example has also been considered in~\cite{fgp15,fghp16,gantner17,gps19}.
In the latter works,  several further examples are found, where \cite{gps19} also studies the hypersingular integral equation arising from Neumann problems.
We choose 
\begin{align*}
\Omega:=\big\{&(r \cos(\varphi),r \sin(\varphi)):
\\
& r\in(0,{1}/{4})\wedge \varphi \in \left(-{\pi}/{2\alpha},{\pi}/{2\alpha}\right)\big\}
\end{align*}
with $\alpha:=4/7$, see Figure \ref{fig:pacman}. 
A parametrization ${\bf F}$ of its boundary $\Gamma$ in terms of rational splines of degree $p=2$ is given, e.g., in \cite[Section~5.3]{fgp15}. 
We prescribe the exact solution of \eqref{eq:Laplace bem2} in polar coordinates $(r,\varphi)$ by
\begin{equation*}
u(x,y):=r^{\alpha}\cos\left(\alpha\varphi\right)\; \text{with }(x,y)=(r\cos(\varphi),r\sin(\varphi)).
\end{equation*}
The fundamental solution of $-\Delta$ is given by 
\begin{align*}
G({\bf z}):=-\frac{1}{2\pi}\log|{\bf z}|\quad\text{for all } {\bf z}\in \R^2 \setminus\{0\}.
\end{align*}
Since $\diam(\Omega)<1$, the corresponding single-layer operator $\mathscr{V}$ is elliptic, see Section~\ref{sec:model problem bem}. 
As in Section~\ref{sec:model problem bem}, \eqref{eq:Laplace bem2} can be equivalently rewritten as integral equation~\eqref{eq:Symmy interior}, i.e., $\mathscr{V}\phi =(\mathscr{K}+1/2) g$, 
where the unique solution is the normal derivative $\phi:=\partial_{\mathbf{{\boldsymbol{\nu}}}} u$ of the  weak solution $u$ of  \eqref{eq:Laplace bem2}. 
For our problem, $\phi$ has a singularity at ${\bf F}(1/2)$ and jumps  at ${\bf F}(1/3)$ and ${\bf F}(2/3)$.

\begin{figure}[h!]
\begin{center}
\includegraphics[width=0.35\textwidth]{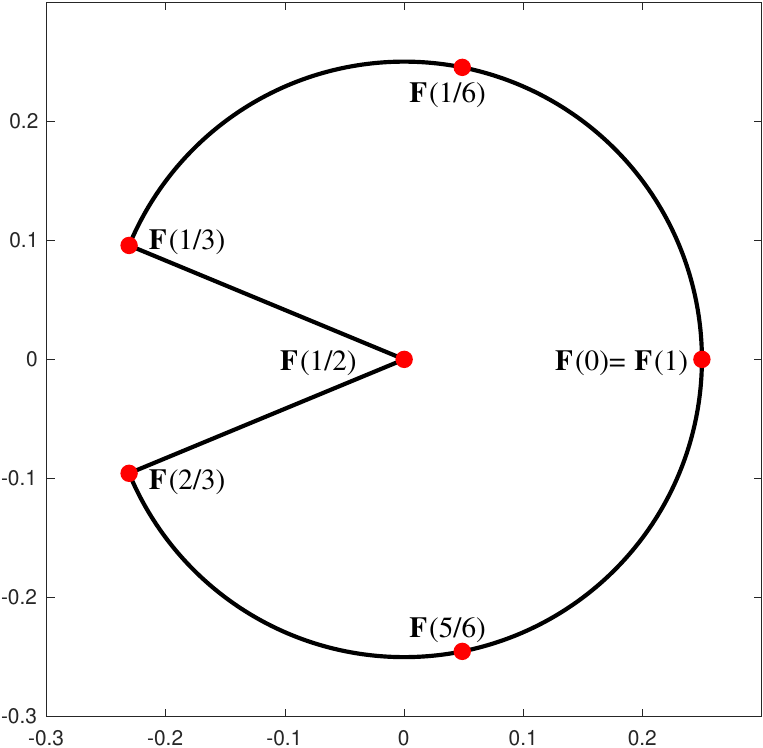}
\end{center}
\caption{Geometry and initial vertices for the experiment of Section~\ref{sec:1D numerics}.} 
\label{fig:pacman}
\end{figure}

To (approximately) calculate the Galerkin matrix, the right-hand side vector, and the weighted-residual error  estimator~\eqref{eq:eta bem2}, we transform the singular integrands into a sum of a smooth part and a logarithmically singular part.
Then, we use adapted Gaussian quadrature to compute the resulting integrals with appropriate accuracy, 
 see \cite[Section 5]{gantner14} for details.
 Moreover, to ease computation, we replace at each step of the adaptive algorithm the term $|\pi_k({\bf z})|$ in the error indicators $\eta_k({\bf z})=\norm{|\pi_k({\bf z})|^{1/2} \nabla_\Gamma(f-\mathscr{V}\Phi_k)}{L^2(\pi_k({\bf z}))}$  by the equivalent term $\diam(\Gamma)\,\widehat h_k$, 
where, $\widehat h_k\in L^\infty(\Gamma)$ denotes the mesh-width function with $\widehat h_k|_{Q}=|{\bf F}^{-1}(Q)|$ for all $Q\in\QQ_k$. 
The error in the energy norm is computed again via Aitken's $\Delta^2$-extrapolation and~\eqref{eq:error calc gal bem1}.

We choose the parameters of the modified Algorithm~\ref{alg:1D algorithm} as $\theta=0.75$ and $\const{min}=1$.
For comparison, we also consider uniform refinement, where we mark all vertices at each step, i.e., $\MM_k=\VV_k$ for all $k\in\N_0$. 
Note that this leads to uniform bisection (without knot multiplicity increase) of all elements.
Given the knot vector defining the parametrization ${\bf F}$, 
\begin{align*}
\kv_{\bf F}=\Big(0,0,0,\frac16,\frac16,\frac26,\frac26,\frac36,\frac36,\frac46,\frac46,\frac56,\frac56,1,1,1\Big),
\end{align*}
we consider splines of degree $p\in\{0,1,2,3\}$ such that at the breakpoints the initial space $\widehat{\mathbb{S}}_0$ is (if possible) as smooth as the space used to construct the parametrization. 
That is, the space is $C^0$ for $p \ge 1$, with the knots repeated exactly $p$ times, and $C^{-1}$ for $p=0$. Note that ${\bf F}$ is continuous but not necessarily differentiable at the breakpoints.

In  Figure~\ref{fig:pacman_p} and Figure~\ref{fig:pacman_pcomp}, we plot   the approximated energy error $\norm{\phi-\Phi_k}{\mathscr{V}}$ and the error estimator $\eta_k$ against the number of degrees of freedom.
Since the solution lacks regularity, uniform refinement leads to the suboptimal rate $\mathcal{O}(N^{-4/7})$ for the energy error, whereas adaptive refinement leads to the optimal rate $\mathcal{O}(N^{-3/2-p})$, see \cite[Corollary~4.1.34]{ss11}.
For adaptive refinement, Figure \ref{fig:pacman knots} provides  a histogram of the knots in the parametric domain $[0,1]$ of the last refinement step.
We observe that at $1/2$, where the singularity occurs, mainly $h$-refinement is used.
Instead, at the two jump points $1/3$ and $2/3$, the adaptive algorithm just increases the multiplicity of the corresponding knots to its maximum allowing for discontinuous ansatz functions.

\begin{figure}
\begin{center}
 \subfigure[Error and estimator for $p=0$]{ 
\includegraphics[width=0.35\textwidth,clip=true]{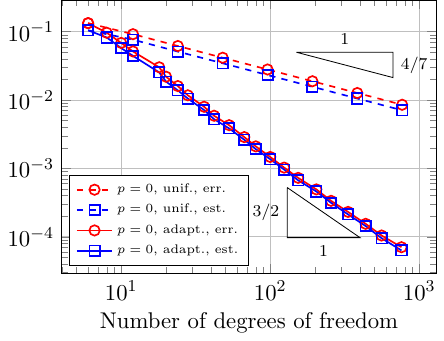}}
\subfigure[Error and estimator for $p=1$]{ 
\includegraphics[width=0.35\textwidth,clip=true]{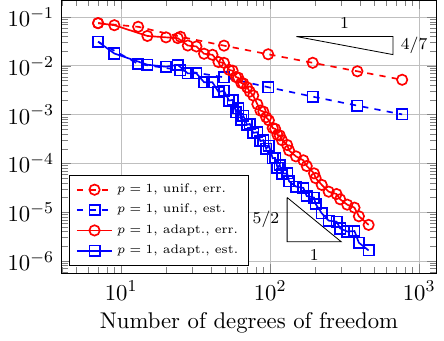}}
\subfigure[Error and estimator for $p=2$]{ 
\includegraphics[width=0.35\textwidth,clip=true]{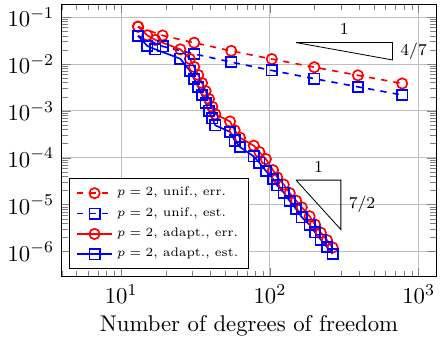}}
\subfigure[Error and estimator for $p=3$]{ 
\includegraphics[width=0.35\textwidth,clip=true]{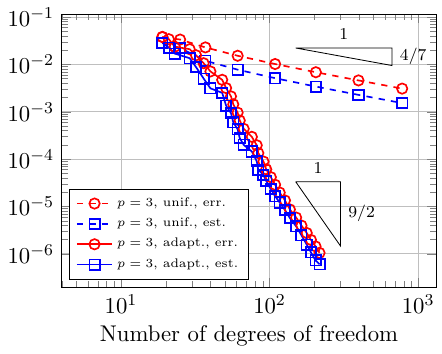}}
\end{center}

\caption{Singularity on pacman:
Energy error $\norm{ \phi- \Phi_k}{\mathscr{V}}$ and estimator $\eta_k$ of Algorithm~\ref{alg:1D algorithm} for splines of degree $p$ are plotted versus the number of degrees of freedom.
Uniform  and adaptive ($\theta=0.75$) refinement is considered.}
\label{fig:pacman_p} 
\end{figure}

\begin{figure}

\centering 
 
\includegraphics[width=0.35\textwidth]{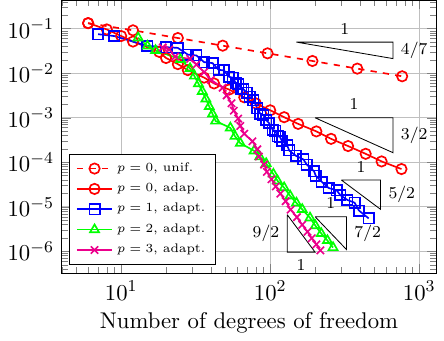}

\caption{Singularity on pacman:
The energy errors $\norm{\phi- \Phi_k}{\mathscr{V}}$  of Algorithm~\ref{alg:1D algorithm} for splines of degree $p\in\{0,1,2,3\}$ are plotted versus the number of degrees of freedom.
Uniform (for $p=0$) and adaptive ($\theta=0.75$ for $p\in\{0,1,2,3\}$) refinement is considered.}
\label{fig:pacman_pcomp} 
\end{figure}

\begin{figure}[h!]
\psfrag{parametric domain}[c][c]{\small parametric domain}
\psfrag{Cars}[l][l]{\tiny $\eta$}
\psfrag{Faer}[l][l]{\tiny $\mu$}
\psfrag{error}[l][l]{\tiny  error}
\psfrag{Gal1est2}[c][c]{\tiny Galerkin with $\rho=\mu$}
\psfrag{Gal0est2}[c][c]{}
\psfrag{Gal1est1}[c][c]{\tiny Galerkin with $\rho=\eta$}
\psfrag{Gal0est1}[c][c]{\tiny collocation with $\rho=\eta$}
\begin{center}
\includegraphics[width=0.4\textwidth]{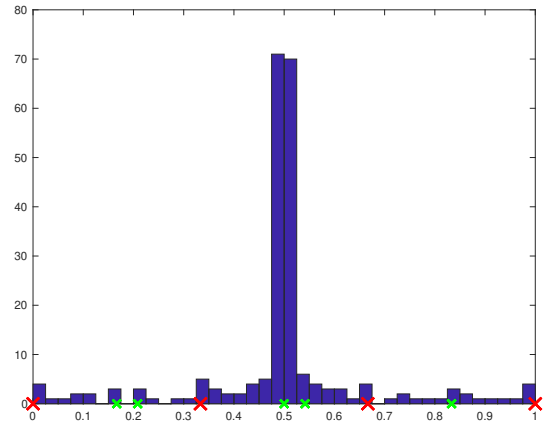}
\end{center}
\caption{Singularity on pacman: 
Histogram of number of knots over the parametric domain for the knot vector $\kv_{29}$ generated in Algorithm~\ref{alg:1D algorithm} (with $\theta=0.75$) for  splines of degree $p=3$. Knots with maximal multiplicity $p+1=4$ are marked with a red cross and knots with multiplicity $3$ are marked with a green smaller cross.}
\label{fig:pacman knots}
\end{figure}

\newpage


\section{Conclusion and open questions}
\label{sec:conclusion}


This work aims to give a state-of-the-art introduction to the numerical analysis of adaptive FEM and BEM in the framework of IGA. The first sections (Section~\ref{sec:splines} and~\ref{sec:model}) introduce the concepts and notation of IGAFEM and IGABEM without adaptivity. Then, Section~\ref{sec:adaptive-splines} gives the description and mathematical properties of two of the most popular \emph{adaptive spline constructions} considered in the recent years, namely (T)HB-splines and T-splines.  


Section~\ref{sec:abstract} provides a brief introduction into the so-called \emph{axioms of adaptivity}~\cite{cfpp14} and the concept of rate-optimal adaptive algorithms. It further provides a framework for finite element and boundary element discretizations, respectively, that guarantees the validity of the axioms of adaptivity. Leveraging on the properties for splines on adaptive meshes summarized in Section~\ref{sec:adaptive-splines}, we prove that  (T)HB-splines on certain admissible hierarchical meshes and T-splines on suitable admissible meshes with alternating directions of refinement fit into this framework. This is verified in Section~\ref{sec:adaptive igafem} for IGAFEM and in Section~\ref{sec:igabem} for IGABEM. It should be noted that the local tensor-product structure of hierarchical splines not only enables the possibility of easily constructing (analysis-suitable) bases but it also simplifies the theoretical analysis of adaptive isogeometric methods. On the other hand, T-splines and T-meshes are more flexible and suited for applications, but restricted mesh configurations are needed for the development of their theoretical analysis and the involved results are more complicated in nature. As a consequence, adaptive isogeometric methods based on (T)HB-splines appear to be the IGA framework most advanced in terms of numerical analysis, while T-splines still appears to be most used in the engineering literature).

Even though adaptive IGA is a ra\-pid\-ly developing research field, many important questions remain open: 

First, we have verified that the abstract properties in Section~\ref{sec:abstract} are satisfied for isogeometric discretizations with (T)HB-splines and T-splines. For instance, it remains open whether these mesh and space properties are also satisfied for other adaptive spline constructions, as the ones briefly mentioned in Section~\ref{subsec:others}, including for example LR-splines, or the different definitions of linearly independent T-splines from Section~\ref{subsec:extend tsplines}, which reduce the strong constraints posed by the dual-compatibility condition. We note that the mathematical study of adaptive methods based on these constructions is at different stages, being probably most advanced for LR-splines. As long as suitable refinement algorithms and interpolation estimates for a given adaptive spline construction are available, the abstract framework presented in this work can be properly exploited to study the resulting adaptive isogeometric method.

Second, it has not been mathematically studied yet how the approximation classes of the PDE solution and thus the resulting convergence rates of the adaptive algorithm depend on the employed adaptive splines. 
In particular, their relation to the classes and rates of standard (only continuous) finite element spaces is theoretically open. 
Our numerical experiments of Section~\ref{sec:numerical igafem} suggest that they might especially depend on the smoothness of the splines. 
A verifiable characterization in terms of the given data and the corresponding PDE solution would be desirable.

Third, the current analysis is implicitly tailored to isotropic meshes through the closure estimate~\eqref{R:closure} and the overlay estimate~\eqref{R:overlay} in Section~\ref{sec:abstract main}. Available proofs of~\eqref{R:closure} (even for standard FEM and BEM) use the relation $\diam(Q)^{\widehat d} \simeq |Q|$ of element diameter and element area and hence mathematically exclude long and thin aniso\-tropic elements, see, e.g.,~\cite{MR2353951} for the seminal work which is transferred to IGA in~\cite{mp15,bgmp16,morgenstern17,ghp17}. However, it is known that only point singularities can optimally be resolved by isotropic elements, while anisotropic elements are mandatory to resolve edge singularities, both in 2D and 3D computations. 
Optimal adaptivity with anisotropic elements is not only theoretically completely open, but also the stable implementation (in particular for BEM) is highly non-trivial.



Finally, the analysis presented for multi-patch domains has to be extended to more general configurations. In particular, for HB-splines we are assuming that there are no hanging nodes on the interface between patches, and the continuity is set to $C^0$. While the first assumption can probably be removed without major issues, as we explained in Remark~\ref{rem:second multi patch}, the construction of hierarchical splines with $C^1$ continuity in general multi-patch geometries remains an open question. For T-splines instead, we have only presented results for BEM by assuming discontinuous functions across patches, which is very restrictive with respect to the standard setting used in the CAD and engineering literature, based on bicubic T-spline surfaces of $C^2$ continuity everywhere except in the vicinity of extraordinary points, i.e., points at the intersection of a number of patches different from four.

\newpage

\begin{acknowledgements}
The authors would like to thank Cesare Bracco and Durkbin Cho for several fruitful discussions on the subject of the paper.
\end{acknowledgements}

\bibliographystyle{spmpsci}      
\bibliography{literature/literature}   
\end{document}